\documentclass[a4paper,12pt,reqno]{amsart}
\usepackage[english]{babel}
\usepackage{amsmath}
\usepackage{amscd}
\usepackage{amsgen}
\usepackage{amsfonts}
\usepackage{amssymb}
\usepackage{amsthm}
\usepackage{latexsym}
\usepackage{url}
\usepackage{slashed}
\usepackage{stmaryrd}
\usepackage[all]{xy}
\usepackage{mathrsfs}
\usepackage{tocvsec2}

\allowdisplaybreaks

\def\A{{\mathcal A}}
\def\B{{\mathcal B}}       
\def\Co{{\mathcal C}}      
\def\D{{\mathcal D}}
\def\F{{\mathcal F}}
\def\H{{\mathcal H}}       
\def\K{{\mathcal K}}       
\def\L{{\mathscr L}}
\def\M{{\mathcal M}}       
\def\R{{\mathcal R}}       

\def\T{{\mathcal T}}
\def\V{{\mathcal V}}
\def\VF{{\mathscr V}}      
\def\W{{\mathscr W}}       
\def\dvol{dv}              
\def\Rt{{\tilde{R}}}       

\def\J{{\sf J}}            
\def\Rho{{\sf P}}          

\def\U{{\mathcal U}}       
\def\X{{\mathfrak X}}
\def\C{{\mathscr C}}       

\def\r{{\mathbb{R}}}       
\def\c{{\mathbb{C}}}       
\def\N{{\mathbb{N}}}       
\def\h{{\mathbb H}}        
\def\s{{\mathbb S}}        

\def\f{\frac{n}{2}}
\def\fr{\frac{n-1}{2}}
\def\st{\stackrel{\text{def}}{=}}

\newcommand{\tr}{\operatorname{tr}}
\newcommand{\Tr}{\operatorname{Tr}}
\newcommand{\id}{\operatorname{id}}

\newcommand{\Ric}{\operatorname{Ric}}

\newcommand{\scal}{\operatorname{scal}}
\newcommand{\res}{\operatorname{res}}
\newcommand{\Hess}{\operatorname{Hess}}
\newcommand{\SD}{\slashed{\nabla}}
\newcommand{\FP}{\operatorname{FP}}

\newtheorem{thm}{Theorem}[section]
\newtheorem{lemm}{Lemma}[section]
\newtheorem{corr}{Corollary}[section]
\newtheorem{defn}{Definition}[section]

\newtheorem{prop}{Proposition}[section]
\newtheorem{rem}{Remark}[section]

\newtheorem{conj}{Conjecture}[section]

\numberwithin{equation}{section} \numberwithin{table}{subsection}

\title[Heat kernels, ambient metrics and conformal invariants]
{Heat kernel expansions, ambient metrics and conformal invariants}

\author{Andreas Juhl}

\address{University Uppsala, Department of Mathematics,
P.O. Box 480, S-75106 Uppsala}
\email{andreasj@math.uu.se}

\address{Humboldt-Universit\"at, Institut f\"ur Mathematik,
Unter den Linden, D-10099 Berlin}
\email{ajuhl@math.hu-berlin.de}

\begin{document}

\begin{abstract} The conformal powers of the Laplacian of a Riemannian
metric which are known as the GJMS-operators admit a combinatorial
description in terms of the Taylor coefficients of a natural second-order
one-parameter family $\H(r;g)$ of self-adjoint elliptic differential
operators. $\H(r;g)$ is a non-Laplace-type perturbation of the conformal
Laplacian $P_2(g) = \H(0;g)$. It is defined in terms of the metric $g$ and
covariant derivatives of the curvature of $g$. We study the heat kernel
coefficients $a_{2k}(r;g)$ of $\H(r;g)$ on closed manifolds. We prove
general structural results for the heat kernel coefficients $a_{2k}(r;g)$
and derive explicit formulas for $a_0(r)$ and $a_2(r)$ in terms of
renormalized volume coefficients. The Taylor coefficients of $a_{2k}(r;g)$
(as functions of $r$) interpolate between the renormalized volume
coefficients of a metric $g$ ($k=0$) and the heat kernel coefficients of
the conformal Laplacian of $g$ ($r=0$). Although $\H(r;g)$ is not
conformally covariant, there is a beautiful formula for the conformal
variation of the trace of its heat kernel. Its proof rests on the
combinatorial relations between Taylor coefficients of $\H(r;g)$ and
GJMS-operators. As a consequence, we give a heat equation proof of the
conformal transformation law of the integrated renormalized volume
coefficients. By refining these arguments, we also give a heat equation
proof of the conformal transformation law of the renormalized volume
coefficients itself. The Taylor coefficients of $a_2(r)$ define a sequence
of higher-order Riemannian curvature functionals with extremal properties
at Einstein metrics which are analogous to those of integrated
renormalized volume coefficients. Among the various additional results the
reader finds a Polyakov-type formula for the renormalized volume of a
Poincar\'e-Einstein metric in terms of $Q$-curvature of its conformal
infinity and additional holographic terms.

The present study of relations between spectral theoretic heat kernel
coefficients and geometric quantities like renormalized volume
coefficients is stimulated by the holographic perspective of the
AdS/CFT-duality (proposing relations between functional determinants and
renormalized volumes).
\end{abstract}


\maketitle

\centerline \today

\footnotetext{The work was supported by the University of Uppsala and the
Collaborative Research Centre 647 ``Space-Time-Matter'' of the German
Research Foundation at Humboldt-University Berlin.}

\tableofcontents

\section{Introduction and formulation of the main results}\label{intro}

Since the pioneering work \cite{MP}, enormous progress has been made in
understanding the heat kernel of elliptic self-adjoint differential
operators on manifolds. In this context, the so-called Laplace-type
operators are best understood. The small time asymptotic expansion of the
restriction of the heat kernel to the diagonal defines the so-called
(local) heat kernel coefficients. For geometric operators, these are of
fundamental interest since they bear geometric information. In fact, heat
kernel coefficients of Laplace-type operators play a central role in one
approach to the index theorem (\cite{G-book}, \cite{BGV}) and the heat
kernel coefficients of the Laplace-Beltrami operator and the
Hodge-Laplacian on forms contain basic information on the relation between
spectral theory and geometry (\cite{Ch}). In addition, heat kernel
coefficients of various operators are of interest in quantum field theory
(see \cite{BD}, \cite{A-book} and the review \cite{V}).

An obstacle in analyzing the properties of heat kernel coefficients is
their quickly growing complexity (as a function of their order). However,
at least in low-order cases, there are well-known explicit formulas in
terms of local invariants (see \cite{BGM}, \cite{MC-S}, \cite{G-spec} and
\cite{G-book}). Note also that Weingart \cite{Wein} found very remarkable
closed combinatorial formulas of heat kernel coefficients of arbitrary
order.

The asymptotic expansion of the heat kernel is closely related to the
notion of the (zeta-regularized) functional determinant. For a
self-adjoint Laplace-type operator on a closed manifold, it gives the
infinite product of its (non-zero) real eigenvalues a sense. The notion
was introduced by Ray and Singer \cite{RS} in order to define an analytic
analog of Reidemeister torsion. In addition, standard constructions in
quantum field theory \cite{H} lead to the consideration of functional
determinants of differential operators (acting on various types of
fields). Of particular interest are operators which are determined by a
background metric and which are covariant with respect to conformal
changes of the metric. Although the functional determinant of such
operators in general is invariant under conformal changes, there are
remarkable instances where the violation of its conformal invariance can
be described in terms of {\em local} quantities. For instance, let
$\Delta_g = - \delta_g d$ be the non-positive Laplacian of $g$\footnote{We
use the convention that $-\Delta = \delta d$ is non-negative. In
particular, the Laplacian $\Delta$ of the Euclidean space $\r^n$ is the
sum of the second-order partial derivatives.} The transformation law
\begin{equation}\label{lap-2}
e^{2\varphi} \Delta_{\hat{g}} = \Delta_g, \quad \hat{g} =
e^{2\varphi}g, \; \varphi \in C^\infty(M)
\end{equation}
for the Laplacian of a closed surface $(M^2,g)$ implies that the behaviour
of its functional determinant $\det(-\Delta_g)$ under conformal changes $g
\to \hat{g}$ of the metric is governed by the Polyakov formula
\begin{equation}\label{pol-det-2}
\log \left(\frac{\det(-\Delta_{\hat{g}})}{\det(-\Delta_g)}\right) = -
\frac{1}{4\pi} \int_M \varphi (a_2(g) \dvol_g + a_2(\hat{g})
\dvol_{\hat{g}})
\end{equation}
(if $g$ and $\hat{g}$ have the same volume).\footnote{For the definition
of the determinant see Section \ref{CV-zeta}.} Here the quantity $a_2(g) =
\scal(g)/6$ is the second heat kernel coefficient of the operator
$\Delta_g$. By the transformation formula
$$
e^{2\varphi} K(e^{2\varphi}g) = K(g) - \Delta_g (\varphi)
$$
for the Gauss curvature $K(g) = \scal(g)/2$, formula
\eqref{pol-det-2} is equivalent to the more familiar version
\cite{OPS}
$$
\log \left(\frac{\det(-\Delta_{\hat{g}})}{\det(-\Delta_g)}\right) = -
\frac{1}{12 \pi} \int_M (2 \varphi K(g) + |d\varphi|_g^2) \dvol_g.
$$
Polyakov's formula is the key to the study of extremal properties of
the functional determinant of the Laplacian of surfaces \cite{OPS}.
For a mathematical discussion of some of the related constructions
of string theory see \cite{AJPS}.

In higher dimensions, \eqref{lap-2} generalizes to the conformal
covariance
$$
e^{(\f+1)\varphi} P_2(\hat{g}) = P_2(g) e^{(\f-1)\varphi}, \;
\hat{g} = e^{2\varphi}g, \; \varphi \in C^\infty(M)
$$
of the conformal Laplacian
\begin{equation}\label{CL-def}
P_2(g) \st \Delta_g - \left(\f\!-\!1\right) \J(g), \quad \J(g) =
\frac{\scal(g)}{2(n\!-\!1)}.
\end{equation}
On closed manifolds $(M^n,g)$ of even dimension $n$, an
infinitesimal version of \eqref{pol-det-2} for the functional
determinant of $P_2(g)$ states that\footnote{In order to simplify
the formulation, we assume that $P_2(g)$ has only positive
eigenvalues.}
\begin{equation}\label{pol-inf}
(d/d\varepsilon)|_{\varepsilon = 0} (\log \det
(-P_2(e^{2\varepsilon\varphi}g)) = -2 \int_M \varphi a_n(g) \dvol_g,
\end{equation}
where the conformal anomaly $a_n \in C^\infty(M^n)$ is the {\em critical}
heat kernel coefficient of the conformal Laplacian (see
\eqref{heat-expansion-CL}).\footnote{Here we refer to the coefficient of
$t^0$ in the heat kernel expansion as to the critical heat kernel
coefficient. Later we shall change the convention and denote the product
of this coefficient with the factor $(4\pi)^\f$ by $a_n$.} The derivation
of a full analog of Polyakov's formula \eqref{pol-det-2} then requires
finding a conformal primitive of the right-hand side of the latter
formula. This is a delicate matter of substantial interest (see
\cite{sharp}). It motivated Branson to introduce the notion of
$Q$-curvature $Q_n$, and experimental evidence in low-order cases led him
to suspect that in some sense the critical $Q$-curvature $Q_n$ plays the
role of a {\em main part} of the critical heat kernel coefficient
(universally also for other conformally covariant operators)
\cite{origin}. The functional $g \mapsto a_n(g)$ is a local Riemannian
invariant of weight $n$ with the property that its total integral on
closed manifolds is conformally invariant (\cite{BO-index}, \cite{PR}).
This is one of the basic consequences of the conformal covariance of $P_2$
for its heat kernel coefficients. Although this property already has the
strong implication that $a_n(g)$ is a linear combination of the Pfaffian
of $g$, a local conformal invariant and a divergence (by the
Deser-Schwimmer decomposition proved in \cite{alex}), the full structure
of its heat kernel coefficients is far from being well-understood. In
particular, the explicit form of the mentioned decomposition of the
critical heat kernel coefficients is not known, and the general principles
which may connect heat kernel coefficients of $P_2$ (and other conformally
covariant operators) and $Q$-curvatures have not yet been found. For a
brief introduction to the numerous aspects of $Q$-curvature see
\cite{what}.

In recent years, speculations on dalities between gauge field theories and
theories of gravity (as the so-called AdS/CFT-correspondence and versions
of it)\footnote{Although that subject has its own huge literature, we only
refer to Witten's seminal work \cite{witt}.} led to the introduction of
the notion of renormalized volume (see \cite{HS} and \cite{G-vol}). The
definition of the renormalized volume rests on the notion of
Poincar\'e-Einstein metrics in the sense of Fefferman and Graham
(\cite{FG-Cartan}, \cite{FG-final}). In some respects, the behaviour of
the renormalized volume functional resembles that of the functional
determinant. In fact, in even dimension its conformal variation is given
by an anomaly $v_n$. Like $a_n$, the critical renormalized volume
coefficient $g \mapsto v_n(g)$ is a scalar curvature invariant of weight
$n$ whose total integral on closed manifolds is a global conformal
invariant. Moreover, in odd dimensions, both the functional determinant of
the conformal Laplacian and the renormalized volume are conformally
invariant. The AdS/CFT-correspondence actually involves (in specific
cases) relations between (critical) heat kernel coefficients and
(critical) renormalized volume coefficients; one of these relations was
confirmed in \cite{HS} (for details see also Section 6.15 of
\cite{juhl-book}). For recent progress on the extremal properties of the
renormalized volume functional we refer to \cite{GMS} and \cite{CFG}.

These developments motivate the study of relations among the elements of
the three series of scalar local Riemannian invariants which (for even
$n$) are given by
\begin{itemize}
\item the heat kernel coefficients $a_2, a_4, \dots, a_n$ of the
conformal Laplacian,\footnote{Here one may also consider heat kernel
coefficients of other conformally covariant operators.}
\item the Branson $Q$-curvatures $Q_2, Q_4, \dots, Q_n$,
\item the renormalized volume coefficients $v_2, v_4, \dots, v_n$.
\end{itemize}

A remarkable common property of the corresponding integral functionals is
their conformal variational formula
\begin{equation}\label{common}
\left(\int_{M^n} F_{2k}(g) \dvol_g \right)^\bullet[\varphi] =
(n\!-\!2k) \int_{M^n} \varphi F_{2k}(g) \dvol_g
\end{equation}
on closed manifolds. Here and in the following we use the notation
$$
\F(g)^\bullet[\varphi] = (d/d\varepsilon)|_0 (\F(e^{2\varepsilon
\varphi}g))
$$
for the conformal variation of a scalar-valued functional $\F$. For the
respective proofs see \cite{BO-index}, \cite{PR} (for $F_{2k} = a_{2k}$),
\cite{sharp} (for $F_{2k} = Q_{2k}$) and \cite{CF}, \cite{G-vol} (for
$F_{2k} = v_{2k}$). Equation \eqref{common} implies that (for even $n$)
the total integral of $F_n$ is a global conformal invariant. Hence the
Deser-Schwimmer decomposition (proved in \cite{alex}) shows that each of
the quantities $a_n$, $Q_n$ and $v_n$ decomposes as a linear combination
of the Pfaffian, local conformal invariants and a divergence.

As explained above, the relation between heat kernel coefficients and
$Q$-curvatures is expected to be important for generalizations of
Polyakov's formula, and the relation of heat kernel coefficients to
renormalized volume coefficients plays a central role in tests of
gauge-gravity dualities.

On the other hand, recent years have seen substantial progress in
understanding the relations between $Q$-curvatures $Q_{2N}$ and
renormalized volume coefficients $v_{2N}$. First of all, there is a
natural decomposition of the critical $Q$-curvature $Q_n$ (in even
dimension) as a linear combination of $v_n$ and a sequence of differential
operators acting on lower-order renormalized volume coefficients
\cite{GJ-holo}. That result extends a result of \cite{GZ} and has natural
generalizations to all $Q$-curvatures \cite{holo-II}. Next, the sequence
of $Q$-curvatures turned out to have a recursive structure which allows to
determine any $Q$-curvature in terms of renormalized volume coefficients
and GJMS-operators acting on respective lower-order $Q$-curvatures
(\cite{Q-recursive}).

We continue with the description of the content of the present paper. The
central point is that we consider the heat kernel of the conformal
Laplacian from a new perspective. Here we develop only some of the aspects
of this point of view, and expect that further investigations will be
fruitful. The new perspective is inspired by the interactions of the
theories mentioned above. More precisely, we use Poincar\'e-Einstein
metrics to introduce a one-parameter family $\H(r;g)$ of self-adjoint
elliptic operators so that
$$
\H(0;g) = P_2(g).
$$
We regard $\H(r;g)$ as a perturbation of the conformal Laplacian. Through
the heat kernel coefficients of the family $\H(r;g)$ we embed the heat
kernel coefficients of $P_2$ into one-parameter families of local
Riemannian invariants. Although the operator $\H(r;g)$ for $r \ne 0$ is
not conformally covariant, its behaviour resembles that of a conformally
covariant one. An important observation is that in the theory of the heat
kernel of $\H(r;g)$ the variables $r$ and $t$ play similar roles. In
particular, that will allow us to give a heat equation proof of the
conformal variational properties of the renormalized volume coefficients.

We continue with the formulation of the main results. We start with the
definition of the holographic Laplacian $\H(r;g)$. On any Riemannian
manifold $(M,g)$ of dimension $n \ge 3$, there is a sequence $P_2(g),
P_4(g), \dots$ of higher-order generalizations of the conformal Laplacian
which are known as the GJMS-operators \cite{GJMS}. These operators are
conformally covariant in the sense that
$$
e^{(\f+N)\varphi} P_{2N}(e^{2\varphi} g) = P_{2N}(g)
e^{(\f-N)\varphi}
$$
for all $\varphi \in C^\infty(M)$, and can be viewed as corrections of
powers of the Laplacian $\Delta_g$ by lower-order terms which, as for
$P_2$, are determined by the covariant derivatives of the curvature of the
metric $g$. The second operator in the sequence is the well-known Paneitz
operator \cite{pan}\footnote{Paneitz discovered that operator in general
dimensions in 1983. Around the same time, this operator in dimension $n=4$
independently appeared in different contexts in \cite{FT}, \cite{Rieg} and
\cite{ES}.}
\begin{equation}\label{pan}
P_4 = \Delta^2 + \delta ((n\!-\!2) \J g - 4 \Rho) d +
\left(\f\!-\!2\right)\left(\f \J^2 - 2|\Rho|^2 - \Delta \J\right),
\end{equation}
where
$$
\Rho = \frac{1}{n\!-\!2}\left(\Ric - \frac{\scal}{2(n\!-\!1)}\right)
$$
is the Schouten tensor of $g$ and $\J$ its trace. The construction of
these operators in \cite{GJMS} works for all orders on manifolds of odd
dimension but is restricted to orders not exceeding the dimension of the
underlying manifold if the dimension is even.

For $N \ge 3$, the structure of these conformally covariant powers of the
Laplacian is quite complicated. However, in the recent work \cite{juhl-ex}
we revealed new structures in the sequence of these operators. In
particular, we found that any GJMS-operator $P_{2N}(g)$ can be written as
a linear combination
$$
P_{2N} = \sum_{|I|=N} n_I \M_{2I}
$$
of compositions of certain natural self-adjoint {\em building-block}
operators $\M_{2N}(g)$ with $\M_2 = P_2$ (see Section \ref{basics}). One
of the main results of \cite{juhl-ex} (see Theorem \ref{base}) is a
description of these building-block operators. It states that the
generating function
\begin{equation*}
\sum_{N \ge 1} \M_{2N}(g) \frac{1}{(N-1)!^2}
\left(\frac{r^2}{4}\right)^{N-1}
\end{equation*}
of the building-block operators coincides with the Schr\"{o}dinger-type
operator
\begin{equation}\label{sch}
\H(r;g) \st - \delta (g(r)^{-1} d) + \U(r;g)
\end{equation}
with the potential
\begin{equation}\label{potential}
\U(r;g) \st - w(r)^{-1} \left(\partial^2/\partial r^2 - (n-1) r^{-1}
\partial/\partial r - \delta (g(r)^{-1} d) \right)(w(r)).
\end{equation}
In particular, all building-block operators are only second-order. In
these definitions, $\delta: \Omega^1(M) \to C^\infty(M)$ denotes the
adjoint of $d: C^\infty(M) \to \Omega^1(M)$ with respect to the Hodge
scalar products defined by $g$; on vector fields it corresponds to the
(negative) divergence with respect to $g$. $g(r)$ is a smooth even
one-parameter family of metrics on $M$ so that $g(0)=g$ and
$$
g_+ = r^{-2}(dr^2 + g(r))
$$
is an asymptotic Einstein metric in the sense that the tensor
$$
\Ric(g_+) + n g_+
$$
vanishes asymptotically at $r=0$. In \eqref{sch} and \eqref{potential},
$g(r)$ is regarded as an endomorphism on one-forms using $g$. The metric
$g_+$ is a Poincar\'e-Einstein metric in the sense of \cite{FG-final}. In
odd dimension $n$, all coefficients in the formal {\em even} power series
expansion
$$
g(r) = g + r^2 g_{(2)} + r^4 g_{(4)} + \cdots + r^{n-1} g_{(n-1)} +
\cdots
$$
at $r=0$ are determined by $g$. In even dimension $n$, the situation
is more subtle. For general metrics $g$, the coefficients $g_{(2)},
\dots, g_{(n-2)}$ in the even expansion
$$
g(r) = g + r^2 g_{(2)} + \cdots + r^{n-2} g_{(n-2)} + \cdots
$$
are determined by $g$. However, there is an obstruction to finding a
solution to the order $n$. More precisely, the trace of the Taylor
coefficient of $r^n$ is still determined by $g$, but the determination of
the trace-free part is obstructed by the so-called obstruction tensor. For
some special metrics, more information on these expansions is available.
In particular, for conformally flat metrics in dimension $n \ge 4$, the
expansion of $g(r)$ terminates at $r^4$. Moreover, for conformally
Einstein metrics $g$ (in even dimension $n \ge 4$), the family $g(r)$ is
uniquely determined (although explicit formulas are not known except for
Einstein metrics). For more details we refer to Section \ref{basics}.

The function $w(r) \in C^\infty(M)$ in \eqref{potential} is defined
as the square root
\begin{equation}\label{square-root}
w(r) \st \sqrt{v(r)}
\end{equation}
of
\begin{equation}\label{volume}
v(r) \st \sqrt{\det g(r)}/\sqrt{\det g} \in C^\infty(M).
\end{equation}
It also plays an important role in structural formulas for
$Q$-curvatures (see \eqref{Q-rec}).

The power series coefficients $v_{2k}$ of $v(r)$ are the renormalized
volume coefficients. For general metrics, they are well-defined for all $k
\ge 0$ in odd dimensions and for $2k \le n$ in even
dimension.\footnote{$g_{(n)}$ contributes to $v_n$ only through its
uniquely determined $g$-trace.} However, for locally conformally flat and
conformally Einstein metrics, all coefficients $v_{2k}$ are well-defined
also in even dimension $n \ge 4$. In all cases, the coefficients $v_{2k}$
are scalar Riemannian invariants of $g$ of weight $2k$.

Now, for general metrics $g$, we define the operator $\H(r;g)$ by the
formula \eqref{sch} using an even Poincar\'e-Einstein metric relative to
$g$. Then the second-order and potential parts of the Taylor coefficient
of $r^{2N}$ of $\H(r;g)$ involve the tensors $g_{(2j)}$ up to $g_{(2N)}$
and $g_{(2N+2)}$, respectively. Here the the potential part depends on the
coefficient $g_{(2N+2)}$ only through $v_{2N+2}$. In particular, for even
$n$, the Taylor coefficients of $\H(r;g)$ up to $r^{n-2}$ are uniquely
determined by the metric $g$. In the discussions below, the parts of
$\H(r;g)$ which are not uniquely determined by $g$ will play no role.
Similarly, for odd $n$, the full formal Taylor expansion of $\H(r;g)$ at
$r=0$ is uniquely determined by $g$. For more details we refer to Section
\ref{basics}.

We regard the second-order operator $\H(r;g)$ as a {\em canonical
geometric} differential operator defined by a Riemannian metric. Since its
definition involves concepts on a space of one higher dimension including
the concept of renormalized volume (or holographic) coefficients
$v_{2k}(g)$, we refer to it as to the {\em holographic} Laplacian of $g$.
We emphasize, that $\H(r;g)$ is not a Laplace-type operator or generalized
Laplacian (with respect to $g$) in the sense of \cite{G-book} and
\cite{BGV}, respectively.

In these terms, the main result of \cite{juhl-ex} can be rephrased by
saying that the holographic Laplacian {\em coincides} with the generating
function of the building-block operators of the GJMS-operators.

In the first section of Section \ref{app}, we display explicit formulas
for the first three terms in the power series expansion of the holographic
Laplacian (as a function of $r^2$) regarded as a perturbation of the
conformal Laplacian $P_2$. In particular, the first-order perturbation is
given by the operator
$$
P_2 + \M_4 \frac{r^2}{4} = P_2 - \delta (\Rho d) r^2 +
(-\J^2\!-\!(n\!-\!4)|\Rho|^2\!+\!\Delta \J) \frac{r^2}{4}.
$$
We also note that, for an Einstein metric $g$, the operator
$\H(r;g)$ takes the particularly simple form
\begin{equation}\label{hol-ein}
(1-\lambda r^2/4)^{-2} P_2(g), \; \lambda = \scal(g)/n(n\!-\!1)
\end{equation}
(see Section 11.2 of \cite{juhl-ex}).

In \cite{FG-J}, Fefferman and Graham revealed the relation of the operator
\eqref{sch} to the Laplacian of the ambient metric and used it for an
alternative proof of the results in \cite{juhl-ex}. For more details we
refer to Section \ref{basics}.

The building-block operators $\M_{2N}(g)$ are natural in $g$ but not
covariant under conformal changes of the metric (except for $N=1$).
However, for constant conformal changes $g \mapsto \lambda^2 g$, i.e., for
rescalings of $g$, they obey the transformation law
$$
\lambda^{2N} \M_{2N}(\lambda^2 g) = \M_{2N}(g), \lambda \ne 0.
$$
These relations imply
\begin{equation}\label{CCC}
\lambda^2 \H(r;\lambda^2 g) = \H(r/\lambda;g), \; \lambda \ne 0.
\end{equation}
Under general conformal changes, the behaviour of the building block
operators is much more complicated. But it is remarkable that the
conformal variation of any $\M_{2N}$ can be described only in terms of
respective lower-order building-block operators. The following result is a
consequence of such a description. It is the first main result of the
present paper and constitutes the basis of the whole work.

\begin{thm}[\bf Conformal variation of $\H(r;g)$]\label{conform-H} For
any metric $g$ and any $\varphi \in C^\infty(M)$, we have
\begin{multline}\label{base-H}
(\partial/\partial t)|_{t=0} \left( e^{(\f+1)t\varphi} \H(r;e^{2t
\varphi} g) e^{-(\f-1)t\varphi} \right) \\ = - \frac{1}{2} r
(\partial/\partial r)(\varphi \H(r;g) + \H(r;g) \varphi) -
[\H(r;g),[\K(r;g),\varphi]]
\end{multline}
for sufficiently small $r$. Here
$$
\K(r;g) \st \sum_{N \ge 1} \M_{2N}(g) \frac{1}{N!(N\!-\!1)!}
\left(\frac{r^2}{4}\right)^N = \frac{1}{2} \int_0^r s \H(s;g) ds.
$$
On the right-hand side of \eqref{base-H}, $\varphi$ is regarded as a
multiplication operator.
\end{thm}

Here we use the following conventions. For odd $n \ge 3$, the
identity \eqref{base-H} is to be interpreted as a relation of formal
power series in $r$. For even $n \ge 4$ and general metrics, it
asserts the equality of the power series expansions of both sides of
\eqref{base-H} up to $r^{n-2}$. However, for locally conformally
flat metrics and conformally Einstein metrics in even dimension $n
\ge 4$, the equality of formal power series is valid without
restrictions on the powers of $r^2$. Similar interpretations will be
used in what follows without mentioning.

Theorem \ref{conform-H} obviously generalizes the infinitesimal
version
$$
(d/dt)|_{t=0} (e^{2\lambda t} \H(r;e^{2\lambda t}g)) = (d/dt)|_{t=0}
(\H(r e^{-\lambda t};g)), \; \lambda \in \r
$$
of \eqref{CCC}. The restriction of the right-hand side of
\eqref{base-H} to $r=0$ vanishes. Thus, in view of $\M_2 = P_2$,
Equation \eqref{base-H} also generalizes the infinitesimal conformal
covariance of the conformal Laplacian. In particular, the conformal
weights $\f+1$ and $\f-1$ in \eqref{base-H} are the same as for the
conformal Laplacian.

Now we consider spectral invariants for the holographic Laplacian.
For this purpose, we assume from now on that the underlying manifold
$M$ is closed. For sufficiently small $r$, $g(r)$ is a Riemannian
metric and the holographic Laplacian is elliptic.

The following result is a conformal variational formula for the
trace of the heat kernel of the (negative) holographic Laplacian. It
is a consequence of the variational formula in Theorem
\ref{conform-H}.

\begin{thm}[\bf Conformal variation of trace of heat kernel of $\H(r;g)$]
\label{conform-heat} For any metric $g$ on a closed manifold $M$, we
have
\begin{multline}\label{base-heat}
(\partial/\partial\varepsilon)|_{\varepsilon=0}
(\Tr(\exp(t\H(r;e^{2\varepsilon \varphi}g)))) \\ = -2t
(\partial/\partial t)(\Tr(\varphi e^{t\H(r;g)})) - r t \frac{1}{2}
\Tr \left((\varphi \dot{\H}(r;g) + \dot{\H}(r;g) \varphi)
e^{t\H(r;g)}\right)
\end{multline}
for sufficiently small $r$ and all $\varphi \in C^\infty(M)$. Here
traces are taken in $L^2(M,g)$ and the dot denotes the derivative
with respect to $r$.
\end{thm}

Two comments on this result are in order. First, the relation
\eqref{base-heat} is to be interpreted as an identity of formal power
series in $r$ (up to $r^{n-2}$ for even $n$ and general metrics). Second,
it is one of the basic features of Theorem \ref{conform-heat} that the
massive double-commutator term on the right-hand side of \eqref{base-H}
does {\em not} contribute to \eqref{base-heat}. This follows from the
observation that for any differential operator $\K$ on $M$ we have
$$
\Tr ([\H,\K] e^{t\H}) = \Tr (\H \K e^{t\H} - \K \H e^{t\H}) = \Tr
(\K e^{t\H} \H)- \Tr(\K \H e^{t\H}) = 0
$$
by the cyclicity of the trace and the fact that $\H$ commutes with
$e^{t\H}$. It is actually this result which makes the trace of the
heat kernel of $\H(r;g)$ a manageable object. On the other hand, the
discussion of the local variational formula for the renormalized
volume coefficients in Section \ref{CV-v-local} will illustrate the
significance of the double-commutator term in \eqref{base-H}.

By the ellipticity of $\H(r;g)$ (for sufficiently small $r$), the
trace of $\exp(t\H(r;g))$ has an asymptotic expansion
\begin{equation}\label{heat-expansion-H}
\Tr (\exp (t\H(r;g)) \sim (4 \pi t)^{-\f} \sum_{j \ge 0} t^j \int_M
a_{2j}(r;g) \dvol_g, \; t \to 0
\end{equation}
which generalizes the asymptotic expansion
\begin{equation}\label{heat-expansion-CL}
\Tr (\exp(tP_2(g)) \sim (4 \pi t)^{-\f} \sum_{j \ge 0} t^j \int_M
a_{2j}(g) \dvol_g, \; t \to 0
\end{equation}
of the trace of the heat kernel of the conformal Laplacian. Indeed, the
heat kernel coefficients $a_{2j}(r;g)$ satisfy the restriction property
$a_{2j}(0;g) = a_{2j}(g)$. Theorem \ref{structure} implies that the
functionals $g \mapsto a_{2k}(r;g)$ are one-parameter families of local
Riemannian invariants of $g$. Let $a_{(2j,2k)}(g)$ be the coefficients in
the Taylor expansion of $a_{2k}(r;g)$:\footnote{For even $n$ and general
metrics $g$, the operator $\H(r;g)$ is uniquely defined only up to order
$r^{n-2}$. Hence, in this case, the expansions of $a_{2j}(r;g)$ are
well-defined also only up to order $r^{n-2}$. However, they are
well-defined to all orders for conformally Einstein metrics.}
\begin{equation}\label{r-expansion-a}
a_{2j}(r;g) \sim \sum_{k \ge 0} a_{(2j,2k)}(g) r^{2k}, \; r \to 0.
\end{equation}

For constant conformal changes $g \mapsto \lambda^2 g$, i.e., for
rescalings of the metric, the relation \eqref{CCC} implies that
\begin{equation}\label{hc-hom}
a_{2j}(r;\lambda^2g) = \lambda^{-2j} a_{2j}(r/\lambda;g), \; \lambda
\ne 0.
\end{equation}
Hence
$$
a_{(2j,2k)}(\lambda^2 g) = \lambda^{-2j-2k} a_{(2j,2k)}(g), \;
\lambda \ne 0.
$$
In other words, the invariant $a_{(2j,2k)}(g) \in C^\infty(M)$ is of
weight $2j+2k$.

In the following, the coefficients $a_{(2j,n-2j)}$ (for even $n$)
will be called {\em critical} heat kernel coefficients of $\H(r)$.
Since these coefficients satisfy
\begin{equation}\label{hom-crit}
a_{(2j,n-2j)}(\lambda^2g) = \lambda^{-n} a_{(2j,n-2j)}(g),
\end{equation}
their total integrals are scale-invariant.

We note that the explicit formula \eqref{hol-ein} for the
holographic Laplacian $\H(r;g)$ of an Einstein metric $g$ shows that
for such a metric the heat kernel coefficients $a_{2j}(r;g)$ are
given by the relation
\begin{equation}\label{hc-einstein}
a_{2j}(r;g) = (1-\lambda r^2/4)^{n-2j} a_{2j}(g).
\end{equation}
For general metrics, the dependence of the coefficients $a_{2j}(r;g)$ on
$r$ is significantly more complicated, however.

We continue with the discussion of consequences of Theorem
\ref{conform-heat} for the behaviour of the heat kernel coefficients of
$\H(r;g)$ under general conformal changes of the metric. Due to the
overall factor $r$ in the second term on right-hand side of
\eqref{base-heat}, this equation specializes for $r=0$ to the conformal
variational formula
$$
(\partial/\partial\varepsilon)|_{\varepsilon=0} (\Tr(\exp(t
P_2(e^{2\varepsilon \varphi}g)))) = - 2t (\partial/\partial t)
(\Tr(\varphi e^{t P_2(g)}))
$$
for the trace of the heat kernel of the conformal Laplacian. In
\cite{BO-index} and \cite{PR} it is shown that the latter formula implies
the conformal variational formulas
\begin{equation}
\left(\int_{M^n} a_{2j}(g) \dvol_g \right)^\bullet[\varphi] =
(n\!-\!2j) \int_{M^n} \varphi a_{2j}(g) \dvol_g.
\end{equation}
In particular, for even $n$, the functional
\begin{equation}\label{index}
g \mapsto \int_{M^n} a_n(g) \dvol_g
\end{equation}
is a global conformal invariant. Since the integrand $a_n$ is the constant
term in the expansion \eqref{heat-expansion-CL}, the integral
\eqref{index} is often referred to as a conformal index \cite{BO-index}.
The second term on the right-hand side of \eqref{base-heat} contributes
only to the conformal transformation laws of the {\em non-constant} Taylor
coefficients of $a_{2j}(r;g)$ with respect to $r$. By the cyclicity of the
trace, it equals
\begin{equation}\label{bulk-term}
-rt \frac{1}{2} \Tr \left( \varphi \left( \dot{\H}(r;g) e^{t
\H(r;g)} + e^{t \H(r;g)} \dot{\H}(r;g) \right) \right).
\end{equation}
If $\H(r;g)$ and $\dot{\H}(r;g)$ commute, this sum can be written in the
form
\begin{equation}\label{r-contribution}
-r(\partial/\partial r) (\Tr (\varphi e^{t\H(r;g)})).
\end{equation}
The similarity of both contributions
$$
-2t(\partial/\partial t) (\Tr(\varphi e^{t\H(r;g)})) \quad
\mbox{and} \quad -r (\partial/\partial r) (\Tr (\varphi
e^{t\H(r;g)}))
$$
suggests to consider $r^2$ as a second time variable. But since $\H(r)$
and $\dot{\H}(r)$ do not commute in general, the right-hand side of
\eqref{base-heat} is the sum of these two contributions and a contribution
given by the difference of \eqref{bulk-term} and \eqref{r-contribution}.
The additional term can be described in terms of the restriction to the
diagonal of the kernel of the operator \label{refo-dc}
\begin{equation}\label{C-term}
\C(t;r) = \frac{1}{2\pi i} \int_\Gamma
[\R(\lambda),[\R(\lambda),\dot{\H}(r)]] e^{-t\lambda} d\lambda
\end{equation}
(Lemma \ref{double-key}), where $\R(\lambda) = \R(r;\lambda)$
denotes the resolvent of $-\H(r)$. Theorem \ref{B} below is
motivated by the analysis of this term.

For special metrics, the operator $\C(t;r)$ may vanish. For instance,
\eqref{hol-ein} implies that for an Einstein metric $g$ the families
$\H(r;g)$ and $\dot{\H}(r;g)$ are proportional by a factor which only
depends on $r$. Hence, for such metrics, the term \eqref{C-term} vanishes
and, by combining Theorem \ref{conform-heat} with the above arguments, we
obtain the following result.

\begin{prop}\label{critical-even} On closed Einstein manifolds $(M^n,g)$
of even dimension $n$, we have
\begin{equation}\label{CV-a}
\left(\int_{M^n} a_{(2j,2k)}(g) \dvol_g \right)^\bullet[\varphi] =
(n\!-\!2j\!-\!2k) \int_{M^n} \varphi a_{(2j,2k)}(g) \dvol_g.
\end{equation}
In particular, the total integrals
\begin{equation}\label{CV-a-c}
\int_{M^n} a_{(2j,n-2j)} \dvol
\end{equation}
of the critical heat kernel coefficients are critical at Einstein
metrics in their conformal class.
\end{prop}

It is natural to regard Proposition \ref{critical-even} as a result
on the two-parameter {\em spectral zeta function}
\begin{equation}\label{zeta-two}
\zeta(r;s) \st \sum_k \frac{1}{\lambda_k(r)^s}, \quad \Re(s) > \f
\end{equation}
of $-\H(g;r)$. In the latter definition, we assume that the eigenvalues
$\lambda_k(r)$ of $-\H(g;r)$ are positive (for sufficiently small $r$).
The function $s \mapsto \zeta(r;s)$ admits a meromorphic continuation to
$\c$ with simple poles in $s \in \{n/2, n/2-1,\dots,1\}$. Moreover, the
corresponding residue in $n/2-k$ is a constant multiple of the total
integral of $a_{2k}(g;r)$. Now the conformal variation of $\zeta(r;s)$ at
an Einstein metric is given by the formula
\begin{equation}\label{zeta-E}
\zeta(r;s)^\bullet[\varphi] = (2s\!-\!r\partial/\partial r)
\zeta(\varphi;r;s),
\end{equation}
where the local zeta function $\zeta(\varphi;r;s)$ is defined by
$$
\zeta(\varphi;r;s) = \frac{1}{\Gamma(s)} \int_0^\infty t^{s-1} \Tr
(\varphi e^{t \H(r)}) dt, \quad \Re(s) > n/2.
$$
Like $\zeta(r;s)$, the local zeta function $\zeta(\varphi;r;s)$ admits a
meromorphic continuation to $\c$ and the relation \eqref{zeta-E} is valid
as an identity of meromorphic functions. For more details we refer to
Section \ref{CV-zeta}.

In particular, expanding the resulting relation for residues in the
variable $r$ for even $n$ yields the relations \eqref{CV-a} for
$j=0,1,\dots,n/2-1$. Similarly, the relations \eqref{CV-a} for $j \ge n/2$
are obtained by expanding the values of both sides of \eqref{zeta-E} at $s
= n/2-j \in -\N_0$.

The formula \eqref{zeta-E} also has an interesting consequence for odd
$n$. Let $\zeta_{2k}(s)$ be the coefficients in the expansion of
$\zeta(r;s)$ in the variable $r$ (as in \eqref{zeta-expand}). Now, for odd
$n$, $\zeta(r;s)$ is regular at $s \in \N$. The functionals
$$
g \mapsto \zeta_{2k}(k;g)
$$
are scale-invariant and \eqref{zeta-E} implies that all zeta values
$$
\zeta_{2k}(k) \quad \mbox{for $k \in \{0,1,\dots,(n-1)/2,\dots\}$}
$$
are {\em critical} at Einstein metrics in their conformal class.

In the variational formula \eqref{base-heat}, the contribution by
\eqref{C-term} vanishes for $r=0$ for all metrics (due to the
overall $r$-factor). Thus, the same arguments prove the conformal
invariance of the total integral of $a_n=a_{(n,0)}$ for general
metrics.

For general metrics, the term \eqref{C-term} does {\em not} contribute to
the leading heat kernel coefficient $a_0(r)$ of $\H(r)$. More precisely,
we shall prove the relation
\begin{equation}\label{base-0}
a_0(r) = v(r)
\end{equation}
(Lemma \ref{top}). The latter result also follows from the well-known
general results in Chapter II of \cite{shubin} and \cite{DG}. The main
point here is that by combining the relation \eqref{base-0} with the above
arguments concerning the behaviour of the total integrals of $a_0(r)$
naturally leads to a heat equation proof of the following variational
formula.

\begin{thm}[\cite{CF}, \cite{G-ext}]\label{conf-reno} For any metric
$g$ on a closed manifold $M$ of dimension $n$ and $2k \le n$ if $n$
is even, we have
\begin{equation}\label{CF-CT}
\left(\int_M v_{2k}(g) \dvol_g \right)^\bullet[\varphi] = (n\!-\!2k)
\int_M \varphi v_{2k}(g) \dvol_g
\end{equation}
for all $\varphi \in C^\infty(M)$.
\end{thm}

In particular, for closed manifolds $M$ of even dimension $n$, the
total integral of $v_n$ is conformally invariant. We also recall
that this integral is proportional to the total integral of the
critical $Q$-curvature (see \cite{GZ}, \cite{GJ-holo},
\cite{juhl-book}). The quantity $v_n$ is the conformal anomaly of
the renormalized volume \cite{G-vol}, \cite{CFG}.

A somewhat subtle point in the heat equation proof of Theorem
\ref{conf-reno} is that, for general metrics in even dimension $n$, it
does not cover the total integral of the critical coefficient $v_n$. The
reason is that for general metrics the holographic Laplacian is only
uniquely defined up to order $r^{n-2}$.

The heat equation proof of Theorem \ref{conf-reno} is an application of
the algebraic relation between the operator $\H(r;g)$ and the
GJMS-operators (Theorem \ref{base}). It completely differs from earlier
arguments in \cite{CF} and \cite{G-ext}.

For locally conformally flat metrics on manifolds of even dimension
$n$, the coefficient $v_n$ is a multiple of the integrand in the
Chern-Gauss-Bonnet theorem, i.e., of the Euler form (or Pfaffian)
(see \cite{GJ-holo}). Hence, for such metrics, the integrated {\em
leading} heat kernel coefficient of $\H(r)$ spectrally detects the
Euler characteristic of the underlying manifold in its Taylor
coefficient of $r^n$.

We continue with a discussion of the {\em sub-leading} heat kernel
coefficient $a_2(r)$ of $\H(r)$. In order to describe the conformal
variation of its total integral, we introduce the {\em corrected}
sub-leading heat kernel coefficient
\begin{equation}\label{Lambda}
\Lambda(r) =  a_2(r) - (\dot{w}(r))^2 \in C^\infty(M).
\end{equation}
The function $\Lambda(r)$ has an even expansion $\Lambda_0 +
\Lambda_2 r^2 + \Lambda_4 r^4 + \cdots$. Note that $\Lambda_0 =
a_2$. In these terms, we state our second main result.

\begin{thm}\label{A} For any metric $g$ on a closed manifold $M$ of
dimension $n$ and $2k \le n-2$ if $n$ is even, we have
\begin{equation}\label{CT-Lambda}
\left(\int_M \Lambda_{2k}(g) \dvol_g \right)^\bullet[\varphi] =
(n\!-\!2\!-\!2k) \int_M \varphi \Lambda_{2k}(g) \dvol_g
\end{equation}
for all $\varphi \in C^\infty(M)$. In particular, for even $n$, the
functional
$$
\L_{n-2}: g \mapsto \int_M \Lambda_{n-2}(g) \dvol_g
$$
is conformally invariant, i.e., $\L_{n-2}(e^{2\varphi}g) =
\L_{n-2}(g)$.
\end{thm}

Theorem \ref{A} shows how to extract from the heat kernel
coefficient $a_2(r)$ of $\H(r)$ a Riemannian curvature integral
which, under conformal changes of the metric, behaves similarly as
the total integrals of the renormalized volume coefficients $v_{2k}$
and the heat kernel coefficients $a_{2k}$ of the conformal
Laplacian. In particular, the conformal invariance of $\L_{n-2}$ is
an analog of the conformal invariance of the total integrals of
$a_n$ (conformal index) and $v_n$. Theorem \ref{A-fine} will
actually explain this result by an independent calculation of the
invariants $\Lambda_{2k}$.

The proof of Theorem \ref{A} rests on Theorem \ref{conform-heat} and
a certain integral formula. In order to state that formula, we
introduce some additional notation. For any $\Phi \in
C^\infty(T^*M)$ which is polynomial on the fibers of $T^*M$, we set
$$
\left\langle \Phi \right\rangle \st \left(\int_{\r^n} \Phi
e^{-\H(r)} d\xi \right) \big/ \left(\int_{\r^n} e^{-\H(r)}
d\xi\right) \in C^\infty(M),
$$
where the quadratic form
\begin{equation}\label{principal}
\H(r) = \H(r)(x,\xi) = \sum_{i,j} g(r,x)_{ij}^{-1} \xi_i \xi_j \in
C^\infty(T^*M)
\end{equation}
is the principal symbol of the (negative) holographic Laplacian,
i.e., of $-\Delta_{g(r)}$, in local coordinates.\footnote{It is easy
to verify that $\left\langle\Phi\right\rangle$ is well-defined,
i.e., independent of the choice of local coordinates.} In these
terms, we prove the following result.

\begin{thm}\label{B} For any metric $g$, we have
\begin{equation}\label{e-value}
\frac{1}{3!} \left\langle \left
\{\H(r),\left\{\H(r),\dot{\H}(r)\right\} \right\} \right\rangle = -
\Delta_{g(r)}((\dot{v}/v)(r))
\end{equation}
for sufficiently small $r$ as an identity of functions on $M$. Here
the brackets denote the Poisson brackets of the standard symplectic
structure on $T^*M$.
\end{thm}

On the other hand, the following variational formula yields a local
conformal primitive of the ``expectation value'' in Theorem \ref{B}.

\begin{thm}\label{var} For any metric $g$ on a closed manifold $M$ of
dimension $n$, we have
\begin{multline*}
\left(\int_M (\dot{w}(r))^2 \dvol \right)^\bullet[\varphi] =
(n\!-\!2\!-\!r \partial/\partial r) \int_M \varphi (\dot{w}(r))^2
\dvol + r \frac{1}{2} \int_M \varphi v(r) \Delta_{g(r)}
((\dot{v}/v)(r)) \dvol
\end{multline*}
for sufficiently small $r$ and all $\varphi \in C^\infty(M)$.
\end{thm}

Theorem \ref{var} is a consequence of a variational formula for
$v(r)$ (Theorem \ref{RVC-CV}) which constitutes a {\em local}
refinement of Theorem \ref{conf-reno}.

Next, we formulate a structural result for the invariants
$a_{2k}(r)$. For this purpose, we first write $\H(r;g)$ in the form
\begin{equation}\label{LT}
\H(r;g) = \Delta_{g(r)} - (d \log v(r),d)_{g(r)} + \U(r;g).
\end{equation}
It relates the non-Laplace-type operator $\H(r;g)$ for the metric
$g$ to a Laplace-type operator for the metric $g(r)$. Indeed, we
identify the latter operator with an operator of the form
$$
\tr_{g(r)}(\nabla^{v(r)} \circ \nabla^{v(r)}) + \mbox{endomorphism}
$$
on the sections of the trivial line bundle $\L = M \times \c$
equipped with the connection $\nabla^{v(r)}$ with connection
one-form
\begin{equation}\label{nabla-v}
\omega (r) = - \frac{1}{2} d\log v(r) \in \Omega^1(M)
\end{equation}
for the trivialization by the section $s_0 = 1$.

For arbitrary $E \in C^\infty(M)$ and $\omega \in \Omega^1(M)$, let
$a_{2k}(L)$ be the $k^{th}$ heat kernel coefficient of the Laplace
type operator $L = \tr_g(\nabla \circ \nabla) + E$ on the sections
of the line bundle $\L = M \times \c$ equipped with the connection
$\nabla^{\omega}$ with connection one-form $\omega$. Then
$a_{2k}(L)$ is an invariant non-commutative polynomial
$$
a_{2k}(R,\Omega,E)
$$
of order $2k$ in the covariant derivatives of the curvature of the
Levi-Civita connection of $g$, of the curvature $\Omega = d\omega$
and of the endomorphism $E$ (see Chapter 4 of \cite{G-book}).

\begin{thm}\label{structure} The heat kernel coefficients $a_{2k}(r)$
have the form
\begin{equation}\label{a-reduced}
a_{2k}(r) = a_{2k}(R(r),0,E(r)) v(r),
\end{equation}
where $R(r) = R(g(r))$ denotes the curvature of the Levi-Civita
connection of $g(r)$ and the endomorphism $E(r)$ is given by the
function
\begin{equation}\label{E-potential}
E(r) = -w(r)^{-1} \left(\frac{\partial^2}{\partial r^2} -
\frac{n\!-\!1}{r} \frac{\partial}{\partial r} \right)(w(r)).
\end{equation}
\end{thm}

Note that the exactness of the connection one-form $d\log v$ implies
that the curvature of $\nabla^v$ vanishes. We also emphasize that
the potential $E(r)$ does not contain derivatives of the function
$w(r)$ along $M$. By Theorem \ref{structure}, explicit formulas for
low-order heat kernel coefficients of Laplace-type operators imply
formulas for low-order coefficients $a_{2k}(r)$.

In particular, Theorem \ref{structure} generalizes \eqref{base-0}
(by using $a_0=1$). Similarly, using
$$
a_2 = \frac{\scal}{6} + E,
$$
we obtain the following result for the sub-leading heat kernel
coefficient $a_2(r)$.

\begin{thm}\label{A-fine} For any metric, the sub-leading heat kernel
coefficient of the holographic Laplacian $\H(r)$ is given by
the formula
\begin{equation}\label{a2-diff}
a_2(r) = - \frac{1}{3} \left( \ddot{v}(r) - \left(\f\!-\!1\right)
r^{-1} \dot{v}(r) \right) + (\dot{w}(r))^2.
\end{equation}
Equivalently, the power series coefficients of the corrected heat
kernel coefficient
$$
\Lambda(r) = a_2(r) - (\dot{w}(r))^2 = \sum_{k \ge 0} \Lambda_{2k}
r^{2k}
$$
satisfy the relations
\begin{equation}\label{L-v}
\Lambda_{2k-2} = \frac{1}{3} k (n\!-\!4k) v_{2k}, \; k \ge 1.
\end{equation}
In particular, we have
$$
\L_{n-2}(g) = - \frac{n^2}{6} \int_{M^n} v_n(g) \dvol_g.
$$
\end{thm}

Some comments on Theorem \ref{A-fine} and its proof are in order.

We recall that for general metrics the relation \eqref{a2-diff} is
to be understood as an identity of formal power series (up to
$r^{n-2}$ for even $n$). However, for locally conformally flat
metrics in dimension $n \ge 3$, we have
$$
g(r) = g - r^2 \Rho + r^4/4 \Rho^2 \quad \mbox{and} \quad v(r) =
\det (1-r^2/2\Rho),
$$
and $\H(r)$ is well-defined as an analytic function in $r$. In that
case, the relation \eqref{a2-diff} is a formula for the analytic
function $a_2(r)$ in terms of $v(r)$.

If $g$ is Einstein, then $v(r)=(1-cr^2)^n$ for some constant $c$ so
that $\J = 2cn$. In this case, Theorem \ref{A-fine} yields
$$
a_2(r) = - \frac{1}{3} c n (n\!-\!4) (1-cr^2)^{n-2} = - \frac{1}{6}
(n\!-\!4) \J (1-cr^2)^{n-2}.
$$
In view of $a_2 = - \frac{1}{6} (n\!-\!4) \J$, this fits with the
consequence $a_2(r) = (1-cr^2)^{n-2} a_2$ of \eqref{hc-einstein}.

The power series coefficients $\omega_{2k}$ of $(\dot{w}(r))^2$ are
polynomials of degree $k$ (without a linear term) in the
coefficients $v_2,v_4,\dots$. For instance, we have $\omega_2 =
v_2^2$ and $\omega_4 = 4 v_2 v_4 - v_2^3$. Thus, the power series of
$a_2(r)$ starts with
$$
\frac{n\!-\!4}{3} v_2 + r^2 \left(\frac{2(n\!-\!8)}{3} v_4 + v_2^2
\right) + r^4 \left((n\!-\!12) v_6 + 4 v_2 v_4 - v_2^3 \right) +
\cdots.
$$

It seems remarkable that the sub-leading heat kernel coefficient
$a_2(r)$ can be described in terms of renormalized volume
coefficients only. In view of $a_0(r) = v(r)$, the relation
\eqref{a2-diff} actually may be regarded as a relation between the
heat kernel coefficients $a_2(r)$ and $a_0(r)$. The structure of
$a_4(r)$ is substantially more complicated, however. In fact, $a_4 =
a_4(0)$ contains a contribution by the norm of the Weyl tensor (see
Section \ref{coeff-L}) and the Taylor coefficients of $a_4(r)$
cannot be written in terms of renormalized volume coefficients only
(see Proposition \ref{a42-final}).

Combining Theorem \ref{A-fine} with Theorem \ref{A} leads to a
second heat kernel proof of the variational formula for integrated
renormalized volume coefficients (Theorem \ref{conf-reno}); that
proof also covers the conformal invariance of the total integral of
$v_n$ for even $n$.

Of course, combining Theorem \ref{A-fine} with Theorem \ref{conf-reno},
yields a second proof of Theorem \ref{A}. Nevertheless, the arguments
leading to Theorem \ref{A} {\em without} invoking an explicit formula for
$a_2(r)$ do not lose their significance since they suggest the existence
(and ways of finding) of analogs of Theorem \ref{A} for the heat kernel
coefficients $a_{2k}$ for $k \ge 2$ (for more details we refer to Section
\ref{open}).

In addition to the above argument, we shall give a proof of Theorem
\ref{A-fine} which only rests on the direct evaluation of the algorithm
for the calculation of the heat kernel coefficients (described in Section
\ref{AE}).

Now we turn to the formulation of the last main result of the paper. For
closed manifolds $M$, the total integrals of the heat kernel coefficients
$a_{2k}(r)$ are spectral invariants. For $k=0$, the Taylor coefficients of
this invariant are given by the total integrals of the renormalized volume
coefficients. These functionals are variational \cite{CF} and have
interesting extremal properties at Einstein metrics \cite{GuLi},
\cite{CFG}. From that perspective, Theorem \ref{A-fine} suggests to study
the extremal properties of the Taylor coefficients $\W_{2k}$ of the
functional
\begin{equation}\label{WF}
\W(r) \st \int_{M^n} (\dot{w}(r))^2 \dvol,
\end{equation}
i.e., of
$$
\W_{2k} = \int_{M^n} \omega_{2k} \dvol
$$
at Einstein metrics. In view of $\omega_2 = v_2^2$, the functionals
$\W_{2k}$ are generalizations of the quadratic curvature functional
$\int_M \scal^2 \dvol$. The local Riemannian invariant $\omega_{2k}$
is a polynomial of degree $2k$ in the curvature with the property
that $\omega_{2k}(\lambda^2 g) = \lambda^{-2k-2} \omega_{2k}(g)$ for
$\lambda \ne 0$. Hence $\W_{2k}$ is homogeneous of degree $n-2-2k$.
Although the invariants $\omega_{2k}$ become very complicated for
large $k$, we prove the following result.

\begin{thm}\label{ext} Let $(M^n,g)$ be a closed unit volume
Einstein manifold of dimension $n \ge 4$. Assume that $\scal(g)
> 0$ and let $2k \le n\!-\!2$. Then the restriction of the functional
$(-1)^k \W_{2k}$ to the set $[g]_1$ of unit volume metrics conformal
to $g$ has a strict local maximum at $g$ unless $(M^n,g)$ is
isometric to a (rescaled) round sphere.
\end{thm}

For the proof of Theorem \ref{ext} we combine the conformal
variational formula for $v(r)$ (Theorem \ref{RVC-CV}) with a
well-known result of Obata \cite{obata}.

Let $n \ge 4$ be even. Then the special case $2k=n\!-\!2$ of Theorem
\ref{ext} implies a corresponding result for the total integral
$\int_M a_{(2,n-2)} \dvol$ by using Theorem \ref{A-fine} and the
conformal invariance of $\int_M v_n \dvol$. On the other hand, the
total integral of $a_{(2,n-2)}$ has a local minimum at Einstein
metrics $g$ with $\scal(g) < 0$. For the details and further results
in this direction we refer to Section \ref{extremal}.

For the readers convenience, we close this section with a review of the
content of the paper. In Section \ref{basics}, we review the theory of
building-block operators from \cite{juhl-ex} and \cite{FG-J}. Section
\ref{variation} contains the proof of the basic conformal variational
formula for the trace of the heat kernel of $\H(r;g)$ (Theorem
\ref{conform-heat}). In Section \ref{AE}, we recall Gilkey's
pseudo-differential algorithm which yields the asymptotic expansion of the
trace of the heat kernel. In Section \ref{dc-term}, we describe the effect
of isolating the contribution \eqref{r-contribution} on the right-hand
side of the variational formula in Theorem \ref{conform-heat}. Section
\ref{proof-A} then contains the heat equation proof of Theorem
\ref{conf-reno}. The discussion of the sub-leading heat kernel coefficient
$a_2(r;g)$ starts in Section \ref{proof-B}. Here we prove Theorem \ref{B}.
In Section \ref{prime}, we determine a local conformal primitive of the
correction term for $a_2(r;g)$ found in Theorem \ref{B}. We combine these
results in Section \ref{proof-AA} to complete the proof of Theorem
\ref{A}. Section \ref{fine} contains the details of the derivation of the
closed formula for $a_2(r;g)$ given in Theorem \ref{A-fine}. In Section
\ref{structure-gen} we return to the discussion of the structure of the
coefficients $a_{2k}(r)$. Here the main idea is to regard $\H(r;g)$ as a
Laplace-type operator with respect to the metric $g(r)$ and to combine
this with structural results for Laplace-type operators. This yields a
proof of Theorem \ref{structure} and a second proof of Theorem
\ref{A-fine}. Section \ref{extremal} is devoted to the study of extremal
properties of various curvature integrals arising from heat kernel
coefficients. The first set of such integrals are defined by the Taylor
coefficients of $a_0(r)$. These are given by the renormalized volume
coefficients $v_{2k}$ and we review and reprove corresponding results of
\cite{CFG}. The remainder of this section contains a proof of Theorem
\ref{ext}. In Section \ref{open}, we collect various additional results
and indicate some open problems. In particular, in Section \ref{CV-zeta}
we introduce the two-parameter spectral zeta function \eqref{zeta-two} and
prove the variational formula \eqref{zeta-E}, and in Section
\ref{CV-v-local} we outline a heat equation proof of the conformal
transformation rule of the non-integrated renormalized volume
coefficients. The final section of the paper is a long appendix which
contains background information and complementary material. In particular,
we display the first few terms in the expansion of the holographic
Laplacian, and discuss in detail the heat kernel coefficients $a_2$, $a_4$
and $a_6$ of the conformal Laplacian. In this connection, we give a new
derivation of an explicit formula for $a_6$ (for locally conformally flat
metrics) which does not depend on Gilkey's formula for $a_6$, and provide
a direct proof of the relation between the local conformal invariant $a_6$
in dimension $n=8$ and the formal heat kernel coefficient $\tilde{a}_6$ of
the ambient Laplacian $\tilde{\Delta}$. In addition, we derive a
Polyakov-type formula for the renormalized volume $\V(g_+;g)$ of a
Poincar\'e-Einstein metric $g_+$ (Theorem \ref{PV-holo}). For even
dimensional conformal infinity, this result expresses the behaviour of the
renormalized volume under conformal changes of the metric $g$ in terms of
a $Q$-curvature term and some holographic correction terms. The formula is
a consequence of the results in \cite{GJ-holo} and \cite{holo-II} and
should be regarded as an analog of Branson's (conjectural) Polyakov-type
formula for the determinant of $P_2$ with an analogous main part in terms
of the critical $Q$-curvature $Q_n$. Finally, we determine the coefficient
$a_{(4,2)}$ of $r^2$ in the expansion of $a_4(r)$ in terms of Riemannian
invariants of $g$ (Proposition \ref{a42-final}). For conformally flat
metrics, the results in Lemma \ref{R-expansion}, Lemma \ref{E-final} and
\eqref{scal} actually allow to derive a formula for the full expansion of
$a_4(r)$ in terms of $\Rho$ (at least in principle). Moreover, we apply
these results to the study of the Hessian of the integral of $a_{(4,2)}$
in the critical dimension $n=6$.

The present paper touches a large variety of topics like heat
kernels, spectral theory, conformal differential geometry, AdS/CFT
duality. Each of these has its own huge literature. We minimize the
number of references by citing only items of immediate relevance.

Some of the results of this paper were presented at the conference {``The
Interaction of Geometry and Representation Theory''} at ESI (Vienna) in
2012, during the program {``Conformal Geometry and Geometric PDE's''} at
CRM (Barcelona) in 2013, at the Winter School {``Geometry and Physics''}
2014 (Srni) and at the conference {``Lorentzian and Conformal Geometry''}
(Greifswald) in 2014. Finally, I am grateful to A. Tseytlin for pointing
out the reference \cite{FT}.

\section{GJMS-operators and the holographic Laplacian}\label{basics}

In the present section, we describe the relation between the holographic
Laplacian $\H(r;g)$ and the GJMS-operators $P_{2N}(g)$. The following
material is based on \cite{juhl-ex} and \cite{FG-J}.

We start with the definition of the concepts involved. For a given
Riemannian manifold $(M,g)$ of dimension $n \ge 3$, a Poincar\'e-Einstein
metric $g_+$ in normal form relative to $g$ is a metric $g_+$ on $M \times
(0,\varepsilon)$ of the form
$$
g_+ = r^{-2}(dr^2 + g(r))
$$
with a smooth one-parameter family $g(r)$ of metrics on $M$ so that
$g(0)=g$, and for which $\Ric(g_+) + ng_+$ vanishes on $M$ in the
following asymptotic sense. We require that $\Ric(g_+) + ng_+ =
O(r^\infty)$ if $n$ is odd and $\Ric(g_+) + ng_+ = O(r^{n-2})$ and
the tangential trace of $r^{2-n}(\Ric(g_+)+ng_+)$ vanishes at $r=0$
if $n$ is even.

The Poincar\'e-Einstein metric $g_+$ associated to $g$ is closely
related to an ambient metric associated to $g$. In normal form, this
is the metric
$$
\tilde{g} = 2\rho dt^2 + 2tdtd\rho + t^2 {\bf g}(\rho)
$$
on $\r_+ \times M \times (-\varepsilon,\varepsilon)$ with ${\bf
g}(-r^2/2)=g(r)$.

As explained in Section \ref{intro}, $g(r)$ has a uniquely
determined formal {\em even} power series for odd $n$ and a uniquely
determined {\em even} power series up to $r^{n-2}$ for even $n$. In
the latter case, also the trace of the coefficient of $r^n$ is
determined by $g$. For general metrics, the Taylor series of $g(r)$
starts with
\begin{equation}\label{g-exp-4}
g(r) = g - r^2 \Rho + r^4/4 (\Rho^2 - \B/(n\!-\!4)) + \cdots,
\end{equation}
where $\B$ is the Bach tensor of $g$ (Chapter 3 of \cite{FG-final}
or Section 6.9 of \cite{juhl-book}). This formula also shows that in
dimension $4$ the obstruction tensor is given by $\B$.

For some exceptional classes of metrics $g$, there are closed
formulas for $g(r)$. For instance, for a locally conformally flat
metric $g$, we have
$$
g(r) = g - r^2 \Rho + r^4/4 \Rho^2.
$$
The same formula holds true for an Einstein metric $g$. In that
case, it further simplifies to
$$
g(r) = (1\!-\!\lambda r^2)^2 g, \quad \lambda = \frac{\scal(g)}{4
n(n\!-\!1)}.
$$
For the corresponding details we refer to Chapter 7 of
\cite{FG-final} or Section 6.14 and Section 6.16 in
\cite{juhl-book}. Now we define
$$
{\bf v}(\rho) = \sqrt{\det {\bf g}(\rho)}/ \sqrt{\det {\bf g}(0)}
\quad \mbox{and} \quad v(r) = \sqrt{\det g(r)}/\sqrt{\det g(0)}.
$$
Then ${\bf v}(-r^2/2) = v(r)$. The power series coefficients
$v_{2k}$ of $v(r)$ are the renormalized volume (or holographic)
coefficients \cite{G-vol}. We also set
$$
{\bf w}(\rho) = \sqrt{{\bf v}(\rho)} \quad \mbox{and} \quad w(r) =
\sqrt{v(r)}.
$$
Let $w_{2k}$ be the power series coefficients of $w(r)$.

The original definition of the GJMS-operators $P_{2N}$ \cite{GJMS}
defines these operators through the composition
\begin{equation}\label{GJMS-def}
P_{2N}(u) = \tilde{\Delta}^N (t^{N-\f} \tilde{u})|_{\rho=0, t=1}, \;
u \in C^\infty(M).
\end{equation}
Here $\tilde{\Delta}$ denotes the Laplace-Beltrami operator of the
ambient metric $\tilde{g}$ and $\tilde{u} \in C^\infty(M \times
(-\varepsilon,\varepsilon))$ is an arbitrary extension of $u$:
$\tilde{u}(x,0) = u(x)$. That construction is well-defined (for
general metrics) under the assumption $2N \le n/2$ for even $n$ and
$N \ge 1$ for odd $n$. Moreover, it is well-defined for all $N \ge
1$ for locally conformally flat metrics and conformally Einstein
metrics in all dimensions $n \ge 3$.

Next, we recall the definition of the building-block operators
$\M_{2N}$ in terms of GJMS-operators.

For this purpose, we shall use the following combinatorial
conventions. A sequence $I=(I_1,\dots,I_r)$ of natural numbers $I_j
\ge 1$ will be regarded as a composition of the natural number
$|I|=I_1+I_2+\cdots+I_r$. In other words, compositions are
partitions in which the order of the summands is considered. $|I|$
is called the size of $I$. To any composition $I=(I_1,\dots,I_r)$,
we associate the multiplicities
\begin{equation}\label{m-form}
m_I = - (-1)^r |I|! \, (|I|\!-\!1)! \prod_{j=1}^r \frac{1}{I_j! \,
(I_j\!-\!1)!} \prod_{j=1}^{r-1} \frac{1}{I_j \!+\! I_{j+1}}
\end{equation}
and
\begin{equation*}
n_I = \prod_{j=1}^r \binom{\sum_{k \le j} I_k -1}{I_j-1}
\binom{\sum_{k \ge j} I_k -1}{I_j-1}.
\end{equation*}
Note that $m_{(N)} = n_{(N)} = 1$ for all $N \ge 1$ and
$n_{(1,\dots,1)} = 1$.

Now, for any composition $I=(I_1,\dots,I_r)$, we set $P_{2I} \st
P_{2I_1} \circ \cdots \circ P_{2I_r}$. Then we define
\begin{equation}\label{M-def}
\M_{2N} \st \sum_{|I|=N} m_I P_{2I} = P_{2N} + \mbox{compositions of
lower-order GJMS-operators}.
\end{equation}
Note that $\M_2 = P_2$. By an inversion formula (Theorem 2.1 in
\cite{juhl-ex}), these definitions are equivalent to the basic
representation formulas
\begin{equation}\label{basic-PM}
P_{2N} = \sum_{|I|=N} n_I \M_{2I} = \M_2^N + \mbox{compositions of
$< N$ building-block operators}
\end{equation}
for GJMS-operators in terms of building-block operators. Here we use
the convention $\M_{2I} \st \M_{2I_1} \circ \cdots \circ \M_{2I_r}$.
In these constructions, the integer $N$ is subject to the usual
condition $2N \le n$ for even $n$.

The following result provides a formula for the building-block
operators.

\begin{thm}[\cite{juhl-ex}]\label{base} The generating function
\begin{equation}\label{GF}
\sum_{N \ge 1} \M_{2N}(g) \frac{1}{(N\!-\!1)!^2} \left(
\frac{r^2}{4} \right)^{N-1}
\end{equation}
of the building-block operators coincides with the operator
$\H(r;g)$.
\end{thm}

The relation \eqref{GF} is to be interpreted as an equality of
formal power series (up to $r^{n-2}$ for even $n$). Theorem
\ref{base} and \eqref{M-def} show that, by comparing coefficients of
power series in the equality of \eqref{GF} and $\H(r;g)$, one
obtains formulas for all GJMS-operators $P_{2N}$ in terms of the
Taylor coefficients of $g(r)$ and $v(r)$.

The proof of Theorem \ref{base} in \cite{juhl-ex} rests on the
theory of residue families (as introduced in \cite{juhl-book}).

In \cite{FG-J}, Fefferman and Graham gave an alternative proof of
these formulas for the operators $P_{2N}$. The idea of that proof is
the following. Instead of appealing to the inversion formula of
\cite{juhl-ex}, {\em define} the operators $\M_{2N}(g)$ by the
identity
$$
\H(r;g) = \sum_{N \ge 1} \M_{2N}(g) \frac{1}{(N\!-\!1)!^2} \left(
\frac{r^2}{4}\right)^{N-1}
$$
(with $\H(r;g)$ given by \eqref{sch} and \eqref{potential}). The
starting point of the proof is a formula which relates a conjugate
of the Laplacian of the ambient metric of $g$ to $\H(r;g)$. In fact,
let
$$
\tilde{\Delta}_{\bf w} \st {\bf w} \circ \tilde{\Delta} \circ {\bf
w}^{-1}.
$$
Then a calculation shows that, for any $\lambda \in \r$,
\begin{equation}\label{con-ambient}
\tilde{\Delta}_{\bf w} (t^\lambda \tilde{u}) = t^{\lambda-2}
\left(-2\rho (\partial/\partial \rho)^2 + (2\lambda\!+\!n\!-\!2)
\partial/\partial \rho + \tilde{\H}(\rho) \right) \tilde{u},
\end{equation}
where $\tilde{\H}(-r^2/2) = \H(r)$. Now the defining equation
\eqref{GJMS-def} is easily seen to be equivalent to
\begin{equation}\label{GJMS-rev}
P_{2N}(u) = \tilde{\Delta}_{\bf w}^N (t^{N-\f} \tilde{u})
|_{\rho=0,t=1}.
\end{equation}
The evaluation of \eqref{GJMS-rev} then yields a formula for
$P_{2N}$ in terms of the power series coefficients of
$\tilde{\H}(\rho)$, and it turns out that the results of both
methods coincide.

Note that there is an analog of the conjugation formula
\eqref{con-ambient} for the Laplacian of the Poincar\'e-Einstein
metric.

\begin{prop}\label{conjugate-PE} For $u \in C^\infty(M \times
(0,\varepsilon))$, we have
\begin{equation}\label{con-PE}
w \Delta_{g_+}(w^{-1} u) = r^2 \frac{\partial^2 u}{\partial r^2} -
(n\!-\!1) r \frac{\partial u}{\partial r} + r^2 \H(r;g) u.
\end{equation}
\end{prop}

The relation \eqref{con-PE} shows that the holographic Laplacian
$\H(r;g)$ is given by the {\em tangential part} of the conjugation
of the Laplacian $\Delta_{g_+}$ by $w(r)$.

An important effect of the conjugation by ${\bf w}$ (in
\eqref{con-ambient}) and $w$ (in \eqref{con-PE}) is that the
coefficients of the respective first-order partial derivative with
respect to $\rho$ and $r$ are constant on $M$.\footnote{The role of
the square-root of $v(r)$ resembles the role of the square-root of
the determinant $j(x)$ of the Jacobian of the exponential map in the
theory of the heat equation (see Chapter 2 of \cite{BGV}).}

Finally, we briefly discuss Branson's $Q$-curvatures \cite{sharp}.
First of all, we recall that the $Q$-curvatures arise through the
zeroth order terms of the GJMS-operators. More precisely, for even
$n$ and $2N < n$, the local Riemannian invariant $Q_{2N}$ is defined
by
$$
P_{2N}(1) = \left(\f\!-\!N\right)(-1)^N Q_{2N}.
$$
Similarly, $Q_{2N}$ is defined for odd $n$ and all $N \ge 1$. In
even dimension $n$, the {\em critical} $Q$-curvature $Q_n$ is
defined by an argument of analytic continuation in $n$. The
functional $g \mapsto Q_{2N}(g)$ is a local Riemannian invariant of
weight $2N$. $Q$-curvatures (of low-order) play an important role in
geometric analysis. Among the important properties of $Q$-curvature,
we emphasize that, in even dimension $n$, the critical $Q$-curvature
$Q_n$ satisfies the conformal transformation law
$$
e^{n\varphi} Q_n(e^{2\varphi}g) = Q_n(g) + (-1)^\f P_n(g)(\varphi).
$$
It is this result which implies that the total integral of $Q_n$ on
closed manifolds is a global conformal invariant. Note that the
first two $Q$-curvatures (in general dimension) are given by $Q_2 =
\J$ and $Q_4 = \f \J^2 - 2 |\Rho|^2 - \Delta \J$ (see \eqref{pan}).
For $N \ge 3$, $Q_{2N}$ is much more complicated, and the following
recent result on its recursive structure \cite{Q-recursive} may be
used to derive explicit formulas in terms of $g$. For $N \ge 1$ and
$2N \le n$ (if $n$ is even), we have
\begin{equation}\label{Q-rec}
(-1)^N Q_{2N} = - \sum_{|I|+a=N, \, a < N} m_{(I,a)} (-1)^a
P_{2I}(Q_{2a}) + N!(N\!-\!1)! 2^{2N} w_{2N}
\end{equation}
where for any composition $I=(I_1,\dots,I_r)$ and $a \in \N$, the
notation $(I,a)$ means the composition $(I_1,\dots,I_r,a)$. An
alternative proof of this formula was given in \cite{FG-J}. We
recall that the contributions $w_{2N}$ are the power series
coefficients of the square-root $w(r)$ of $v(r)$. For more details
on $Q$-curvatures we refer to \cite{juhl-book}, \cite{GJ-holo} and
\cite{holo-II}.

\section{Conformal variation}\label{variation}

In the present section we prove Theorem \ref{conform-H} and Theorem
\ref{conform-heat}. The proof of Theorem \ref{conform-H} rests on the
identification of $\H(r;g)$ with the generating function \eqref{GF}
(Theorem \ref{base}) and a formula for the conformal variation of the
building-block operators $\M_{2N}$ found in \cite{juhl-power}. We first
recall the formulation and the proof of that result.

For $N \ge 2$, the operators $\M_{2N}$ are not conformally
covariant. However, the conformal variation of any building-block
operator $\M_{2N}$ is given by an expression in terms of respective
lower-order building-block operators.

\begin{thm}\label{M-var} Let $N \ge 1$ for $n$ odd and $2N \le n$
for $n$ even. Then, for all metrics and all $\varphi \in
C^\infty(M)$,\footnote{For $N =1$, the right-hand side of
\eqref{M-CV} is defined to be $0$.}
\begin{multline}\label{M-CV}
(d/dt)|_{t=0} \big(e^{(\f+N)t\varphi} \M_{2N}(e^{2t\varphi}g) e^{-(\f-N)t\varphi}\big) \\
= - \sum_{j=1}^{N-1} \binom{N\!-\!1}{j\!-\!1}^2 (N\!-\!j)
\left[\M_{2j}(g),[\M_{2N-2j}(g),\varphi]\right].
\end{multline}
\end{thm}

\begin{proof} The assertion follows from the conformal covariance
\begin{equation}\label{conf-cov}
e^{(\f+N)\varphi} \circ P_{2N}(e^{2\varphi}g) = P_{2N}(g) \circ
e^{(\f-N)\varphi}
\end{equation}
of the GJMS-operators. The right-hand side of \eqref{M-CV} is a
weighted sum of terms of the form
\begin{equation*}
\M_{2j} \circ \M_{2N-2j} \circ \varphi - \M_{2j} \circ \varphi \circ
\M_{2N-2j} - \M_{2N-2j} \circ \varphi \circ \M_{2j} + \varphi \circ
\M_{2N-2j} \circ \M_{2j}.
\end{equation*}
By \eqref{M-def}, on the left-hand side of \eqref{M-CV} the term
$P_{2I}$ with $I=(I_1,I_2,\dots,I_r)$ contributes a constant
multiple of
\begin{multline}\label{lhs}
(N\!-\!I_1) \varphi \circ P_{2I} - (I_1\!+\!I_2) P_{2I_1} \circ
\varphi \circ P_{2I_2} \cdots P_{2I_r} \\
- \cdots - (I_{r-1}\!+\!I_r) P_{2I_1} \cdots P_{2I_{r-1}} \circ
\varphi \circ P_{2I_r} + (N\!-\!I_r) P_{2I} \circ \varphi.
\end{multline}
For the term $P_{2I} \circ \varphi$, the claim is
\begin{multline}\label{c-last}
-(N\!-\!I_r) m_I = \binom{N\!-\!1}{I_1\!-\!1}^2
(N\!-\!I_1) \; m_{(I_1)} m_{(I_2,\dots,I_r)} \\
+ \binom{N\!-\!1}{I_1\!+\!I_2\!-\!1}^2
(N\!-\!I_1\!-\!I_2) \; m_{(I_1,I_2)} m_{(I_3,\dots,I_r)} \\
+ \cdots + \binom{N\!-\!1}{I_1\!+\!I_2\!+ \cdots +\!I_{r-1}\!-\!1}^2
(N\!-\!I_1\!-\!I_2-\cdots-I_{r-1}) \; m_{(I_1,I_2,\dots,I_{r-1})}
m_{(I_r)}.
\end{multline}
In order to prove this identity, we use the explicit formula for the
coefficients $m_I$ (see \eqref{m-form}) to write the terms in the
sum as multiples of $m_I$. We find
\begin{multline}\label{comb}
-\frac{1}{N} \Big[ I_1(I_1+I_2) + (I_1+I_2)(I_2+I_3) +
(I_1+I_2+I_3)(I_3+I_4) \\ + \cdots +
(I_1+I_2+\cdots+I_{r-1})(I_{r-1}+I_r)\Big] m_I.
\end{multline}
Now the relation
\begin{equation*}
I_1(I_1+I_2) + \cdots + (I_1+I_2+\cdots+I_{r-1})(I_{r-1}+I_r)  =
(I_1+\cdots+I_r)(I_1+\cdots+I_{r-1})
\end{equation*}
(which easily follows by induction) implies that in \eqref{comb} the
sum in brackets equals $N(N-I_r)$ if $|I|=N$. Thus, \eqref{comb}
equals $-(N-I_r)m_I$. This proves the assertion.

Next, for the term $\varphi \circ P_{2I}$ with $|I|=N$, the claim is
\begin{multline*}
-(N\!-\!I_1) m_I = \binom{N\!-\!1}{I_2\!+\cdots+\!I_r\!-\!1}^2
(N\!-\!I_2\!-\cdots-I_r) \; m_{(I_1)} m_{(I_2,I_2,\dots,I_{r})} \\
+ \cdots + \binom{N\!-\!1}{I_r\!-\!1}^2 (N\!-\!I_r) \;
m_{(I_1,\dots,I_{r-1})} m_{(I_r)}.
\end{multline*}
This identity follows by applying \eqref{c-last} to the inverse
composition $I^{-1}$ of $I$ and using the relations $m_{I^{-1}} =
m_I$ for all compositions $I$.

It remains to prove the corresponding identities for the
coefficients of the terms
$$
P_{2I_1} \cdots P_{2I_a} \circ \varphi \circ P_{2I_{a+1}} \dots
P_{2I_r}.
$$
In that case, the claim is
\begin{multline*}
-(I_a+I_{a+1}) m_I = \Big[
\binom{N\!-\!1}{I_1\!+\cdots+\!I_a\!-\!1}^2
(N\!-\!I_1\!-\cdots-I_a) \\
+ \binom{N\!-\!1}{I_{a+1}\!+\cdots+\!I_r\!-\!1}^2
(N\!-\!I_{a+1}\!-\cdots-I_r) \Big] m_{(I_1,\dots,I_a)}
m_{(I_{a+1},\dots,I_r)}.
\end{multline*}
By \eqref{m-form}, the right-hand side reduces to
$$
-\frac{1}{N} (I_a+I_{a+1}) [(I_1+\cdots+I_a) + (N-I_1-\cdots-I_a)]
m_I,
$$
i.e., to $-(I_a + I_{a+1)}) m_I$. This completes the proof.
\end{proof}

For locally conformally flat metrics, Theorem \ref{M-var} extends to
even $n$ and all $N \ge 1$.

\smallskip

We continue with the {\bf proof of Theorem \ref{conform-H}}.

We rewrite \eqref{M-CV} as
\begin{equation}\label{rew}
(d/dt)|_{t=0} \big(e^{(\f+N)t\varphi}
\tilde{\M}_{2N}(e^{2t\varphi}g) e^{-(\f-N)t\varphi} \big) = -
\sum_{a=1}^{N-1} \frac{1}{a}
\left[\tilde{\M}_{2N-2a}(g),[\tilde{\M}_{2a}(g),\varphi]\right]
\end{equation}
with
$$
\tilde{\M}_{2N} \st \frac{\M_{2N}}{(N\!-\!1)!^2}.
$$
We recall that for $N=1$ the sum on the right-hand side of
\eqref{rew} is to be understood as $0$ (by definition). Now we
combine the relation
\begin{multline*}
(d/dt)|_{t=0} \left( e^{(\f+1)t\varphi} \H(r;e^{2t \varphi} g)
e^{-(\f-1)t\varphi} \right) \\
= (d/dt)|_{t=0} \left( \sum_{N \ge 1} e^{-(N-1)t\varphi}
\left(e^{(\f+N) t\varphi} \tilde{\M}_{2N}(e^{2t\varphi}g)
e^{-(\f-N)t\varphi}\right) e^{-(N-1)t\varphi}
\left(\frac{r^2}{4}\right)^{N-1}\right)
\end{multline*}
with \eqref{rew} and find
\begin{multline}\label{alm}
(d/dt)|_{t=0} \left( e^{(\f+1)t\varphi} \H(r;e^{2t \varphi} g)
e^{-(\f-1)t\varphi} \right) \\
= \sum_{N\ge 1} \left(-(N\!-\!1) \varphi \tilde{\M}_{2N}(g) -
(N\!-\!1) \tilde{\M}_{2N}(g) \varphi \right) \left(\frac{r^2}{4}\right)^{N-1} \\
- \sum_{N \ge 1} \left( \sum_{a=1}^{N-1} \frac{1}{a}
\left[\tilde{\M}_{2N-2a}(g),[\tilde{\M}_{2a}(g),\varphi]\right]
\right) \left(\frac{r^2}{4}\right)^{N-1}.
\end{multline}
But the sums on the right-hand side of \eqref{alm} can be rewritten
as
$$
- \frac{1}{2} r (\partial/\partial r) (\varphi \H(r;g) + \H(r;g)
\varphi) - [\H(r;g),[\K(r;g),\varphi]]
$$
with
$$
\K(r;g) = \sum_{N \ge 1} \tilde{\M}_{2N}(g) \frac{1}{N}
\left(\frac{r^2}{4}\right)^N = \sum_{N \ge 1} \M_{2N}(g) \frac{1}{N!
(N\!-\!1)!} \left(\frac{r^2}{4}\right)^N.
$$
This completes the proof of Theorem \ref{conform-H}. \hfill
$\square$ \smallskip

We recall that, for even $n$, Theorem \ref{conform-H} asserts the
equality of two power series in $r$ up to $r^{n-2}$. This
interpretation also applies to the above arguments.

The above proof shows that the first term on the right-hand side of
\eqref{base-H} is caused by the differences of the conformal weights
in the respective variational formulas of the operators $\M_{2N}$
and $\H(r)$.

We finish the discussion of Theorem \ref{conform-H} with an {\bf
independent proof} in the special case of conformal changes of the
round metric on the spheres $\s^n$ which arise from the conformal
action of the group $G=SO(1,n\!+\!1)$. We recall that $G$ acts on
the round metric $g_0$ by conformal diffeomorphisms, i.e., we have
\begin{equation}
\gamma_*(g_0) = e^{2\varphi_\gamma} g_0
\end{equation}
for all $\gamma \in G$ and certain functions $\varphi_\gamma \in
C^\infty(M)$. In particular, we obtain $\gamma_*(\dvol) =
e^{n\varphi_\gamma} \dvol$ for the volume form $\dvol$ of the round
metric. For any $\lambda \in \c$, the map
$$
\pi_\lambda: \gamma \mapsto
\left(\frac{\gamma_*(\dvol)}{\dvol}\right)^\frac{\lambda}{n} \gamma_*(u) =
e^{\lambda \varphi_\gamma} \gamma_*(u)
$$
defines a principal series representation of $G$ on
$C^\infty(\s^n)$. Using
$$
P_2(e^{2\varphi_\gamma} g_0) = P_2 (\gamma_*(g_0)) = \gamma_* P_2(g_0)
\gamma^*,
$$
the conformal covariance
$$
e^{(\f+1)\varphi} \circ P_2(e^{2\varphi}g) = P_2(g) \circ
e^{(\f-1)\varphi}, \; \varphi \in C^\infty(M^n)
$$
of the conformal Laplacian $P_2$ implies the intertwining relation
\begin{equation}\label{intertwine}
\pi_{\f+1}(\gamma) \circ P_2(g_0) = P_2(g_0) \circ \pi_{\f-1}(\gamma), \;
\gamma \in G
\end{equation}
for $P_2(g_0)$ on $\s^n$. For more details see \cite{juhl-book}. We
also recall that
\begin{equation}\label{special-H}
\H(r;g_0) = (1\!-\!r^2/4)^{-2} P_2(g_0)
\end{equation}
(see \cite{juhl-ex}). Now let $\gamma_t \in G$ be a one-parameter family
with $\gamma_0=e$ and let $e^{2\varphi_t}$ be the corresponding conformal
factor. By \eqref{intertwine}, the variation
$$
(d/dt)|_{t=0}(\pi_{\f+1}(\gamma_t) P_2(g_0) \pi_{\f-1}(\gamma_t^{-1}))
$$
vanishes. Hence, using \eqref{special-H}, we obtain
\begin{align*}
& (d/dt)|_0 (e^{(\f+1)t \psi} \H(r;e^{2t\psi} g_0) e^{-(\f-1)t\psi}) \\
& = (d/dt)|_0 (e^{(\f+1) \varphi_t} (\gamma_t)_* \H(r;g_0) \gamma_t^* e^{-(\f-1)\varphi_t}) \\
& = (1\!-\!r^2/4)^{-2} (d/dt)|_0 (e^{(\f+1)\varphi_t} (\gamma_t)_*
P_2(g_0) \gamma_t^* e^{-(\f-1) \varphi_t}) \\
& = (1\!-\!r^2/4)^{-2} (d/dt)|_0 (\pi_{\f+1}(\gamma_t) P_2(g_0)
\pi_{\f-1}(\gamma_t^{-1})) \\
& = 0
\end{align*}
for $\psi = (d/dt)|_{t=0}(\varphi_t)$. Therefore, Theorem
\ref{conform-H} states that the sum
\begin{multline*}
- \frac{1}{2} (r \partial/\partial r) ((1\!-\!r^2/4)^{-2})
(\psi P_2(g_0) + P_2(g_0) \psi) \\
-(1\!-\!r^2/4)^{-2} \int_0^r s(1\!-\!s^2/4)^{-2} ds \;
[P_2(g_0),[P_2(g_0),\psi]]
\end{multline*}
vanishes, too. A calculation shows that the latter sum coincides
with the product of
$$
-\frac{r^2}{4} (1\!-\!r^2/4)^{-3}
$$
and the operator
\begin{equation}\label{rhs-sphere}
2 (\psi P_2(g_0) + P_2(g_0) \psi) + [P_2(g_0),[P_2(g_0),\psi]].
\end{equation}
In order to verify the vanishing of \eqref{rhs-sphere}, we note that
the operator $\M_4 = P_4 - P_2^2$ satisfies the conformal
variational formula
$$
(d/dt)|_{t=0} (e^{(\f+1)t\varphi} \M_4(e^{2t\varphi}g)
e^{-(\f-1)t\varphi}) = - (\varphi \M_4(g) + \M_4(g) \varphi) -
[P_2(g),[P_2(g),\varphi]]
$$
for all metrics $g$ and all $\varphi \in C^\infty(M)$.\footnote{This is a
consequence of the special case $N=2$ of Theorem \ref{M-var}.} We consider
this relation for the round sphere and conformal changes which are induced
by the action of one-parameter groups $\gamma_t \in G$ with $\gamma_0 =
e$. Using $\M_4 (g_0) = 2 P_2(g_0)$,\footnote{This formula for $\M_4(g_0)$
is a consequence of the product formula $P_4(g_0)= P_2(g_0) (P_2(g_0)
\!+\! 2)$.} we find
$$
0 = - 2 (\psi P_2(g_0) + P_2(g_0) \psi) -
[P_2(g_0),[P_2(g_0),\psi]].
$$
This proves the vanishing of \eqref{rhs-sphere} and henceforth
confirms Theorem \ref{conform-H} in this special case.\footnote{A
more direct proof of the vanishing of \eqref{rhs-sphere} follows
from the fact that $\psi$ is an eigenfunction of $\Delta$ for the
smallest non-zero eigenvalue: $\Delta \psi = -n \psi$ and the
relation $n \Hess(\psi) = g \Delta \psi$ \cite{obata}.}

Now we are ready to complete the {\bf proof of Theorem
\ref{conform-heat}}. A generalization of a formula of Ray and Singer
(see Proposition 3.5 in \cite{BO-index} or Corollary 2.50 in
\cite{BGV}) shows that
$$
(\partial/\partial\varepsilon)|_{\varepsilon=0} (\Tr(\exp(t
\H(r;e^{2\varepsilon \varphi} g)))) = t \Tr((\partial/\partial
\varepsilon)|_{\varepsilon=0} (\H(r;e^{2\varepsilon \varphi}g))
\exp(t \H(r;g))).
$$
Now Theorem \ref{conform-H} implies that
\begin{multline}\label{almost-0}
(\partial/\partial\varepsilon)|_{\varepsilon=0}
\left(\H(r;e^{2\varepsilon \varphi} g)\right) +
\left(\f\!+\!1\right) \varphi \H(r;g)
- \left(\f\!-\!1\right) \H(r;g) \varphi \\
= - r \frac{1}{2} (\varphi \dot{\H}(r;g) + \dot{\H}(r;g) \varphi)-
[\H(r;g),[\K(r;g),\varphi]].
\end{multline}
Hence we obtain
\begin{multline}\label{almost}
(\partial/\partial \varepsilon)|_{\varepsilon=0}
(\Tr(\exp(t \H(r;e^{2\varepsilon \varphi} g)))) \\
= t \Tr \left(\left[-2\varphi \H(r;g) - r \frac{1}{2} (\varphi
\dot{\H}(r;g) + \dot{\H}(r;g)\varphi) \right] e^{t \H(r;g)} \right)\\
+ t \Tr \left(\left[\left(\f\!-\!1\right) [\H(r;g),\varphi] -
[\H(r;g),[\K(r;g),\varphi]]\right] e^{t\H(r;g)} \right).
\end{multline}
By the cyclicity of the trace, the terms which involve the
commutators with $\H(r;g)$ vanish. Therefore, the right-hand side of
\eqref{almost} simplifies to
\begin{multline*}
-2t \Tr (\varphi \H(r;g) e^{t \H(r;g)}) - r t \frac{1}{2} \Tr
((\varphi \dot{\H}(r;g) + \dot{\H}(r;g)\varphi) e^{t\H(r;g)}) \\
= -2 t(\partial/\partial t) (\Tr (\varphi e^{t \H(r;g)})) - r t
\frac{1}{2} \Tr ((\varphi \dot{\H}(r;g) + \dot{\H}(r;g)\varphi)
e^{t\H(r;g)}).
\end{multline*}
This completes the proof. \hfill $\square$ \smallskip

\begin{rem} For even $n$, the arguments of the above proof are to be
interpreted in the following sense. In this case, Theorem
\ref{conform-H} yields an equality of power series up to $r^{n-2}$.
It implies the equality of the corresponding finite Taylor series of
both sides of \eqref{almost-0}. In turn, this shows the equality of
the finite Taylor expansions up to $r^{n-2}$ of both sides of
\eqref{almost}. The latter argument rests on the repeated
application of the obvious identity $T_N (p \circ q) = T_N (T_N(p)
\circ q)$ for the finite Taylor series $T_N(\cdot)$ up to $r^N$ of
smooth one-parameter families of operators.
\end{rem}

The operator $\H(r;g)$ is defined through a Poincar\'e-Einstein
metric $g_+$ associated to $g$. The Poincar\'e-Einstein metrics
$g_+$ and $\hat{g}_+$ of the conformally equivalent metrics $g$ and
$\hat{g}= e^{2\varphi}g$ are related by
$$
\hat{g}_+ = \kappa^* (g_+)
$$
with a diffeomorphism $\kappa$ of a neighbourhood $[0,\varepsilon)
\times M$ of $M$ which restricts to the identity for $r=0$. It
follows that the left-hand side of \eqref{base-heat} can be
expressed in terms of $g_+$ and $\kappa$. However, from that
perspective, it is surprising that this yields a result as stated in
Theorem \ref{conform-heat}.

\section{Asymptotic expansions}\label{AE}

For small $r$, the operator $\H(r)$ is a self-adjoint elliptic
differential operator. By the compactness of $M$, it has a discrete
spectrum consisting of real eigenvalues $\lambda_k$ of finite
multiplicity. Moreover, we find a constant $\varepsilon > 0$ so that
for $r \in [-\varepsilon,\varepsilon]$, the spectrum of $-\H(r)$ is
uniformly bounded below by a constant $L$. We choose some open cone
$C \subset \c$ of small slope about the ray
$$
\{x \in \r \,|\, x \ge \min(L,0) \}.
$$
The cone $C$ contains the spectrum of $-\H(r)$ for $|r| \le
\varepsilon$. Let $\Gamma$ be the boundary of $C$ and $\Omega$ its
complement. We orient $\Gamma$ so that the spectrum of $-\H(r)$ is
on the right-hand side.

The analytic arguments rest on a theory of pseudo-differential
operators depending on parameters.

\begin{defn}\label{symbol} For $m \in \r$, we let $S^m(U)$ be the
space of symbols $p_m(x,\xi,\lambda;r)$ with $(x,\xi) \in T^*(U)$,
$U \subset \r^n$ open, $\lambda \in \Omega$ and $r \in
[-\varepsilon,\varepsilon]$ depending on the parameters $\lambda$
and $r$ so that
\begin{itemize}
\item [(i)] $p_m(x,\xi,\lambda;r)$ is smooth in the variables
$(x,\xi,r)$,
\item [(ii)] $p_m(x,\xi,\lambda;r)$ is holomorphic in $\lambda \in
\Omega$,
\item [(iii)] for all multi-indices $\alpha, \beta$ and
$\gamma,\delta$, there is a constant
$c_{\alpha,\beta,\gamma,\delta}$ so that
$$
|\partial_x^\alpha \partial_\xi^\beta \partial_\lambda^\gamma
\partial_r^\delta p_m(x,\xi,\lambda;r)| \le c_{\alpha,\beta,\gamma,\delta}
(1+|\xi|^2 + |\lambda|)^{\frac{m-|\beta|}{2} - \gamma}.
$$
\end{itemize}
\end{defn}

We call a symbol $p_m \in S^m$ {\em homogeneous} of order $m \in \r$
if
$$
p_m(x,s\xi,s^2\lambda;r) = s^m p_m(x,\xi,\lambda;r), \; s \ge 1.
$$

\begin{lemm}\label{homogen} If $p_m \in S^m$ is homogeneous of
order $m$, then
\begin{equation}\label{integral-hom} \int_{\r^n}
\int_\Gamma p_m(x,\xi,\lambda;r) e^{-t\lambda} d\lambda d\xi =
t^{-\frac{n+m}{2}-1} \int_{\r^n} \int_\Gamma p_m(x,\xi,\lambda;r)
e^{-\lambda} d\lambda d\xi
\end{equation}
for small $t$.
\end{lemm}

\begin{proof} For $0 < t < 1$, the curve $t^{-1}\Gamma$ is contained in the
complement of the cone $C$. By Cauchy's theorem, the integral on the
left-hand side of \eqref{integral-hom} equals
$$
\int_{\r^n} \int_{t^{-1}\Gamma} p_m(x,\xi,\lambda;r) e^{-t\lambda}
d\lambda d\xi.
$$
Now we replace $\xi$ by $t^{-1/2} \xi$ and $\lambda$ by $t^{-1}
\lambda$. Then the integral equals
$$
t^{-\f-1} \int_{\r^n} \int_\Gamma p_m(x,t^{-1/2}\xi,t^{-1}\lambda;r)
e^{-\lambda} d\lambda d\xi.
$$
Therefore, the assertion follows from the homogeneity
$$
p_m(x,s\xi,s^2\lambda;r) = s^m p_m(x,\xi,\lambda;r), \; s \ge 1.
$$
The proof is complete.
\end{proof}

Note that for $m=-4$, the exponent of $t$ on the right-hand side of
\eqref{integral-hom} is $-\f+1$. That case will be relevant in
Section \ref{dc-term}.

Now Theorem 3.3 of \cite{BO-index} implies

\begin{thm}\label{hk-expansion} Let $K(x,y,t;r)$ be the kernel of
$e^{t\H(r)}$, i.e.,
$$
e^{t \H(r)}(u)(x) = \int_M K(x,y,t;r) u(y) \dvol, \; u \in C^\infty(M).
$$
Then $K(x,x,t;r)$ has an asymptotic expansion
\begin{equation}\label{kernel}
K(x,x,t;r) \sim (4 \pi t)^{-\f} \sum_{j \ge 0} t^j a_{2j}(x;r), \; t
\to 0
\end{equation}
with $a_{2j}(x;r) \in C^\infty (M \times
(-\varepsilon,\varepsilon))$. In particular,
\begin{equation}\label{heat-kernel-as}
\Tr (e^{t\H(r)}) \sim (4\pi t)^{-\f} \sum_{j \ge 0} t^j \int_M a_{2j}(x;r)
\dvol,
\end{equation}
where the trace is taken in $L^2(M,g)$. The coefficients
$a_{2j}(x;r)$ are polynomials in the jets of the symbol of
$\H(r;g)$. Moreover, the left-hand side of \eqref{heat-kernel-as} is
smooth in $r$ and its asymptotic expansion can be differentiated
with respect to $r$ term by term.
\end{thm}

In order to simplify the notation, we suppressed the dependence of
$K(x,y,t;r)$ and $a_{2j}(x;r)$ on the metric. We shall use this
convention throughout.

\begin{proof} It suffices to apply the general results of \cite{BO-index}.
For more details we refer to \cite{BO-index}. The proof of Theorem
3.3 in \cite{BO-index} is a refined version of Gilkey's proof in
\cite{G-book} which establishes the uniformity of the arguments in
an external parameter. Here we only recall the main ideas and the
resulting algorithm for the calculation of the heat kernel
coefficients. The operator $e^{t\H(r)}$ is an operator with smooth
kernel. One can approximate that kernel arbitrarily well by the
smooth kernel of a pseudo-differential operator so that the
asymptotic expansions of the restriction of both kernels to the
diagonal have the same coefficients. We choose local coordinates and
write the symbol of the operator $-\H(r)$ as
$$
\sigma(-\H(r)) = p_2(r) + p_1(r) + p_0(r)
$$
with homogeneous $p_j(r)$. Then we define
$$
r_{-2}(x,\xi,\lambda;r) = (p_2(x,\xi;r)-\lambda)^{-1} =
(\H(r)(x,\xi) - \lambda)^{-1}
$$
and
$$
r_{-2-j} = - r_{-2} \left( \sum_{a > -2-j \atop a+b-|\alpha|=-j}
\frac{1}{\alpha!} \partial_\xi^\alpha (p_b) D_x^\alpha (r_a) \right)
$$
for $j \ge 1$. Here $D_x = i^{-1} \partial_x$. Then $r_{-2-j} \in
S^{-2-j}$ is homogeneous of order $-2-j$. The definitions have the
consequence that the pseudo-differential operator $\R_N(\lambda;r)$
with symbol $r_{-2}(\lambda;r) + \cdots + r_{-2-N}(\lambda;r)$
satisfies
$$
\sigma((-\H(r)-\lambda) \R_N(\lambda;r) - 1) \in S^{-N-1}.
$$
The operators $\R_N(\lambda;r)$ serve as approximations of the
resolvent $\R(\lambda;r)$ of $-\H(r)$. The operators
$$
E_N(t;r) = \frac{1}{2\pi i} \int_\Gamma \R_N(\lambda;r) e^{-\lambda
t} d\lambda
$$
have smooth kernels $K_N(x,y,t;r)$ in the sense that
$$
E_N(t;r)(u)(x) = \int K_N(x,y,t;r) u(y) dy
$$
and their restriction to the diagonal are given by
$$
K_N(x,x,t;r) = (2\pi)^{-n} \frac{1}{2\pi i} \int_{\r^n} \int_\Gamma
(r_{-2} + \cdots + r_{-2-N})(x,\xi,\lambda;r) e^{-\lambda t}
d\lambda d\xi.
$$
Lemma \ref{homogen} shows that
$$
K_N(x,x,t;r) = \sum_{j=0}^N t^{-\frac{n-j}{2}} a_j(x;r) \sqrt{\det
g}
$$
with
\begin{equation}\label{algo}
a_j(x;r) \sqrt{\det g} = (2\pi)^{-n} \frac{1}{2\pi i} \int_{\r^n}
\int_\Gamma r_{-2-j}(x,\xi,\lambda;r) e^{-\lambda} d\lambda d\xi.
\end{equation}
For odd $j$, the integrals in \eqref{algo} vanish. This implies the
existence of the desired asymptotic expansion. \end{proof}

The latter proof also yields an algorithm for the calculation of the
heat kernel coefficients.

\begin{corr}\label{lead} The coefficients $a_{2j}(x;r)$ in
\eqref{heat-kernel-as} are given by the formula
\begin{equation}\label{heat-algo}
\pi^\f a_{2j}(x;r) \sqrt{\det g} = \frac{1}{2\pi i} \int_{\r^n}
\int_\Gamma r_{-2-2j}(x,\xi,\lambda;r) e^{-\lambda} d\lambda d\xi.
\end{equation}
\end{corr}

Now the power series expansions
\begin{equation*}
a_{2j}(x;r) \sim \sum_{k \ge 0} a_{(2j,2k)} r^{2k}
\end{equation*}
define scalar Riemannian invariants $a_{(2j,2k)} \in C^\infty(M)$.

Finally, we note that an easy calculation in polar coordinates for
$\r^n$ shows that \eqref{heat-algo} can be rewritten in the form
\begin{align*}
\pi^\f a_{2j}(x;r) \sqrt{\det g} & = \frac{1}{2\pi i} \int_{S^{n-1}}
\int_0^\infty t^{n-1-2j} \left( \int_\Gamma
r_{-2-2j} (x,\xi,t^{-2} \lambda;r) e^{-\lambda} d\lambda \right) dt d\sigma(\xi) \\
& = \frac{1}{2\pi i} \Gamma\left(\f\!-\!j\right) \frac{1}{2}
\int_{S^{n-1}} \left(\int_\Gamma r_{-2-2j}(x,\xi,\lambda;r)
\lambda^{j-\f} d\lambda \right) d\sigma(\xi) \\
& = \Gamma\left(\f\!-\!j\right) \frac{1}{2} \int_{S^{n-1}}
r_{-n}^{(2j)}(x,\xi;r) d\sigma(\xi),
\end{align*}
where
$$
r_{-n}^{(2j)}(x,\xi;r) \st \frac{1}{2\pi i} \int_\Gamma
r_{-2-2j}(x,\xi,\lambda;r) \lambda^{j-\f} d\lambda
$$
is homogeneous of order $-n$ in $\xi$. These formulas are equivalent
to the formulas for the residues of the spectral zeta-function at
$s=\f-j$ (see \cite{shubin}, Theorem 13.1).

\section{The double-commutator term}\label{dc-term}

In the present section, we discuss the influence of the term
\begin{equation}\label{CT}
-rt \frac{1}{2} \Tr \left((\varphi \dot{\H}(r) + \dot{\H}(r)\varphi)
e^{t\H(r)} \right)
\end{equation}
in \eqref{base-heat} on the leading and sub-leading heat kernel
coefficients of $\H(r)$. By the cyclicity of the trace, the term
\eqref{CT} equals
\begin{equation}\label{CT2}
- rt \frac{1}{2} \Tr \left( \varphi (\dot{\H}(r) e^{t \H(r)} + e^{t
\H(r)} \dot{\H}(r)) \right).
\end{equation}
The further evaluation of this terms rest on the following result.
Let $\R(r;\lambda)$ be the resolvent of $-\H(r)$.\footnote{In the
following, we shall suppress the dependence of the resolvent on
$r$.}

\begin{lemm}\label{double-key} We have
\begin{multline}\label{key-form}
(\partial/\partial r)(e^{t\H(r)}) - t \frac{1}{2} \left(\dot{\H}(r)
e^{t\H(r)} + e^{t\H(r)} \dot{\H}(r)\right) \\ = - \frac{1}{2}
\frac{1}{2\pi i}\int_\Gamma [\R(\lambda),[\R(\lambda),\dot{\H}(r)]]
e^{-t\lambda} d\lambda.
\end{multline}
\end{lemm}

\begin{proof} Differentiating the relation
$$
(-\H(r)\!-\!\lambda) \circ \R(\lambda) = 1
$$
with respect to $r$, yields
$$
-\dot{\H}(r) \circ \R(\lambda) + (-\H(r)\!-\!\lambda) \circ
\dot{\R}(\lambda) = 0,
$$
i.e.,
$$
\dot{\R}(\lambda) = \R(\lambda) \circ \dot{\H}(r) \circ \R(\lambda).
$$
Hence the definition
$$
e^{t\H(r)} = \frac{1}{2 \pi i} \int_\Gamma e^{- t\lambda}
\R(\lambda) d\lambda
$$
implies
$$
(\partial/\partial r)(e^{t\H(r)}) = \frac{1}{2\pi i} \int_\Gamma
e^{-t \lambda} (\partial/\partial r)(\R(\lambda)) d\lambda =
\frac{1}{2\pi i} \int_\Gamma e^{-t \lambda} \R(\lambda) \dot{\H}(r)
\R(\lambda) d \lambda.
$$
On the other hand, by partial integration, we find that
$$
t \left(\dot{\H}(r) e^{t\H(r)} + e^{t\H(r)} \dot{\H}(r)\right)
$$
equals
\begin{align*}
& - \frac{1}{2\pi i} \int_\Gamma (\partial /\partial \lambda)
(e^{-t\lambda}) (\dot{\H}(r) \R(\lambda) + \R(\lambda) \dot{\H}(r)) d\lambda \\
& = \frac{1}{2\pi i} \int_\Gamma e^{-t\lambda} \left(\dot{\H}(r)
\partial \R(\lambda)/ \partial \lambda +
\partial \R(\lambda)/\partial \lambda \dot{\H}(r) \right) d\lambda \\
& = \frac{1}{2\pi i} \int_\Gamma e^{-t\lambda} (\dot{\H}(r)
\R^2(\lambda) + \R^2(\lambda)\dot{\H}(r)) d\lambda.
\end{align*}
Thus, the left-hand side of \eqref{key-form} equals
\begin{align*}
& \frac{1}{2\pi i} \int_\Gamma e^{-t \lambda} \left(\R(\lambda)
\dot{\H}(r) \R(\lambda) - \frac{1}{2} \left(\dot{\H}(r)
\R^2(\lambda) + \R^2(\lambda) \dot{\H}(r) \right) \right) d\lambda \\
& = - \frac{1}{2} \frac{1}{2\pi i} \int_\Gamma e^{-t \lambda}
[\R(\lambda),[\R(\lambda),\dot{\H}(r)]] d\lambda.
\end{align*}
The proof is complete. \end{proof}

Lemma \ref{double-key} is an identity of operators with smooth
kernels. By taking traces, it implies the relation
\begin{multline}\label{base-red}
-t\frac{1}{2} \Tr \left(\varphi (\dot{\H}(r) e^{t\H(r)} + e^{t\H(r)}
\dot{\H}(r)) \right) = -(\partial/\partial r)(\Tr(\varphi
e^{t\H(r)})) - \frac{1}{2} \Tr (\varphi \C(t;r)),
\end{multline}
where
\begin{equation}\label{dc}
\C(t;r) = \frac{1}{2\pi i} \int_\Gamma
[\R(\lambda),[\R(\lambda),\dot{\H}(r)]] e^{-t\lambda} d\lambda.
\end{equation}
$\C(t;r)$ is an operator with smooth kernel. Let $\C(x,y,t;r)$ be
its kernel. Note that \eqref{base-red} for $\varphi = 1$ and the
cyclicity of the trace show that
$$
\int_M \C(x,x,t;r) \dvol_g = 0,
$$
i.e., the function $x \mapsto \C(x,x,t;r)$ is a total divergence
(see also Remark \ref{div-quick}).

Now we combine the variational formula in Theorem \ref{conform-heat}
with \eqref{base-red} and find
\begin{multline*}
(\Tr(\exp (t \H(r;g))))^\bullet[\varphi] = -2t(\partial/\partial t)
(\Tr(\varphi e^{t\H(r;g)})) \\ -r(\partial/\partial r) (\Tr(\varphi
e^{t\H(r;g)})) - r \frac{1}{2} \int_M \varphi(x) \C(x,x,t;r)\dvol_g.
\end{multline*}
By comparing coefficients in the small $t$ expansion, we obtain the
conformal variational formulas
\begin{equation}\label{conf-hc}
\left(\int_M a_{2j}(x;r) \dvol \right)^\bullet[\varphi] =
(n\!-\!2j\!-\!r \partial/\partial r) \int_M \varphi a_{2j}(x;r)
\dvol -r \frac{1}{2} \int_M \varphi c_{2j}(x;r) \dvol,
\end{equation}
where
\begin{equation}\label{corr-as}
\C(x,x,t;r) \sim (4\pi t)^{-\f} \sum_{j \ge 1} t^j c_{2j}(x;r), \; t
\to 0
\end{equation}
is the asymptotic expansion of the restriction of the kernel $\C$ to
the diagonal. By the above arguments, the functions $c_{2j}(x;r)$
are total divergencies on $M$. Thus, by expansion in $r$, we have
proved the following result. Let
\begin{equation}\label{taylor-c}
c_{2j}(r) \sim \sum_{k \ge 1} c_{(2j,2k)} r^{2k-1}, \; c_{(2j,2k)}
\in C^\infty(M).
\end{equation}

\begin{prop}\label{crit-van} For closed manifolds $M$ of even dimension $n$
and $j=0,\dots,n/2$, the conformal variation
$$
\left(\int_M a_{(2j,n-2j)} \dvol \right)^\bullet[\varphi]
$$
vanishes at $g$ iff $c_{(2j,n-2j)}=0$ at $g$.
\end{prop}

Since $\C$ vanishes for Einstein metrics, the latter result proves
the second part of Proposition \ref{critical-even}.

Moreover, a further set of metrics with vanishing $\C$ is provided
by metrics of the form $g = g_1 + g_2$ on product spaces $M_1^p
\times M_2^q$ with Einstein metrics $g_i$ on the factors so that
$$
\frac{\scal(g_1)}{p(p\!-\!1)} = - \frac{\scal(g_2)}{q(q\!-\!1)}.
$$
In fact, the latter vanishing result is a consequence of
\eqref{hol-spec}.

We also emphasize that the expansion \eqref{corr-as} does not
contain a contribution by $t^{-\f}$. In order to prove the claim, we
approximate the resolvent $\R(\lambda)$ by a pseudo-differential
operator (as in Section \ref{AE}). The coefficients of the expansion
are given in terms of the symbol of the resulting
pseudo-differential integrand. Since the double-commutator integrand
yields a pseudo-differential operator of order $-4$, the vanishing
of $c_0$ follows from Lemma \ref{homogen}.

The proof of Theorem \ref{A} requires to find the divergence term
$c_2(x;r)$. It is given in terms of the principal symbol of the
(approximation of the) double-commutator
\begin{equation}\label{dcc}
[\R(\lambda),[\R(\lambda),\dot{\H}(r)]].
\end{equation}
More precisely, we have
\begin{equation}\label{c2}
c_2(x;r) = \pi^{-\f} \frac{1}{2\pi i} \int_{\r^n} \left( \int_\Gamma
\sigma_{-4}([\R(\lambda),[\R(\lambda),\dot{\H}(r)]])(x,\xi)
e^{-\lambda} d\lambda \right) d\xi / \sqrt{\det g}.
\end{equation}
Here we abuse notation by identifying the resolvent with its
pseudo-differential approximations.

Now we recall (see \cite{G-book} or \cite{shubin}) that the
principal symbol of the commutator $[A,B]$ of two classical
pseudo-differential operators $A$ and $B$ with respective principal
symbols $\sigma_A$ and $\sigma_B$ is given by
$$
i \left\{ \sigma_A,\sigma_B \right\}.
$$
Here for $f,g \in C^\infty(T^*M)$, their Poisson bracket $\{f,g\}$
is given in local coordinates by
\begin{equation}\label{Poisson}
\{ f,g \} = \sum_i \frac{\partial f}{\partial x_i} \frac{\partial
g}{\partial \xi_i} - \frac{\partial f}{\partial \xi_i}
\frac{\partial g}{\partial x_i}.
\end{equation}

It follows that the principal symbol of the double-commutator
\eqref{dcc} is given by the {\em negative} of the double Poisson
bracket of the principal symbols of the operators involved. For its
calculation we apply the following result.

\begin{lemm}\label{DPB} For functions $G,H \in C^\infty(T^*M)$
so that $G$ does not vanish on $T^*M$, we have
$$
\left\{ 1/G, \left\{ 1/G, H \right\} \right\} = G^{-4}
\left\{G,\left\{G,H\right\}\right\}.
$$
\end{lemm}

\begin{proof} The result follows by direct calculation. For homogeneous
polynomials $G$ and $H$ as in the assertion, the identity follows by
comparing principal symbols in the operator identity
$$
\left[P^{-1},\left[P^{-1},Q\right]\right] = P^{-2} [P,[P,Q]] P^{-2}
$$
for pseudo-differential operators $P$ and $Q$ with the respective
principal symbols $H$ and $G$.
\end{proof}

\begin{rem}\label{div-quick} The operator-argument in the proof of
Lemma \ref{DPB} also yields a quick alternative proof of the fact
that
$$
\Tr \C(t;r) = 0.
$$
Indeed, the relation
$$
[\R(\lambda),[\R(\lambda),\dot{\H}]] = \R(\lambda)^2
[-\H\!-\!\lambda,[-\H\!-\!\lambda,\dot{\H}]] \R(\lambda)^2 =
\R(\lambda)^2 [\H,[\H,\dot{\H}]] \R(\lambda)^2
$$
implies
$$
\Tr \C(t;r) = \frac{1}{2\pi i} \Tr \left( \int_\Gamma \R(\lambda)^4
[\H,[\H,\dot{\H}]] e^{-t\lambda} d\lambda \right).
$$
By Cauchy's theorem, this integral is a multiple of
$$
\Tr \left( [\H,[\H,\dot{\H}]] e^{t\H}\right) = 0.
$$
\end{rem}

Now we apply Lemma \ref{DPB} to the symbols $G =
\H(r)(x,\xi)-\lambda$ and $H=\dot{\H}(r)(x,\xi)$.

The principal symbol of the operator
$[\R(\lambda),[\R(\lambda),\dot{\H}(r)]]$ is given by the double
Poisson bracket
$$
\{(\H(r)(x,\xi)-\lambda)^{-1},\{(\H(r)(x,\xi)-\lambda)^{-1},
\dot{\H}(r)(x,\xi)\}\};
$$
recall that the operator $-\H(r)$ has principal symbol
$\H(r)(x,\xi)$. Lemma \ref{DPB} shows that this function equals
\begin{align*}
& (\H(r)(x,\xi)\!-\!\lambda)^{-4}
\{\H(r)(x,\xi)\!-\!\lambda,\{\H(r)(x,\xi)\!-\!\lambda, \dot{\H}(r)(x,\xi)\}\} \\
& = (\H(r)(x,\xi)\!-\!\lambda)^{-4} \{\H(r)(x,\xi),\{\H(r)(x,\xi),
\dot{\H}(r)(x,\xi)\}\}.
\end{align*}
Hence
$$
c_2(x) = \pi^{-\f} \frac{1}{2\pi i} \int_{\r^n} \left( \int_\Gamma
\frac{\{\H(r)(x,\xi),\{\H(r)(x,\xi),
\dot{\H}(r)(x,\xi)\}\}}{(\H(r)(x,\xi)\!-\!\lambda)^4} e^{-\lambda}
d\lambda \right) d\xi / \sqrt{\det g}.
$$
By Cauchy's formula, the latter equation is equivalent to
\begin{equation}\label{c2-inter}
c_2(x) \sqrt{\det g} = \pi^{-\f} \frac{1}{3!} \int_{\r^n}
\{\H(r)(x,\xi),\{\H(r)(x,\xi), \dot{\H}(r)(x,\xi)\}\}
e^{-\H(r)(x,\xi)} d\xi.
\end{equation}
The integral in \eqref{c2-inter} will be calculated in Section
\ref{proof-B}.

\section{Proof of Theorem \ref{conf-reno}}\label{proof-A}

In order to prove Theorem \ref{conf-reno}, we first determine the
leading heat kernel coefficient of the holographic Laplacian.

\begin{lemm}\label{top} The leading heat kernel coefficient
of $\H(r)$ satisfies
\begin{equation}\label{top-heat}
a_0 (r) = v(r).
\end{equation}
\end{lemm}

\begin{proof} By Corollary \ref{lead}, the leading coefficient in
\eqref{kernel} is given by
$$
\pi^{-\f} \int_{\r^n} \int_\Gamma (\H(r)(x,\xi)-\lambda)^{-1}
e^{-\lambda} d\lambda d\xi / \sqrt{\det{g}}.
$$
By Cauchy's formula, the latter integral equals
$$
\pi^{-\f} \int_{\r^n} e^{-\H(r)(x,\xi)} d \xi / \sqrt{\det{g}}.
$$
But for the Gaussian integral we obtain\footnote{see (22) on page 15
of \cite{zee}}
\begin{equation}\label{gaussian}
\int_{\r^n} e^{-\H(r)(x,\xi)} d\xi = 2^{-\f} (2\pi)^\f
(\sqrt{\det{g(r)^{-1}}})^{-1} = \pi^\f \sqrt{\det g(r)}.
\end{equation}
Thus, we finally arrive at
$$
a_0(r) = \sqrt{\det g(r)}/ \sqrt{\det g} = v(r).
$$
The proof is complete.
\end{proof}

\begin{rem} The coefficient
$$
(4\pi)^{-\f} \int_{M^n} v(r) \dvol_g
$$
of $t^{-\f}$ can be rewritten as the symplectic volume
$$
(2\pi)^{-n} \frac{1}{2} \Gamma \left(\f\right) \int_{\H(r)=1} dx
d\xi,
$$
of the hypersurface $\{(x,\xi) \in T^*M \,|\, \H(r)(x,\xi)=1\}$ in
$T^*M$. In fact, an easy calculation in local coordinates shows that
$$
2 \int_{\r^n} e^{-\H(r)(x,\xi)} d\xi = \Gamma \left(\f\right)
\int_{\H(r)(x,\xi) = 1} d\sigma(\xi).
$$
Comparison with \eqref{gaussian} yields
$$
\Gamma \left(\f\right) \int_{\H(r)(x,\xi)=1} d\sigma(\xi) = 2
\pi^{\f} \sqrt{\det (g(r))}.
$$
Hence
$$
\frac{1}{2} (2\pi)^{-n} \Gamma\left(\f\right) \int_{\H(r)(x,\xi)=1}
d\sigma(\xi) = (4\pi)^{-\f} \sqrt{\det (g(r))}
$$
and the assertion follows. This observation shows that Lemma
\ref{top} is a special case of the formula for the leading
coefficient in the asymptotic expansion of the trace of the heat
kernel for general elliptic self-adjoint operators given in
Corollary 2.2' of \cite{DG}.
\end{rem}

Now we complete the {\bf proof of Theorem \ref{conf-reno}}. By
\eqref{conf-hc} and $c_0 = 0$, we have
\begin{equation*}
\left(\int_M a_0(x;r) \dvol_g \right)^\bullet[\varphi] \\
= (n\!-\!r\partial/\partial r) \int_M \varphi a_0(x;r) \dvol_g.
\end{equation*}
Thus, by expansion in $r$, Lemma \ref{top} implies
\begin{equation*}
\left(\int_M v_{2k} \dvol_g \right)^\bullet[\varphi] \\
= (n\!-\!2k) \int_M \varphi v_{2k} \dvol_g.
\end{equation*}
This completes the proof. \hfill$\square$
\medskip

Some comments on the above proof of Theorem \ref{conf-reno} are in order.
In the mathematical literature, the variational formula \eqref{CF-CT}
first appeared in \cite{CF} (see (23)). The proof of Chang and Fang rests
on a description of the infinitesimal variation of the Fefferman-Graham
diffeomorphism of the bulk-space which is associated to a conformal change
of the metric on the boundary. Graham's proof in \cite{G-ext} may be
considered also as a version of that argument. The present proof
completely differs from these arguments. It replaces the consideration of
the diffeomorphisms by the conformal variational formula in Theorem
\ref{conform-H} which in turn is based on the identification of the
holographic Laplacian with the generating function of the building-block
operators. It seems difficult to prove Theorem \ref{conform-H} by using
the Fefferman-Graham diffeomorphism of the bulk-space.

\section{The Gaussian integral}\label{proof-B}

The present section is devoted to a proof of Theorem \ref{B}.

We start by fixing conventions concerning curvature and normal
coordinates. For a given metric $g$, let
$$
R(X,Y)(Z) = \nabla^g_X \nabla^g_Y (Z) - \nabla^g_Y \nabla^g_X (Z) -
\nabla^g_{[X,Y]}(Z)
$$
for vector fields $X,Y,Z \in \X(M)$ and set
$$
R(\partial_i,\partial_j)(\partial_k) = \sum_{l} R_{ijk}{}^l
\partial_l.
$$
for a local frame $\{\partial_i\}$ defined by local coordinates
$\left\{x_1,\dots,x_n\right\}$. Then in normal coordinates
$\left\{x_1, \dots, x_n\right\}$ for $g$ around $m$, we have the
Taylor expansion
\begin{equation}\label{normal}
g_{ij}(x) = g_{ij}(m) + \frac{1}{3} \sum_{r,s} x_r x_s R_{risj}(m) +
\cdots
\end{equation}
near $m$ with $g_{ij}(m) = \delta_{ij}$. In particular, all
first-order derivatives $\partial_k (g_{ij})$ vanish at $m$ and
\begin{equation}\label{normal-der}
\partial_a \partial_b (g_{ij})(m) = \frac{1}{3}(R_{aibj}(m) + R_{ajbi}(m)).
\end{equation}
For a proof of \eqref{normal} see \cite{G-GP}, Lemma 1.11.4 or
\cite{BGV}, Proposition 1.28. Finally, we define the Weyl tensor $W$
by
\begin{equation}\label{R-W}
R = W - (\Rho \owedge g).
\end{equation}
Here $a \owedge b$ denotes the Kulkarni-Nomizu product of the symmetric
bilinear forms $a$ and $b$ which is defined by
\begin{equation}\label{KN-product}
(a \owedge b)_{ijkl} \st a_{ik} b_{jl} - a_{jk} b_{il} + a_{jl}
b_{ik} - a_{il} b_{jk}.
\end{equation}

In the first step of the proof of Theorem \ref{B}, we make explicit
the double Poisson bracket defining the integrand. We use local
coordinates $\{x_i,\xi_i\}$. Then, using \eqref{Poisson}, the
integrand is given by the sum of
\begin{multline}\label{DP-1}
\sum_{a,b,c,d} \sum_{\alpha,\beta,\gamma,\delta}
\frac{\partial}{\partial x_b} \left(g(r)^{-1}_{cd}\right) \xi_c
\xi_d \Big[\frac{\partial}{\partial x_a} \left(g(r)^{-1}_{\alpha
\beta}\right) \dot{g}(r)^{-1}_{\gamma \delta}
\frac{\partial}{\partial \xi_b} \left(\xi_\alpha \xi_\beta
\frac{\partial}{\partial \xi_a}(\xi_\gamma \xi_\delta)\right) \\ -
g(r)^{-1}_{\alpha \beta} \frac{\partial}{\partial x_a}
\left(\dot{g}(r)^{-1}_{\gamma \delta}\right)
\frac{\partial}{\partial \xi_b} \left(\frac{\partial}{\partial
\xi_a} (\xi_\alpha \xi_\beta) \xi_\gamma \xi_\delta \right)\Big]
\end{multline}
and
\begin{multline}\label{DP-2}
-\sum_{a,b,c,d} \sum_{\alpha,\beta,\gamma,\delta} g(r)^{-1}_{cd}
\frac{\partial}{\partial \xi_b} (\xi_c \xi_d) \Big[
\frac{\partial}{\partial x_b} \left(\frac{\partial}{\partial x_a}
\left(g(r)^{-1}_{\alpha \beta}\right) \dot{g}(r)^{-1}_{\gamma
\delta} \right) \xi_\alpha \xi_\beta \frac{\partial}{\partial
\xi_a}(\xi_\gamma \xi_\delta) \\ - \frac{\partial}{\partial x_b}
\left( g(r)^{-1}_{\alpha \beta} \frac{\partial}{\partial x_a} \left(
\dot{g}(r)^{-1}_{\gamma \delta} \right)\right)
\frac{\partial}{\partial \xi_a} (\xi_\alpha \xi_\beta) \xi_\gamma
\xi_\delta \Big].
\end{multline}
Since the left-hand side of \eqref{e-value} does not depend on the
choice of coordinates, it suffices to prove that relation for any
fixed $m \in M$, sufficiently small $r \ge 0$ and normal coordinates
around $m$ for the metric $g(r)$ on $M$. Then the sum \eqref{DP-1}
vanishes and the sum \eqref{DP-2} simplifies to
\begin{multline}\label{integrand}
- \sum_{a,b,c,d} \sum_{\alpha,\beta,\gamma,\delta} \delta_{cd}
\xi_\alpha \xi_\beta \frac{\partial}{\partial \xi_b}(\xi_c \xi_d)
\frac{\partial}{\partial \xi_a}(\xi_\gamma \xi_\delta) \left[
\frac{\partial^2}{\partial x_b \partial x_a}
\left(g(r)^{-1}_{\alpha \beta}\right) \dot{g}(r)^{-1}_{\gamma \delta}\right] \\
+ \sum_{a,b,c,d} \sum_{\alpha,\beta,\gamma,\delta} \delta_{cd}
\delta_{\alpha \beta} \xi_\gamma \xi_\delta \frac{\partial}{\partial
\xi_b}(\xi_c \xi_d) \frac{\partial}{\partial \xi_a}(\xi_\alpha
\xi_\beta) \left[\frac{\partial^2}{\partial x_a \partial
x_b}(\dot{g}(r)^{-1}_{\gamma\delta})\right].
\end{multline}
The first sum can be expressed in terms of curvature. In fact, by
\eqref{normal-der} it equals
\begin{multline*}
- \sum_{a,b} \sum_{\alpha,\beta,\gamma} \xi_\alpha \xi_\beta
\frac{\partial}{\partial \xi_b}(|\xi|^2) \frac{\partial}{\partial
\xi_a}(\xi_\gamma \xi_\delta) \left[ \frac{\partial^2}{\partial x_a
\partial x_b} \left(g(r)^{-1}_{\alpha \beta} \right)
\dot{g}(r)^{-1}_{\gamma a} \right] \\
= - 4 \sum_{a,b} \sum_{\alpha,\beta,\gamma} \xi_\alpha \xi_\beta
\xi_\gamma \xi_b \frac{\partial^2}{\partial x_a \partial x_b}
\left( g(r)^{-1}_{\alpha \beta} \right) \dot{g}(r)^{-1}_{\gamma a} \\
= \frac{4}{3} \sum_{a,b} \sum_{\alpha,\beta,\gamma} \xi_\alpha
\xi_\beta \xi_\gamma \xi_b (R_{a \alpha b \beta} + R_{a \beta b
\alpha} ) \dot{g}(r)^{-1}_{\gamma a}.
\end{multline*}
Now after integration these terms sum up to $0$. In fact, by the
well-known formula\footnote{see (24) on page 15 of \cite{zee} or
(83) on page 259 of \cite{Fo}.}
\begin{equation}\label{quartic}
\int_{\r^n} (\xi_a \xi_b \xi_c \xi_d) e^{-|\xi|^2} d\xi \big/ \int_{\r^n}
e^{-|\xi|^2} d\xi = \frac{1}{4} (\delta_{ab} \delta_{cd} + \delta_{ac}
\delta_{bd} + \delta_{ad} \delta_{bc}),
\end{equation}
the integral is a multiple of
\begin{align*}
& \sum_{a,b} \sum_{\alpha,\beta,\gamma} (\delta_{\alpha \beta}
\delta_{\gamma b} + \delta_{\alpha \gamma} \delta_{\beta b} +
\delta_{\alpha b} \delta_{\beta \gamma}) (R_{a\alpha b \beta} + R_{a
\beta b \alpha}) \dot{g}(r)^{-1}_{\gamma a} \\
& = \sum_{a,\gamma} \sum_{\alpha} (R_{a \alpha \gamma \alpha} + R_{a
\alpha \gamma \alpha}) \dot{g}(r)^{-1}_{\gamma a} \\
& + \sum_{a,\gamma} \sum_{b} (R_{\alpha \gamma b b} + R_{a b b
\gamma}) \dot{g}(r)^{-1}_{\gamma a} \\
& + \sum_{a,\gamma} \sum_{b} (R_{a b b \gamma} + R_{a \gamma bb})
\dot{g}(r)^{-1}_{\gamma a} \\
& = 0.
\end{align*}

It remains to evaluate the terms which involve the second-order
derivatives of $\dot{g}(r)$. The following results will play a
central role in this evaluation.\footnote{In what follows, $\delta$
also denotes the divergence on symmetric bilinear forms.}

\begin{lemm}\label{d-div} We have
\begin{equation}\label{higher}
\delta_{g(r)}(\dot{g}(r)) =  d \tr_{g(r)}(\dot{g}(r))
\end{equation}
as an identity of one-parameter families of $1$-forms on $M$.
\end{lemm}

\begin{proof} Let $\bar{\Rho}$ be the Schouten tensor of the metric
$\bar{g} = dr^2 + g(r)$. Then
\begin{equation}\label{Sch-bar}
\bar{\Rho} = - \frac{1}{2r}\dot{g}(r)
\end{equation}
(see \cite{juhl-ex}, Lemma 7.1). Hence
\begin{equation}\label{J-bar}
\bar{\J} = \tr_{\bar{g}}(\bar{\Rho}) = - \frac{1}{2r} \tr_{g(r)}
(\dot{g}(r)) = - \frac{1}{2r} \sum_{i,j} \dot{g}(r)_{ij} g(r)^{ij}.
\end{equation}
We combine these formulas with the relation
\begin{equation}\label{B-Sch-bar}
\delta_{\bar{g}}(\bar{\Rho}) = d \bar{\J}
\end{equation}
of $1$-forms on the product space $[0,\varepsilon) \times M$. We evaluate
the relation \eqref{B-Sch-bar} on tangential vectors of $M$. Let
$\{\partial_0 = \partial/\partial r, \partial_i\}$ be a local frame with
tangential vectors $\partial_i$, $i=1,\dots,n$ of $M$. Let $\nabla_i =
\nabla_{\partial_i}$ and $\nabla_0^{\bar{g}} =
\nabla^{\bar{g}}_{\partial_0}$. Then, using $\nabla^{\bar{g}}_0
(\partial_0) = 0$ and $g(r)(\partial_0,\cdot)=0$, we find
\begin{equation}\label{div}
\delta_{\bar{g}}\left(-\frac{1}{2r} \dot{g}(r)\right)(\cdot) =
-\frac{1}{2r} \sum_{i,j = 1}^n g(r)^{ij} \nabla_i^{\bar{g}}
(\dot{g}(r))(\partial_j,\cdot)
\end{equation}
on all tangent vectors. Now by
$$
\nabla_i^{\bar{g}}(\partial_j) = \nabla_i^{g(r)}(\partial_j) +
\bar{\Gamma}_{ij}^0 \partial_0,
$$
the right-hand side of \eqref{div} coincides with
$$
-\frac{1}{2r} \delta_{g(r)}(\dot{g}(r))(\cdot)
$$
on tangential vectors of $M$. Hence \eqref{Sch-bar}, \eqref{J-bar} and
\eqref{B-Sch-bar} imply that
$$
-\frac{1}{2r} \delta_{g(r)}(\dot{g}(r)) = \delta_{\bar{g}}(\bar{\Rho}) = d
\bar{\J} = - \frac{1}{2r} d \tr_{g(r)}(\dot{g}(r))
$$
on tangential vectors of $M$. This proves the asserted identity.
\end{proof}

\begin{rem} We evaluate \eqref{B-Sch-bar} on the normal vector field
$\partial_0 = \partial/\partial r$. First, \eqref{div} and
$$
\bar{\Gamma}_{i0}^k = \frac{1}{2} \sum_{k,l} \dot{g}(r)_{li} g(r)^{lk}
$$
imply
\begin{align}\label{eval-a}
\delta_{\bar{g}}\left(-\frac{1}{2r} \dot{g}(r) \right)(\partial_0) & =
-\frac{1}{2r} g(r)^{ij} \nabla_i^{\bar{g}}(\dot{g}(r))(\partial_j,\partial_0) \\
& = \frac{1}{2r} \sum_{i,j,k} g(r)^{ij} \bar{\Gamma}_{i0}^k \dot{g}(r)_{jk} \nonumber \\
& = \frac{1}{4r} \sum_{i,j,k,l} g(r)^{ij} \dot{g}(r)_{li} g(r)^{lk}
\dot{g}(r)_{jk} \nonumber \\
& = \frac{1}{4r} \sum_{j,l} \dot{g}(r)_l^j \dot{g}(r)_j^l.
\end{align}
Here we raised indices using $g(r)$. Next, a calculation using
\eqref{J-bar} shows that
\begin{align}\label{eval-b}
(d\bar{\J})(\partial_0) & = (\partial/\partial r)
\left(-\frac{1}{2r} \sum_{i,j} \dot{g}(r)_{ij} g(r)^{ij} \right) \nonumber \\
& = \frac{1}{2r^2} \sum_i \dot{g}(r)_i^i - \frac{1}{2r} \sum_i
\ddot{g}(r)_i^i + \frac{1}{2r} \sum_{i,j} \dot{g}(r)_i^j
\dot{g}(r)_j^i.
\end{align}
Hence the equality of \eqref{eval-a} and \eqref{eval-b} is
equivalent to the trace identity
\begin{equation}\label{eval-normal}
\sum_i \ddot{g}(r)_i^i = \frac{1}{r} \sum_i \dot{g}(r)_i^i +
\frac{1}{2} \sum_{i,j} \dot{g}(r)_i^j \dot{g}(r)_j^i,
\end{equation}
i.e.,
$$
\tr (\ddot{g}(r)) = \frac{1}{r} \tr (\dot{g}(r)) + \frac{1}{2} \tr
(\dot{g}(r)^2)
$$
where all traces are taken with respect to $g(r)$.
\end{rem}

In addition, we need the following two general identities.

\begin{lemm}\label{fund} In normal coordinates
$\left\{x_1, \dots, x_n\right\}$ for $g$ around $m$, any symmetric
bilinear form $h$ satisfies the relations
\begin{equation}\label{gen-id-1}
\sum_{a,b} \partial_a \partial_a (h_{bb}) = \Delta_{g} (\tr_{g}(h))
- \frac{2}{3} \sum_{a,b,c} h_{ab} R_{cabc}
\end{equation}
and
\begin{equation}\label{gen-id-2}
\sum_{a,b} \partial_a \partial_b (h_{ab}) =
\delta_{g} (\delta_{g}(h)) + \frac{1}{3} \sum_{a,b,c} h_{ab}
R_{cabc}
\end{equation}
at $m$.
\end{lemm}

\begin{proof} The relation \eqref{gen-id-1} follows from the calculation
\begin{align*}
\Delta_g (\tr_g (h)) = \sum_a \partial_a \partial_a \left(\sum_{b,c}
h_{bc} g^{bc}\right) & = \sum_{a,b,c} \partial_a \partial_a (h_{bc})
\delta_{bc} + \sum_{a,b,c} h_{bc} \partial_a \partial_a (g^{bc}) \\
& = \sum_{a,b} \partial_a \partial_a(h_{bb}) - \frac{2}{3}
\sum_{a,b,c} h_{bc} R_{abac}
\end{align*}
at $m$. For the proof of the second relation, we recall that
\begin{align*}
\delta_g(h)(X) & = \sum_{i,j} g^{ij} \nabla_i^g (h)(\partial_j,X) \\
& = \sum_{i,j} g^{ij} (\partial_i (h(\partial_j,X)) - h(\nabla^g_i
(\partial_j), X) - h(\partial_j, \nabla^g_i (X)))
\end{align*}
and calculate
\begin{align*}
\delta_g (\delta_g(h)) & = \sum_{i,j} g^{ij} \nabla^g_{i}
(\delta_g(h))(\partial_j) \\
& = \sum_{i,j} g^{ij} (\partial_i (\delta_g(h)(\partial_j))
- \delta_g(h) (\nabla^g_{i} (\partial_j))) \\
& = \sum_{i,j} g^{ij} \partial_i \left( \sum_{r,s} g^{rs} (\partial_r
(h_{sj}) - h(\nabla^g_r (\partial_s),\partial_j) - h(\partial_s,\nabla^g_{r}(\partial_j))) \right)\\
& - \sum_{i,j} g^{ij} \left( \sum_{r,s} g^{rs} (\partial_r
(h(\partial_s,\nabla^g_i (\partial_j))) - h (\nabla^g_{r} (\partial_s),
\nabla^g_{i} (\partial_j)) - h(\partial_s, \nabla^g_r \nabla^g_i (
\partial_j)))\right).
\end{align*}
Since the first-order derivatives of the metric $g$ vanish at $m$, all
Christoffel symbols vanish at $m$, and the latter sum simplifies to
\begin{multline*}
\sum_{i,r} \partial_i \partial_r (h_{ri}) \\ - \sum_{i,r} \partial_i
(h(\nabla^g_r (\partial_r),\partial_i)) + \partial_i
(h(\partial_r,\nabla^g_{r} (\partial_i))) +
\partial_r(h(\partial_r,\nabla^g_i (\partial_i))) -
h(\partial_r,\nabla^g_{r} \nabla^g_{i} (\partial_i)).
\end{multline*}
Now the vanishing of the Christoffel symbols implies the relation
$$
\sum_i \partial_i (h(\partial_i,X)) = \delta(h)(X) + \sum_i
h(\partial_i,\nabla^g_{i}(X))
$$
at $m$. Hence we get
\begin{multline*}
\delta_g (\delta_g(h)) = \sum_{i,r} \partial_i \partial_r (h_{ir}) -
\sum_{i,r} \partial_i (h(\partial_r,\nabla^g_{r} (\partial_i))) \\
- \sum_{i,r} h(\nabla^g_i \nabla^g_r (\partial_r),\partial_i) - \sum_{i,r}
h(\partial_r, \nabla^g_{r} \nabla^g_i (\partial_i)) + \sum_{i,r}
h(\partial_r,\nabla^g_r \nabla^g_{i}(\partial_i)),
\end{multline*}
i.e.,
\begin{equation}\label{inter}
\delta_g (\delta_g(h)) = \sum_{i,r} \partial_i \partial_r (h_{ir})
-\sum_{i,r} \partial_i (h(\partial_r,\nabla^g_r (\partial_i))) -
\sum_{i,r} h(\partial_i,\nabla^g_i \nabla^g_r (\partial_r)) .
\end{equation}
Now
$$
\partial_r (\Gamma^s_{ij}) = \frac{1}{2} \left(\partial_i \partial_r (g_{js})
+ \partial_j \partial_r (g_{si}) - \partial_s \partial_r
(g_{ij})\right)
$$
and \eqref{normal-der} imply that the derivatives of Christoffel
symbols are given by
$$
\partial_r (\Gamma^s_{ij}) = \frac{1}{3} \left( R_{jrsi} + R_{jsri} \right)
$$
at $m$. Hence
$$
\sum_r \nabla_i^g \nabla_r^g (\partial_r) = \sum_{r,s}
\partial_i (\Gamma_{rr}^s) \partial_s = -\frac{2}{3} \sum_{r,s} R_{rirs}
\partial_s
$$
and \eqref{inter} simplifies to
\begin{align*}
\delta_g (\delta_g(h)) & = \sum_{i,r} \partial_i
\partial_r (h_{ir}) - \sum_{i,r,s} \partial_i (\Gamma^s_{ri}) h_{rs}
+ \frac{2}{3} \sum_{i,r,s} R_{rirs} h_{is} \\
& = \sum_{i,r} \partial_i \partial_r (h_{ir}) - \frac{1}{3}
\sum_{i,r,s} R_{isir} h_{rs} + \frac{2}{3} \sum_{i,r,s} R_{rirs}
h_{is} \\
& = \sum_{i,r} \partial_i \partial_r (h_{ir}) - \frac{1}{3}
\sum_{a,b,c} R_{cabc} h_{ab}.
\end{align*}
This proves \eqref{gen-id-2}.
\end{proof}

Applying Lemma \ref{fund} for the metric $g=g(r)$ and the symmetric
bilinear form $h=\dot{g}(r)$, proves

\begin{lemm}\label{sum-1} In normal coordinates
$\left\{x_1, \dots, x_n\right\}$ for $g(r)$ around $m$, we have
\begin{equation}
\sum_{a,b} \partial_a \partial_a (\dot{g}(r)_{bb}) + 2 \sum_{a,b}
\partial_a \partial_b (\dot{g}(r)_{ab})
= \Delta_{g(r)} \left(\tr_{g(r)}(\dot{g}(r))\right)
+ 2 \delta_{g(r)} (\delta_{g(r)}(\dot{g}(r)))
\end{equation}
at $m$.
\end{lemm}

Now Lemma \ref{sum-1} yields

\begin{lemm}\label{sum-2} In normal coordinates
$\left\{x_1, \dots, x_n\right\}$ for $g(r)$ around $m$, we have
\begin{equation*}
\sum_{a,b} \partial_a \partial_a (\dot{g}(r)^{-1}_{bb})
+ 2 \sum_{a,b} \partial_a \partial_b (\dot{g}(r)^{-1}_{ab})
= -\Delta_{g(r)} \left(\tr_{g(r)}(\dot{g}(r))\right) -
2 \delta_{g(r)} (\delta_{g(r)}(\dot{g}(r)))
\end{equation*}
at $m$.
\end{lemm}

\begin{proof} By Lemma \ref{sum-1}, the assertion follows from the relations
\begin{equation}\label{id-1}
\sum_{a,b} \partial_a \partial_a \left(\dot{g}(r)^{-1}_{bb}\right) = -
\sum_{a,b} \partial_a \partial_a (\dot{g}(r)_{bb}) + \frac{4}{3}
\sum_{a,b,c} R_{a b a c}(r) \dot{g}(r)_{bc}
\end{equation}
and
\begin{equation}\label{id-2}
\sum_{a,b} \partial_a \partial_b \left(\dot{g}(r)^{-1}_{ab}\right) = -
\sum_{a,b} \partial_a \partial_b (\dot{g}(r)_{ab}) + \frac{2}{3}
\sum_{a,b,c} R_{acba}(r) \dot{g}(r)_{cb}
\end{equation}
at $m$. Here $R(r)$ denotes the curvature of $g(r)$. We start by noting
that
\begin{equation}\label{diff-norm}
\partial_a \partial_b (g(r)_{cd}) + \partial_a \partial_b \left(g(r)^{-1}_{cd}\right) = 0
\end{equation}
at $m$ and
\begin{equation}\label{diff-r}
\dot{g}(r)^{-1} + g^{-1}(r) \dot{g}(r) g^{-1}(r) = 0.
\end{equation}
Hence
\begin{align*}
\partial_a \partial_a  \left(\dot{g}(r)^{-1}_{bb}\right)
& = -\partial_a \partial_a \left( \sum_{\alpha, \beta} g(r)^{-1}_{b
\alpha} \dot{g}(r)_{\alpha \beta} g(r)^{-1}_{\beta b} \right) \\ & =
\sum_{\alpha,\beta} \partial_a \partial_a (g(r)_{b\alpha})
\dot{g}(r)_{\alpha \beta} \delta_{\beta b} + \sum_{\alpha,\beta}
\delta_{b\alpha} \dot{g}(r)_{\alpha \beta}
\partial_a \partial_a (g(r)_{\beta b}) \\
& - \sum_{\alpha, \beta} \delta_{b \alpha} \partial_a \partial_a
\left(\dot{g}(r)_{\alpha \beta}\right) \delta_{\beta b} \\
& = \sum_\alpha \partial_a \partial_a \left(g(r)_{b \alpha}\right) \dot{g}(r)_{\alpha b}
+ \sum_\beta \dot{g}(r)_{b \beta} \partial_a \partial_a \left(g(r)_{\beta b}\right) \\
& - \partial_a \partial_a \left(\dot{g}(r)_{bb}\right).
\end{align*}
Now \eqref{normal-der} implies
\begin{equation}
\partial_a \partial_a \left(\dot{g}(r)^{-1}_{bb}\right)
= \frac{4}{3} \sum_c R_{abac}(r) \dot{g}(r)_{bc} - \partial_a \partial_a
(\dot{g}(r)_{bb}).
\end{equation}
This proves \eqref{id-1}. Similarly, we find
\begin{align*}
\partial_a \partial_b \left(\dot{g}(r)^{-1}_{ab}\right) & = \sum_{c,d}
\partial_a \partial_b (g(r)_{ac}) \dot{g}(r)_{cd} \delta_{db}
+ \sum_{c,d} \delta_{ac} \dot{g}(r)_{cd} \partial_a \partial_b (g(r)_{db}) \\
& - \sum_{c,d} \delta_{ac} \partial_a \partial_b (\dot{g}(r)_{cd}) \delta_{db} \\
& = \sum_c \partial_a \partial_b (g(r)_{ac}) \dot{g}(r)_{cb} + \sum_d \dot{g}(r)_{ad}
\partial_a \partial_b (g(r)_{db}) - \partial_a \partial_b (\dot{g}(r)_{ab})
\end{align*}
and \eqref{normal-der} implies
\begin{multline*}
\partial_a \partial_b \left(\dot{g}(r)^{-1}_{ab}\right) \\
= \frac{1}{3} \sum_c (R_{aabc}(r) + R_{acba}(r)) \dot{g}(r)_{cb} +
\frac{1}{3} \sum_d (R_{adbb}(r) + R_{abbd}(r)) \dot{g}(r)_{ad} -\partial_a
\partial_b  (\dot{g}(r)_{ab}).
\end{multline*}
Hence
$$
\sum_{a,b} \partial_a \partial_b \left(\dot{g}(r)^{-1}_{ab}\right) =
\frac{1}{3} \sum_{a,b,c} (R_{abca}(r) + R_{baac}(r)) \dot{g}(r)_{bc} -
\sum_{a,b} \partial_a \partial_b (\dot{g}(r)_{ab}).
$$
This proves \eqref{id-2}. The proof is complete.
\end{proof}

With these preparations we are able to complete the {\bf proof of
Theorem \ref{B}}.

It remains to determine the Gaussian integral of
\begin{align*}
& \sum_{a,b,c,d} \sum_{\alpha,\beta,\gamma,\delta} \delta_{cd}
\delta_{\alpha \beta} \xi_\gamma \xi_\delta \frac{\partial}{\partial
\xi_b}(\xi_c \xi_d) \frac{\partial}{\partial \xi_a}(\xi_\alpha
\xi_\beta) \left(\frac{\partial^2}{\partial x_a \partial
x_b}(\dot{g}(r)^{-1}_{\gamma \delta}\right) \\
& = \sum_{a,b} \sum_{\gamma,\delta} \xi_\gamma \xi_\delta
\frac{\partial}{\partial \xi_b}(|\xi|^2) \frac{\partial}{\partial
\xi_a}(|\xi|^2) \left(\frac{\partial^2}{\partial x_a
\partial x_b}(\dot{g}(r)^{-1}_{\gamma\delta})\right) \\
& = 4 \sum_{a,b} \sum_{\gamma,\delta} \xi_a \xi_b \xi_\gamma
\xi_\delta \left(\frac{\partial^2}{\partial x_a
\partial x_b}(\dot{g}(r)^{-1}_{\gamma\delta})\right).
\end{align*}
Now integration using \eqref{quartic} yields
\begin{multline*}
\sum_{a,b} \sum_{\gamma,\delta} (\delta_{ab} \delta _{\gamma \delta}
+ \delta_{a\gamma} \delta_{b\delta} + \delta_{a\delta}
\delta_{b\gamma}) \frac{\partial^2}{\partial x_a
\partial x_b}(\dot{g}(r)^{-1}_{\gamma\delta}) \\
= \sum_{a,\gamma} \frac{\partial^2}{\partial x_a \partial
x_a}(\dot{g}(r)^{-1}_{\gamma\gamma}) + \sum_{a,b}
\frac{\partial^2}{\partial x_a \partial x_b} (\dot{g}(r)^{-1}_{ab})
+ \sum_{a,b} \frac{\partial^2}{\partial x_a \partial x_b}
(\dot{g}(r)^{-1}_{ba}).
\end{multline*}
By Lemma \ref{sum-2}, the latter sum equals
$$
-\left(\Delta_{g(r)} \left(\tr_{g(r)}(\dot{g}(r))\right) + 2
\delta_{g(r)}(\delta_{g(r)}(\dot{g}(r)))\right).
$$
But, by Lemma \ref{d-div}, this sum further simplifies to
$$
-3 \Delta_{g(r)} \left(\tr_{g(r)}(\dot{g}(r))\right).
$$
This yields the assertion by using
$$
\tr_{g(r)}(\dot{g}(r)) = 2 (\dot{v}/v) (r).
$$
The proof is complete. \hfill $\square$

\section{A conformal primitive of the Gaussian integral}\label{prime}

In the present section we prove Theorem \ref{var}.

The arguments rest on a local variational formula for the
renormalized volume coefficients. This formula was first stated in
the equations (2.4) and (3.8) of \cite{ISTY} and was rederived in
Theorem 1.5 of \cite{G-ext}. The following result is equivalent to
equation (3.5) in \cite{G-ext}.\footnote{The present conventions
concerning the coefficients $v_{2N}$ differ from those used in
\cite{G-ext} and are closer to those of \cite{ISTY}. In particular,
we replace formulations in terms of ambient space coordinates by
formulations in terms of coordinates on Poincar\'e-Einstein spaces.}

\begin{thm}[\cite{G-ext}]\label{RVC-CV} For any $\varphi \in C^\infty(M)$,
\begin{equation}\label{V-CV}
v(r)^\bullet[\varphi] = - r \dot{v}(r) \varphi - \delta (L(r) d
\varphi),
\end{equation}
where the symmetric bilinear form
\begin{equation}\label{L}
L(r) \st v(r) \int_0^r s g(s)^{-1} ds
\end{equation}
acts on $1$-forms (using $g$).
\end{thm}

In the following, the first term on the right-hand side of
\eqref{V-CV} will be referred to as the {\em scaling} term.

Some comments are in order. By expansion into power series of $r$,
\eqref{V-CV} is equivalent to
\begin{equation}\label{CT-v}
v_{2N}^\bullet[\varphi] = -2N v_{2N} \varphi - \delta \left(
\frac{1}{2} v_{2N-2} g^{-1}_{(0)} + \cdots + \frac{1}{2N} v_0
g^{-1}_{(2N-2)} \right) d \varphi,
\end{equation}
where $g^{-1}_{(2k)}$ are the power series coefficients of
$g(r)^{-1}$. Formula \eqref{CT-v} refines the variational formula
$$
\left( \int_{M^n} v_{2N}(g) \dvol_g \right)^\bullet[\varphi] =
(n\!-\!2N) \int_{M^n} \varphi v_{2N}(g) \dvol_g
$$
for closed $M$. Theorem \ref{RVC-CV} implies that the conformal
variation of $v(r)$ at the round metric on $\s^n$ is given by
\begin{equation}
v(r)^\bullet[\varphi] = \frac{r^2}{2} (1-r^2/4)^{n-1} (n
-\Delta)(\varphi).
\end{equation}

We continue with the {\bf proof of Theorem \ref{var}}.

We start by noting that
$$
(\dot{w})^2 = \frac{1}{4} \frac{\dot{v}^2}{v}.
$$
Hence
$$
((\dot{w})^2)^\bullet[\varphi] = \frac{1}{4} \left( \frac{
\dot{v}^2}{v} \right)^\bullet[\varphi].
$$
Moreover, we have
\begin{equation}\label{red}
\left(\frac{\dot{v}^2}{v} \right)^\bullet[\varphi] = 2 \left(\frac{
\dot{v}}{v}\right) (\dot{v})^\bullet[\varphi] -
\left(\frac{\dot{v}}{v} \right)^2 v^\bullet[\varphi].
\end{equation}
By Theorem \ref{RVC-CV}, the right-hand side of \eqref{red} equals
\begin{multline*}
- 2 \left(\frac{\dot{v}}{v}\right) \delta_g (\dot{L}(r) d\varphi) +
\left(\frac{\dot{v}}{v}\right)^2 \delta_g (L(r) d\varphi) \\
= - 2 \left(\frac{\dot{v}}{v}\right) \delta_g \left( \dot{v}
\int_0^r s g(s)^{-1} ds + r v g(r)^{-1} \right) d\varphi +
\left(\frac{\dot{v}}{v}\right)^2 \delta_g (L(r) d\varphi),
\end{multline*}
up to a contribution by the scaling term. Now partial integration
implies
\begin{multline*}
\left(\int_M \left( \frac{\dot{v}^2}{v} \right) \dvol_g
\right)^\bullet[\varphi] = n \int_M \varphi
\left(\frac{\dot{v}^2}{v} \right) \dvol_g \\ - 2 \int_M \varphi
\delta_g \left( \left(\frac{\dot{v}}{v}\right) L(r) + r v g(r)^{-1}
\right) d \left( \frac{\dot{v}}{v} \right) \dvol_g \\ + 2 \int_M
\varphi \delta_g \left( L(r) \left(\frac{\dot{v}}{v}\right) \right)
d \left(\frac{\dot{v}}{v}\right) \dvol_g,
\end{multline*}
up to a contribution by the scaling term. The contributions by $L$
cancel and we obtain
\begin{equation*}
\left(\int_M \left( \frac{\dot{v}^2}{v} \right) \dvol_g
\right)^\bullet[\varphi] = n \int_M \varphi
\left(\frac{\dot{v}^2}{v} \right) \dvol_g - 2 r \int_M \varphi
\delta_g (v g(r)^{-1}) d \left( \frac{\dot{v}}{v} \right) \dvol_g.
\end{equation*}
Now Lemma \ref{conjugate} shows that the last term on the right-hand
side of this relation can be rewritten as
$$
2 r \int_M \varphi v \Delta_{g(r)} \left( \frac{\dot{v}}{v} \right)
\dvol_g.
$$
Therefore, it only remains to determine the contribution by the scaling
term. By \eqref{red}, this term is given by
\begin{multline*}
\int_M \left( 2 \left(\frac{\dot{v}}{v}\right) \frac{\partial}{\partial r}
(-r \dot{v}) \varphi + \left(\frac{\dot{v}}{v} \right)^2 r \dot{v} \varphi \right) \dvol_g \\
= \int_M \varphi \left( - 2 \left( \frac{\dot{v}}{v} \right) r \ddot{v} -
2 \left(\frac{\dot{v}}{v}\right) \dot{v} + \left(\frac{\dot{v}}{v}
\right)^2 r \dot{v} \right) \dvol_g  = -(2 + r \partial/\partial r) \int_M
\varphi \left(\frac{\dot{v}^2}{v}\right) \dvol_g.
\end{multline*}
The proof is complete. \hfill $\square$
\medskip

The above proof made use of the following conjugation formula.

\begin{lemm}\label{conjugate} We have
\begin{equation}\label{c-div}
v(r) \circ \delta_{g(r)} = \delta_g \circ v(r) \circ g(r)^{-1}
\end{equation}
as an identity of operators on $1$-forms.
\end{lemm}

\begin{proof} We recall that a symmetric bilinear form $b$ on $TM$
acts by
$$
\langle b(\omega),X \rangle = b(\omega^\#,X), \; X \in \X(M)
$$
on $\Omega^1(M)$, where the vector-field $\omega^\#$ is dual to
$\omega$ in the sense that $g(\omega^\#,X) = \langle \omega,X
\rangle$ for all $X \in \X(M)$. In particular, the one-parameter
family $g(r)$ of metrics acts on $\Omega^1(M)$ using $g$. As before,
the latter endomorphism of $\Omega^1(M)$ will also be denoted by
$g(r)$. In local coordinates, it is given by
\begin{equation}\label{op-g}
g(r): \sum_a \omega_a dx^a \mapsto \sum_{a,b,c} g(r)_{bc}
g^{-1}_{ac} \omega_a dx^b.
\end{equation}
Similarly, the inverse of the endomorphism $g(r)$ is given by
\begin{equation}\label{op-g-inverse}
g(r)^{-1}: \sum_a \omega_a dx^a \mapsto \sum_{a,b,c} g_{bc}
g(r)^{-1}_{ac} \omega_a dx^b.
\end{equation}
Now the left-hand side of \eqref{c-div} acts on $\omega = \sum_a
\omega_a dx^a$ by
\begin{multline}\label{left-a}
v(r) (\sqrt{\det g(r)})^{-1} \sum_{a,b} \partial_a (\sqrt{\det g(r)}
g(r)^{-1}_{ab} \omega_b) \\[-3mm] = (\sqrt{\det g})^{-1} \sum_{a,b}
\partial_a (\sqrt{\det g(r)} g(r)^{-1}_{ab} \omega_b).
\end{multline}
On the other hand, the right-hand side of \eqref{c-div} acts on
$\omega$ by
$$
(\sqrt{\det g})^{-1} \sum_{a,b} \partial_a (\sqrt{\det g} v(r)
g^{-1}_{ab} g(r)^{-1} (\omega)_b).
$$
By \eqref{op-g-inverse}, the latter sum equals
\begin{multline}\label{right-a}
(\sqrt{\det g})^{-1} \sum_{a,b,c,d} \partial_a (\sqrt{\det g(r)}
g^{-1}_{ab} g_{bc} g(r)^{-1}_{dc} \omega_d) \\[-3mm]
= (\sqrt{\det g})^{-1} \sum_{a,d} \partial_a (\sqrt{\det g(r)}
g(r)^{-1}_{da} \omega_d).
\end{multline}
Comparing \eqref{left-a} and \eqref{right-a}, completes the proof.
\end{proof}


\section{Proof of Theorem \ref{A}}\label{proof-AA}

In the present section we complete the proof of Theorem \ref{A}.

We first observe that \eqref{conf-hc} implies
\begin{equation*}
\left(\int_M a_2(r) \dvol \right)^\bullet[\varphi] = (n\!-\!2\!-\!r
\partial/\partial r) \int_M \varphi a_2(r) \dvol - r \frac{1}{2}
\int_M \varphi c_2(r) \dvol.
\end{equation*}
Now by \eqref{c2-inter} the latter formula is equivalent to
\begin{multline}\label{a2-var}
\left(\int_M a_2(r) \dvol \right)^\bullet[\varphi] = (n\!-\!2\!-\!r
\partial/\partial r) \int_M \varphi a_2(r) \dvol \\
- r \pi^{-\f} \frac{1}{2} \frac{1}{3!} \int_M \varphi
\left(\int_{\r^n} \{\H(r),\{\H(r), \dot{\H}(r)\}\} e^{-\H(r)} d\xi
\right) dx.
\end{multline}
Hence, using Theorem \ref{B} and \eqref{gaussian}, we find
\begin{multline*}
\left(\int_M a_2(r) \dvol \right)^\bullet[\varphi] \\ =
(n\!-\!2\!-\!r \partial/\partial r) \int_M \varphi a_2(r) \dvol + r
\frac{1}{2} \int_M \varphi v(r) \Delta_{g(r)} ((\dot{v}/v) (r))
\dvol.
\end{multline*}
Combining this relation with Theorem \ref{var}, yields
$$
\left(\int_M \Lambda(r) \dvol \right)^\bullet[\varphi] =
(n\!-\!2\!-\!r \partial/\partial r) \int_M \varphi \Lambda(r) \dvol.
$$
The assertion follows by expansion in $r$. \hfill $\square$

\section{Proof of Theorem \ref{A-fine}}\label{fine}

The following proof of Theorem \ref{A-fine} utilizes the algorithm
of Section \ref{AE}. An alternative proof will be given in Section
\ref{structure-gen}.

For $j=2$, the relation \eqref{heat-algo} is a formula for $a_2(r)$
in terms of $r_{-4}(\lambda;r)$. For any second-order operator with
complete symbol $p_2(x,\xi) + p_1(x,\xi) + p_0(x)$, the algorithm in
the proof of Theorem \ref{hk-expansion} yields
$$
r_{-2}(\lambda) = (p_2 \!-\! \lambda)^{-1}
$$
and
\begin{align*}
r_{-3}(\lambda) & = -(p_2 \!-\! \lambda)^{-1} \left(p_1
r_{-2}(\lambda) + \sum_i \partial_{\xi_i}(p_2) D_{x_i} (r_{-2}(\lambda))\right) \\
& = - (p_2 \!-\! \lambda)^{-2} p_1 + (p_2 \!-\! \lambda)^{-3} \sum_i
\partial_{\xi_i}(p_2) D_{x_i} (p_2).
\end{align*}
Moreover, we find
\begin{align*}
r_{-4}(\lambda) = - r_{-2}(\lambda) & \Big[ p_1 r_{-3}(\lambda) + p_0 r_{-2}(\lambda) \\
& + \sum_i \left(\partial_{\xi_i}(p_2) D_{x_i} (r_{-3}(\lambda)) +
\partial_{\xi_i}(p_1) D_{x_i}(r_{-2}(\lambda))\right) \\
& + \frac{1}{2} \sum_{i,j} \partial^2_{\xi_i \xi_j} (p_2) D^2_{x_i
x_j} (r_{-2}(\lambda)) \Big].
\end{align*}
These formulas imply that $r_{-4}(\lambda)$ is given by the sum
\begin{align}
& - (p_2\!-\!\lambda)^{-4} p_1 \sum_i
\partial_{\xi_i}(p_2) D_{x_i}(p_2) + (p_2\!-\!\lambda)^{-3} \left(p_1^2
+ \sum_i \partial_{\xi_i} (p_1) D_{x_i}(p_2)\right) \\
& - (p_2\!-\!\lambda)^{-2} p_0 \label{pot} \\
& - (p_2\!-\!\lambda)^{-1} \sum_i \partial_{\xi_i}(p_2)
D_{x_i}(r_{-3}(\lambda)) - (p_2\!-\!\lambda)^{-1} \frac{1}{2}
\sum_{i,j} \partial^2_{\xi_i \xi_j}(p_2) D^2_{x_i
x_j}(r_{-2}(\lambda)). \nonumber
\end{align}

First, we use the latter expression for $r_{-4}(\lambda)$ to
determine the contribution of the potential $\U(r)$ to $a_2(r) =
a_2(\H(r))$. For the operator $-\H(r)$, we have $p_0 = -\U(r)$, and
the term in \eqref{pot} yields
\begin{align*}
\pi^\f a_2(x;r) \sqrt{\det g} & = \U(x;r) \frac{1}{2\pi i}
\int_{\r^n} \int_\Gamma (\H(r)(x,\xi)\!-\!\lambda)^{-2} e^{-\lambda} d\lambda d\xi \\
& = \U(x;r) \int_{\r^n} e^{-\H(r)(x,\xi)} d\xi \\
& = \pi^\f \U(x;r) \sqrt{\det g(r)}
\end{align*}
using \eqref{gaussian}. Hence $\U(r)$ contributes to $a_2(r)$ by
$\U(r) v(r)$. In other words, we have the decomposition
\begin{equation}\label{a2-deco}
a_2(r) = a_2(\D(r)) + \U(r) v(r),
\end{equation}
where $a_2(\D(r))$ denotes the sub-leading heat kernel coefficient
of the operator
\begin{equation}\label{D(r)}
\D(r) = -\delta (g(r)^{-1} d).
\end{equation}
By the definition of the potential $\U(r)$, it follows that
\begin{align*}
\Lambda(r) & = a_2(r) - (\dot{w}(r))^2 \\
& = a_2(\D(r)) + \delta (g(r)^{-1}d)(w)w  -\left(
\frac{\partial^2}{\partial r^2} (w) - (n\!-\!1) r^{-1}
\frac{\partial}{\partial r}(w)\right) w - (\dot{w}(r))^2.
\end{align*}
Simplification yields the formula
\begin{equation}\label{L-simple}
\Lambda(r) = a_2(\D(r)) + \delta (g(r)^{-1}d)(w(r)) w(r) -
\frac{1}{2} \left( \ddot{v}(r) - (n\!-\!1) r^{-1} \dot{v}(r)
\right).
\end{equation}
Note that the last term in \eqref{L-simple} is {\em linear} in
$v(r)$, i.e., in $a_0(r)$. By expansion in $r$, \eqref{L-simple}
yields
$$
\Lambda_{2k-2} = \left( a_2(\D(r)) + \delta (g(r)^{-1}d)(w(r))
w(r)\right)[2k\!-\!2] + (n\!-\!2k) k v_{2k};
$$
here $u(r)[k]$ denotes the coefficient of $r^k$ in the formal power
series expansion of $u$. In particular, we have
\begin{equation}\label{L-reduced}
\Lambda_{n-2} = \left( a_2(\D(r)) + \delta (g(r)^{-1}d)(w(r))
w(r)\right)[n\!-\!2].
\end{equation}
Thus, it remains to make explicit the sub-leading heat kernel
coefficient of the operator $\D(r)$. For this purpose, we observe
that
\begin{equation}\label{diff}
-\D(r)(u) + \Delta_{g(r)}(u) = (d \log v(r),du)_{g(r)}.
\end{equation}
Then we apply the above formula for $r_{-4}$ to express the
difference
$$
a_2(\D(r)) - a_2(\Delta_{g(r)})
$$
in terms of a certain Gaussian integral and calculate this integral.
For the following considerations, we fix $m \in M$ and choose normal
coordinates for $g(r)$ around $m$. We compare the contributions to
$r_{-4}(x,\xi)$ for the operators $\D(r) = -\delta (g(r)^{-1} d)$
and $\Delta_{g(r)}$. These two operators are related by
\eqref{diff}. Note that $r_{-4}$ depends quadratically on $p_1$. By
$$
\Delta_{g(r)}(u)(m) = \sum_i \partial^2_{x_i x_i}(u)(m),
$$
the linear part $p_1^{\Delta}(x,\xi)$ of the complete symbol of
$-\Delta_{g(r)}$ vanishes at $m$. Hence the linear part
$p_1^\D(x,\xi)$ of the complete symbol of
$$
-\D(r) = \delta(g(r)^{-1}d) = (d\log v(r),d)_{g(r)} - \Delta_{g(r)}
$$
equals
$$
p_1^\D (m,\xi) = i \sum_j \partial_{x_j}(\log v)(m) \xi_j
$$
at $m$. Let
$$
d_1(x,\xi) \st p_1^\D (x,\xi) - p_1^{\Delta}(x,\xi).
$$
Now the above formula for $r_{-4}(\lambda)$ shows that the
difference
$$
r_{-4}^\D(m,\xi,\lambda) - r_{-4}^{\Delta}(m,\xi,\lambda)
$$
is given by the sum
\begin{align*}
& -(\H(r)\!-\!\lambda)^{-4} d_1 \sum_j \partial_{\xi_j}(\H(r)) D_{x_j}(\H(r)) \\
& + (\H(r)\!-\!\lambda)^{-3} ((p_1^\D)^2 + \sum_j \partial_{\xi_j}(d_1) D_{x_j}(\H(r))) \\
& + (\H(r)\!-\!\lambda)^{-1} \sum_j \partial_{\xi_j}(\H(r)) D_{x_j}
((\H(r)\!-\!\lambda)^{-2} d_1)
\end{align*}
at $m$. By $\H(r)(m,\xi)=|\xi|^2$ and the vanishing of all
derivatives $\partial_{x_j}(\H(r))(m,\xi)$, the latter sum equals
\begin{multline}\label{d1-part}
(|\xi|^2\!-\!\lambda)^{-3} (p_1^\D)^2(m,\xi) + 2
(|\xi|^2\!-\!\lambda)^{-3} \sum_{j,k} \xi_j \xi_k
\partial^2_{x_j x_k} (\log v)(m) \\
= (|\xi|^2\!-\!\lambda)^{-3} \left( -\left( \sum_j \partial_{x_j}
(\log v)(m) \xi_j \right)^2 + 2 \sum_{j,k}
\partial^2_{x_j x_k}(\log v)(m)\xi_j \xi_k \right).
\end{multline}
Now we calculate the integral
\begin{equation}\label{int-d1}
I_{-4}(m) = \frac{1}{2\pi i} \int_{\r^n} \int_\Gamma \left(
r_{-4}^\D(m,\xi,\lambda) - r_{-4}^{\Delta}(m,\xi,\lambda)\right)
e^{-\lambda} d\lambda d\xi.
\end{equation}
First, by \eqref{d1-part} and Cauchy's formula, it simplifies to the
sum of
\begin{equation}\label{int-d1-1}
-\frac{1}{2!} \int_{\r^n} \left( \sum_j \partial_{x_j}(\log v)(m)
\xi_j \right)^2 e^{-|\xi|^2} d\xi
\end{equation}
and
\begin{equation}\label{int-d1-2}
\int_{\r^n} \left( \sum_{j,k}
\partial^2_{x_j x_k}(\log v)(m) \xi_j \xi_k \right) e^{-|\xi|^2} d\xi.
\end{equation}
Next, we evaluate the integrals \eqref{int-d1-1} and
\eqref{int-d1-2} using the identities
$$
\int_{\r^n} \xi_j \xi_k e^{-|\xi|^2} d\xi \big/ \int_{\r^n}
e^{-|\xi|^2} d\xi = \frac{1}{2} \delta_{jk} \quad \mbox{and} \quad
\int_{\r^n} e^{-|\xi|^2} d\xi = \pi^\f.
$$
It follows that
$$
I_{-4}(m) = -\frac{1}{4} \pi^\f |d\log v|^2(m) + \frac{1}{2} \pi^\f
\Delta_{g(r)}(\log v)(m).
$$
Therefore, we obtain
\begin{align*}
\pi^\f a_2(\D) \sqrt{\det g} & = \pi^\f a_2(\Delta) \sqrt{\det g} + I_{-4}(m) \\
& = \pi^\f a_2(\Delta) \sqrt{\det g} + \frac{1}{2} \pi^\f
\Delta_{g(r)}(\log v) - \frac{1}{4} \pi^\f |d\log v|^2.
\end{align*}
Hence we get the relation
\begin{equation}\label{a2-compare}
a_2(\D) = a_2(\Delta_{g(r)})  + \frac{1}{2} \Delta_{g(r)}(\log v) v
- \frac{1}{4} |d\log v|^2 v.
\end{equation}
We combine \eqref{a2-compare} with the well-known formula for $a_2$
of the Laplacian (see Section \ref{coeff-L}). We find
\begin{equation}\label{a2-r}
a_2(\Delta_{g(r)}) = \frac{\scal(g(r))}{6} v(r),
\end{equation}
where the additional factor $v(r)$ comes from the fact that we
regard $\Delta_{g(r)}$ as an operator on $L^2(M,g)$. Now
simplification yields the formula
\begin{equation}\label{a2-almost}
a_2(\D) = \frac{\scal(g(r))}{6} v(r) + \frac{1}{2} \Delta_{g(r)}(v)
- \frac{3}{4} \frac{|d v|^2}{v}.
\end{equation}
Thus, by \eqref{L-simple}, we find
\begin{multline}\label{L-simp}
\Lambda(r) = \frac{\scal (g(r))}{6} v(r) - \frac{1}{2} \left(
\ddot{v}(r) - (n\!-\!1) r^{-1} \dot{v}(r)\right)\\
+ \frac{1}{2} \Delta_{g(r)}(v) - \frac{3}{4} \frac{|d v|^2}{v} +
\delta (g(r)^{-1}d)(w(r)) w(r).
\end{multline}
Now \eqref{diff} implies
$$
\delta(g(r)^{-1} d w) w = - \Delta_{g(r)}(w) w + (d \log v,dw) w.
$$
Some calculations show that
$$
\Delta_{g(r)}(w) w = \frac{1}{2} \Delta_{g(r)}(v) - \frac{1}{4}
\frac{|dv|^2}{v}
$$
and
$$
(d \log v,dw) w = \frac{1}{2} \frac{|dv|^2}{v}.
$$
Hence we find the relation
$$
\delta(g(r)^{-1}dw)w = -\frac{1}{2} \Delta_{g(r)}(v) + \frac{3}{4}
\frac{|dv|^2}{v}.
$$
It follows that the last three terms in \eqref{L-simp} cancel.
Therefore, we conclude that
\begin{equation}\label{L-inter}
\Lambda(r) = \frac{\scal (g(r))}{6} v(r) - \frac{1}{2} \left(
\ddot{v}(r) - (n\!-\!1) r^{-1} \dot{v}(r)\right).
\end{equation}
In order to calculate the scalar curvature of $g(r)$, we recall the
construction of the Poincar\'e-Einstein metric. In terms of the
coordinates $r^2 = \rho$, the metric $g_+$ reads
\begin{equation}\label{new-coo}
g_+ = \frac{d\rho^2}{4\rho^2} + \frac{1}{\rho} h(\rho), \quad
h(\rho) = g_{(0)} + \rho g_{(2)} + \rho^2 g_{(4)} + \cdots.
\end{equation}
We evaluate the vanishing of the tangential part of $\Ric(g_+) +
ng_+$. By Equation (3.17) in \cite{FG-final}, this leads to the
identity
\begin{equation}\label{Ricci}
- \rho \left[2 \ddot{h} - 2 \dot{h} h^{-1} \dot{h} +
\tr(h^{-1}\dot{h}) \dot{h} \right] + (n\!-\!2) \dot{h} + \tr(h^{-1}
\dot{h}) h + \Ric(h) = 0
\end{equation}
(see also the discussion on page 241 of \cite{juhl-book}), where $h
= h(\rho)$ and the dots denote derivatives with respect to $\rho$.
The relation \eqref{Ricci} implies
\begin{equation}\label{R}
\scal(h) = -(2n-2) \tr (h^{-1}\dot{h}) + \rho \left( 2 \tr (h^{-1}
\ddot{h}) - 2 \tr (h^{-1} \dot{h} h^{-1} \dot{h}) + \tr
(h^{-1}\dot{h})^2 \right).
\end{equation}
Now
$$
\tr (h^{-1} \dot{h}) = \partial_\rho \tr \log h = \partial_\rho \log
\det h = 2 \partial_\rho \log V = 2 \dot{V}/V,
$$
where $V$ is defined by $V(r^2) = v(r)$. In turn, the latter
equation gives
$$
2 \partial_\rho (\dot{V}/V) = \tr (h^{-1} \ddot{h}) - \tr (h^{-1}
\dot{h} h^{-1} \dot{h}).
$$
Therefore, \eqref{R} can be rewritten in the form
$$
\scal(h) = -4(n\!-\!1) \dot{V}/V + 4 \rho \partial_\rho(\dot{V}/V) +
4 \rho (\dot{V}/V)^2.
$$
Hence
$$
\scal(h) V = -4(n-1) \dot{V} + 4 \rho \ddot{V}.
$$
In other words,
\begin{equation}\label{scal}
\scal(g(r)) v(r) = \ddot{v}(r) - (2n\!-\!1) r^{-1} \dot{v}(r).
\end{equation}
Together with \eqref{L-inter}, we find
$$
\Lambda(r) = - \frac{1}{3} \ddot{v}(r) + \frac{n\!-\!2}{6} r^{-1}
\dot{v}(r).
$$
This completes the proof of Theorem \ref{A-fine}. \hfill $\square$
\medskip

Finally, we use Theorem \ref{A-fine} to give a {\bf second proof of
Theorem \ref{B}}.

On the one hand, Theorem \ref{conform-heat} leads to the variational
formula \eqref{a2-var}. On the other hand, Theorem \ref{conf-reno}
and Theorem \ref{var} together with Theorem \ref{A-fine} show that
\begin{multline*}
\left(\int_M a_2(r) \dvol \right)^\bullet[\varphi] = (n\!-\!2\!-\!r
\partial/\partial r) \int_M \varphi a_2(r) \dvol + r \frac{1}{2}
\int_M \varphi v(r) \Delta_{g(r)} ((\dot{v}/v)(r)) \dvol.
\end{multline*}
Since both identities hold true for arbitrary $\varphi \in
C^\infty(M)$, the assertion follows by comparison.

\section{Proof of Theorem \ref{structure}}\label{structure-gen}

We start with an alternative proof of Theorem \ref{A-fine}.

First, we note that the arguments in Section \ref{fine} proved the
relation
$$
a_2(r) = \left(\frac{\scal(g(r))}{6} + \frac{1}{2}
\Delta_{g(r)}(\log v) - \frac{1}{4} |d\log v|^2 + \U(r)\right) v(r)
$$
(see \eqref{a2-deco}, \eqref{a2-compare} and \eqref{a2-r}). The
latter formula also may be seen as a special case of Gilkey's
formula for the sub-leading heat kernel coefficient of a general
second-order operator. In fact, from that perspective, it takes the
form
\begin{equation}\label{a2-E}
a_2(r) = \left(\frac{\scal(g(r))}{6} + E(r) \right) v(r)
\end{equation}
with
\begin{equation}\label{E}
E(r) = \Delta_{g(r)}(\eta(r)) - |d\eta(r)|^2 + \U(r), \quad \eta(r)
= 1/2 \log v(r).
\end{equation}

In order to recognize \eqref{a2-E} and \eqref{E} as a special case
of Gilkey's formula, we recall this formula for the heat kernel
coefficient $a_2$ of a second-order operator of the form
\begin{equation}\label{L-gen}
L = - \left( \sum_{i,j} g^{ij} \partial^2/\partial x_i
\partial x_j + \sum_k A^k \partial/\partial x_k + B  \right)
\end{equation}
defined by a metric $g$ and coefficients $A^k, B \in C^\infty(M)$.
Let
\begin{equation}\label{one-form-gen}
\omega_i = \frac{1}{2} g_{ij} (A^j + g^{kl} \Gamma_{kl}^j)
\end{equation}
and
\begin{equation}\label{E-gen}
E = B - g^{ij} (\partial \omega_i/\partial x_j + \omega_i \omega_j -
\omega_k \Gamma_{ij}^k),
\end{equation}
where $\Gamma_{ij}^k$ are the Christoffel symbols of the Levi-Civita
connection of $g$. Then the operator $L$ can be identified with the
operator
$$
-\tr_g (\nabla \circ \nabla) - E
$$
on the sections of the trivial line bundle $\L = M \times \c$
equipped with the connection $\nabla$ defined by the connection
one-form $\omega = \sum_k \omega_k dx^k$. Here we use the notation
$\nabla$ also for the coupled connection on $TM \otimes \L$ defined
by $\nabla$ on $\L$ and the Levi-Civita connection on $TM$. For the
proof see Lemma 4.8.1 and Corollary 4.8.2 in \cite{G-book}. In these
terms, we have the following result.

\begin{lemm}[\cite{G-book}, Theorem 4.8.16/(b)]\label{a2-gen}
$$
a_2 (L) = \frac{\scal(g)}{6} + E.
$$
\end{lemm}

Hence for an operator $L$ of the form
\begin{equation}\label{L-geom}
L = - (\Delta_g  + g(d\eta,d) + b)
\end{equation}
with $b, \eta \in C^\infty(M)$, we obtain the following formula.

\begin{corr}\label{a2-L-geom} The second heat kernel coefficient of
the operator $L$ as in \eqref{L-geom} is given by
$$
a_2(L) = \frac{\scal(g)}{6} - \frac{1}{2} \Delta_g (\eta) -
\frac{1}{4} |d \eta|_g^2 + b.
$$
\end{corr}

\begin{proof} By Lemma \ref{a2-gen}, it suffices to determine the
potential $E$. First, we observe that for $L$ as in \eqref{L-geom}
the connection form $\omega$ is given by
\begin{equation}\label{c-form}
\omega = \frac{1}{2} d\eta.
\end{equation}
Indeed, \eqref{one-form-gen} implies
$$
\omega_r = \frac{1}{2} g_{rj} \left( (g^{ij})_i + \frac{1}{2} (\log
\det(g))_i g^{ij} + g^{kl} \Gamma_{kl}^j + A^j \right)
$$
with $A^i = g^{ij} \eta_j$. We simplify the latter sum using
\begin{equation}\label{christ}
\Gamma_{ij}^k = \frac{1}{2} g^{kl} ((g_{il})_j + (g_{jl})_i -
(g_{ij})_l).
\end{equation}
Then
\begin{align*}
\omega_r & = \frac{1}{2} g_{rj} \left( (g^{ij})_i + \frac{1}{2}
(g_{ab})_i g^{ab} g^{ij} + \frac{1}{2} g^{kl} g^{js} ((g_{ks})_l +
(g_{ls})_k - (g_{kl})_s) + A^j \right) \\
& = \frac{1}{2} g_{rj} \left( (g^{ij})_i + g^{kl} g^{js} (g_{ks})_l
+ A^j \right) \\
& = \frac{1}{2} g_{rj} A^j.
\end{align*}
This proves \eqref{c-form}. Next, \eqref{christ} yields
$$
g^{ij} \Gamma_{ij}^k = - (g^{kr})_r - g^{kl} \frac{1}{2} (\log \det
(g))_l.
$$
Hence \eqref{E-gen} gives
\begin{align}\label{E-g}
E & = b - \frac{1}{2} g^{ij} \partial^2 \eta/\partial x_i
\partial x_j - \frac{1}{4} g^{ij} \eta_i \eta_j - \frac{1}{2}
\left((g^{kr})_r + g^{kl} \frac{1}{2} (\log \det (g))_l \right) \eta_k  \nonumber \\
& = b - \frac{1}{4} |d\eta|_g^2 - \frac{1}{2} \Delta_g (\eta).
\end{align}
The proof is complete.
\end{proof}

In order to apply Corollary \ref{a2-L-geom} to the operator
$\H(r;g)$, we write $\H(r;g)$ in the form
\begin{equation}\label{key}
\H(r;g) = \Delta_{g(r)} - (d\log v(r),d)_{g(r)} + \U(r;g)
\end{equation}
(for details see the proof of Lemma 8.1 in \cite{juhl-ex}). This
yields the formula \eqref{a2-E} with $E(r)$ as given in \eqref{E};
the coefficient $v(r)$ comes from the fact that the heat kernel of
$\H(r;g)$ is defined by integration against the volume form of $g$
(instead of $g(r)$).

The calculations in Section \ref{fine} already proved the following
remarkable formula for the new potential $E(r)$. It shows that
$E(r)$ does not involve differentiations along $M$.

\begin{lemm}\label{E-final} The function $E(r)$ as defined in
\eqref{E} can be written in the form
\begin{equation}\label{E-closed}
E(r) = - \frac{1}{2} \frac{\ddot{v}}{v} +
\frac{1}{2}\frac{n\!-\!1}{r} \frac{\dot{v}}{v} + \frac{1}{4}
\frac{\dot{v}^2}{v^2} = - w^{-1} \left( \frac{\partial^2}{\partial
r^2} - \frac{n\!-\!1}{r} \frac{\partial}{\partial r} \right)(w).
\end{equation}
\end{lemm}

\begin{proof} We repeat the arguments (in a slightly different form).
By \eqref{E}, we have
\begin{equation*}
E(r) = \frac{1}{2} \Delta_{g(r)}(\log v) - \frac{1}{4} |d\log v|^2 +
\U(r).
\end{equation*}
Now we observe that for any metric $g$ and any non-negative function
$u \in C^\infty(M)$,
$$
\Delta_g (\log u) = \delta_g (du/u) = \Delta_g (u) /u - |du|^2_g
/u^2.
$$
Therefore, by the definition of $\U(r)$, $E(r)$ is the sum of
$$
- w^{-1} \left( \frac{\partial^2}{\partial r^2} - \frac{n\!-\!1}{r}
\frac{\partial}{\partial r} \right)(w)
$$
and
\begin{equation}\label{cancel}
\frac{1}{2} (\Delta_{g(r)}(v)/v - |dv|^2_{g(r)}/v^2) - \frac{1}{4}
|dv|^2_{g(r)}/v^2  + \delta_g(g(r)^{-1}d)(w)/w.
\end{equation}
But
\begin{align*}
\delta_g(g(r)^{-1}d)(w)/w & = - \Delta_{g(r)}(w)/w + (d\log v,dw)_{g(r)}/w \\
& = - \frac{1}{2} \Delta_{g(r)}(v) / v + \frac{1}{4} |dv|_{g(r)}^2
/v^2 + \frac{1}{2} |dv|_{g(r)}^2 /v^2.
\end{align*}
Hence the sum \eqref{cancel} vanishes. This completes the proof.
\end{proof}

Finally, we give a {\bf proof of Theorem \ref{structure}}.

The above arguments prove that $\H(r;g)$ can be identified with the
operator
$$
\tr_{g(r)} (\nabla^{v(r)} \circ \nabla^{v(r)}) + E(r)
$$
on the sections of $\L$ equipped with the connection $\nabla^{v(r)}$
with the connection one-form $\omega = -\frac{1}{2} d\log v(r)$.
Here the potential $E(r)$ (see \eqref{E-closed}) is regarded as an
endomorphism of $\L$. In particular, $\H(r;g)$ is a Laplace-type
operator for the metric $g(r)$. Now we apply Lemma 4.8.6 in
\cite{G-book}. It follows that $a_{2k}(r)$ is an invariant
polynomial of homogeneity $2k$ in the non-commuting variables
$$
\{\nabla^{g(r)}_{i_1} \cdots \nabla^{g(r)}_{i_r} R_{ijkl}(r),
\nabla^{v(r)}_{i_1} \cdots \nabla^{v(r)}_{i_r} E(r)\}
$$
of the covariant derivatives of the curvature of the Levi-Civita
connection of $g(r)$ and the covariant derivatives of the
endomorphism $E(r)$ of $\L$. Here we use the fact that the curvature
of $\nabla^{v(r)}$ vanishes. Finally, by definition, the variables
$R_{ijkl}(r)$ and $E(r)$ have homogeneity $2$, and each derivative
increases the homogeneity by $1$. This completes the proof of
Theorem \ref{structure}. \hfill $\square$
\medskip

Since the coefficients in the Taylor expansion of $g(r)$ are
universal functionals of $g$, Theorem \ref{structure} implies that
the coefficients $a_{2k}(r) = a_{2k}(g(r))$ may be regarded also as
functions in $r$ with coefficients given by scalar local Riemannian
invariants of the metric $g$.

\section{Extremal properties of some curvature functionals}\label{extremal}

In the present section, we study extremal properties of the critical
heat kernel coefficients $a_{(2,n-2)}$ and some related functionals.

In view of $a_0(r) = v(r)$, the following result on extremal
properties of renormalized volume coefficients may be regarded as a
result on the first heat kernel coefficient of $\H(r)$.

\begin{thm}[\cite{CFG}]\label{extremal-RVC} Let $(M^n,g)$ be a
closed unit volume Einstein manifold of dimension $n \ge 3$. Assume that
$g$ has non-zero scalar curvature. Then the restriction of the functional
$$
\F_{2k}(g) = \int_M v_{2k}(g) \dvol_g, \; 2k < n
$$
to the set $c_1=[g]_1$ of unit volume metrics conformal to $g$ has a
local extremum at $g$. More precisely, $(-1)^k \F_{2k}$ has a local
minimum at $g$ if $\scal(g)>0$, and $\F_{2k}$ has a local maximum at
$g$ if $\scal(g)<0$. In the case $\scal(g)>0$, the local minimum is
strict unless $(M^n,g)$ is isometric to a (rescaled) round sphere.
\end{thm}

We also recall that $\F_n$ is conformally invariant. For $2k>n$,
there is an analogous result with maxima and minima interchanged. In
the locally conformally flat case, we have $v_{2k} = (-1/2)^k \tr
(\wedge^k\Rho)$ and the study of the corresponding functionals
$\F_{2k}$ was initiated by Viaclovsky \cite{viac}.

For later reference, it will be convenient to recall the main steps
of the proof of Theorem \ref{extremal-RVC}.

\begin{proof} First, Theorem \ref{RVC-CV} implies that
$$
(\F_{2k})^\bullet[\varphi] = (n-2k) \int_M \varphi v_{2k} \dvol, \;
\varphi \in C^\infty(M).
$$
It follows that the restriction of $\F_{2k}$ ($2k < n$) to the
subset $c_1 = [g]_1 \subset c$ is critical at $g$ iff $v_{2k}$ is
constant at $g$. In fact, the method of Lagrange multipliers shows
that $g$ is critical iff
$$
\left(\F_{2k} - c \int_M \dvol\right)^\bullet[\varphi] = 0
$$
at $g$ for some constant $c$ and all $\varphi \in C^\infty(M)$.
Hence
$$
\int_M \varphi ((n-2k) v_{2k} - n c) \dvol = 0, \; \varphi \in
C^\infty(M)
$$
at $g$. This proves the claim. The assertions on extremal values are
consequences of a formula for the second conformal variation of the
restrictions of the functionals $\F_{2k}$ to $c_1$. In order to
derive this formula, let $\gamma(t) = e^{2\varphi(t)}g$ be a curve
in the conformal class $c$ of $g$ with $\varphi(0)=0$, $\varphi'(0)
= \varphi$ and $\varphi''(0) = \psi$; here $'$ denotes the
derivative with respect to $t$. Let
$$
(\F_{2k}(g))^{\bullet \bullet}[\gamma] \st (\partial^2/\partial
t^2)|_0 (\F_{2k}(\gamma(t))).
$$
We also set
$$
\F(r;g) \st \int_M v(r;g) \dvol_g \quad \mbox{and} \quad
(\F(r;g))^{\bullet \bullet}[\gamma] = (\partial^2/\partial t^2)|_0
(\F(r;\gamma(t))).
$$
Now Theorem \ref{RVC-CV} implies
\begin{align*}
(\F(r;g))^{\bullet \bullet}[\gamma] & = (\partial/\partial t)|_0
((\partial/\partial t) (\F(r;\gamma(t)))) \\
& = (\partial/\partial t)|_0 ((\partial/\partial s)|_0 (\F(r;\gamma(t+s)))) \\
& = (\partial/\partial t)|_0 \left((n\!-\!r\partial/\partial r)
\int_M \varphi'(t) v(r;\gamma(t)) \dvol_{\gamma(t)} \right).
\end{align*}
A second application of Theorem \ref{RVC-CV} yields
\begin{equation}\label{F-second-var}
(\F(r;g))^{\bullet \bullet}[\gamma]  = (n\!-\!r\partial/\partial r)
\left(\int_M [-\varphi (\varphi r \dot{v}(r) + \delta(L(r) d\varphi))+
(\psi \!+\! n \varphi^2) v(r)] \dvol_g \right).
\end{equation}
If the curve $\gamma(t)$ is in $c_1$, i.e., preserves volumes, the
integral $\int_M e^{n \varphi(t)} \dvol_g$ does not depend on $t$.
Hence for such curves we have
$$
\int_M \varphi \dvol_g = 0 \quad \mbox{and} \quad \int_M (\psi + n
\varphi^2) \dvol_g = 0.
$$
In particular, this has the consequence that, if $v(r)$ is constant
on $M$, the last two terms in the second variational formula
\eqref{F-second-var} vanish for such curves $\gamma(t)$. In this
case, the second variation does not depend on the second derivative
of $\varphi(t)$ at $t=0$ and we simply write $(\F(r;g))^{\bullet
\bullet}[\varphi]$ for the second variation.

Now we evaluate the formula \eqref{F-second-var} for a curve
$\gamma(t)$ of unit volume metrics conformal to an Einstein metric
$g$. As we have just seen, it suffices to evaluate the first two
terms at an Einstein metric. We recall from Section \ref{basics}
that for an Einstein metric $g$,
$$
g(r) = (1\!-\!cr^2)^2 g \;\; \mbox{with} \;\; c =
\scal(g)/(4n(n\!-\!1)).
$$
Hence
\begin{equation}\label{v-Einstein}
v(r) = (1\!-\!cr^2)^n.
\end{equation}
In particular, $v(r)$ is constant on $M$. Moreover, we find
\begin{equation}\label{L-Einstein}
L(r) = (1\!-\!cr^2)^n \int_0^r s (1\!-\!cs^2)^{-2} ds =
\frac{r^2}{2} (1\!-\!cr^2)^{n-1}.
\end{equation}
Therefore, \eqref{F-second-var} implies the variational formula
\begin{equation}\label{second-GF}
(\F(r;g))^{\bullet \bullet}[\varphi] = (n\!-\!r\partial/\partial r)
\left((1\!-\!cr^2)^{n-1} \frac{r^2}{2} \int_M \varphi (4nc \varphi +
\Delta \varphi) \dvol_g \right),
\end{equation}
where $\Delta = -\delta d$. Note that $4nc = \scal(g)/(n\!-\!1)$. By
expansion of \eqref{second-GF} into power series in $r$, we obtain
\begin{equation}\label{second-F2k}
(\F_{2k}(g))^{\bullet \bullet}[\varphi] = (n\!-\!2k) \frac{1}{2}
\binom{n-1}{k-1} (-c)^{k-1} \int_M \varphi (4nc \varphi + \Delta
\varphi) \dvol_g.
\end{equation}

Next we recall some results of Obata \cite{obata}. The smallest
possible non-zero eigenvalue $\lambda_1$ of $-\Delta$ on a closed
Einstein manifold $(M^n,g)$ of positive scalar curvature equals
$\scal(g)/(n\!-\!1)$. Moreover, this value is an eigenvalue iff
$(M^n,g)$ is isometric to a (rescaled) round sphere.

Thus, if $(M^n,g)$ is not isometric to a (rescaled) round sphere,
\eqref{second-GF} has the following consequences for the signs of
the quadratic forms $(\F_{2k})^{\bullet \bullet}[\varphi]$:
\begin{itemize}
\item For $2k<n$ and $\scal(g)>0$ the sign of
$(-1)^k(\F_{2k})^{\bullet \bullet}[\varphi]$ is positive.
\item For $2k<n$ and $\scal(g)<0$ the sign of
$(\F_{2k})^{\bullet \bullet}[\varphi]$ is negative.
\end{itemize}
This proves the assertions in this case.

On the other hand, if $(M^n,g)$ is isometric to a round sphere, then
$\F_{2k}(g)^{\bullet \bullet}[\varphi]$ vanishes iff $\varphi$ is an
eigenfunction for $\lambda_1=n$. By the conformal transformation law
of scalar curvature, these eigenfunctions correspond to
infinitesimal conformal factors of conformal
diffeomorphisms.\footnote{In fact, let $g_0$ be the round metric on
$\s^n$ and assume that $\varphi_t^*(g_0) = e^{2\Phi_t} g_0$. Then
$\J(e^{2\Phi_t} g_0) = \J(\varphi_t^*(g_0)) = \varphi_t^* (\J(g_0))
= \J(g_0)$. Combining this with the relation $e^{2\Phi_t}
\J(e^{2\Phi_t} g_0) = \J(g_0) - \Delta (\Phi_t) + \dots$, yields
$e^{2\Phi_t} \J(g_0) = \J(g_0) - \Delta(\Phi_t) + \dots$. Now
differentiation at $t=0$ implies
$$
2\psi \J(g_0) = - \Delta(\psi), \; \psi = (d/dt)|_0(\Phi_t).
$$
But $\J(g_0)=n/2$.} Hence the assertion on local minima also follows
in this case.
\end{proof}

By expansion of \eqref{F-second-var} into power series in $r$, we
obtain second conformal variational formulas for the functionals
$\F_{2k}$ on the full conformal class. These formulas admit a more
natural formulation in terms of the second conformal variations of
the {\em rescaled} functionals
$$
\tilde{\F}_{2k}(g) \st \F_{2k}(g)/V(g)^{(n-2k)/n} \quad \mbox{with}
\quad V(g) = \int_M \dvol_g
$$
at its critical points. Note that $\tilde{\F}_{2k}(\lambda g) =
\tilde{\F}_{2k}(g)$ for any $\lambda > 0$. In the following, it will
be convenient to use the notation
$$
\bar{\varphi} \st \varphi - \left(\int_M \varphi \dvol\right) / V.
$$
Let $L_{2k}$ be the coefficient of $r^{2k}$ in the $r$-expansion of
$L(r)$.

\begin{prop}\label{tilde-F-cv} Let $\gamma(t)$ be a curve in the conformal
class of $g$ with $\gamma(0)=g$ and $(\partial/\partial t)|_0
(\gamma(t)) = 2 \varphi g$. Let $2k < n$. Then $\tilde{\F}_{2k}$ is
critical at $g$, i.e.,
$$
(\partial/\partial t)|_0 (\tilde{\F}_{2k}(\gamma(t))) = 0,
$$
iff $v_{2k}(g)$ is constant. Moreover, we have
$$
(\partial^2/\partial t^2)|_0 (\tilde{\F}_{2k}(\gamma(t))) = -
V^{-(n-2k)/n} (n\!-\!2k) \int_M (2k v_{2k} \bar{\varphi}^2 +
\bar{\varphi} \delta(L_{2k} d \bar{\varphi})) \dvol
$$
at a critical metric $g$.
\end{prop}

\begin{proof} We first prove that critical points of $\tilde{\F}_{2k}$
are metrics with constant $v_{2k}$. The assertion follows from the
following general argument. Consider a functional $\F(g) = \int_M
F(g) \dvol_g$ so that
$$
\F(g)^\bullet[\varphi] = m \int_M \varphi F(g) \dvol_g, \; m \ne 0.
$$
Then conformal variation of $ \tilde{\F} \st \F / V^{m/n}$ at $g$
yields
$$
m \frac{\int_M \varphi F(g) \dvol_g}{V(g)^{m/n}} - \frac{\int_M F(g)
\dvol_g}{V(g)^{m/n+1}} \frac{m}{n} n \int_M \varphi \dvol_g.
$$
Hence $\tilde{\F}$ is critical at $g$ iff
$$
\int_M \varphi F(g) \dvol_g = \frac{\int_M F(g) \dvol_g}{V(g)}
\int_M \varphi \dvol_g,
$$
i.e.,
$$
\int_M \varphi \left(F(g) - \frac{\int_M F(g) \dvol_g}{V(g)}\right)
\dvol_g = 0.
$$
Since this relation is valid for all $\varphi \in C^\infty(M)$, it
follows that $F(g)$ is constant. Now we apply the result to the
functional $\F_{2k}$ (with $m=n-2k)$. In order to derive the second
variational formula for $\tilde{\F}_{2k}$, we use
\eqref{F-second-var} and the fact that $v_{2k}$ is constant at a
critical point $g$ of $\tilde{\F}_{2k}$. First, we have
$$
(\tilde{\F}_{2k})^{\bullet \bullet}[\gamma] = (\F_{2k})^{\bullet
\bullet}[\gamma] V^{-\frac{n-2k}{n}} + 2 (\F_{2k})^\bullet[\varphi]
(V^{-\frac{n-2k}{n}})^\bullet[\varphi] + \F_{2k}
(V^{-\frac{n-2k}{n}})^{\bullet \bullet}[\gamma].
$$
Now a simple calculation shows that
$$
(V^{-\frac{n-2k}{n}})^\bullet[\varphi] = -(n-2k)
V^{-\frac{2n-2k}{n}} \int_M \varphi dv
$$
and
\begin{multline*}
(V^{-\frac{n-2k}{n}})^{\bullet \bullet}[\gamma] = -(n-2k)
V^{-\frac{2n-2k}{n}} \int_M (\psi + n \varphi^2) \dvol \\ +
(n-2k)(2n-2k) V^{-\frac{3n-2k}{n}} \left(\int_M \varphi \dvol
\right)^2.
\end{multline*}
Hence $(\F_{2k})^\bullet[\varphi] = (n\!-\!2k) \int_M \varphi v_{2k}
\dvol$ yields
\begin{multline*}
(\tilde{\F}_{2k})^{\bullet \bullet}[\gamma] = V^{-\frac{n-2k}{n}}
\Big[(\F_{2k})^{\bullet\bullet}[\gamma] \\ + 2k(n\!-\!2k) V^{-1}
v_{2k} \left(\int_M \varphi \dvol \right)^2 - n(n\!-\!2k) v_{2k}
\int_M \varphi^2 \dvol - (n\!-\!2k) v_{2k} \int_M \psi \dvol\Big];
\end{multline*}
here we utilized the fact that $v_{2k}$ is constant at $g$. Now
combining this formula with the consequence
$$
(\F_{2k})^{\bullet\bullet}[\gamma] = (n\!-\!2k) \left(-2k v_{2k}
\int_M \varphi^2 \dvol - \int_M \varphi \delta(L_{2k} d\varphi)
\dvol + v_{2k} \int_M (\psi\!+\!n\varphi^2) \dvol\right)
$$
of \eqref{F-second-var} gives
\begin{multline*}
(\tilde{\F}_{2k})^{\bullet \bullet}[\gamma] = - V^{-\frac{n-2k}{n}}(n\!-\!2k) \\
\left(2k v_{2k} \int_M \varphi^2 \dvol - 2k v_{2k} V^{-1}
\left(\int_M \varphi \dvol_g \right)^2 + \int_M \varphi
\delta(L_{2k} d\varphi) \dvol \right).
\end{multline*}
Hence the relation
$$
\int_M \left(\varphi - V^{-1} \int_M \varphi \dvol \right)^2 \dvol =
\int_M \varphi^2 \dvol - V^{-1} \left(\int_M \varphi \dvol\right)^2
$$
completes the proof.
\end{proof}

The arguments in the latter proof generalize those in \cite{GuLi}
(in the case $k=3$). Now a similar reasoning as above yields the
following result.

\begin{corr}\label{critical-revo} Let $g$ be an Einstein metric with
non-zero scalar curvature. Let $\gamma(t)$ be a curve in the
conformal class of $g$ with $\gamma(0)=g$ and $(\partial/\partial
t)|_0(\gamma(t)) = 2 \varphi g$. Then
\begin{multline}\label{second-tilde-F}
(\partial^2/\partial t^2)|_0 (\tilde{\F}_{2k}(\gamma(t))) \\
= V(g)^{-(n-2k)/n} (n\!-\!2k) \binom{n-1}{k-1} \frac{1}{2}
(-c)^{k-1} \int_M \bar{\varphi} (4n c \bar{\varphi} + \Delta
\bar{\varphi}) \dvol_g,
\end{multline}
where $c = \scal(g)/(4n(n\!-\!1))$. In particular, if $2k < n$ and
$\scal(g) > 0$, then $(-1)^k \tilde {\F}_{2k}$ has a local minimum
at $g$.
\end{corr}

If $\scal(g)>0$, $2k<n$ and $(M^n,g)$ is not isometric to a (rescaled)
round sphere $\s^n$, the right-hand side of \eqref{second-tilde-F}
vanishes iff $\bar{\varphi}=0$, i.e., iff $\varphi$ is
constant. Therefore, in this case, the functional
$(-1)^k\tilde{\F}_{2k}$ has a strict minimum (modulo rescalings). The
special case $k=3$ of this result is due to Guo and Li \cite{GL}.

Next, we outline the arguments which prove an analogous result for
the renormalized volume functional $\V_n(g_+;g)$.

We start by recalling the definition of $\V_n(g_+;g)$ following
\cite{G-vol}. Assume that $(M,c)$ is the conformal infinity of a
conformally compact Einstein manifold $(X,g_+)$. Then for any choice
of $g \in c$ there is a boundary defining function $0 < r \in
C^\infty(X)$ so that $g_+ = r^{-2}(dr^2+g(r))$ near the boundary $M$
and $g(0)=g$. Let
\begin{equation}\label{RV-def}
\V(g_+;g) \st \FP_{\varepsilon=0} \int_{r > \varepsilon}
\dvol_{g_+}.
\end{equation}
For more details see \cite{G-vol}, \cite{BJ}. As noted in
\cite{albin}, this Hadamard regularization of the volume $\int_X
\dvol_{g_+}$ coincides with the Riesz regularization
$$
\FP_{\lambda=0} \int_X r^\lambda \dvol_{g_+}.
$$

\begin{thm}[\cite{CFG}]\label{extremal-RV} Let $(M^n,g)$ be a
closed unit volume Einstein manifold of even dimension $n \ge 4$.
Assume that $(M^n,[g])$ is the conformal infinity of a
Poincar\'e-Einstein manifold $(X^{n+1},g_+)$. Then the restriction
of the renormalized volume functional $\V_n(g_+;\cdot)$ to the set
$[g]_1$ is critical at $g$ and the second variation at $g$ is given
by
\begin{equation}
\V_n^{\bullet\bullet}[\varphi] = (-c)^{\f-1} \frac{1}{2}
\binom{n-1}{\f-1} \int_M \varphi (4nc \varphi + \Delta \varphi) dv.
\end{equation}
\end{thm}

\begin{proof} By the general variational formula
$$
\V_n^\bullet[\varphi] = \int_M \varphi v_n dv
$$
(see \cite{G-vol}) and the arguments in the proof of Theorem
\ref{extremal-RVC}, it suffices to prove that
$$
v_n^\bullet[\varphi] = (-c)^{\f-1} \frac{1}{2} \binom{n-1}{\f-1}
(4nc \varphi + \Delta \varphi)
$$
at the Einstein metric $g_0$. But this is a consequence of Theorem
\ref{RVC-CV}, \eqref{v-Einstein} and \eqref{L-Einstein}.
\end{proof}

Again, Theorem \ref{extremal-RV} implies that $(-1)^\f
\V_n(g_+;\cdot)$ has a local minimum at $g$ if $\scal(g)>0$ and
$\V_n(g_+;\cdot$ has a local maximum at $g$ if $\scal(g)<0$. The
latter results also appear in \cite{GMS} (Lemma
4.4).\footnote{However, the arguments in \cite{GMS} are flawed by
the fact that the Hessian of the restriction of $\V_n$ to $[g_0]_1$
is incorrectly calculated using paths of the form $e^{t\varphi}g_0$.
For the round sphere $\s^n$, this has the effect that the incorrect
Hessian does not have the infinitesimal conformal factors of
conformal diffeomorphisms in its kernel.}

We continue with an analogous discussion of the functional $\int_M
a_{(2,n-2)} \dvol$.

We first recall that the functional
\begin{equation}\label{W}
\W(r;g) \st \int_M (\dot{w}(r))^2 \dvol_g
\end{equation}
is critical at Einstein metrics. This follows from Theorem \ref{var}
and the fact that $v(r)$ is constant for Einstein metrics.

\begin{thm}\label{extremal-crit} Let $(M^n,g)$ be a closed
Einstein manifold of even dimension $n \ge 4$. Assume that $g$ has
non-zero scalar curvature. Then the restriction of the (integrated)
critical heat kernel coefficient $a_{(2,n-2)}$ to the full conformal
class $c=[g]$ has a local extremum at $g$. More precisely,
\begin{itemize}
\item $(-1)^{\f-1} \int_M a_{(2,n-2)} \dvol$ has a local maximum at $g$
if $\scal(g)>0$, and
\item $\int_M a_{(2,n-2)} \dvol$ has a local minimum at $g$ if
$\scal(g)<0$.
\end{itemize}
If $(M^n,g)$ is not isometric to a (rescaled) round sphere then the
extreme values are strict (modulo rescaling).
\end{thm}

We recall that the total integral of $a_{(2,n-2)}$ is invariant
under rescalings (see \eqref{hom-crit}).

\begin{proof} The proof rests on a formula for the second conformal
variation of the functional
$$
\int_M a_{(2,n-2)} \dvol
$$
at an Einstein metric. By Theorem \ref{A-fine}, this functional
equals the sum of
$$
\VF_n = -\frac{n^2}{6} \int_M v_n \dvol
$$
and
$$
\W_{n-2} = \int_M (\dot w(r))^2_{n-2} \dvol.
$$
Since the functional $\VF_n$ is conformally invariant, it only
remains to determine the second conformal variation of $\W_{n-2}$ at
Einstein metrics. For this purpose, we need the following result.

\begin{prop}\label{second-hom} Let $g$ be an Einstein metric and let
$\gamma(t) = e^{2\varphi(t)}g$ be a curve in the full conformal
class of $g$ so that $\varphi(0)=0$, $\varphi'(0) = \varphi$. Then
\begin{equation}\label{W-crit-eval}
(\partial^2/\partial t^2)|_0 (\W_{n-2}(\gamma(t))) = \frac{1}{2}
\binom{n-4}{\f-2} (-c)^{\f-2} \int_M \varphi \Delta (4cn \!+\!
\Delta)(\varphi) \dvol_g,
\end{equation}
where $c=\scal(g)/(4n(n\!-\!1))$ and $\Delta = -\delta d$.
\end{prop}

\begin{proof} We note that
\begin{multline*}\label{W-second-var}
(\partial^2/\partial t^2)|_0 (\W(r;\gamma(t))) \\ =
(\partial/\partial t)|_0 ((\partial/\partial t) (\W(r;\gamma(t)))) =
(\partial/\partial t)|_0 ((\partial/\partial s)|_0
(\W(r;\gamma(t+s)))).
\end{multline*}
We evaluate the latter derivative at an Einstein metric. By Theorem
\ref{var} and Lemma \ref{conjugate}, we find that
$(\partial^2/\partial t^2)|_0 (\W(r;\gamma(t)))$ is the sum of the
two terms
\begin{equation} \label{second-var-w1}
(n\!-\!2\!-\!r\partial/\partial r) (\partial/\partial t)|_0 \left(
\int_M \varphi'(t) \dot{w}^2(r;\gamma(t)) \dvol_{\gamma(t)}\right)
\end{equation}
and
\begin{equation}\label{second-var-w2}
-\frac{1}{2} r \int_M \varphi \delta_g (g^{-1}(r) v(r)) d
((\dot{v}/v(r))^\bullet[\varphi]) \dvol_g.
\end{equation}
Here we utilized the fact that, for an Einstein metric, $v(r)$ is
constant on $M$ and thus $(\dot{v}/v)(r)$ is annihilated by $d$. Now
the first term does not contribute to the homogeneous part
$\W_{n-2}$. Thus, it only remains to analyze the corresponding
contribution by the second term. But, for an Einstein metric $g$,
\eqref{second-var-w2} simplifies to
$$
\frac{1}{2} r (1\!-\!cr^2)^{n-2} \int_M \varphi \Delta_g
((\dot{v}/v(r))^\bullet[\varphi]) \dvol_g.
$$
Next we prove that
\begin{equation}\label{log-v-var}
((\dot{v}/v)(r))^\bullet[\varphi] = r (1-cr^2)^{-2} (4cn \!+\!
\Delta)(\varphi), \; \Delta = -\delta d.
\end{equation}
For that purpose, we combine Theorem \ref{RVC-CV} and its
consequence
$$
(\dot{v})^\bullet[\varphi] = - \varphi(\dot{v} + r \ddot{v}) -
\delta (\dot{L} d\varphi)
$$
with the formula
$$
\left(\frac{\dot{v}}{v}\right)^\bullet[\varphi] =
\frac{(\dot{v})^\bullet[\varphi]}{v} - \frac{\dot{v}}{v^2}
v^\bullet[\varphi].
$$
We find that
\begin{equation*}
\left(\frac{\dot{v}}{v}\right)^\bullet[\varphi]  = \varphi \left( r
\frac{\dot{v}^2}{v^2} - r \frac{\ddot{v}}{v} - \frac{\dot{v}}{v}
\right) - \frac{1}{r} \delta (\dot{L} d\varphi) +
\frac{\dot{v}}{v^2} \delta(L d\varphi).
\end{equation*}
Now \eqref{log-v-var} follows from this formula by a direct
calculation using \eqref{v-Einstein} and \eqref{L-Einstein}. The
assertion follows by combining these results.
\end{proof}

Now by combining Proposition \ref{second-hom} with the spectral
decomposition of $\Delta$ on $M$, we deduce the following facts.
\begin{itemize}
\item [1.] Let $\scal(g)<0$. Then $c<0$ and the integral on the right-hand side
of \eqref{W-crit-eval} is strictly positive except for constant
$\varphi$.
\item [2.] Let $\scal(g)>0$. Then $c>0$. If $(M^n,g)$ is not isometric
to a (rescaled) round sphere, then Obata's estimate $\lambda_1 >
\scal(g)/(n\!-\!1)$ for the smallest non-trivial eigenvalue
$\lambda_1$ of $-\Delta$ implies that the integral on the right-hand
side of \eqref{W-crit-eval} is strictly positive for non-constant
$\varphi$. If $(M^n,g)$ is isometric to a (rescaled) round sphere,
then the right-hand side of \eqref{W-crit-eval} vanishes for
constant $\varphi$ and $\varphi \in \ker (\Delta + \lambda_1)$.
\end{itemize}
In all cases, non-trivial kernels of the right-hand side of
\eqref{W-crit-eval} reflect obvious invariance properties of
$\W_{n-2}$ (under rescalings and conformal diffeomorphisms). This
proves the assertions.
\end{proof}

Finally, we discuss the extremal behaviour of the functionals
$$
\W_{2k}(g) = \int_M (\dot w(r))^2_{2k} \dvol_g, \; 2k \ne n-2
$$
at Einstein metrics.

\begin{thm}\label{ext-w-subcrit} Let $(M^n,g)$ be a closed
unit volume Einstein manifold of dimension $n \ge 3$. Assume that
$\scal(g) > 0$ and let $2k < n\!-\!2$. Then the restriction of the
functional $(-1)^k \W_{2k}$ to $c_1 = [g]_1$ has a local maximum at
$g$. Similarly, if $\scal(g) < 0$ and $2k > n\!-\!2$, the
restriction of the functional $\W_{2k}$ to $c_1$ has a local minimum
at $g$. In the case $\scal(g)>0$, the local maximum is strict unless
$(M^n,g)$ is isometric to a (rescaled) round sphere.
\end{thm}

\begin{proof} In the following, we shall use the notation of the
proof of Theorem \ref{second-hom}. That proof shows that the second
conformal variation
$$
(\partial^2/\partial t^2)|_0 (\W(r;\gamma(t)))
$$
is the sum of \eqref{second-var-w1} and \eqref{second-var-w2}. The
first of these two terms equals
\begin{equation*}
(n\!-\!2\!-\!r\partial/\partial r) \left( \int_M  [(\psi + n
\varphi^2) \dot{w}^2(r) + \varphi (\partial/\partial t)|_0
(\dot{w}^2(r;\gamma(t)))] \dvol_g \right).
\end{equation*}
Since the path $\gamma(t)$ consists of unit volume metrics and
$\dot{w}^2(r)$ is constant on $M$, this contribution simplifies to
\begin{equation}\label{w-sub-h}
(n\!-\!2\!-\!r\partial/\partial r) \left( \int_M \varphi
(\dot{w}^2(r))^\bullet[\varphi] \dvol_g \right).
\end{equation}
Now the proof of Theorem \ref{var} in Section \ref{prime} shows that
at general metrics
$$
4 (\dot{w}^2(r))^\bullet[\varphi] = (-2 \dot{v}/v \delta (\dot{L}
d\varphi) + (\dot{v}/v)^2 \delta (L d\varphi) - 2 \dot{v}/v
(\partial/\partial r)(r \dot{v})\varphi + (\dot{v}/v)^2 r \dot{v}
\varphi).
$$
But a calculation using \eqref{v-Einstein} and \eqref{L-Einstein}
shows that for an Einstein metric the latter sum simplifies to
$$
-2n (2cr^2 \!-\! n c^2 r^4) (1\!-\!cr^2)^{n-3} (4cn \!+\!
\Delta)(\varphi).
$$
For $k>1$, the coefficient of $r^{2k}$ in the $r$-expansion of this
sum equals
$$
(-c)^k 2n \left(2 \binom{n-3}{k-1} + n \binom{n-3}{k-2}\right)(4cn
\!+\! \Delta)(\varphi);
$$
the same formulas holds true also for $k=1$ with the convention that
the second binomial coefficient vanishes. Hence \eqref{w-sub-h}
contributes by
\begin{equation}\label{var-h1}
(n\!-\!2\!-\!2k) (-c)^k \frac{n}{2} \left(2 \binom{n-3}{k-1} + n
\binom{n-3}{k-2}\right) \int_M \varphi (4cn \!+\! \Delta)(\varphi)
\dvol_g
\end{equation}
to the second variation of $\W_{2k}$ at $g$. Next, by arguments as in the
proof of Proposition \ref{second-hom}, \eqref{second-var-w2} yields
$$
\frac{r^2}{2} (1\!-\!cr^2)^{n-4} \int_M \varphi \Delta (4cn \!+\!
\Delta)(\varphi) \dvol_g.
$$
This term contributes by
\begin{equation}\label{var-h2}
\frac{1}{2} \binom{n-4}{k-1} (-c)^{k-1} \int_M \varphi \Delta
(4cn\!+\!\Delta)(\varphi) \dvol_g
\end{equation}
to the second variation of $\W_{2k}$ at $g$.

Now assume that $\scal(g) > 0$ and that $(M^n,g)$ is not isometric to a
(rescaled) round sphere. Then Obata's estimate $\lambda_1 >
\scal(g)/(n\!-\!1) = 4cn$ implies that the integrals in \eqref{var-h1} and
\eqref{var-h2} are negative and non-negative, respectively. Thus, in the
case of positive scalar curvature, the second variation of $(-1)^k
\W_{2k}$ at $g$ is negative if $2k < n-2$. Similarly, if $\scal(g) < 0$,
the second variation of $\W_{2k}$ at $g$ is positive if $2k>n-2$. This
proves the assertions in this case. Finally, by the same arguments as in
the proof of Theorem \ref{extremal-RVC}, the assertions hold true also if
$(M^n,g)$ is isometric to a (rescaled) round sphere .
\end{proof}

Note that if $\scal(g) < 0$ and $2k<n-2$, the contributions
\eqref{var-h1} and \eqref{var-h2} have opposite signs, and there is
no analogous result.

For explicit formulas for the low-order scalar Riemannian invariants
$\omega_{2k} = (\dot w(r))^2_{2k}$ in terms of renormalized volume
coefficients we refer to Section \ref{corr-v}.

\section{Further results, comments and open problems}\label{open}

In the present section, we briefly discuss a number of further
results and point out some open problems.

\subsection{Global conformal invariants}\label{GCI}

Theorem \ref{A} yields a sequence of conformal variational
integrals. These are defined by integration of local Riemannian
invariants which are given by {\em correcting} the sub-leading heat
kernel coefficient $a_2(r)$ of $\H(r)$ by $(\dot{w}(r))^2$. In fact,
for even $n$ and $2k < n\!-\!2$, Theorem \ref{A} implies that the
functional
\begin{equation}\label{L-func}
g \mapsto \int_M \Lambda_{2k}(g) \dvol_g \Big/ \left(\int_M \dvol_g
\right)^{\frac{n-2-2k}{n}}
\end{equation}
is {\em critical} at $g$ in the conformal class of $g$ iff
$$
\Lambda_{2k}(g) = \mbox{constant}
$$
(see also the discussion in Section \ref{extremal}). In the same
way, $a_0(r) = v(r)$ and Theorem \ref{conf-reno} show that the
Taylor coefficients of $a_0(r)$ give rise to a sequence of conformal
variational integrals. In this case, no correction terms are needed.
Of course, Theorem \ref{A-fine} tells that the resulting two
sequences essentially coincide.

It is natural to ask for generalizations of these results to all
heat kernel coefficients of $\H(r)$. Similarly, as for $a_2(r)$,
this leads to the problem of finding local conformal primitives of
the coefficients in the asymptotic expansion (for $t \to 0$) of the
restriction $\C(x,x,t;r)$ to the diagonal of the smooth kernel of
the operator
\begin{equation*}
\C(t;r) = \frac{1}{2\pi i} \int_\Gamma
[\R(\lambda),[\R(\lambda),\dot{\H}(r)]] e^{-t\lambda} d\lambda.
\end{equation*}
More precisely, let
$$
\sum_{k \ge 0} a_{(2j,2k)} r^{2k} \quad \mbox{and} \quad \sum_{k \ge
1} c_{(2j,2k)} r^{2k-1}
$$
be the respective Taylor expansions of $a_{2j}(r)$ for $j \ge 0$ and
$c_{2j}(r)$ for $j \ge 1$; we recall that $c_0 = 0$ (see also
\eqref{taylor-c}). The relation \eqref{conf-hc} implies the
conformal variational formula
\begin{equation}\label{CV-ac}
\left(\int_M a_{(2j,2k)} \dvol \right)^\bullet[\varphi] \\ =
(n\!-\!2j\!-\!2k) \int_M \varphi a_{(2j,2k)} \dvol - \frac{1}{2}
\int_M \varphi c_{(2j,2k)} \dvol.
\end{equation}
We rephrase that relation by saying that $a_{(2j,2k)}$ is a local
conformal primitive of $-1/2 c_{(2j,2k)}$. It follows that
$$
\mbox{the difference of $a_{(2j,2k)}$ and any {\em other} local
conformal primitive of $-1/2 c_{(2j,2k)}$}
$$
(in the described sense) is a non-trivial Riemannian invariant
$\Lambda_{(2j,2k)}$ so that
\begin{equation}\label{conf-Lambda}
\left(\int_{M^n} \Lambda_{(2j,2k)} \dvol \right)^\bullet[\varphi] =
(n\!-\!2j\!-\!2k) \int_{M^n} \varphi \Lambda_{(2j,2k)} \dvol.
\end{equation}
In particular for even $n$ and $2 \le 2j \le n-2$ the integrals
$$
\int_{M^n} \Lambda_{(2j,n-2j)} \dvol
$$
are conformally invariant.


\subsection{Spectral zeta functions}\label{CV-zeta}

The sub-critical renormalized volume coefficients $v_{2k}$ (for $2k
\le n-2$) have local conformal primitives given by multiples of
$v_{2k}$, respectively. In contrast, for even $n$, the critical
renormalized volume coefficient $v_n$ is the conformal anomaly of
the non-local renormalized volume functional (see \eqref{RV-def}):
$$
(d/dt)|_0 (\V_n(g_+;e^{2t\varphi}g)) = \int_{M^n} \varphi v_n(g)
\dvol_g.
$$
This fact may be regarded as an analog of the fact that the
(logarithm of the) non-local functional determinant of the conformal
Laplacian is a conformal primitive of the critical heat kernel
coefficient $a_n$:
\begin{equation}\label{zeta-conform}
(\log {\det}(-P_2(g)))^\bullet[\varphi] = -2 \int_{M^n} \varphi a_n(g)
\dvol_g
\end{equation}
(see \cite{PR}, \cite{BO-index})\footnote{ In the present section we
absorb the coefficient $(4\pi)^{-\f}$ into the definition of the heat
kernel coefficients. In particular, with this convention, $a_n \in
C^\infty(M)$ denotes the constant term in the asymptotic expansion of the
restriction to the diagonal of the heat kernel of $P_2$.}; here we assume
that the kernel of $P_2(g)$ is trivial.

We briefly recall the definition of the functional determinant
$\det(-P_2)$. Assume that the spectrum of $-P_2$ consists only of positive
eigenvalues $\lambda_k$. Then the spectral zeta function
$$
\zeta(s) \st \sum_k \lambda_k^{-s}, \quad \Re(s) > n/2
$$
admits a meromorphic continuation to $\c$ with at most simple poles
off $s=0$. Since the continuation of $\zeta(s)$ is holomorphic at
$s=0$, it can be used to define the determinant of $-P_2$ by the
formula
\begin{equation}\label{det}
\det (-P_2) \st \exp(-\zeta'(0)).
\end{equation}
In general, $-P_2$ has a non-trivial kernel and finitely many negative
eigenvalues. The definition of the determinant extends to that case by
$$
\det (-P_2) = (-1)^{\# (negative \; eigenvalues)} \exp (-\zeta'(0))
$$
with the modified zeta function
$$
\zeta(s) \st \sum_{\lambda_k \ne 0} |\lambda_k|^{-s}, \quad \Re(s) > n/2.
$$

It is natural to look at the analogy between the determinant and the
renormalized volume from the perspective of the two-parameter spectral
zeta-function
\begin{equation}\label{zeta}
\zeta(r;s) \st \Tr ((-\H(r))^{-s}) = \frac{1}{\Gamma(s)}
\int_0^\infty t^{s-1} \Tr(e^{t\H(r)}) dt, \; \Re(s) > n/2;
\end{equation}
here we assume again that the spectrum of $-\H(r)$ consists only of
positive eigenvalues $\lambda_k(r)$ (for sufficiently small $r$).
Then
$$
\zeta(r;s) = \sum_k \frac{1}{\lambda_k(r)^s}, \quad \Re(s) > n/2.
$$
Again, by standard arguments, the function $\zeta(r;s)$ admits a
meromorphic continuation to $\c$ (in the variable $s$) with at most
simple poles $s \ne 0$. Obviously, $\zeta(0;s)$ coincides with the
spectral zeta-function of the conformal Laplacian $P_2$.

Now we consider the behaviour of $\zeta(s) = \zeta(0;s)$ near $s=0$.
The constant term in the Laurent series
$$
\zeta(s) = \zeta(0) + s \zeta'(0) + s^2 ( \cdot )
$$
near $s=0$ satisfies
$$
\zeta(0) = \int_M a_n \dvol.
$$
The conformal variational formula for the trace of the heat kernel
implies that
\begin{equation}\label{zeta-conform-0}
\zeta(s)^\bullet[\varphi] = 2 s \zeta(\varphi;s)
\end{equation}
with the local zeta function
$$
\zeta(\varphi;s) \st \frac{1}{\Gamma(s)} \int_0^\infty t^{s-1} \Tr(
\varphi e^{t\H(0)}) dt, \; \Re(s) > n/2
$$
of the conformal Laplacian $\H(0)=P_2$. Hence, by the conformal
invariance of $\zeta(0)$, we find that
$$
s (\zeta'(0))^\bullet[\varphi] + s^2 (\cdot) = 2 s \zeta(\varphi;s).
$$
Thus
\begin{equation}\label{CV-zeta-0}
\zeta'(0)^\bullet[\varphi] = 2 \int_M \varphi a_n \dvol.
\end{equation}
This is the basic argument which yields the infinitesimal Polyakov
formula \eqref{zeta-conform}.

Next, we consider the behaviour of $\zeta(r;s)$ near $s=n/2$. The
conformal variational formula in Theorem \ref{conform-heat} shows
that for $\Re(s) > \f$
$$
\zeta(r;s)^\bullet[\varphi] = \frac{1}{\Gamma(s)} \int_0^\infty
t^{s-1} \Tr(e^{t\H(r)})^\bullet[\varphi] dt
$$
equals
\begin{equation*}
\frac{1}{\Gamma(s)} \Big(-2 \int_0^\infty t^s (\partial/\partial t)
(\Tr (\varphi e^{t\H(r)})) dt - \frac{r}{2} \int_0^\infty t^s \Tr
(\varphi (\dot{\H}(r) e^{t\H(r)} + e^{t\H(r)} \dot{\H}(r))) dt
\Big).
\end{equation*}
By the arguments on page \pageref{refo-dc} (and partial integration
in the first integral), the latter sum can be rewritten as
\begin{equation}\label{variation-zeta}
\frac{2s}{\Gamma(s)} \int_0^\infty t^{s-1} \Tr(\varphi e^{t\H(r)})
dt - \frac{r}{\Gamma(s)} \int_0^\infty t^{s-1} (\partial/\partial r)
(\Tr(\varphi e^{t\H(r)})) dt,
\end{equation}
up to the double-commutator term
\begin{equation}\label{dct}
-\frac{1}{\Gamma(s)} \frac{r}{2} \int_0^\infty t^{s-1} \int_M
\varphi(x) \C(x,x,t;r) \dvol dt.
\end{equation}
We interpret the sum in \eqref{variation-zeta} in terms of local
zeta functions as
$$
(2s\!-\!r\partial/\partial r) \zeta(\varphi;r;s).
$$

Now we restrict the considerations to conformal variations at an
Einstein metric $g_0$. This has two effects: 1. the holographic
Laplacian $\H(r;g)$ is defined for all metrics $g$ in the conformal
class $[g_0]$ to {\em all} orders in $r$, and 2. the double
commutator term \eqref{dct} vanishes at $g_0$. It follows that the
conformal variation of the zeta function at an Einstein metric is
given by the formula
\begin{equation}\label{zeta-CVE}
\zeta(r;s)^\bullet[\varphi] = (2s\!-\!r\partial/\partial r)
\zeta(\varphi;r;s), \; \Re(s) > n/2.
\end{equation}
This relation generalizes \eqref{zeta-conform-0}.

Now let $n$ be even. We expand $\zeta(r;s)$ in the variable $r$:
\begin{equation}\label{zeta-expand}
\zeta(r;s) = \zeta_0(s) + r^2 \zeta_2(s) + \dots + r^n \zeta_n(s) +
\cdots.
\end{equation}
A similar expansion of the local zeta function $\zeta(\varphi;r;s)$
defines $\zeta_n(\varphi;s)$. Now \eqref{zeta-CVE} implies
\begin{equation}\label{zeta-n}
\zeta_n(s)^\bullet[\varphi] = (2s-n) \zeta_n(\varphi;s).
\end{equation}
The residue of $\zeta(r;s)$ at $s=n/2$ is proportional to the total
integral of $a_0(r)$, i.e., to the total integral of $v(r)$. Hence
the function $\zeta_n(s)$ has a Laurent series of the form
$$
\zeta_n(s) = \frac{\res_{\f}(\zeta_n)}{s-\f} + Z_{(0,n)} +
O(s\!-\!n/2)
$$
near $s=\f$ with a residue which is proportional to the total
integral of $v_n$. Similarly, the residue of $\zeta_n(\varphi;s)$ at
$s=n/2$ is proportional to $\int_M \varphi v_n \dvol$. We recall
that the total integral of $v_n$ is conformally invariant.
Therefore, the relation \eqref{zeta-n} implies the variational
formula
\begin{equation}\label{CV-zeta-n}
Z_{(0,n)}^\bullet[\varphi] = 2 c_n \int_{M^n} \varphi v_n \dvol,
\quad c_n^{-1} = \left(\f\!-\!1\right)! (4\pi)^{\f}
\end{equation}
at an Einstein metric. Since $v_n$ is constant at Einstein metrics,
\eqref{CV-zeta-n} implies that, under volume preserving conformal
variations, the functional $Z_{(0,n)}$ is {\em critical} at Einstein
metrics.

The above observation admits the following generalization to {\em
all} critical heat kernel coefficients $a_{(2j,n-2j)}$. For
$j=0,\dots,n/2-1$, the function $\zeta(r;s)$ has a Laurent series of
the form
$$
\zeta(r;s) = \frac{\res_{\f-j}(\zeta(r;\cdot))}{s-(\f-j)} + Z_j(r) +
O(s\!-\!n/2+j)
$$
near $s=n/2-j$. Let $Z_j(r) = Z_{(j,0)} + r^2 Z_{(j,2)} + \cdots +
r^n Z_{(j,n)} + \cdots$. In these terms, \eqref{zeta-CVE} implies
the variational formulas
\begin{equation}\label{CV-residue}
\res_{\f-j} (\zeta_{2k}(\cdot))^\bullet[\varphi] = (n\!-\!2j\!-\!2k)
\res_{\f-j}(\zeta_{2k}(\varphi;\cdot))
\end{equation}
and
\begin{equation}\label{CV-zeta-k}
Z_{(j,n-2j)}^\bullet[\varphi] = 2 \res_{\f-j}
(\zeta_{n-2j}(\varphi;\cdot))
\end{equation}
at Einstein metrics. Since the residue
$\res_{\f-j}(\zeta_{2k}(\varphi;\cdot))$ is proportional to
$$
\int_{M^n} \varphi a_{(2j,2k)} \dvol,
$$
\eqref{CV-residue} just restates \eqref{conf-hc} for Einstein
metrics.

Now let $n \ge 3$ be odd. In this case, the functions $s \mapsto
\zeta(r,s)$ and $s \mapsto \zeta(\varphi;r,s)$ are regular at
$$
s \in \{0,1,\dots,(n-1)/2,\dots\}.
$$
The relation \eqref{zeta-CVE} implies that the restriction of {\em
all} zeta-values
\begin{equation}\label{zeta-v}
\zeta_{2k}(k) \quad \mbox{for $k \in \{0,1,\dots,(n-1)/2,\dots\} =
\N_0$}
\end{equation}
to the conformal class of an Einstein metric $g$ are critical at
$g$. This result extends the well-known conformal invariance of the
zeta-value $\zeta(0) = \zeta_0(0)$ in odd dimensions. Note that the
function $\zeta_2(s)$ is given by the weighted Dirichlet series
$$
-s \sum_k \frac{\dot{\lambda}_k(0)}{\lambda_k(0)}
\frac{1}{\lambda_k(0)^s}, \quad \Re(s) > n/2,
$$
where $\dot{\lambda}_k = \partial \lambda_k /\partial(r^2)$. It is
an interesting open problem to study the {\em second} variation of
the zeta values $\zeta_{2k}(k)$ at Einstein metrics.

An extension of the above results to general metrics would require a
proper understanding of the term \eqref{dct} (being holomorphic at
$s=n/2$ in view of $c_0=0$). In this direction, it would be of
interest to characterize the metrics for which the zeta values
$\zeta_{2k}(k)$ (for odd $n$ and $k \in \N_0$) are critical in its
conformal class.

Finally, we emphasize that in odd dimensions the functionals
$$
g \mapsto \zeta_{2k}(k;g)
$$
are scale-invariant. In fact, \eqref{CCC} shows that
$$
\zeta(r;s;\lambda^2 g) = \lambda^{2s} \zeta(r/\lambda;s;g).
$$
Hence
$$
\zeta_{2k}(k;\lambda^2 g) = \zeta_{2k}(k;g).
$$
The latter observation should be regarded as a natural analog of the
scale-invariance of the total integrals of the critical heat kernel
coefficients $a_{(2j,n-2j)}$ in even dimensions (see
\eqref{hom-crit}).

\subsection{A heat kernel proof of the variational formula of $v(r)$}
\label{CV-v-local}

We have seen in Section \ref{proof-A} that the conformal variational
formulas for the integrated renormalized volume coefficients
(Theorem \ref{conf-reno}) are consequences of the variational
formula for the {\em trace} of the heat kernel of $\H(r;g)$ (Theorem
\ref{conform-heat}). The proof of Theorem \ref{conform-heat} itself
rests on Theorem \ref{conform-H}. But Theorem \ref{conform-H}
actually implies also a conformal variational formula for the local
(i.e., non-integrated) leading heat kernel coefficient $a_0(r;g) \in
C^\infty(M)$ of $\H(r;g)$. In fact, a version of Lemma
\ref{double-key} states that the variation
\begin{equation}\label{var-heat-l}
(\partial /\partial \varepsilon)|_0 (\exp(t \H(r;e^{2\varepsilon
\varphi} g)))
\end{equation}
equals the sum
\begin{multline}\label{var-heat-r}
\frac{1}{2} t ((\partial/\partial \varepsilon)|_0
(\H(r;e^{2\varepsilon \varphi}g)) e^{t\H(r;g)} + e^{t\H(r;g)}
(\partial/\partial \varepsilon)|_0 (\H(r;e^{2\varepsilon \varphi}g))) \\
- \frac{1}{2} \frac{1}{2\pi i} \int_\Gamma
[\R(\lambda),[\R(\lambda),(\partial/\partial \varepsilon)|_0
(\H(r;e^{2\varepsilon \varphi}g))]] e^{-t\lambda} d\lambda.
\end{multline}
We determine the coefficients of $(4 \pi t)^{-n/2}$ in the
asymptotic expansions for $t \to 0$ of the restriction to the
diagonal of the kernels of both sides of this identity. On the one
hand, for \eqref{var-heat-l} we obtain
$$
(\partial/\partial \varepsilon)|_0 (a_0(r;e^{2\varepsilon \varphi}
g)) + n \varphi a_0(r;g).
$$
On the other hand, the integral in \eqref{var-heat-r} has an
asymptotic expansion which starts with a multiple of $(4 \pi
t)^{-n/2} t$. Therefore, only the first two terms in
\eqref{var-heat-r} contribute. Since $\H(r;g)$ and
$(\partial/\partial \varepsilon)|_0 (\H(r;e^{2\varepsilon
\varphi}g))$ are self-adjoint, the contributions by these two terms
coincide with the contribution by
$$
t ((\partial/\partial \varepsilon)|_0 (\H(r;e^{2\varepsilon
\varphi}g)) e^{t\H(r;g)}.
$$
Hence, by Theorem \ref{conform-H}, it only remains to determine the
corresponding coefficient in the asymptotic expansion of the
restriction to the diagonal of the kernel of the operator
\begin{align}\label{sum-cont}
& - \frac{1}{2}rt (\varphi \dot{\H}(r;g) + \dot{\H}(r;g) \varphi)
e^{t\H(r;g)} - 2 t \varphi \H(r;g) e^{t\H(r;g)} \nonumber \\
& + \left(\frac{n}{2}-1\right) t [\H(r;g),\varphi] e^{t\H(r;g)} - t
[\H(r;g),[\K(r;g),\varphi]] e^{t\H(r;g)}.
\end{align}
By Lemma \ref{double-key}, the sum of the contributions by the
kernels of the operators in the first line of \eqref{sum-cont}
coincides with that of the kernel of
$$
-\varphi r (\partial/\partial r) (e^{t\H(r;g)}) - 2 t \varphi
(\partial/\partial t) (e^{t\H(r;g)}).
$$
Thus, the kernels of the operators in the first line of
\eqref{sum-cont} contribute by
$$
- \varphi r (\partial/\partial r) (a_0(r;g)) + n \varphi a_0(r;g)
$$
to the coefficient of $(4\pi t)^{-n/2}$. Now, since the first term
in the second line of \eqref{sum-cont} does not contribute to that
coefficient, we have proved the relation
\begin{equation}\label{leading-conform}
(\partial/\partial \varepsilon)|_0 (a_0(r;e^{2\varepsilon \varphi}
g)) = - \varphi r (\partial/\partial r) (a_0(r;g)) -
\kappa_0(r;g;\varphi),
\end{equation}
where $\kappa_0(r;g;\varphi)$ defines the {\em leading} coefficient
in the asymptotic expansion of the restriction to the diagonal of
the kernel of
$$
t[\H(r;g),[\K(r;g),\varphi]] (\exp (t\H(r;g))),
$$
i.e.,
\begin{equation}\label{DC-heat}
t [\H(r;g),[\K(r;g),\varphi]] (\exp (t\H(r;g)))|_{y=x} \sim (4\pi
t)^{-\f} (\kappa_0(r;g;\varphi) + O(t))
\end{equation}
for $t \to 0$. Hence by combining the relation $a_0(r) = v(r)$ with
the formula
\begin{equation}\label{alpha-form}
\kappa_0(r;g;\varphi) = \delta_g (v(r) \int_0^r s g(s)^{-1} ds
d\varphi)
\end{equation}
we complete the {\bf heat kernel proof} of Theorem \ref{RVC-CV}. In the
following, we shall outline the arguments which prove \eqref{alpha-form}.

We start with a discussion of the special case of conformal
variations of renormalized volume coefficients at the round metric
$g_0$ on $\s^n$. In that case, the double-commutator term
$[\H(r;g),[\K(r;g),\varphi]]$ reduces to
$$
\frac{r^2}{4} (1\!-\!r^2/4)^{-3} [\Delta_{g_0},
[\Delta_{g_0},\varphi]].
$$
(see the discussion in Section \ref{variation}). Moreover, the
off-diagonal expansion of the heat kernel
\begin{equation*}
\exp (t \H(r;g_0))(x,y) = \exp (t(1\!-\!r^2/4)^{-2}P_2(g_0))(x,y),
\; x \sim y
\end{equation*}
for $t \to 0$ starts with
$$
(4\pi t)^{-\f} (1\!-\!r^2/4)^n \exp
(-d(x,y)^2/(4t(1\!-\!r^2/4)^{-2})) a_0(x,y)
$$
with $a_0(x,x) = 1$; here $d(\cdot,\cdot)$ denotes the distance
function of $g_0$. Hence the identity
$$
[\Delta, [\Delta,\varphi]] (u) = 4 (\Hess(\varphi),\Hess(u)) +
\mbox{lower-order terms}
$$
and the general formula
\begin{equation}\label{Hess}
\Hess_x (\exp (-d^2(x,y)/4t))|_{y=x} = - 1/2 t g_x
\end{equation}
imply that the leading coefficient in the asymptotic expansion of
\eqref{DC-heat} equals
$$
-(4\pi t)^{-\f} \frac{r^2}{2} (1\!-\!r^2/4)^{n-1}
\tr(\Hess(\varphi)) = -(4\pi t)^{-\f} \frac{r^2}{2}
(1\!-\!r^2/4)^{n-1} \Delta(\varphi).
$$
Thus, we find
$$
\kappa_0(r;g;\varphi) = -\frac{r^2}{2} (1\!-\!r^2/4)^{n-1}
\Delta(\varphi).
$$
Now an easy calculation shows that the latter result coincides with
$$
\delta (v(r) \int_0^r s g(s)^{-1} ds d\varphi).
$$
Hence \eqref{leading-conform} reads
$$
v(r)^\bullet[\varphi] = - \varphi r \dot{v}(r) - \delta (v(r)
\int_0^r s g(s)^{-1} ds d\varphi).
$$
This is a special case of Theorem \ref{RVC-CV}.

Finally, we outline the arguments which prove \eqref{alpha-form} in
the general case. On general grounds, the differential operator
$[\H(r;g),[\K(r;g),\varphi]]$ is only second-order. It suffices to
determine its leading part. For that purpose, we use Lemma
\ref{conjugate} to write the operators $\H(r;g)$ and $\K(r;g)$ in
the form
$$
\H(r;g) = v(r) \delta_{g(r)}(v(r)^{-1} d) + \mbox{potential}
$$
and
$$
2 \K(r;g) = v(r)\delta_{g(r)}(v(r)^{-1} E(r) d) + \mbox{potential},
$$
where $E(r)$ is the endomorphism
$$
E(r) \st g(r) \int_0^r s g(s)^{-1} ds
$$
on $\Omega^1(M)$. Now a (long) calculation shows that
$$
[\H(r;g),[2\K(r;g),\varphi]](u) = 4 (\nabla^{g(r)}(E(r)
d\varphi),\Hess^{g(r)}(u))_{g(r)} + \mbox{lower-order terms}.
$$
Hence the relation \eqref{Hess} (for $g(r)$) implies that
$$
t[\H(r;g),[\K(r;g),\varphi]]_x(\exp (-d^2(r)(x,y)/4t))|_{y=x} =
\delta_{g(r)} (E(r) d \varphi),
$$
where $d(r)$ denotes the distance function of $g(r)$. By Lemma
\ref{conjugate} and the definition of $E(r)$, the latter result
coincides with
$$
v(r)^{-1} \delta_g (v(r) \int_0^r s g(s)^{-1} ds d \varphi).
$$
Since the off-diagonal asymptotic expansion of the heat kernel of
$\H(r;g)$ starts with
$$
(4\pi t)^{-\f} v(r) \exp (-d^2(r)(x,y)/4t) (a_0(r;g)(x,y) + O(t)),
\; t \to 0
$$
with $a_0(r;g)(x,x) = 1$,\footnote{Here we use that $\H(r;g)$ is a
Laplace-type operator for $g(r)$. We also recall that we define its
kernel by integration against the volume of $g$. This brings in the
additional factor $v(r)$.} the above results prove the formula
\eqref{alpha-form}. This completes the heat kernel proof of Theorem
\ref{RVC-CV}.

\subsection{Hessians}\label{hess}

Although the variational formulas for the integrated renormalized
volume coefficients and the integrated heat kernel coefficients have
a similar structure, the second conformal variations of these
functionals reveal important differences. On the one hand, we have
seen in Section \ref{extremal} that the (conformal) Hessian forms of
the functionals
$$
\F_{2k} = \int_M v_{2k} \dvol, \; k=1,\dots,n/2-1
$$
at Einstein metrics are quadratic forms which are given by a
second-order differential operator (see \eqref{second-F2k}). For
even $n$, a natural substitute for the conformally invariant
integral $\F_n$ is given by the renormalized volume $\V(g_+;\cdot)$.
The structure of the Hessian form of this functional at Einstein
metrics is similar to that of the Hessian forms of $\F_{2k}$
(Theorem \ref{extremal-RV}). On the other hand, for the functional
determinant of $P_2$, the following conjecture extrapolates from
results in low dimensions. For the definition of the determinant we
refer to Section \ref{CV-zeta}.

\begin{conj}\label{det-local-ext} Let $M^n$ be a closed manifold of
even dimension $n \ge 4$ with a locally conformally flat Einstein metric
$g$ of unit volume. Assume that $\ker (P_2(g))$ is trivial. Then the
restriction of the functional $\log \det (-P_2)$ to $[g]_1$ is critical at
$g$ and its second variation at $g$ is given by
\begin{equation}
(4\pi)^\f (\log \det (-P_2))^{\bullet\bullet}[\varphi] = (-1)^{\f+1}
\int_{M^n} \varphi (\Delta+4nc) \Pi_{n-2}(c;-\Delta)(\varphi) dv
\end{equation}
for some real polynomial $\lambda \mapsto \Pi_{n-2}(c;\lambda)$ of
degree $n/2-1$ so that
$$
\Pi_{n-2}(ac;a\lambda) = a^{\f-1} \Pi_{n-2}(c;\lambda), \; a \in \r.
$$
Here $\Delta = -\delta d$ is the non-negative Laplacian and $c=
\scal(g)/(4n(n-1))$. Moreover, if $c>0$, we have
\begin{equation}\label{positive}
\Pi_{n-2}(c;\lambda) > 0 \quad \mbox{for $0 < \lambda \in
\sigma(-\Delta)$}.
\end{equation}
\end{conj}

Some comments are in order. The factor $\Delta+4nc$ reflects the
invariance of the functional under conformal diffeomorphism of the
round sphere. In low dimensions, Conjecture \ref{det-local-ext} is
supported by the results in Section \ref{holo-heat}.

The positivity \eqref{positive} for $\scal(g)>0$ would imply that
the restriction of $\log \det (-P_2)$ to $[g]_1$ has a local maximum
or local minimum at $g$ if $n \equiv 2 \!\! \mod{4}$ or $n \equiv 0
\! \! \mod{4}$, respectively. For the round sphere $\s^n$, the
latter max/min-pattern would correspond to the deeper (conjectural)
global behaviour of the determinant of $P_2$ on the conformal class
of the round metric \cite{spec-ineq}, \cite{sigma}.

In order to describe the relation of Conjecture \ref{det-local-ext}
to well-known conjectures, we recall that on a locally conformally
flat manifold $(M^n,g)$ of even dimension $n \ge 4$ and with
$\ker(P_2)=0$, it is expected that \cite{origin}
\begin{multline}\label{polyakov-g}
\log\left(\frac{\det(-P_2(\hat{g}))}{\det(-P_2(g)) }\right) = a
\int_{M^n} \left(\frac{1}{2} (-1)^\f \varphi P_n(g)(\varphi) +
Q_n(g) \varphi \right)\dvol_g \\ + \int_{M^n} \left(F(\hat{g})
\dvol_{\hat{g}} - F(g)\dvol_g \right).
\end{multline}
Here $\hat{g} = e^{2\varphi}g$, $a$ is a non-vanishing real
constant, $Q_n$ is the critical $Q$-curvature and $F$ is some local
scalar Riemannian invariant.\footnote{The sign $(-1)^\f$ in
\eqref{polyakov-g} is due to the convention $P_n = \Delta^\f +
\cdots$ with $\Delta=-\delta d$.} A basic feature of this {\em
global} Polyakov formula for $P_2$ is that it connects the
determinant with $Q$-curvature. In this connection, we also note
that the first integral on the right-hand side of \eqref{polyakov-g}
can be written in the form
$$
\frac{1}{2} \int_{M^n} \varphi (Q_n(\hat{g}) \dvol_{\hat{g}} +
Q_n(g) \dvol_g).
$$
Explicit formulas for $a$ and the invariant $F$ are only known for
$n \le 6$. In particular, in dimension $n=4$ we have
$$
a = (4\pi)^{-2} \frac{1}{45}  \quad \mbox{and} \quad F = (4\pi)^{-2}
\frac{1}{45} \J^2
$$
(see Section \ref{pol-det}). In the situation of Conjecture
\ref{det-local-ext}, the formula \eqref{polyakov-g} implies that,
for volume preserving conformal variations $\gamma_t =
e^{2\varphi_t}g$ with $\varphi_t = t\varphi + \dots$, the second
variational formula reads
$$
(\log \det (-P_2))^{\bullet\bullet}[\varphi] = a \H_Q (\varphi) +
\H_F (\varphi)
$$
with the quadratic form
$$
\H_Q (\varphi) = \int_{M^n} \varphi ((-1)^\f P_n - n Q_n)(\varphi)
\dvol.
$$
Thus, in order to study extremal values of the determinant of $P_2$
in the volume preserving conformal class of an Einstein metric $g$,
it is important to understand the definiteness of the quadratic
forms $\H_Q$ and $\H_F$. While we shall see below that for $\H_Q$
this is easy in general dimensions, the Hessian $\H_F$ is subtle and
its study requires additional structural insight.

The positive semi-definiteness of $\H_Q$ at an Einstein metric $g$
of {\em positive} scalar curvature can be seen as follows. For an
Einstein metric $g$, $P_n$ and $Q_n$ are given by
$$
P_n(g) = \prod_{j=\f}^{n-1} (\Delta_g - j(n\!-\!1\!-\!j)
\scal(g)/(n(n\!-\!1))), \; \Delta_g = -\delta_g d
$$
and
$$
Q_n(g) = (n\!-\!1)! \left(\scal(g)/(n(n\!-\!1))\right)^{\f}
$$
(see \cite{gover-product} and \cite{juhl-book}, Section 6.16). It
follows that the difference
$$
(-1)^\f P_n(g) - n Q_n(g) = (-1)^\f P_n(g) - (4c)^{\f} n!
$$
factors into the product of $-(\Delta+4nc)$ and a polynomial in
$-\Delta$ with {\em positive} coefficients. Now it suffices to apply
Obata's estimate $\lambda_1 \ge 4nc$ for the smallest non-trivial
eigenvalue of $-\Delta$.

On the sphere $\s^n$, the positive semi-definiteness of $\H_Q$
actually reflects the fact that the first integral on the right-hand
side of \eqref{polyakov-g} is non-negative. In fact, this follows
from an inequality of Beckner (see \cite{sharp}, section 2). For
more details see the proof of Corollary \ref{RV-global}.

In Section \ref{extremal}, we have seen that the extremal properties
of the total integral of $a_{(2,n-2)}$ at Einstein metrics (Theorem
\ref{extremal-crit}) are consequences of a beautiful second
variational formula for $\W_{n-2}$ (Proposition \ref{second-hom}).
The following conjecture extends this to the total integrals off all
critical heat kernel coefficients.

\begin{conj}\label{hess-inter} Let $M^n$ be a closed manifold of
even dimension $n \ge 4$ with a locally conformally flat Einstein
metric $g$. Then for $k=1,\dots,n/2-1$ the restriction of the
functional
$$
\int_{M^n} a_{(2k,n-2k)} \dvol
$$
to $[g]$ is critical at $g$ and its second variation at $g$ has the
form
\begin{equation}
c^{\f-k-1} (-1)^\f \int_{M^n} \varphi \Delta (\Delta+4nc) H_{2k-2}
(c;-\Delta) (\varphi) \dvol, \quad \Delta = -\delta d
\end{equation}
with a polynomial $\lambda \mapsto H_{2k-2}(c;\lambda)$ of degree
$k-1$ so that
$$
H_{2k-2}(ac;a\lambda) = a^{k-1} H_{2k-2}(c;\lambda), \; a \in \r.
$$
Moreover, if $c>0$, we have
$$
H_{2k-2}(c;\lambda) > 0 \quad \mbox{for $0 < \lambda \in
\sigma(-\Delta)$}.
$$
\end{conj}

The results in Section \ref{a4-holo} confirm the special case
$(n,k)=(6,2)$. Again, the factors $\Delta$ and $\Delta+4nc$ reflect
the scale-invariance of the functional and on the round sphere its
invariance under conformal diffeomorphisms, respectively.

Finally, we emphasize that according to Theorem \ref{extremal-RV},
Conjecture \ref{det-local-ext} and Conjecture \ref{hess-inter} the
Hessian forms of the functionals $\V_n(g_+,\cdot)$, $\log \det
(-P_2)$ and of the total integrals of $a_{(2k,n-2k)}$ at positive
Einstein metrics are {\em all} positive or negative semi-definite if
$n \equiv 0 \!\! \mod{4}$ or $n \equiv 2 \! \! \mod{4}$,
respectively.

\subsection{Spectral geometry}\label{spectral-geometry}

Spectral geometry studies the relations between the spectrum of the
Laplace-Beltrami operator of a Riemannian manifold and the geometry
of the underlying metric. In particular, it is of central interest
to understand which geometric quantities are determined by the
spectrum and to describe the sets of metrics for which the
corresponding Laplace-Beltrami operators have the same spectrum (the
so-called {\em isospectral problem}). In this context, spectral
invariants like the heat kernel coefficients play an important role
\cite{GPS}, \cite{Z-inverse}.

Analogous questions may be asked for other geometric operators like
the conformal Laplacian and the Dirac operator. In addition, the
present paper suggests to study inverse spectral problems for the
holographic Laplacian $\H(r;g)$. In the latter case, the spectrum
consists of an infinite sequence $\lambda_k(r)$ of eigenvalues
depending on a small parameter $r$. The asymptotic expansion
$$
\sum_k e^{t\lambda_k(r)} \sim (4 \pi t)^{-\f} \sum_{j \ge 0} t^j
\int_M a_{2j}(r;g) \dvol_g, \; t \to 0
$$
shows that the (integrated) heat kernel coefficients of $\H(r;g)$
are spectral invariants.

From the perspective of spectral geometry, the presence of the
parameter $r$ gives rise to new effects. We illustrate the new
quality of the situation by some examples of spectrally determined
quantities. For that purpose, we recall that the integrated heat
kernel coefficient $a_4$ of the Laplace-Beltrami operator is a
multiple of
$$
\int_M (2 |R|^2 - 2 |\Ric|^2 + 5 \scal^2) \dvol
$$
(see Section \ref{coeff-L}). Although this integral is a spectral
invariant, it is well known that the three quantities
\begin{equation}\label{special-invariants}
\int_M \scal^2 \dvol, \; \int_M |\Ric|^2 \dvol \quad \mbox{and}
\quad \int_M |R|^2 \dvol
\end{equation}
are {\em not determined} by the spectrum of the Laplacian (see
Section 5.2 of \cite{GPS}). On the other hand, \eqref{base-0} and
Theorem \ref{A-fine} show that the first two of the integrated heat
kernel coefficients of $\H(r;g)$ are given by the total integrals of
polynomials in renormalized volume coefficients. In particular, the
respective coefficients of $r^4$ and $r^2$ are given by the
integrals
$$
\int_M v_4 \dvol \quad \mbox{and} \quad  \int_M (2/3 (n\!-\!8) v_4 +
v_2^2) \dvol.
$$
Now the explicit formulas $2v_2 = -\J$, $8v_4 = \J^2-|\Rho|^2$ and
Proposition \ref{a-holo} (for $a_4$) imply that all quantities in
\eqref{special-invariants} are {\em determined} by the spectrum of
$\H(r;g)$. Similar arguments using $a_0(r) = v(r)$, Theorem
\ref{A-fine}, Proposition \ref{a42-final} and Proposition
\ref{a-holo} (for $a_6$) show that for locally conformally flat
metrics in dimension $n \ne 6,8$ the quantities
$$
\int_M \J^3 \dvol, \; \int_M \J|\Rho|^2 \dvol, \; \int_M \tr
(\Rho^3) \dvol  \quad  \mbox{and} \quad \int_M |d\J|^2 \dvol
$$
are determined by the spectrum of $\H(r;g)$.

\subsection{Relation to $Q$-curvature}\label{Q-curvature}

We recall from \cite{sharp} that under conformal variations the
total integral of Branson's $Q$-curvature $Q_{2N}$ satisfies
$$
\left( \int_{M^n} Q_{2N} \dvol \right)^\bullet[\varphi]= (n\!-\!2N)
\int_{M^n}\varphi Q_{2N} \dvol.
$$
Therefore, the relations \eqref{conf-Lambda} show that under
conformal variations the total integrals of the local Riemannian
invariants
$$
v_{2N}, \; \Lambda_{(2,2N-2)}, \; \cdots, \; \Lambda_{(2N-2,2)} \;\;
\mbox{and} \;\; a_{2N}
$$
behave similarly as $Q_{2N}$. In particular, it is natural to ask
for relations among these invariants.

In the low-order cases $N=1$ and $N=2$, we find the following
(possibly accidental) results.

\begin{lemm}\label{Q-a-4} We have
$$
Q_2 = - \frac{6}{n\!-\!4} a_2
$$
for general metrics and
\begin{equation}\label{magic-4}
Q_4 = \frac{60 a_4}{n\!-\!6} - 4 \Lambda_2
\end{equation}
for locally conformally flat metrics.
\end{lemm}

\begin{proof} The first assertion is obvious. By combining
$$
Q_4 = \f \J^2 - 2|\Rho|^2 - \Delta \J \quad \mbox{and} \quad v_4 =
\frac{1}{8} (\J^2-|\Rho|^2)  \quad \mbox{(see Section
\ref{holo-heat})}
$$
with
$$
\Lambda_2 = \frac{2}{3} (n\!-\!8) v_4 \quad \mbox{(see \eqref{L-v})}
$$
and the explicit formula for $a_4$ (displayed in Section
\ref{coeff-L}), it follows that \eqref{magic-4} is equivalent to the
obvious identity
$$
6\left(\f \J^2 - 2|\Rho|^2 - \Delta \J\right) = ((5n\!-\!16) \J^2 -
2(n\!-\!2)|\Rho|^2 - 6 \Delta \J) - 2 (n\!-\!8) (\J^2-|\Rho|^2).
$$
The proof is complete.
\end{proof}

Further results in this direction may help to clarify the relation
between heat kernel coefficients of the conformal Laplacian and
$Q$-curvatures.

\subsection{First-order perturbation}\label{perturbation-1}

From the perspective of perturbation theory, it seems natural to
perform a systematic study of the properties of the {\em
lowest-order} perturbation of the heat kernel coefficients $a_{2k}$
of the conformal Laplacian which are given by the respective
coefficients of $r^2$ in the power series expansions of the
coefficients $a_{2k}(r)$. By Duhamel's formula, the total integrals
of these coefficients are given by the coefficients in the expansion
of
$$
t \Tr (\M_4 e^{tP_2})
$$
for $t \to 0$. This expansions contains basic information on the
off-diagonal structure of the heat-kernel coefficients $a_{2k}(x,y)$
of $P_2$. In this connection, the results in \cite{GSZ} are
relevant.

\subsection{Quantization and Einstein condition}\label{quantization}

Theorem \ref{B} may be regarded as a relation of a quantity which is
defined in terms of classical mechanics of the Hamiltonian
$\H(r;g)(x,\xi)$ and a quantity the definition of which involves its
quantization $\Delta_{g(r)}$. It seems natural to ask whether one can {\em
characterize} the families $g(r)$ for which the integral formula
\eqref{e-value} holds true. What is the role of the Einstein condition for
$g_+$ in this context? We recall that in the given proof of Theorem
\ref{B} the relation \eqref{Sch-bar} plays an important role, and the
proof of this relation uses the Einstein property of $g_+$.

\subsection{Non-Laplace-type operators}\label{NLT}

We emphasize again that $\H(r;g)$ is a family of non-Laplace-type
operators for the metric $g$. The spectral theory of such operators is
much less understood than that of Laplace-type operators. For some
pioneering work on the asymptotic expansion of the trace of the heat
kernel of certain classes of elliptic self-adjoint non-Laplace-type
operators we refer to \cite{AB}, \cite{A-GG}. Although the most serious
problems arise for such operators on sections of vector bundles, already
the scalar-valued case offers substantial difficulties. The cited
references use Volterra series (see \cite{BGV}) to determine the
(integrated) leading and sub-leading heat kernel coefficients (partly
under additional conditions on the operators). Although these methods can
also be applied to the operator $\H(r;g)$, the corresponding consequences
remain to be explored.

\section{Appendix}\label{app}

In this section we collect basic material which illustrates and
complements the general theory. We start with explicit displays of the
first few terms in the $r$-expansion of the holographic Laplacian
$\H(r;g)$. Then we proceed with a detailed discussion of the low-order
heat kernel coefficients $a_2, a_4, a_6$ of the conformal Laplacian. Here
we review classical results and provide a new derivation of an explicit
formula for $a_6$ (for locally conformally flat metrics). In this
connection, we emphasize a holographic perspective, i.e., we describe heat
kernel coefficients in terms of renormalized volume coefficients. In a
sense, this point of view is dual to the description of renormalized
volume coefficients as heat kernel coefficients (as in Lemma \ref{top}).
The holographic perspective naturally leads to a discussion of the
relation between the functional determinant (of the conformal Laplacian)
and the renormalized volume (in low-order cases), and includes a
derivation of an analog for the renormalized volume of Branson's
conjectural Polyakov formula for the determinant of $P_2$. Finally, we
determine the coefficient of $r^2$ in the expansion of the coefficient
$a_4(r)$.

\subsection{The holographic Laplacian}\label{HL}

We display explicit formulas for the first three terms in the
Taylor-expansion of the holographic Laplacian $\H(r;g)$ in $r$. Its main
part is given by $-\delta_g (g(r)^{-1} d)$. Now the expansion of $g(r)$
starts with
\begin{equation}\label{PE-metric}
g(r)_{ij} = g_{ij} - r^2 \Rho_{ij} + r^4 \frac{1}{4}
\left(\Rho^2_{ij} \!-\! \frac{1}{n\!-\!4} \B_{ij} \right) + \cdots,
\end{equation}
where $\Rho^2_{ij} = \Rho_{ik} \Rho^k_j$ and $\B$ is the Bach tensor
(see \cite{FG-final}, \cite{juhl-book}); here we raised indices
using $g$. It follows that
\begin{equation}\label{inverse}
g(r)^{-1}_{ij} = g^{ij} + r^2 \Rho^{ij} + r^4 \frac{1}{4} \left(3
(\Rho^2)^{ij} \!+\! \frac{1}{n\!-\!4} \B^{ij} \right) + \cdots.
\end{equation}
Hence the Taylor-expansion of the main part of $\H(r;g)$ starts with
\begin{equation}\label{holo-main}
\Delta - r^2 \delta (\Rho d) - r^4 \frac{1}{4} \delta \left(3 \Rho^2
\!+\! \frac{1}{n\!-\!4} \B \right) d
\end{equation}
(for $n \ge 5$). Next, the potential of the holographic Laplacian
has the expansion
$$
\U(r) = \sum_{N \ge 1} \mu_{2N} \frac{1}{(N\!-\!1)!^2}
\left(\frac{r^2}{4}\right)^{N-1} \quad \mbox{with} \quad \mu_{2N} =
\M_{2N}(1).
$$
Calculations show that the first three non-trivial coefficients
$\mu_{2N}$ are given by the formulas
\begin{equation*}
\mu_2 = - \left(\f\!-\!1\right) \J, \quad  \mu_4 = -\J^2 \!-\!
(n\!-\!4)|\Rho|^2 \!+\! \Delta \J
\end{equation*}
and
\begin{equation*}
\mu_6 = - 8 \frac{n\!-\!6}{n\!-\!4} (\B,\Rho) - 8 (n\!-\!6) \tr
\Rho^3 - 16 \J |\Rho|^2 - 4 |d\J|^2 + 4 \Delta |\Rho|^2 - 16 \delta
(\Rho d\J)
\end{equation*}
(for $n \ge 5$); for details we refer to \cite{juhl-ex}. Hence the
expansion of the potential of $\H(r;g)$ starts with
\begin{align*}
& - r^0 \frac{1}{2} (n\!-\!2) \J \\
& + r^2 \frac{1}{4} \left( -\J^2 \!-\! (n\!-\!4)|\Rho|^2 \!+\!
\Delta \J \right) \\
& + r^4 \frac{1}{16} \left(-2 \frac{n\!-\!6}{n\!-\!4} (\B,\Rho) -
2(n\!-\!6) \tr \Rho^3 - 4 \J |\Rho|^2 - |d\J|^2 + \Delta |\Rho|^2 -
4 \delta (\Rho d\J) \right)
\end{align*}
in dimension $n \ge 5$. In particular, for general metrics in
dimension $n=4$, we define
$$
\H(r;g) = P_2 + r^2 \frac{1}{4} (-4\delta(\Rho d) \!-\! \J^2 \!+\!
\Delta \J).
$$

\subsection{The heat kernel coefficients $a_0$, $a_2$ and $a_4$}\label{coeff-L}

The first three local heat kernel coefficients of the Laplacian
$\Delta$ are given by $a_0=1$ and
\begin{align*}
6 a_2 & = \scal, \\
360 a_4 & = 2 |R|^2 - 2 |\Ric|^2 + 5 \scal^2 - 12 \Delta \scal,
\quad \Delta = \delta d
\end{align*}
Here $R$ and $\Ric$ are the curvature tensor and the Ricci tensor of the
given metric, respectively. For proofs of these formulas we refer to
\cite{G-book}. Detailed proofs of the integrated versions of the above
formulas also can be found in Chapter III of \cite{BGM}. Note that the
coefficients in these formulas are universal in the dimension $n$. This is
a consequence of the functoriality of heat kernel coefficients for direct
products. \footnote{These formulas are already stated in Section 17 of
\cite{dewitt}. Accordingly, in the physical literature, the heat kernel
coefficients are often referred to as the DeWitt coefficients.}

By $R = W - \Rho \owedge g$, the first three heat kernel
coefficients of the {\em Laplacian} can be written in terms of the
Weyl tensor $W$, the Schouten tensor $\Rho$ and its trace $\J$. Then
we obtain $a_0=1$ and
\begin{align*}
3 a_2 & = (n\!-\!1) \J, \\
180 a_4 & = -(n\!-\!2)(n\!-\!6) |\Rho|^2 + (10n^2\!-\!23n\!+\!18)
\J^2 - 12(n\!-\!1) \Delta \J + |W|^2.
\end{align*}

The analogous formulas for the first three heat kernel coefficients
of the {\em conformal Laplacian} read $a_0=1$ and
\begin{align}
6 a_2 & = -(n\!-\!4) \J, \label{CL-a2} \\
360 a_4 & = -(n\!-\!6) (2(n\!-\!2)|\Rho|^2 - (5n\!-\!16) \J^2 - 6
\Delta \J) + 2 |W|^2, \quad \Delta = \delta d \label{CL-a4}
\end{align}
(see \cite{PR}, \cite{BO-index} and the references therein). The result
for $a_4$ also follows from the calculations in Section \ref{a4-holo}.

In particular, these explicit formulas imply that for $n=4$ and $n=6$ the
respective coefficients $a_2$ and $a_4$ are local conformal invariants.
Moreover, they confirm that for $n=2$ and $n=4$ the respective total
integrals of $a_2$ and $a_4$ are global conformal invariants (conformal
indices). In fact, we have
$$
\int_{M^4} a_4 \dvol = - 1/45 \int_{M^4} (\J^2\!-\!|\Rho|^2) \dvol +
1/180 \int_{M^4} |W|^2 \dvol,
$$
and both of the latter integrals are conformally invariant. Using
$$
\chi(M^4) = \frac{1}{32\pi^2}
\int_{M^4}(8\J^2-8|\Rho|^2+|W|^2)\dvol,
$$
the latter formula can be written also in the form
$$
\frac{1}{(4\pi)^2} \int_{M^4} a_4 \dvol = - \frac{1}{180} \chi(M) +
\frac{1}{120} \frac{1}{(4\pi)^2} \int_{M^4} |W|^2 \dvol
$$
(see Proposition 4.3 in \cite{PR}).

\subsection{The heat kernel coefficient $a_6$}\label{CL-a6}

Explicit formulas for the local heat kernel coefficient $a_6$ of the
conformal Laplacian are substantially more involved. In the literature,
various formulas for $a_6$ and its total integral on closed manifolds can
be found. In the pioneering work \cite{G-spec} (see also Theorem 4.8.16 in
\cite{G-book}), Gilkey derived a formula for $a_6$ for general
Laplace-type operators on vector bundles. It effectively can be used to
deduce explicit formulas for the heat kernel coefficient $a_6$ of the
conformal Laplacian.

The complexity of the situation is illustrated by the fact that Gilkey's
general formula (in the scalar case) involves $17$ local Riemannian
invariants and a basis of $12$ invariants which are relevant for the
contributions of the potential term. Since there are no canonical bases of
invariants, the formulation of any resulting formula for $a_6$ depends on
the choice of such a basis. Gilkey uses a basis of Riemannian invariants
which is defined in terms of the curvature tensor $R$, the Ricci tensor
$\Ric$ and the scalar curvature $\scal$ (and their covariant derivatives).
Alternatively, one may use the Weyl tensor $W$, the trace-less Ricci
tensor $B$ and the scalar curvature $\scal$ (and their covariant
derivatives), or a basis in terms of the Weyl tensor, the Schouten tensor
$\Rho$ and its trace $\J$ (and their covariant derivatives). The actual
comparison of results stated in terms of different bases usually is a
non-trivial issue.

By specialization of Gilkey's formula, Parker and Rosenberg derived in
\cite{PR} a formula for $a_6$ in terms of $W$, $B$ and $\scal$ (and its
covariant derivatives) in general dimensions.\footnote{In \cite{PR}, the
convention for the signs of the components $R_{ijkl}$ is opposite to that
of Gilkey.} Unfortunately, the displayed formulas in \cite{PR} contain
some misprints even in the locally conformally flat case. We use the
chance to remove these inaccuracies (in Section \ref{a6-c-flat}).

For locally conformally flat metrics, the coefficient $a_6$ is somewhat
less complicated. For such metrics, Branson (see Lemma 6.9 in
\cite{sharp}) used an integrated version of Gilkey's formula to prove that
\begin{multline}\label{branson-a6}
7! a_6 = (n\!-\!8) \Big( -\frac{1}{9} (35n^2\!-\!266n\!+\!456) \J^3
+ \frac{2}{3} (n\!-\!1)(7n\!-\!30) \J|\Rho|^2 \\ - \frac{2}{9}
(5n^2\!-\!2n\!-\!48) \tr(\Rho^3) - 3 (n\!-\!6) |d\J|^2 \Big)
\end{multline}
modulo a divergence term. By the conformal covariance of the
conformal Laplacian, that results suffices to derive a formula for
the hidden divergence term (see the discussion in Section
\ref{holo-heat}). Equation \eqref{branson-a6} shows that the
coefficient $a_6$ vanishes in dimension $n=8$ in the locally
conformally flat case. This is yet another example of the general
fact that the local heat kernel coefficient $a_{n-2}$ is a local
conformal invariant.

For general metrics, comparing Gilkey's formula for $a_6$ with Lemma
\ref{FGI} identifies the local conformal invariant $a_6$ in dimension
$n=8$ as the element
\begin{equation}\label{CI-8}
\frac{1}{7! 9} (81 \Phi_8 + 64 I_1 + 352 I_2)
\end{equation}
in the three-dimensional space of local conformal invariants $I$ of weight
$-6$, i.e., $I$ satisfies $e^{6 \varphi} I(e^{2\varphi} g) = I(g)$ (see
Proposition 4.2 of \cite{PR}). Here, $\Phi_n \st |\tilde{\nabla}
\tilde{R}|^2$ (restricted to $\rho=0$ and $t=1$) is the Fefferman-Graham
invariant (see Proposition 3.4 in \cite{FG-Cartan}) being defined in terms
of the Levi-Civita connection $\tilde{\nabla}$ of the ambient metric
$\tilde{g}$ and its curvature $\tilde{R}$. Moreover,
$$
I_1 = W^{ij}{}_{kl} W^{kl}{}_{mn} W^{mn}{}_{ij} \quad \mbox{and}
\quad I_2 = W_{ijkl} W^j{}_m{}^l{}_n W^{imkn}.
$$
For more details of the relevant invariant theory see also Section 4
of \cite{B-spec}.

In dimension $n=6$, the total integral of $a_6$ can be written as a
linear combination of the Euler characteristic and the total
integrals of the local conformal invariants $I_1$, $I_2$ and $I_3 =
\Phi_6$ (see Proposition 4.3 in \cite{PR}), i.e.,
$$
\frac{1}{(4\pi)^3} \int_{M^6} a_6 \dvol = \frac{1}{7!} \frac{10}{3}
\chi(M) + \sum_{i=1}^3 c_i \int_{M^6} I_i \dvol.
$$
This is a special case of the Deser-Schwimmer decomposition of
global conformal invariants proved in \cite{alex}.

The following identity was already stated in \cite{PR} without a
proof. Although some arguments of the following proof are taken from
\cite{Erd}, the actual result differs.

\begin{lemm}\label{FGI} The Fefferman-Graham invariant $\Phi_n$ can be
written in the form
\begin{equation}\label{PR-FG}
\Phi_n = \Omega_n - \frac{8}{n\!-\!2} (I_1 + 4 I_2),
\end{equation}
where
$$
\Omega_n \st |\nabla W|^2 - 4(n\!-\!10) |\Co|^2 +
\frac{8}{n\!-\!2}(W,\Delta W) - \frac{16}{n\!-\!2} \J |W|^2
$$
and $\Co$ is the Cotton tensor.
\end{lemm}

\begin{rem}\label{PR-Phi} The right-hand side of \eqref{PR-FG}
coincides with the invariant in Equation (4.1) of \cite{PR}. This
follows by a direct calculation using the relation
$$
\Rho = \frac{1}{n\!-\!2} B + \frac{1}{2n(n\!-\!1)} \scal.
$$
\end{rem}

Note that the identity $2 \Omega_{10} = (\Delta - 4 \J) (|W|^2) =
P_2 (|W|^2)$ immediately proves the conformal invariance of
$\Omega_{10}$. Lemma \ref{FGI} shows that \eqref{CI-8} equals
\begin{equation}
\frac{1}{7!} \left(9 \Omega_8 - \frac{44}{9} I_1 - \frac{80}{9}
I_2\right).
\end{equation}

\begin{proof} We first recall from \cite{FG-Cartan} (and \cite{G-vol})
an explicit formula for the conformal invariant $\Phi_n$. Let
\begin{equation}\label{cotton}
\Co_{ijk} = \nabla_k \Rho_{ij} - \nabla_j \Rho_{ik}
\end{equation}
be the Cotton tensor. Then $\Co_{ijk} = -\Co_{ikj}$ and $\Co_{ijk} +
\Co_{jki} + \Co_{kij} = 0$. Now set\footnote{Note that our sign
conventions for $R$ and $W$ coincide with those in \cite{FG-Cartan}
and are opposite to those of \cite{G-vol}.}
$$
V_{ijklm} = \nabla_m W_{ijkl} - g_{im} \Co_{jkl} + g_{jm} \Co_{ikl}
- g_{km} \Co_{lij} + g_{lm} \Co_{kij}.
$$
Then we have
$$
\Phi_n = |V|^2 + 16 W^{ijkl} (\nabla_i \Co_{jkl} + \Rho^m_i
W_{mjkl}) + 16 |\Co|^2.
$$
The second Bianchi identity implies
\begin{multline}\label{bianchi-2}
\nabla_m W_{ijkl} + \nabla_k W_{ijlm} + \nabla_l W_{ijmk} \\ =
\Co_{ikm} g_{jl} + \Co_{ilk} g_{jm} + \Co_{iml} g_{jk} - \Co_{jkm}
g_{il} - \Co_{jlk}g_{im} - \Co_{jml}g_{ik}.
\end{multline}
It follows that
\begin{equation}\label{div-W}
\nabla^m W_{mjkl} = (n\!-\!3) \Co_{jkl}.
\end{equation}
These results yield
$$
|V|^2 = |\nabla W|^2 - 8(n\!-\!3) |C|^2 + 4n |C|^2
$$
and we obtain the alternative formula
\begin{equation}\label{Phi-alt}
\Phi_n = |\nabla W|^2 - 4(n\!-\!10)|C|^2 + 16(W^{ijkl} \nabla_i
\Co_{jkl} + \Rho^m_i W^{ijkl}W_{mjkl}).
\end{equation}
Now since
$$
\Phi_n - \Omega_n = -\frac{8}{n\!-\!2}(W,\Delta W)  + 16(W^{ijkl}
\nabla_i \Co_{jkl} + \Rho^m_i W^{ijkl}W_{mjkl}) + \frac{16}{n\!-\!2}
\J |W|^2,
$$
it only remains to verify the identity
\begin{multline}\label{myst}
-\frac{1}{2} (W,\Delta W) + (n\!-\!2) W^{ijkl} \nabla_i \Co_{jkl} \\
= - \frac{1}{2} I_1 - 2 I_2 - (n\!-\!2) \Rho^m_i W^{ijkl}W_{mjkl} -
\J |W|^2.
\end{multline}
But \eqref{myst} is a consequence of two different calculations of
the commutator term
\begin{equation}\label{ct}
W^{ijkl} [\nabla_i,\nabla_m] W^m{}_{jkl}.
\end{equation}
On the one hand, the commutator term \eqref{ct} equals the left-hand
side of \eqref{myst}. In fact, \eqref{bianchi-2} implies
\begin{multline*}
\nabla_i W_{mjkl} + \nabla_j W_{imkl} \\ = -\nabla_m W_{jikl} +
(\Co_{kmi} g_{lj} + \Co_{kjm} g_{li} + \Co_{kij} g_{lm} - \Co_{lmi}
g_{kj} - \Co_{ljm} g_{ki} - \Co_{lij} g_{km}).
\end{multline*}
We apply $\nabla^m$ to that identity, multiply with  $W^{ijkl}$ and
sum over repeated indices. Since $W$ is trace-free, we obtain
$$
2 W^{ijkl} \nabla^m \nabla_i W_{mjkl} = W^{ijkl} \Delta W_{ijkl} + 2
W^{ijkl} \nabla_l \Co_{kij}
$$
using the symmetries. Together with \eqref{div-W} we find
\begin{align*}
W^{ijkl} [\nabla_i,\nabla_m] W^m{}_{jkl} & = (n\!-\!3) W^{ijkl}
\nabla_i \Co_{jkl} - W^{ijkl} \nabla_l \Co_{kij} - \frac{1}{2}
W^{ijkl}
\Delta W_{ijkl} \\
& = (n\!-\!2) W^{ijkl} \nabla_i \Co_{jkl} - \frac{1}{2} W^{ijkl}
\Delta W_{ijkl}.
\end{align*}
On the other hand, the commutator term \eqref{ct} can be written in
the form
\begin{equation*}
- W^{ijkl} R_{im}{}^m{}_a W^a{}_{jkl} - W^{ijkl} R_{imja}
W^{ma}{}_{kl} - W^{ijkl} R_{imka} W^m{}_j{}^a{}_l - W^{ijkl}
R_{imla} W^m{}_{jk}{}^a.
\end{equation*}
By $R = W - \Rho \owedge g$ (see \eqref{R-W}), this sum equals the
sum of
\begin{multline*}
- W^{ijkl} R_{im}{}^m{}_a W^a{}_{jkl} - W^{ijkl} W_{imja}
W^{ma}{}_{kl} - W^{ijkl} W_{imka} W^{m}{}_j{}^a{}_l - W^{ijkl}
W_{imla} W^{m}{}_{jk}{}^a \\
= - \Ric_{ia} W^{ijkl} W^{a}{}_{jkl} - W^{ijkl} W_{imja}
W^{ma}{}_{kl} - I_2 - I_2
\end{multline*}
and
\begin{multline*}
W^{ijkl} ((\Rho \owedge g)_{imja} W^{ma}{}_{kl} +  (\Rho \owedge
g)_{imka} W^m{}_j{}^a{}_l) + (\Rho \owedge g)_{imla} W^m{}_{jk}{}^a) \\
= - W^{ijkl} (\Rho_{mj} W^m{}_{ikl} + \Rho_{mk} W^m{}_{jil} +
\Rho_{ml} W^m{}_{jki}) - W^{ijkl} \Rho_{ia} (W_j{}^a{}_{kl} +
W_{kj}{}^a{}_l + W_{ljk}{}^a).
\end{multline*}
Now the first Bianchi identity (for $W$) shows that the latter sum
vanishes. In fact, it implies that the last three terms sum up to
$0$ and that the first three terms are given by
$$
(W^{ijkl} \Rho_{mj} W^m{}_{kli} + W^{ijkl} \Rho_{mj} W^m{}_{lik}) -
(W^{ijkl} \Rho_{mk} W^m{}_{jil} + W^{ijkl} \Rho_{ml} W^m{}_{jki}) =
0
$$
using the symmetries. Thus, the commutator term \eqref{ct} equals
$$
- \Ric_{ia} W^{ijkl} W^{a}{}_{jkl} - W^{ijkl} W_{imja} W^{ma}{}_{kl}
- 2 I_2.
$$
But the first Bianchi identity (for $W$) shows that $W_{imja} =
W_{ijma} + W_{iajm}$. Hence
$$
W^{ijkl} W_{imja} W^{ma}{}_{kl} = I_1 - W^{ijkl} W_{imja}
W^{ma}{}_{kl},
$$
i.e.,
$$
2 W^{ijkl} W_{imja} W^{ma}{}_{kl} = I_1.
$$
Therefore, the commutator term \eqref{ct} equals
$$
-(n\!-\!2) \Rho^a_i W^{ijkl} W_{ajkl} - \J |W|^2 - \frac{1}{2} I_1 -
2 I_2,
$$
i.e., coincides with the right-hand side of \eqref{myst}. This
proves \eqref{myst} and completes the proof.
\end{proof}

The following result relates the conformal invariant $a_6$ in dimension
$n=8$ to the coefficient $\tilde{a}_6$ which arises by substituting the
ambient curvature $\tilde{R}$ into the formula for the heat kernel
coefficient $a_6$ of the Laplacian of a Ricci flat metric. It is a special
case of Theorem 7.3 announced in \cite{log}.

\begin{lemm}\label{a66} For a Riemannian manifold $(M,g)$ of dimension
$n=8$, we have
$$
a_6 = \tilde{a}_6 |_{t=1,\rho=0}.
$$
\end{lemm}

\begin{proof} By Gilkey's formula for the heat kernel coefficient
$a_6$ of the Laplacian (Theorem 4.8.16 in \cite{G-book}), we have
$$
7! \tilde{a}_6 = 9 (\tilde{\nabla}(\Rt),\tilde{\nabla}(\Rt)) + 12
(\Rt,\tilde{\Delta} \Rt) - \frac{44}{9} \tilde{I}_1 - \frac{80}{9}
\tilde{I}_2
$$
with
$$
\tilde{I}_1 = \Rt^{ij}{}_{kl} \Rt^{kl}{}_{mn} \Rt^{mn}{}_{ij} \quad
\mbox{and} \quad \tilde{I}_2 = \Rt^{ijkl} \Rt_{jmln}
\Rt_i{}^m{}_k{}^n.
$$
Here we have used the Ricci flatness of $\tilde{g}$. The
restrictions of $\tilde{I}_1$ and $\tilde{I}_2$ to $t=1$ and
$\rho=0$ are given by $I_1$ and $I_2$, respectively. Therefore, by
\eqref{CI-8}, the assertion is equivalent to
$$
(9 (\tilde{\nabla}(\Rt),\tilde{\nabla}(\Rt)) + 12
(\Rt,\tilde{\Delta} \Rt))|_{t=1,\rho=0} - \frac{44}{9} I_1 -
\frac{80}{9} I_2 = 9 \Phi_8 + \frac{64}{9} I_1 + \frac{352}{9} I_2,
$$
i.e.,
\begin{equation}\label{cubic}
(\Rt,\tilde{\Delta}\Rt)|_{t=1,\rho=0} = I_1 + 4 I_2.
\end{equation}
In order to prove this relation, we show that any Ricci flat metric
$g$ satisfies the relation
\begin{equation}\label{bianchi-flat}
(R,\Delta R) = R^{ijkl} R_{ij}{}^{mn} R_{mnkl} + 4 R^{ijkl} R_{jmln}
R_i{}^m{}_k{}^n.
\end{equation}
In fact, the second Bianchi identity implies
\begin{equation}\label{bf2}
\nabla_i R_{mjkl} + \nabla_j R_{imkl} = - \nabla_m R_{jikl}.
\end{equation}
We apply $\nabla^m$ to that equation, multiply with $R^{ijkl}$ and sum
over repeated indices. This yields
$$
2 R^{ijkl} \nabla^m \nabla_i R_{mjkl} = R^{ijkl} \Delta R_{ijkl}
$$
using the symmetries. Next we commute derivatives
$$
\nabla^m \nabla_i R_{mjkl} = \nabla_i \nabla^m R_{mjkl} +
\left[\nabla^m,\nabla_i\right] R_{mjkl}
$$
and observe that $\nabla^m R_{mjkl} = \nabla_l \Ric_{kj} - \nabla_k
\Ric_{lj} = 0$ (by \eqref{bf2} and the Ricci flatness). Hence
\begin{multline*}
R^{ijkl} \Delta R_{ijkl} = 2 R^{ijkl} \left[\nabla^m,\nabla_i\right] R_{mjkl}\\
= - 2 R^{ijkl} (R^{m}{}_{im}{}^n R_{njkl} + R^m{}_{ij}{}^n R_{mnkl}
+ R^m{}_{ik}{}^n R_{mjnl} + R^{m}{}_{il}{}^n R_{mjkn}).
\end{multline*}
Now the Ricci flatness and the symmetries yield
$$
R^{ijkl} \Delta R_{ijkl} = -2 R^{ijkl} R^m{}_{ij}{}^n R_{mnkl} + 4
R^{ijkl} R_i{}^m{}_k{}^n R_{jmln}.
$$
Finally, an application of the first Bianchi identity (similarly as
in the proof of Lemma \ref{FGI}) proves \eqref{bianchi-flat}. This
completes the proof.
\end{proof}

The arguments in the proof of Lemma \ref{a66} also show that the
elements in Gilkey's basis of Riemannian invariants of weight $6$
are linearly dependent.

\subsection{$a_6$ for locally conformally flat metrics}\label{a6-c-flat}

In the present section we discuss $a_6$ for locally conformally flat
metrics in general dimensions. In particular, we correct some formulas in
\cite{PR} and verify that the result is consistent with Branson's formula
\eqref{branson-a6}.\footnote{These corrections may be of independent
interest in connection with future explicit
calculations.}\,\footnote{\cite{sharp} does not contain \cite{PR} among
its references and accordingly does not discuss the relations to the
results in \cite{PR}.}

For locally conformally flat metrics, Table \ref{heat-6} displays the
local heat kernel coefficient $7! a_6$ as a linear combination of $10$
Riemannian invariants. These results follow by a reformulation of Gilkey's
formula for $a_6$ in terms of $W$, $B$, $\scal$ and are basically taken
from \cite{PR}, up to corrections of the coefficients of $B_6$ and
$B_{13}$ in general dimensions. The last two columns display the weights
of the contributions in the special dimensions $n=6$ and $n=8$.

Two comments are in order. The coefficient of $B_1 = \Delta^2
(\scal)$ alternatively follows by combining Lemma 1.1 and Theorem
1.5 of \cite{BGO}. Moreover, for the round spheres $\s^n$, the only
non-trivial contribution in Table \ref{heat-6} is $B_{10} =
\scal^3$. The corresponding coefficient coincides with that given in
Section \ref{heat-spheres}.

\begin{table}[ht]
\centerline{
\begin{tabular}{l|c|c|c}
invariant & coefficient & $n=6$ & $n=8$ \\
\hline & & \\[-3mm]
$B_1 = \Delta^2 (\scal)$ & $-\frac{3(n-8)}{(n-1)}$ & $\frac{6}{5}$ & $0$ \\[2mm]
$B_2 = |\nabla \scal|^2$ &
$\frac{5n^4-76n^3+288n^2-128n+16}{4n^2(n-1)^2}$ & $-\frac{4}{45}$ & $-\frac{9}{112}$\\[2mm]
$B_3 = |\nabla B|^2$ & $-\frac{2(n-20)}{(n-2)}$ & $7$ & $4$ \\[2mm]
$B_4 = \nabla_i B_{jk} \nabla_k B_{ij}$ & $-4$ & $-4$ & $-4$ \\[2mm]
$B_6 = \scal \Delta (\scal)$ & $\frac{(n-8)(7n^2-34n+12)}{2n(n-1)^2}$ & $-\frac{2}{5}$ & $0$ \\[2mm]
$B_7 = (B,\Delta B)$ & $-\frac{8(n-8)}{(n-2)}$ & $4$ & $0$ \\[2mm]
$B_8 = B_{ij} \nabla_k \nabla_j B_{ik}$ & $-\frac{4n(n-8)}{(n-1)(n-2)}$ & $\frac{12}{5}$ & $0$ \\[2mm]
$B_{10} = \scal^3$ & $-\frac{(n-8)(35n^4-308n^3+688n^2-184n-96)}{72n^2(n-1)^3}$ & $\frac{2}{135}$ & $0$ \\[2mm]
$B_{11} = \scal |B|^2$ & $\frac{(n-8)(7n^3-17n^2-2n+24)}{3n(n-1)^2(n-2)}$ & $-\frac{76}{75}$ & $0$ \\[2mm]
$B_{13} = \tr (B^3)$ & $ \frac{4(n-8) (11n^3-28n^2+32n-24)} {9(n-1)
(n-2)^3}$ & $-\frac{64}{15}$ & $0$
\end{tabular}} \vspace{0.5cm}
\caption{The heat kernel coefficient $7! a_6$ for the conformal
Laplacian for locally conformally flat metrics}
\label{heat-6}
\end{table}

We continue by proving that the results in Table \ref{heat-6} imply
the formula \eqref{branson-a6} for the integrated heat kernel
coefficient
\begin{equation}\label{int-a6}
\int_{M^n} a_6 \dvol
\end{equation}
for closed and locally conformally flat manifold $(M,g)$ of
dimension $n \ge 3$. We start by establishing the following
identity.

\begin{lemm}\label{base-int} For a closed and locally conformally
flat manifold $(M^n,g)$ ($n \ge 3$), we have the relation
$$
\int_M |\nabla \Rho|^2 \dvol = \int_M (|\nabla \J|^2 + \J |\Rho|^2 -
n \tr (\Rho^3)) \dvol.
$$
\end{lemm}

\begin{proof} By conformal flatness, the Cotton tensor vanishes, i.e.,
$\nabla_k(\Rho)_{ij} = \nabla_j(\Rho)_{ik}$. Hence, by partial
integration, we have
$$
\int_M |\nabla \Rho|^2 \dvol = \int_M \nabla_k (\Rho)_{ij} \nabla^k
(\Rho)^{ij} \dvol = - \int_M \nabla^k \nabla_j (\Rho)_{ik} \Rho^{ij}
\dvol.
$$
Now we interchange the derivatives. The relation
$$
\nabla^k \nabla_j (\Rho)_{ik} = \Hess_{ij}(\J) + \Ric_{jl} \Rho^l_i
- (\Rho \owedge g)_{jkil} \Rho^{lk}
$$
(see Equation (6.9.17) in \cite{juhl-book}) implies
\begin{align*}
- \nabla^k \nabla_j (\Rho)_{ik} \Rho^{ij} & = - \Hess_{ij}(\J)
\Rho^{ij} - \Ric_{jl} \Rho^l_i \Rho^{ij} + (\Rho \owedge g)_{jkil}
\Rho^{lk} \Rho^{ij} \\ & = \delta (\Rho d\J) + |\nabla \J|^2 -
\Ric_{jl} (\Rho^2)^{jl} + (\Rho \owedge g)_{jkil} \Rho^{lk}
\Rho^{ij}.
\end{align*}
Therefore, we obtain
\begin{multline*}
\int_M |\nabla \Rho|^2 \dvol = \int_M |\nabla \J|^2 \dvol - \int_M
((n-2) \Rho_{jl} + \J g_{jl}) (\Rho^2)^{jl} \dvol + \int_M (\Rho
\owedge g)_{jkil} \Rho^{lk} \Rho^{ij} \dvol
\end{multline*}
and the assertion follows by simplification.
\end{proof}

Lemma \ref{base-int} implies the following analogous results for $|\nabla
\Ric|^2$ and $|\nabla B|^2$.

\begin{corr} Under the assumptions of Lemma \ref{base-int}, we have
\begin{equation}\label{B-iden}
\int_M |\nabla \Ric|^2 \dvol = \int_M (n(n\!-\!1) |\nabla \J|^2 +
(n\!-\!2) \J |\Rho|^2 - n(n\!-\!2)^2 \tr (\Rho^3)) \dvol
\end{equation}
and
\begin{multline}\label{PR-iden}
\int_M |\nabla B|^2 \dvol = \int_M \left(\frac{n\!-\!2}{4n(n\!-\!1)^2}
|\nabla \scal|^2 - \frac{1}{n\!-\!1} \scal |B|^2 - \frac{n}{n\!-\!2}
\tr(B^3) \right) \dvol.
\end{multline}
\end{corr}

\begin{proof} The assertions follow by direct calculations using
\begin{equation}\label{refo}
\Ric = (n\!-\!2) \Rho + \J g, \; \J = \frac{\scal}{2(n\!-\!1)} \quad
\mbox{and} \quad B = (n\!-\!2) \left(\Rho - \frac{1}{n} \J g
\right).
\end{equation}
We omit the details.
\end{proof}

Note that \eqref{PR-iden} is equivalent to
\begin{equation}\label{S0}
\int_M \left(B_3 - \frac{n\!-\!2}{4n(n\!-\!1)^2} B_2 +
\frac{1}{n\!-\!1} B_{11} + \frac{n}{n\!-\!2} B_{13} \right) \dvol =
0.
\end{equation}

\begin{lemm}\label{PR-iden-2} We have
$$
B_3 - B_4 = \frac{(n\!-\!2)^2}{4n^2(n\!-\!1)} |\nabla \scal|^2.
$$
\end{lemm}

\begin{proof} A calculation using the formula for $B$ in \eqref{refo}
yields
$$
B_3 = |\nabla B|^2 = (n\!-\!2)^2 \left(|\nabla \Rho|^2 - \frac{1}{n}
|\nabla \J|^2\right).
$$
A similar calculation using $\nabla_k(\Rho)_{ij} =
\nabla_j(\Rho)_{ik}$ gives
$$
B_4 = \nabla_i B_{jk} \nabla_k B_{ij} = (n\!-\!2)^2 \left( |\nabla
\Rho|^2 + \left(\frac{1}{n^2} - \frac{2}{n} \right) |\nabla \J|^2
\right).
$$
The assertion follows by combining these two results.
\end{proof}

Now, in order to calculate the integral \eqref{int-a6}, we combine
the results in Table \ref{heat-6} with the fact that
\begin{equation*}
B_4 - B_3 + \frac{(n\!-\!2)^2}{4n^2(n\!-\!1)} B_2 \quad \mbox{(by
Lemma \ref{PR-iden-2})}
\end{equation*}
and
$$
B_6 + B_2, \quad B_7 + B_3, \quad B_8 + B_4  \quad \mbox{(by partial
integration)}
$$
are total derivatives, i.e., integrate to $0$. It follows that the
total integral of  $7!a_6$ equals the sum of the total integrals of
\begin{align*}
& \frac{5n^4-76n^3+288n^2-128n+16}{4n^2(n-1)^2} B_2, \\
& -\frac{2(n-20)}{(n-2)} B_3, \\
& -4 B_3 + \frac{(n\!-\!2)^2}{n^2(n\!-\!1)} B_2, \\
& -\frac{(n-8)(7n^2-34n+12)}{2n(n-1)^2} B_2, \\
& \frac{8(n-8)}{(n-2)} B_3, \\
& \frac{4n(n-8)}{(n-1)(n-2)} \left( B_3 -
\frac{(n\!-\!2)^2}{4n^2(n\!-\!1)} B_2 \right)
\end{align*}
and the total integral of some linear combination of $B_{10}$,
$B_{11}$ and $B_{13}$. A direct calculation shows that the latter
three contributions can be written as the sum of
\begin{multline}\label{S1}
-\frac{(n\!-\!8)(35n^4-308n^3+688n^2-184n-96)}{72n^2(n\!-\!1)^3} B_{10} \\
+ \frac{(n\!-\!8)(7n^2\!-\!28n\!-\!24)}{3n(n\!-\!1)(n\!-\!2)} B_{11}
- \frac{(n\!-\!8)(10n^2\!-\!4n\!-\!96)}{9 (n\!-\!2)^2} B_{13}
\end{multline}
and
\begin{equation}\label{S2}
2 \frac{(n\!-\!8)(3n\!-\!1)}{(n\!-\!2)(n\!-\!1)} \left(
\frac{1}{n\!-\!1} B_{11} + \frac{n}{n\!-\!2} B_{13} \right).
\end{equation}
By \eqref{S0}, the total integral of \eqref{S2} coincides with the
total integral of
$$
2 \frac{(n\!-\!8)(3n\!-\!1)}{(n\!-\!2)(n\!-\!1)} \left(-B_3 +
\frac{(n\!-\!2)^2}{4n(n\!-\!1)} B_2 \right).
$$
Therefore, by summing up all terms, we find that $B_2$ contributes
with the coefficient
$$
-\frac{3}{4} \frac{(n\!-\!8)(n\!-\!6)}{(n\!-\!1)^2}
$$
and that the contribution of $B_3$ vanishes. Now since the sum
\eqref{S1} can be rewritten as
\begin{equation*}
(n\!-\!8) \left(-\frac{1}{9} (35n^2\!-\!260n\!+\!456) \J^3 +
\frac{2}{3} (n\!-\!1)(7n\!-\!30) \J |\Rho|^2 - \frac{2}{9}
(5n^2\!-\!2n\!-\!48) \tr(\Rho^3) \right),
\end{equation*}
we have proved that $7! \times \mbox{\eqref{int-a6}}$ coincides with
the sum of the total integrals of
$$
-\frac{3}{4} \frac{(n\!-\!8)(n\!-\!6)}{(n\!-\!1)^2} |\nabla \scal| =
- 3 (n\!-\!8)(n\!-\!6) |\nabla \J|^2
$$
and of the latter sum. This completes the proof of
\eqref{branson-a6}.

\subsection{Holographic formulas for heat kernel coefficients and applications}
\label{holo-heat}

In the present section, we express the heat kernel coefficients
$a_2,\dots,a_6$ for the conformal Laplacian in terms of the
renormalized volume coefficients $v_2,\dots,v_6$. The discussion of
$a_6$ will be restricted to locally conformally flat metrics. For
this purpose, we shall use the relations
$$
v_2 = -\frac{1}{2} \tr(\Rho) = -\frac{1}{2} \J, \quad v_4 =
\frac{1}{4} \tr (\wedge^2 \Rho) = \frac{1}{8} (\J^2 \!-\! |\Rho|^2)
$$
and
\begin{equation}\label{v6-flat}
v_6 = -\frac{1}{8} \tr (\wedge^3 \Rho) = -\frac{1}{48} (\J^3 \!-\!
3\J|\Rho|^2 \!+\! 2 \tr (\Rho^3))
\end{equation}
(in the locally conformally flat case). The resulting formulas will be
simpler than formulas in other bases and have applications to renormalized
volumes and functional determinants.

\begin{prop}\label{a-holo} We have
\begin{align*}
3 a_2 & = (n\!-\!4) v_2, \\
180 a_4 & = (n\!-\!6) \left(8(n\!-\!2) v_4 + 6(n\!-\!4) v_2^2 - 6
\Delta v_2 \right) + |W|^2, \quad \Delta = \delta d
\end{align*}
for $n \ge 3$ and, in the locally conformally flat case,
\begin{multline}\label{av-6}
7! a_6 = (n\!-\!8) \frac{8}{3} \\
\times \left((10 n^2\!-\!4n\!-\!96) v_6 + (n\!-\!6) \left(18
(n\!-\!2) v_2 v_4 + (n\!-\!10) v_2^3 - 9/2 |d v_2|^2 \right)\right)
\end{multline}
modulo a divergence.
\end{prop}

\begin{proof} The assertions follow from \eqref{CL-a2}, \eqref{CL-a4} and
\eqref{branson-a6}. We omit the details of the direct calculations.
\end{proof}

In view of Theorem \ref{RVC-CV}, the formula \eqref{av-6} is
particularly convenient to make explicit the divergence terms in
$a_6$. The idea of the proof of the following result goes back to
Branson and {\O}rsted (see Section 3 in \cite{BO-def} and Remark 5.12
in \cite{sharp}).

\begin{prop}\label{a6-div} For locally conformally flat metrics, the
divergence term of $7! a_6$ is given by the sum
$$
(n\!-\!8) \left( 12 \Delta^2(v_2) -12(n\!-\!2) \delta (\Rho dv_2) -
24 (n\!-\!2) \Delta (v_4) - 8 (2n\!-\!11) \Delta (v_2^2)\right).
$$
Here we use the convention $\Delta = \delta d$.
\end{prop}

\begin{proof} Let $M$ be closed. We first recall that
\begin{equation}\label{cv-a6}
\left(\int_{M^n} a_6 \dvol\right)^\bullet[\varphi] = (n\!-\!6)
\int_{M^n} \varphi a_6 \dvol.
\end{equation}
Moreover, we have the variational formulas (see \eqref{CT-v})
\begin{align*}
v_2^\bullet[\varphi] & = -2 \varphi v_2 - 1/2 \delta d \varphi = -2
\varphi v_2 - 1/2 \Delta \varphi, \\
v_4^\bullet[\varphi] & = -4 \varphi v_4 - 1/2 \delta (v_2 + 1/2
\Rho) d\varphi, \\
\left(\int_M v_6 \dvol\right)^\bullet[\varphi] & = (n\!-\!6) \int_M
\varphi v_6 \dvol
\end{align*}
and
\begin{equation}\label{norm-cv}
\left(\int_M |dv_2|^2 \dvol \right)^\bullet[\varphi] = (n\!-\!6)
\int_M \varphi |dv_2|^2 \dvol - \int_M \varphi (2 \Delta v_2^2 +
\Delta^2 v_2) \dvol.
\end{equation}
In order to verify \eqref{norm-cv}, we note that the transformation
rule $v_2^\bullet[\varphi] = -2 \varphi v_2 - 1/2 \Delta \varphi$
implies
\begin{equation*}
\left(\int_M |dv_2|^2 \dvol \right)^\bullet[\varphi] = (n\!-\!6)
\int_M \varphi |dv_2|^2 \dvol - 4 \int_M (d\varphi,v_2 dv_2) \dvol -
\int_M (d \Delta \varphi,dv_2) dv.
\end{equation*}
Then partial integration proves the claim. Now we determine the
variation on the left-hand side of \eqref{cv-a6} by combining these
formulas with the explicit formula \eqref{av-6}. Partial integration
shows that the conformal variation
$$
7! \left(\int_{M^n} a_6 \dvol\right)^\bullet[\varphi]
$$
coincides with
\begin{multline*}
(n\!-\!6) \int_M \varphi \; \mbox{(right-hand side of \eqref{av-6})} \; \dvol \\
+ (n\!-\!8)(n\!-\!6) \frac{8}{3} \int_M \varphi
\big[ 18 (n\!-\!2) (-1/2 \delta (v_2 + 1/2 \Rho) dv_2 - 1/2 \Delta v_4) \\
- 3/2 (n\!-\!10) \Delta (v_2^2) + 9/2 (2 \Delta (v_2^2) + \Delta^2
(v_2)) \big] \dvol.
\end{multline*}
By comparing this result with \eqref{cv-a6}, we conclude that
$7!a_6$ equals the sum of the right-hand side of \eqref{av-6} and
\begin{multline*}
(n\!-\!8) \frac{8}{3} \big[ 18 (n\!-\!2) (-1/2 \delta (v_2 + 1/2
\Rho) dv_2 - 1/2 \Delta v_4) \\ - 3/2 (n\!-\!10) \Delta (v_2^2) +
9/2 (2 \Delta (v_2^2) + \Delta^2 (v_2))\big].
\end{multline*}
Now simplification yields the assertion.
\end{proof}

An analogous argument can be used to recover the divergence term in
the holographic formula for $a_4$ in Proposition \ref{a-holo}. We
omit the details.

In particular, we obtain the following holographic formulas for
$a_4$ and $a_6$ in the respective critical dimensions.

\begin{corr}\label{a4-a6-c} For locally conformally flat metrics,
we have
$$
180 a_4 = - 32 v_4 + 12 \Delta (v_2)
$$
in dimension $n=4$ and
$$
7! a_6 = - 1280 v_6 - 24 \Delta^2(v_2) + 96 \delta(\Rho dv_2) + 192
\Delta (v_4) + 16 \Delta(v_2^2)
$$
in dimension $n=6$. Here we use the convention $\Delta = \delta d$.
\end{corr}

Next, we outline an alternative proof of \eqref{av-6} which is {\em
independent} of Gilkey's formula for the heat kernel coefficient
$a_6$ of general Laplace-type operators. In fact, the following
arguments will basically utilize only structural properties of the
heat kernel coefficient $a_6$ of the conformal Laplacian which are
known by the general theory.

\begin{proof} In a basis of scalar Riemannian invariants defined in
terms of $R$, $\Ric$ and $\scal$, the coefficients of $a_6$ do not
depend on the dimension $n$. It follows that, in a basis using $W$,
$\Rho$ and $\J$, the coefficients are cubic polynomials in $n$. Now,
in dimension $n=8$, $a_6$ is a local conformal invariant. Hence, in
this dimension, the functions $a_6$ vanishes for locally conformally
flat metrics. Therefore, by Lemma 6.1 in \cite{sharp}, we can write
$$
7! a_6 = (n\!-\!8) (\pi_3(n) \J^3 + \pi_2(n) \J |\Rho|^2 + \pi_1(n)
\tr(\Rho^3) + \pi_0(n) |d\J|^2),
$$
up to a divergence term, with unknown quadratic polynomials $\pi_i$.
Equivalently, it can be written in the form
$$
7! a_6 = (n\!-\!8) [\pi_3(n) v_6 + \pi_2(n) v_2 v_4 + \pi_1(n) v_2^3
+ \pi_0(n) |d v_2|^2],
$$
up to a divergence term and with other unknown quadratic polynomials
$\pi_i$. Now we use the fact that for $n=6$ the total integral of
$a_6$ is a conformal invariant. It implies that we can refine the
above ansatz into
\begin{equation}\label{refined}
7! a_6 = (n\!-\!8) [\pi_3(n) v_6 + (n\!-\!6)(\pi_2(n) v_2 v_4 +
\pi_1(n) v_2^3 + \pi_0(n) |dv_2|^2)]
\end{equation}
with an unknown quadratic polynomial $\pi_3$ and unknown monomials
$\pi_0, \pi_1, \pi_2$. In order to determine the polynomials $\pi_i$
for $i=1,2,3$, we consider locally conformally flat product spaces
of the form $\s^p \times \h^q$ (with $p+q=n$). On these spaces, the
identity $P_2(g_1 + g_2) = P_2(g_1) + P_2(g_2)$ (see
\eqref{hol-spec} for $r=0$) implies the relation
\begin{equation}\label{func-6}
a_6(g_1 \!+\! g_2) = a_6 (g_1) + a_4 (g_1) a_2 (g_2) + a_2 (g_1) a_4
(g_2) + a_6 (g_2) \in C^\infty(\s^p \times \h^q)
\end{equation}
of local heat kernel coefficients of conformal Laplacians. Using
this identity and the known heat kernel coefficients $a_2,\dots,a_6$
on round spheres $\s^n$ and hyperbolic spaces $\h^n$ (see Section
\ref{heat-spheres}), we determine $a_6$ for the product metrics. The
condition that the resulting polynomial in the variables $n$ and $p$
can be written in the form \eqref{refined} is a system of linear
relations for the unknown coefficients of the polynomials $\pi_i$;
the contribution by $|dv_2|^2$ vanishes in this case. Now a
calculation shows that these conditions are satisfied exactly by the
one-parameter family
\begin{multline}\label{fam}
(n\!-\!8) \frac{1}{3} \frac{1}{7!} \Big((224\!-\!3\beta) n^2 \!-\!
(896\!-\!18\beta)n \!-\! 768) v_6 \\ + (n\!-\!6) (3\beta n
\!-\!(192\!+\!2\beta)) v_2 v_4 + (n\!-\!6)((56\!-\!\beta)n \!-\!
(128\!-\!\beta)) v_2^3 \Big)
\end{multline}
with $\beta \in \r$. For $n=6$, the latter sum specializes to
$-16/63 v_6$ (for any $\beta$). Next, we prove that $\beta = 48$. We
note that in Gilkey's universal formula for $a_6$ (of $P_2$) the
contribution by $|\nabla \Ric|^2$ has a coefficient of the form
$$
a + \frac{b}{n\!-\!2}
$$
with some numerical coefficients $a$, $b$. In fact, this contribution
comes from the contributions $|\nabla R|^2$ and $|\nabla \Ric|^2$ in
Gilkey's formula. But $W=0$ implies $R=-\Rho \owedge g$ and a simple
calculation shows that $|\nabla R|^2 = 4(n-2) |\nabla \Rho|^2 + \dots$.
This proves the claim. Now the arguments in the proof of Proposition
\ref{a6-div} show that the ansatz \eqref{fam} for $a_6$ leads to a
divergence contribution of the form
$$
c (n\!-\!8) (3\beta n \!-\!(192\!+\!2\beta)) \Delta (v_4).
$$
By $8 v_4 = \J^2 - |\Rho|^2$, this term contains a contribution of
the form
$$
c (n\!-\!8) \frac{3\beta n \!-\!(192\!+\!2\beta)}{(n\!-\!2)^2}
|\nabla \Ric|^2.
$$
But the condition that the latter expression does not have a
second-order pole at $n=2$ is equivalent to $\beta =48$. Now, for
$\beta = 48$, \eqref{fam} reads
$$
(n\!-\!8) \frac{1}{7!} \frac{8}{3} \left((10 n^2\!-\!4n\!-\!96) v_6
+ 18 (n\!-\!6)(n\!-\!2) v_2 v_4 + (n\!-\!6)(n\!-\!10) v_2^3 \right).
$$
The latter sum yields \eqref{av-6} modulo the contribution by
$|dv_2|^2$. It remains to determine the monomial $\pi_0$ in
\eqref{refined}. For this purpose, we combine some results of
\cite{BGO} with the observation that $|d\J|^2$ is the only term
which leads to a contribution of $\Delta^2(\J)$ in $a_6$ (see the
proof of Proposition \ref{a6-div}). In fact, as noted in Section
\ref{a6-c-flat}, results in \cite{BGO} imply that $\Delta^2 (\J)$
contributes to $7! a_6$ with the coefficient $-6(n-8)$. Hence the
arguments in the proof of Proposition \ref{a6-div} show that
$\pi_0(n)=-12$.
\end{proof}

It would be interesting to extend the above type of holographic
formulas for $a_6$ beyond the conformally flat case.

\begin{rem}\label{holo-inv} The local conformal invariant $|W|^2$
in the holographic formula for $a_4$ (Proposition \ref{a-holo}) can
be interpreted as the restriction of the coefficient
$$
\tilde{a}_4 = 180 |\tilde{R}|^2
$$
to $t=1$ and $\rho=0$. This coefficient arises by formally
substituting the ambient curvature $\tilde{R}$ into the heat kernel
coefficient of the Laplacian of a Ricci flat metric. Similarly,
Lemma \ref{a66} shows that the conformal invariant given by $a_6$ in
dimension $n=8$ coincides with the restriction to $t=1$ and $\rho=0$
of the coefficient $\tilde{a}_6$ obtained by substituting
$\tilde{R}$ into the formula for $a_6$ for a Ricci flat metric.
These are special cases of a result announced in \cite{log} (Theorem
7.3).
\end{rem}

A fully explicit formula for $a_6$ for general metrics in dimension
$n=6$ is an important ingredient in a test of a version of the
AdS/CFT duality in \cite{BFT}. In fact, in \cite{BFT} it was found
that, for closed spin manifold $M^6$, the total integral of the
linear combination
$$
10 a_6(P_2) - 2 a_6(\SD^2) + (a_6(\Omega^2) - 2 a_6(\Omega^1) + 3
a_6 (\Omega^0))
$$
is proportional to a linear combination of $\chi(M)$ and the total
integral of $v_6$. Here $a_6(\SD^2)$ and $a_6(\Omega^p)$ are the
respective heat kernel coefficients of the square of the Dirac
operator $\SD$ and the Hodge-Laplacian on $p$-forms. By the
Deser-Schwimmer decomposition (see \cite{alex}), all three
individual integrals are linear combinations of the Euler
characteristic and the total integrals of three local conformal
invariants. The remarkable aspect of the above linear combination is
that the resulting weights of all three local conformal invariants
coincide with the weights in the corresponding decomposition of a
multiple of $v_6$. We also note that the linear combination
$$
3 a_6(\Omega^0) - 2 a_6(\Omega^1) + a_6 (\Omega^2)
$$
of heat kernel coefficients of Hodge-Laplace operators naturally
appears in the conformal variational formula of the Cheeger {\em
half-torsion} of $M^6$ (see Section 3 of \cite{B-spec}).

We close this section with two applications of the above results to
the determinant of $P_2$. First, we relate the determinant of the
conformal Laplacian to the renormalized volume.

\begin{prop}\label{det-vol-6} Let $(M^6,g)$ be a closed manifold
with a locally conformally flat metric $g$ so that $\ker(P_2(g)) = 0$.
Assume that $[g]$ is the conformal infinity of a Poincar\'e-Einstein
manifold $(X,g_+)$. Let $\V_6 = \V_6(g_+;\cdot)$ be the renormalized
volume functional. Then
\begin{multline}\label{FD-V-6}
(4\pi)^3 (\log \det (-P_2))^\bullet[\varphi] \\ = \frac{32}{63}
\V_6^\bullet[\varphi] + \frac{1}{7!} \frac{16}{3} \left(\int_{M^6}
(144 v_2 v_4 - 8 v_2^3 - 9 |dv_2|^2) \dvol \right)^\bullet[\varphi].
\end{multline}
\end{prop}

\begin{proof} The left-hand side of \eqref{FD-V-6} equals
$$
-2\int_{M^6} \varphi a_6  \dvol.
$$
Moreover, we have
$$
\V_6^\bullet[\varphi] = \int_{M^6} \varphi v_6 \dvol
$$
(by \cite{G-vol}). The proof of Proposition \ref{a6-div} contains
the variational formula
\begin{multline*}
\left(\int_{M^6} (72 v_2 v_4 - 4 v_2^3 - 9/2 |dv_2|^2) \dvol
\right)^\bullet[\varphi] \\ = -\frac{3}{16} \int_{M^6} \varphi (-24
\Delta^2(v_2) + 192 \Delta(v_4) + 96 \delta(\Rho dv_2) + 16
\Delta(v_2^2)) \dvol.
\end{multline*}
Hence the right-hand side of \eqref{FD-V-6} equals the integral
against $\varphi$ of
$$
\frac{32}{63} v_6 + \frac{1}{7!} (-2) (-24 \Delta^2(v_2) + 192
\Delta(v_4) + 96 \delta(\Rho dv_2) + 16 \Delta(v_2^2)).
$$
The assertion follows by combining this result with Corollary
\ref{a4-a6-c}.
\end{proof}

Now integration yields the following result.

\begin{corr}\label{det-vol-6c} In the situation of Proposition
\ref{det-vol-6}, the functional
$$
\det (-P_2) \exp \left(-\frac{1}{126 \pi^3} \V_6(g_+;\cdot) -
\frac{1}{7!} \frac{1}{12\pi^3} \int_{M^6} (144 v_2 v_4 - 8 v_2^3 - 9
|dv_2|^2) \dvol \right)
$$
is constant on $[g]$.
\end{corr}

\begin{proof} We apply the identity
$\Phi(e^{2\varphi}g)-\Phi(g)=\int_0^1\Phi(e^{2s\varphi}g)^\bullet[\varphi]
ds$ to the functional
$$
\Phi = (4\pi)^3 (\log \det(-P_2)) - \frac{32}{63} \V_6(g_+;\cdot) -
\frac{1}{7!} \frac{16}{3} \left(\int_{M^6} (144 v_2 v_4 - 8 v_2^3 -
9 |dv_2|^2) \dvol \right)
$$
and use \eqref{FD-V-6}. Simplification yields the claim.
\end{proof}

Similar arguments using Corollary \ref{a4-a6-c} prove the following
result. Let $(M^4,g)$ be a locally conformally flat closed $(M^4,g)$ with
$\ker (P_2(g)) = 0$. Assume that $[g]$ is the conformal infinity of a
Poincar\'e-Einstein manifold $(X,g_+)$. Then
\begin{equation}\label{FD-V-4}
(4\pi)^2(\log \det(-P_2))^\bullet[\varphi] \\ = \frac{16}{45}
\V_4^\bullet[\varphi] + \frac{2}{15} \left(\int_{M^4} v_2^2
\dvol\right)^\bullet[\varphi],
\end{equation}
and integration yields

\begin{corr}\label{det-vol-4c} In the above situation, the functional
$$
\det(-P_2) \exp \left(-\frac{1}{45 \pi^2} \V_4(g_+;\cdot) -
\frac{1}{120 \pi^2} \int_{M^4} v_2^2 \dvol \right)
$$
is constant on $[g]$.
\end{corr}

\begin{rem}\label{det-vol-2c} Corollary \ref{det-vol-6c} and Corollary
\ref{det-vol-4c} generalize the following result for closed Riemann
surfaces $(M^2,g)$. Assume that $[g]$ is the conformal infinity of a
Poincar\'e-Einstein space $(X^3,g_+)$. Then the functional
$$
{\det}(-\Delta) \exp \left(-\frac{1}{3\pi} \V_2(g_+;\cdot) \right) /
\int_{M^2} \dvol
$$
is constant on $[g]$. In fact, the assertion follows by comparing the
Polyakov formula
$$
-\log \left( \frac{{\det}(-\Delta_{\hat{g}})}{{\det}(-\Delta_g)}\right) =
- \log \left(\frac{\int \dvol_{\hat{g}}}{\int \dvol_g}\right) +
\frac{1}{12\pi} \int_{M^2} \varphi(-\Delta_g \varphi + \scal(g)) \dvol_g
$$
(see \cite{sharp}, Theorem 5.8) with \cite{G-vol}
\begin{equation}\label{pol-v2}
\V_2(g_+;\hat{g}) - \V_2(g_+;g) = \int_{M^2} (\varphi v_2(g) - 1/4
|d\varphi|_g^2) \dvol_g.
\end{equation}
\end{rem}

Second, we consider {\em local} extremal properties of the determinant of
$P_2$ in dimensions $n=4$ and $n=6$.

\begin{prop}\label{det-6-extreme} Let $M^6$ be a closed
manifold with a locally conformally flat Einstein metric $g$.\footnote{In
other words, $M$ is a closed space form.} Assume that $P_2(g)$ has trivial
kernel. Then the restriction of the functional $\log \det (-P_2)$ to
$[g]_1$ is critical at $g$ and the second variation at $g$ is given by
\begin{equation}\label{det-6-scv}
(4\pi)^3 (\log \det (-P_2))^{\bullet\bullet}[\varphi] = 1/630
\int_{M^6} \varphi (\Delta+24c) (3\Delta^2 - 120 c\Delta + 1600 c^2)
(\varphi) dv.
\end{equation}
Here $\Delta = -\delta d$ is the non-positive Laplacian and
$c=\scal(g)/120$.
\end{prop}

\begin{proof} Let $M^n$ be a closed manifold of even dimension $n$.
Assume that the kernel of $P_2(g)$ is trivial; this is a conformally
invariant condition. Then we recall the variational formula
$$
(4\pi)^\f (\log \det (-P_2))^{\bullet}[\varphi] = -2 \int_{M^n} \varphi
a_n dv
$$
at all metrics in $[g]$ (see \eqref{zeta-conform}). Thus Corollary
\ref{a4-a6-c} and the fact that the coefficients $v_{2k}$ are constant at
Einstein metrics imply the asserted criticality at $g$. Moreover, by
arguments as in the proof of Theorem \ref{extremal-RVC}, we obtain
$$
(4\pi)^3 (\log \det (-P_2))^{\bullet\bullet}[\varphi] = -2
\int_{M^6} \varphi a_6^\bullet[\varphi] dv
$$
at the Einstein metric $g$. These arguments again utilize the fact that
$a_6$ is constant at Einstein metrics. Now the formula for $a_6$ in
Corollary \ref{a4-a6-c} implies that
$$
7! a_6^\bullet[\varphi] = -1280 v_6^\bullet[\varphi] - 24 \Delta^2
(v_2^\bullet[\varphi]) - 192 \Delta (v_4^\bullet[\varphi]) + 96
\delta(\Rho d)(v_2^\bullet[\varphi]) - 32 v_2 \Delta
(v_2^\bullet[\varphi])
$$
at $g$; here we again use the fact that the coefficients $v_{2k}$ are
constant at Einstein metrics. Now $g(r)=(1-cr^2)^2g$ gives
$v(r)=(1-cr^2)^6$ and hence
$$
v_2=-6c, \quad v_4=15c^2, \quad v_6 = -20c^3.
$$
Therefore, \eqref{CT-v} yields the conformal variational formulas
\begin{equation*}
v_2^\bullet[\varphi] = - 2 v_2 \varphi + 1/2 \Delta \varphi, \quad
v_4^\bullet[\varphi] = - 4 v_4 \varphi - 5/2 c \Delta \varphi, \quad
v_6^\bullet[\varphi] = - 6 v_6 \varphi + 5 c^2 \Delta \varphi.
\end{equation*}
Then a calculation shows that
$$
7! a_6^\bullet[\varphi] = -4 (\Delta+24c) (3\Delta^2 - 120c \Delta +
1600c^2)(\varphi), \; \Delta = -\delta d.
$$
The proof is complete.
\end{proof}

For any $c \in \r$, the polynomial $3 \lambda^2 - 120 c \lambda + 1600
c^2$ is positive for all real $\lambda$. Since $\varphi$ is orthogonal to
$1$ (by preservation of volume), the quadratic form on the right-hand side
of \eqref{det-6-scv} is negative semi-definite. More precisely, it is
negative definite iff $\ker(\Delta+24c)=0$, i.e., if $(M^6,g)$ is not
isometric to a rescaled round sphere. Hence Proposition
\ref{det-6-extreme} shows that the restriction of $\log \det (-P_2)$ to
$[g]_1$ has a {\em local maximum} at $g$. This is a {\em local} version of
a global maximum result at the round metric on $\s^6$ due to Branson
\cite{sharp}.

Similar arguments yield the following result.

\begin{prop}\label{det-4-extreme} Let $(M^4,g)$ be a closed locally
conformally flat Einstein manifold. Assume that $P_2(g)$ has trivial
kernel. Then the restriction of the functional $\log \det (-P_2)$ to
$[g]_1$ is critical at $g$ and the second variation at $g$ is given by
\begin{equation}\label{det-4-scv}
(4\pi)^2 (\log \det (-P_2))^{\bullet\bullet}[\varphi] = \frac{1}{15}
\int_{M^4} \varphi (\Delta-8c)(\Delta+16c)(\varphi) dv.
\end{equation}
Here $\Delta = -\delta d$ is the non-positive Laplacian and
$c=\scal(g)/48$.
\end{prop}

\begin{proof} By arguments as in the proof of Proposition
\ref{det-6-extreme}, it suffices to prove the variational formula
$$
180 a_4^\bullet[\varphi] = -6 (\Delta-8c)(\Delta+16c)(\varphi)
$$
at $g$. In order to prove this relation, we note that the explicit formula
for $a_4$ in Corollary \ref{a4-a6-c} yields
$$
180 a_4^\bullet[\varphi] = -32 v_4^\bullet[\varphi] - 12 \Delta
(v_2^\bullet[\varphi]).
$$
Now at the metric $g$ we have $v_2=-4c$, $v_4=6c^2$, and \eqref{CT-v}
shows that
$$
v_2^\bullet[\varphi] = -2 \varphi v_2 + 1/2 \Delta \varphi \quad
\mbox{and} \quad v_4^\bullet[\varphi] = -4 \varphi v_4 - 3/2 c
\Delta \varphi.
$$
Then the assertion follows by a simple calculation. \end{proof}

Proposition \ref{det-4-extreme} implies that, if $\scal(g)>0$, then the
restriction of $\log \det (-P_2)$ to $[g]_1$ has a {\em local minimum} at
$g$. This is a {\em local} version of a global minimum result at the round
metric on $\s^4$ (see \cite{sharp}, \cite{CY}). In contrast, if $\scal(g)
< 0$, then small eigenvalues of $\Delta$ may lead to a non-definite second
variation. In other words, there is no local extremal result on compact
quotients $\Gamma \backslash \h^4$.

Finally, in dimension $n=2$, we have $P_2 = \Delta$ and the analogous
relation
\begin{equation}\label{det-2-scv}
4\pi (\log \det (-\Delta))^{\bullet\bullet}[\varphi] = 2/3 \int_{M^2}
\varphi (\Delta+8c)(\varphi) dv, \quad c=\scal(g)/8
\end{equation}
at an Einstein metric $g$ shows that the restriction of the functional
$\log \det (-\Delta)$ to $[g]_1$ has a {\em local maximum} at $g$. In
particular, $\log \det (-\Delta)$ has a local maximum at the round metric
on $\s^2$. This is a {\em local} version of Onofri's global maximum result
at the round metric on $\s^2$ (see \cite{sharp} and the reference
therein). For general global results see \cite{OPS}.

\subsection{Polyakov formulas for functional determinants of $P_2$}\label{pol-det}

Here we recall the {\em global} Polyakov formulas for the determinant of
the conformal Laplacian in dimension $n=4$ and $n=6$. For the
two-dimensional case see \eqref{pol-det-2}. These results support
\eqref{polyakov-g}. For proofs of the following results we refer to
\cite{sharp}.

In the following, we use the notation $P_2 = P_2(g)$, $\hat{P}_2 =
P_2(\hat{g})$ and $Q_n = Q_n(g)$, $\hat{Q}_n = Q_n(\hat{g})$ for $\hat{g}
= e^{2\varphi} g$.

\begin{prop}\label{pol-det-4} For any locally conformally flat closed
manifold $(M^4,g)$ with trivial $\ker (P_2(g))$, we have
$$
(4\pi)^2 \log \left( \frac{\det (-\hat{P}_2)}{\det (-P_2)} \right) =
\frac{1}{90} \int_{M^4} \varphi (\hat{Q}_4 \dvol_{\hat{g}} + Q_4
\dvol_g) + \int_{M^4} (I_4(\hat{g}) \dvol_{\hat{g}} - I_4(g)
\dvol_g)
$$
with
$$
I_4 = \frac{1}{45} \J^2.
$$
\end{prop}

\begin{prop}\label{pol-det-6} For any locally conformally flat closed
manifold $(M^6,g)$ with trivial $\ker (P_2(g))$, we have
\begin{equation*}
(4\pi)^3 \log \left( \frac{\det (-\hat{P}_2)}{\det (-P_2)} \right) = -
\frac{10}{7! \, 3} \int_{M^6} \varphi (\hat{Q}_6 \dvol_{\hat{g}} + Q_6
\dvol_g) - \int_{M^6} (I_6(\hat{g}) \dvol_{\hat{g}} - I_6(g) \dvol_g)
\end{equation*}
with
$$
I_6 = \frac{1}{7!} \frac{1}{3} (68 \J^3 - 64 \J |\Rho|^2 + 26 |d\J|^2).
$$
\end{prop}

\subsection{Polyakov type formulas for renormalized volumes}\label{pol-volume}

In the present section, we prove a global Polyakov type formula for the
renormalized volume of Poincar\'e-Einstein metrics with conformal
infinities of {\em even} dimension (see \eqref{RV-def}). We recall that
$\V(g_+;\cdot)$ is a conformal invariant in odd dimensions \cite{G-vol}.
The main result will be a consequence of the holographic formulas for
$Q$-curvatures \cite{GJ-holo}, \cite{holo-II}. It is an improved version
of the Polyakov formulas stated in \cite{GMS}.

We first recall the holographic formulas for $Q$-curvatures. For generic
$\lambda$, assume that
$$
u \sim \sum_{j \ge 0} r^{\lambda+2j} a_{2j}(\lambda) + \sum_{j \ge
0} r^{n-\lambda+2j} b_{2j}(\lambda), \quad a_{2j}, b_{2j} \in
C^\infty(M)
$$
is a formal approximate solution of the equation
$$
-\Delta_{g_+}u = \lambda(n-\lambda)u
$$
with $i^*(r^2 g_+) = g$. Then the coefficients $a_{2j}(g;\lambda)$ (for
$2j \le n$) are recursively determined by $a_0$ through the relations
$$
a_{2j}(g;\lambda) = \T_{2j}(g;\lambda)(a_0), \; a_0 \in C^\infty(M)
$$
with certain differential operators $\T_{2j}(g;\lambda)$ of respective
order $2j$ which are natural in $g$ and rational in $\lambda$. Let
$$
c_{2N} = (-1)^{N} (2^{2N-1} N!(N\!-\!1)!)^{-1}.
$$
In these terms, the holographic formula for $Q_{2N}$ states that
\begin{equation}\label{hol-form-g}
2N c_{2N} Q_{2N}(g) = 2N v_{2N}(g) + \sum_{j=1}^{N-1} (2N\!-\!2j)
\T_{2j}^* \left(g;\f\!-\!N \right) (v_{2N-2j}(g)).
\end{equation}
In particular, for the critical $Q$-curvature $Q_n$, we have the
relation
\begin{equation}\label{hol-form-c}
n c_n Q_n(g) = n v_n(g) + \sum_{j=1}^{\f-1} (n\!-\!2j)
\T_{2j}^*(g;0) (v_{2N-2j}(g)).
\end{equation}

\begin{thm}\label{PV-holo} Let $n$ be even. Assume that $(M^n,[g])$
is the conformal infinity of a Poincar\'e-Einstein space $(X^{n+1},g_+)$.
Then
\begin{multline}\label{PV-holo-f}
\V_n(g_+;\hat{g}) - \V_n(g_+;g) \\ = c_n \frac{1}{2} \int_{M^n}
\varphi (Q_n(g) \dvol_g + Q_n(\hat{g}) \dvol_{\hat{g}}) - \int_{M^n}
(\Phi_n(\hat{g}) \dvol_{\hat{g}} - \Phi_n(g) \dvol_g),
\end{multline}
where
\begin{equation}\label{def-Phi}
\Phi_n(g) \st \frac{1}{2n} \sum_{j=1}^{\f-1} (n\!-\!2j)
\dot{\T}_{2j}(g;0)(1) v_{2n-2j}(g).
\end{equation}
Here the dot denotes the derivative with respect to $\lambda$.
\end{thm}

\begin{proof} We first note that
\begin{align*}
\V_n(g_+;\hat{g}) - \V_n(g_+;g) & = \int_0^1 (d/dt)
(\V_n(g_+;e^{2t\varphi}g)) dt = \int_0^1 (d/ds)|_0 (\V_n(g_+;e^{2s\varphi} e^{2t\varphi}g)) dt \\
& = \int_0^1 \left(\int_M \varphi v_n (e^{2t\varphi}g)
\dvol_{e^{2t\varphi}g} \right) dt.
\end{align*}
Now the holographic formula \eqref{hol-form-c} and the
transformation law
$$
e^{n\varphi} Q_n(e^{2\varphi}g) = Q_n(g) + (-1)^\f P_n(g)(\varphi)
$$
imply
\begin{align*}
\V_n(g_+;\hat{g}) - \V_n(g_+;g) & = c_n \int_0^1 \left(\int_M
\varphi Q_n(e^{2t\varphi}g) e^{nt\varphi} \dvol_g \right) dt + \cdots \\
& = c_n \int_0^1 \left( \int_M \varphi (Q_n(g) + (-1)^\f
P_n(g)(t\varphi)) \dvol_g \right) dt + \cdots \\
& = c_n \int_M \varphi \left(Q_n(g) + (-1)^\f \frac{1}{2}
P_n(g)(\varphi)\right) \dvol_g + \cdots \\
& = c_n \frac{1}{2} \int_M \varphi (Q_n(g) \dvol_g + Q_n(\hat{g})
\dvol_{\hat{g}}) + \cdots.
\end{align*}
This yields the $Q_n$-term in \eqref{PV-holo-f}. In the latter
calculation, the hidden terms are given by
$$
-\frac{1}{n} \int_0^1 \int_M \varphi F_n(e^{2t\varphi}g)
\dvol_{e^{2t\varphi}g} dt
$$
with the local scalar invariant
\begin{equation}\label{Fn}
F_n(g) = \sum_{j=1}^{\f-1} (n\!-\!2j) \T_{2j}^* (g;0)(v_{2N-2j}(g)).
\end{equation}
But, if $\Psi_n(g)$ is a conformal primitive of $F_n(g)$, i.e., if
$\Psi_n^\bullet(g)[\varphi] = \int_M \varphi F_n(g) \dvol_g$, then
we have
\begin{multline*}
\int_0^1 \int_M \varphi F_n(e^{2t\varphi}g) \dvol_{e^{2t\varphi}g} dt \\ =
\int_0^1 \Psi_n^\bullet(e^{2t\varphi}g)[\varphi] dt = \int_0^1 (d/dt)
(\Psi_n(e^{2t\varphi} g) dt = \Psi_n(e^{2\varphi}g) - \Psi_n(g).
\end{multline*}
Therefore, it remains to verify that the functional $\Psi_n \st
\int_M \Phi_n \dvol$ is a conformal primitive of $\frac{1}{n} F_n$.
But
$$
\left(\int_M (Q_{2N}-c_{2N}^{-1}v_{2N}) \dvol
\right)^\bullet[\varphi] = (n-2N) \int_M \varphi (Q_{2N} -
c_{2N}^{-1}v_{2N}) \dvol
$$
and the holographic formula \eqref{hol-form-g} imply
$$
\left(\int_M F_{2N} \dvol\right)^\bullet[\varphi] = (n-2N) \int_M
\varphi F_{2N} \dvol,
$$
where
$$
F_{2N}(g) = \sum_{j=1}^{N-1} (2N\!-\!2j) \T_{2j}^* \left(g;\f\!-\!N
\right) (v_{2N-2j}(g)).
$$
We divide the latterrelation by $n-2N$ and set $n=2N$; note that
$\T_{2j}(\lambda)$ does not have a pole at $\lambda=0$. This gives
$$
\frac{1}{2} \left(\int_M \sum_{j=1}^{\f-1} (n\!-\!2j)
\dot{\T}_{2j}^*(0) (v_{n-2j}) \dvol \right)^\bullet[\varphi] =
\int_M \varphi F_n \dvol.
$$
Now partial integration completes the proof.
\end{proof}

\begin{rem}\label{FG-T} There is a unique solution $U \in C^\infty(X)$
(modulo $O(r^n)$) of the equation
$$
-\Delta_{g_+} U = n
$$
being of the form
$$
U = \log r + A + B r^n \log r + O(r^n).
$$
In the latter expansion, the coefficient $A$ is completely determined by
$g$. In fact, the terms $\dot{\T}_{2j}(0)(1)$ for $j=1,\dots,n/2-1$ are
just the Taylor coefficients of $A$. Moreover, $B|_{r=0} = -c_n Q_n$. For
the details we refer to \cite{FG-P}.
\end{rem}

Theorem \ref{PV-holo} implies a global result on extremal values.

\begin{corr}\label{RV-global} Let $n$ be even and let
$g_+ = 4/(1-|x|^2)^2 \sum_i dx_i^2$ be the hyperbolic metric of sectional
curvature $-1$ on the unit ball. Then, for all metrics on $\s^n$ which are
conformal to the round metric $g_0$ and have the same volume, the
functional
$$
(-1)^\f \left(\V_n(g_+;\cdot) + \int_M \Phi_n \dvol\right)
$$
is minimal exactly at the pull-backs $\kappa^*(g_0)$ of $g_0$ under
conformal diffeomorphisms.
\end{corr}

\begin{proof} By Corollary 3.8 in \cite{sharp}, we have the inequality
$$
\frac{n}{2} \int_{\s^n} \varphi (Q_n(g) \dvol_{g} + Q_n(g_0)
\dvol_{g_0}) \ge  Q_n(g_0) \log \left(\int_{\s^n} \dvol_g /
\int_{\s^n} \dvol_{g_0}\right)
$$
for $g = e^{2\varphi} g_0$ with equality exactly for $g = \kappa^*(g_0)$.
It suffices to combine this result with Theorem \ref{PV-holo}. Note that
$\V_n(g_+;g_0)$ can be defined by using the boundary defining function $2
\frac{1-|x|}{1+|x|}$.
\end{proof}

Finally, we make Theorem \ref{PV-holo} explicit in low-order cases.
In fact, calculations using
$$
\T_2(g;\lambda) = \frac{1}{2(n\!-\!2\!-\!2\lambda)} (\Delta \!-\!
\lambda \J)
$$
and
\begin{multline*}
\T_4(g;\lambda) =
\frac{1}{8(n\!-\!2\!-\!2\lambda)(n\!-\!4\!-\!2\lambda)}
[(\Delta-(\lambda\!+\!2)\J)(\Delta-\lambda\J) \\
- \lambda(n\!-\!2\!-\!2\lambda) |\Rho|^2 - 2(n\!-\!2\!-\!2\lambda)
\delta(\Rho d) - (n\!-\!2\!-\!2\lambda) (d\J,d)]
\end{multline*}
(see \cite{BJ}, \cite{juhl-book}) yield the following results. Let
$Q_n = Q_n(g)$, $\hat{Q}_n = Q_n(\hat{g})$ and $v_{2k} = v_{2k}(g)$,
$\hat{v}_{2k}=v_{2k}(\hat{g})$.

\begin{lemm}\label{PV-2}
$$
\V_2(g_+;\hat{g}) - \V_2(g_+;g) = -\frac{1}{4} \int_{M^2} \varphi
(\hat{Q}_2 \dvol_{\hat{g}} + Q_2 \dvol_g).
$$
\end{lemm}

\begin{lemm}\label{PV-4}
$$
\V_4(g_+;\hat{g}) - \V_4(g_+;g) = \frac{1}{32} \int_{M^4} \varphi
(\hat{Q}_4 \dvol_{\hat{g}} + Q_4 \dvol_g) - \frac{1}{8} \int_{M^4}
\left(\hat{v}_2^2 \dvol_{\hat{g}} - v_2^2 \dvol_g\right).
$$
\end{lemm}

\begin{lemm}\label{PV-6}
\begin{multline*}
\V_6(g_+;\hat{g}) - \V_6(g_+;g) = -\frac{1}{768} \int_{M^6} \varphi
(\hat{Q}_6 \dvol_{\hat{g}} + Q_6 \dvol_g) \\ + \frac{1}{192}
\int_{M^6} ((\hat{v}_2 \hat{P}_2 (\hat{v}_2) - 32 \hat{v}_2
\hat{v}_4) \dvol_{\hat{g}} - (v_2 P_2 (v_2) - 32 v_2 v_4) \dvol_g).
\end{multline*}
\end{lemm}

By combining these results with the integrated versions of \eqref{FD-V-6}
and \eqref{FD-V-4} for the determinants one may reprove the global
Polyakov formulas in Section \ref{pol-det}.

\subsection{Heat kernel coefficients of round spheres and hyperbolic spaces}
\label{heat-spheres}

One can use harmonic analysis on the spheres $\s^n$ of curvature $1$ to
derive the following description of heat kernel coefficients of the
Laplace operators. These results are closely related to analogous results
for the hyperbolic spaces $\h^n$ of curvature $-1$. For details we refer
to \cite{CW} and \cite{camp}.

\begin{prop}\label{heat-sphere} For even $n$ and $0 \le k \le n/2-1$,
we define the numbers $\beta_{k,n}$  by the relation
$$
\prod_{j=\frac{1}{2}}^{\frac{n-3}{2}} (j^2-t^2) = \sum_{k=0}^{\f-1}
\beta_{k,n} t^{2k}
$$
and set
$$
b_k = (-1)^{\f-1} \frac{1}{(\f-1) \cdots (\f-k)} \beta_{\f-1-k,n}.
$$
Similarly, for odd $n$ and $0 \le k \le (n-3)/2$, we define the
numbers $\beta_{k,n}$ by the relation
$$
\prod_{j=1}^{\frac{n-3}{2}} (j^2-t^2) = \sum_{k=0}^{\frac{n-3}{2}}
\beta_{k,n} t^{2k}
$$
and set
$$
b_k = (-1)^{\frac{n-3}{2}} \frac{1}{(\f-1) \cdots (\f-k)}
\beta_{\frac{n-3}{2}-k,n}.
$$
Then the heat kernel coefficients $a_{2k}$ with $0 \le 2k \le n-2$
of the Laplacian of $\s^n$ are given by the formula
\begin{equation}
a_{2k} = b_k + \left( \fr \right)^2 b_{k-1} + \cdots + \frac{1}{k!}
\left( \fr \right)^{2k} b_0.
\end{equation}
\end{prop}

For any $k$, Proposition \ref{heat-sphere} yields a formula for
$a_{2k}$ as a polynomial of degree $2k$ in $n$ which is valid for
sufficiently large $n$. The same formula continues to hold for small
$n$ although the analogs of the formulas in Proposition
\ref{heat-sphere} may look differently. For instance, on $\s^2$ we
have the classical formula
$$
b_k = (-1)^k \frac{1}{k!} (2^{1-2k}-1) B_{2k}
$$
with the Bernoulli numbers $B_{2k}$.

Now Proposition \ref{heat-sphere} can be used to derive explicit
formulas for the first few heat kernel coefficients of the Laplace
operators on the round spheres of curvature $1$. We find that
$a_0=1$ and
\begin{align*}
a_2 & = \frac{n(n\!-\!1)}{6}, \\
a_4 & = \frac{n(n\!-\!1)(5n^2-7n+6)}{3 \cdot 5!}, \\
a_6 & = \frac{n(n\!-\!1)(35n^4-112n^3+187n^2-110n+96)}{9 \cdot 7!}.
\end{align*}
By $P_2 = \Delta - n/2(n/2-1)$ and $e^{t(\Delta + c)} = e^{t\Delta}
e^{tc}$ (for any real constant $c$), these results imply that the
first four heat kernel coefficients of the conformal Laplacian on
the sphere $\s^n$ are given by $a_0=1$ and
\begin{align*}
a_2 & = -\frac{n(n\!-\!4)}{2 \cdot 3!}, \\
a_4 & = \frac{n(n\!-\!6)(5 n^2 - 18 n + 4)}{2 \cdot 6!}, \\
a_6 & = -\frac{n(n\!-\!8)(35 n^4 - 308 n^3 + 688n^2 - 184n
-96)}{9!}.
\end{align*}
In particular, in the respective critical dimensions, we find
\begin{equation*}
a_2 = 1/3, \quad a_4 = - 1/15 \quad \mbox{and} \quad a_6 = 5/63.
\end{equation*}
Hence we obtain
\begin{equation*}
\frac{1}{(4\pi)} \int_{S^2} a_2 \dvol = \frac{1}{6} \chi(S^2), \quad
\frac{1}{(4\pi)^2} \int_{S^4} a_4 \dvol = - \frac{1}{180} \chi(S^4)
\end{equation*}
and
$$
\frac{1}{(4\pi)^3} \int_{S^6} a_6 \dvol = \frac{1}{1512} \chi(S^6)
$$
using
$$
\int_{S^2} \dvol = 4\pi, \quad \int_{S^4} \dvol = \frac{8}{3} \pi^2
\quad \mbox{and} \quad \int_{S^6} \dvol = \frac{16}{15} \pi^3.
$$
These results for the round spheres $\s^2, \dots, \s^6$ are special
cases of Proposition 4.3 in \cite{PR} (up to a correction of the
sign for $\s^6$). Note also that the values
$$
v_2 = -1/2, \quad  v_4 = 3/8 \quad \mbox{and} \quad v_6 = -5/16
$$
(in the respective critical dimensions $n=2,4,6$) yield the
relations
\begin{equation}\label{a-v}
a_2 = -2/3 v_2, \quad a_4 = - 8/45 v_4 \quad \mbox{and} \quad a_6 =
-16/63 v_6.
\end{equation}
These relations may serve as a cross-check of the coefficient of
$\V^\bullet[\varphi]$ in \eqref{FD-V-4} and \eqref{FD-V-6}.

Next, we observe that the identity $\H(r) = (1-r^2/4)^{-2} P_2 $ on
$\s^n$ (see \eqref{hol-ein}) implies
$$
\Tr (e^{t\H(r)}) = \Tr (e^{t(1-r^2/4)^{-2}P_2}).
$$
Hence the heat kernel coefficients of $\H(r)$ for the round sphere
$\s^n$ are given by the closed formula
$$
a_{2N}(r) = (1\!-\!r^2/4)^{n-2N} a_{2N}(\s^n)
$$
in terms of the heat kernel coefficients of the conformal Laplacian
on $\s^n$. In particular, we obtain
\begin{equation}\label{hcc-sphere}
a_0(r) = (1\!-\!r^2/4)^n \quad \mbox{and} \quad a_2(r) =
-n(n\!-\!4)/12 (1\!-\!r^2/4)^{n-2}.
\end{equation}
The latter formulas can be used for a cross-check of Lemma \ref{top}
and Theorem \ref{A-fine}. Note also that, for even $n$, the critical
coefficient $a_n(r)$ does not depend on $r$.

Finally, the above results for the conformal Laplacian on the spheres
$\s^n$ of curvature $1$ imply analogous results for the conformal
Laplacian on the hyperbolic spaces $\h^n$ of curvature $-1$ by using the
well-known duality
\begin{equation}\label{dual}
a_{2N}(\s^n) = (-1)^N a_{2N}(\h^n).
\end{equation}

\subsection{The correction terms as polynomials in $v_{2k}$}\label{corr-v}

The coefficients $\omega_{2k}$ in the expansion of the correction term
$(\dot{w}(r))^2$ of the second heat-kernel coefficient $a_2(r)$ can be
written as non-linear polynomials in renormalized volume coefficients. The
first few of these formulas are
\begin{align*}
\omega_0 & = 0, \\
\omega_2 & = v_2^2, \\
\omega_4 & = 4 v_2 v_4 - v_2^3, \\
\omega_6 & = 6 v_2 v_6 + 4 v_4^2 - 5 v_4 v_2^2 + v_2^4, \\
\omega_8 & = 8 v_2 v_8 + 12 v_4 v_6 - 8 v_2 v_4^2 - 7 v_6 v_2^2 + 6
v_4 v_2^3 - v_2^5.
\end{align*}

\subsection{Product metrics and consequences}\label{special}

We describe an alternative method for the derivation of an explicit
formula for the second Taylor coefficient $a_{(2,2)}$ of $a_2(r)$. For
this purpose, we consider product manifolds $M^n = M_1^p \times M_2^q$
($n=p+q$ with $p, q \ge 3$) with the product metrics $g = g_1 + g_2$ given
by Einstein metrics $g_i$ on the factors so that
\begin{equation}\label{tractor}
\lambda \st \frac{\scal(g_1)}{4 p(p\!-\!1)} = - \frac{\scal(g_2)}{4
q(q\!-\!1)}.
\end{equation}
An important feature of these product metrics is that their Schouten
tensors decompose as the sum of the respective Schouten tensors of
the factors. In fact, we have
\begin{equation}\label{mixed-Schouten}
\Rho (g) = \begin{pmatrix} \Rho (g_1) & 0 \\ 0 & \Rho (g_2)
\end{pmatrix} = \begin{pmatrix} 2 \lambda g_1 & 0 \\ 0 & - 2 \lambda g_2
\end{pmatrix}.
\end{equation}
In particular, we find $\J(g) = 2\lambda(p-q)$. The following result
describes the associated Poincar\'e-Einstein metrics.

\begin{prop}[\cite{GL}]\label{GL} The metric
\begin{equation}\label{mixed-PM}
g_+ = r^{-2}\left(d r^2 + (1\!-\!\lambda r^2)^2 g_1 + (1\!+\!\lambda
r^2)^2 g_2 \right)
\end{equation}
satisfies $\Ric(g_+) = -n g_+$.
\end{prop}

Equation \eqref{mixed-PM} implies that
\begin{equation}\label{vw-pq}
v(r) = (1-\lambda r^2)^p (1+\lambda r^2)^q \quad \mbox{and} \quad
w(r) = (1-\lambda r^2)^{p/2} (1+\lambda r^2)^{q/2}.
\end{equation}

Now let $\H(r;g)$ be the corresponding holographic Laplacian. A
calculation yields the explicit formula
\begin{equation}\label{hol-spec}
\H(r;g) = (1\!-\!\lambda r^2)^{-2} P_2(g_1) + (1\!+\!\lambda
r^2)^{-2} P_2(g_2) = \H(r;g_1) + \H(r;g_2).
\end{equation}
Thus, by the functoriality of heat kernel coefficients (see Section
4.8 of \cite{G-book}), the coefficient $a_2(r)$ for $\H(r;g)$ equals
\begin{multline*}
a_0((1-\lambda r^2)^{-2} P_2(g_1)) a_2((1+\lambda r^2)^{-2}P_2(g_2))
\\ + a_2((1-\lambda r^2)^{-2}P_2(g_1)) a_0((1+\lambda
r^2)^{-2}P_2(g_2)),
\end{multline*}
i.e.,
\begin{multline*}
(1-\lambda r^2)^p (1+\lambda r^2)^{q-2} a_0(P_2(g_1)) a_2(P_2(g_2))
\\ + (1-\lambda r^2)^{p-2} (1+\lambda r^2)^q a_2(P_2(g_1))
a_0(P_2(g_2)).
\end{multline*}
Hence, using $a_0 = 1$ and $a_2 = -(n\!-\!4)/6 \J$ (for any manifold
of dimension $n$), we find
\begin{equation}\label{a-pq}
a_2(r) = \frac{q\!-\!4}{6} 2\lambda q (1\!-\!\lambda r^2)^p
(1\!+\!\lambda r^2)^{q-2} -\frac{p\!-\!4}{6} 2 \lambda p
(1\!-\!\lambda r^2)^{p-2} (1\!+\!\lambda r^2)^q.
\end{equation}

We apply the latter formula to prove the following result.

\begin{prop}\label{a22} The coefficient $a_{(2,2)}$ is given by
\begin{equation}\label{a-22}
a_{(2,2)} = \frac{n\!-\!5}{12} \J^2 - \frac{n\!-\!8}{12} |\Rho|^2.
\end{equation}
\end{prop}

\begin{proof} On the one hand, the coefficient $a_{(2,2)}$ is a
universal scalar Riemannian invariant of weight $4$. Hence it can be
written as a linear combination
$$
a \J^2 + b |\Rho|^2 + c \Delta \J + d |W|^2
$$
with universal coefficients depending on $n$. On the other hand, the
coefficient of $r^2$ in the Taylor expansion of \eqref{a-pq} is
given by
$$
\frac{1}{3} \lambda^2 ((n^3\!-\!6n^2\!+\!8n) - p(4n\!-\!20)n +
p^2(4n\!-\!20)).
$$
Now we have
$$
a \J^2 + b |\Rho|^2 = 4 \lambda^2 a (p-q)^2 + 4 \lambda^2 b n = 4
\lambda^2 ((an^2+bn) - 4 p a n + 4 p^2 a)
$$
and a comparison of the decompositions $R = W - \Rho \owedge g$ on the
product space and on the factors yields the relation $|W|^2 = |W_1|^2 +
|W_2|^2$ of the Weyl tensors.\footnote{In particular, the product space is
conformally flat if the factors (of opposite curvature) are. The conformal
flatness of the product $\s^p \times \h^q$ also follows directly from its
well-known realization as the complement of a sphere in a sphere.} But the
universal equality of
$$
4 \lambda^2 ((an^2+bn) - 4 p a n + 4 p^2 a) + d |W|^2
$$
and$$ \frac{1}{3} \lambda^2 ((n^3\!-\!6n^2\!+\!8n) - p(4n\!-\!20)n +
p^2(4n\!-\!20))
$$
for {\em all} $p$ holds precisely if $d=0$ and $a$ and $b$ take the
given values. In order to determine $c$, we use the fact that the
integrands
$$
\Lambda_2 = a_{(2,2)} - v_2^2 = a_{(2,2)} - \J^2/4
$$
define variational integrals. By Lemma \ref{structure-4}, this
property implies the relation
$$
2(a+b-1/4) + (n-4) c = 0.
$$
Hence $c=0$.
\end{proof}

\begin{lemm}\label{structure-4} The linear combination
$$
I_4 = a \J^2 + b |\Rho|^2 + c \Delta \J
$$
satisfies the conformal variational formula
$$
\left(\int_{M^n} I_4 \dvol \right)^\bullet [\varphi] = (n-4)
\int_{M^n} \varphi I_4 \dvol
$$
iff
\begin{equation}\label{rel-inv}
2(a+b) + (n-4) c = 0.
\end{equation}
\end{lemm}

\begin{proof} The claim easily follows from the usual conformal variational
formulas for $\J$ and $\Rho$. Note that the assertion already can be found
as Equation (5.4) in \cite{BO-def}.
\end{proof}

Note that $8 v_4 = \J^2 -|\Rho|^2$ and $Q_4 = \f \J^2 - 2|\Rho|^2 -
\Delta \J$ both satisfy the relation \eqref{rel-inv}.

\begin{rem}\label{obs} By $-2 v_2 = \J$, the formula
$\Lambda_0 = -(n-4)/6 \J$ can be restated as
$$
\Lambda_0 = (n\!-\!4)/3 v_2.
$$
Similarly, using $8 v_4 = \J^2-|\Rho|^2$, Proposition \ref{a22}
implies that the difference
$$
\Lambda_2 = a_{(2,2)} - v_2^2 = a_{(2,2)} - \J^2/4
$$
equals
\begin{equation*}
\frac{n\!-\!8}{12} (\J^2 - |\Rho|^2) = (n\!-\!8) \frac{2}{3} v_4.
\end{equation*}
These are the special cases $k=1$ and $k=2$ of Theorem \ref{A-fine}.
\end{rem}

For the product metrics in Proposition \ref{GL}, the identity
\begin{equation*}
a_2(r) - (\dot{w}(r))^2 = - \frac{1}{3} \left(\ddot{v}(r) -
\frac{n\!-\!2}{2} r^{-1} \dot{v}(r)\right)
\end{equation*}
can be verified by a straightforward calculation using the explicit
formulas \eqref{vw-pq} and \eqref{a-pq}. This is a special case of Theorem
\ref{A-fine}.

It seems difficult to determine the higher-order Taylor coefficients of
$a_2(r)$ by a generalization of the above method. The main reason is that
special metrics with explicitly known Poincar\'e-Einstein metrics are
rare. In particular, considerations of the product spaces $\s^p \times
\h^q$ do not suffice to determine the coefficient of $r^4$ in $a_2(r)$.
Nevertheless, this method can be used to obtain the following information
on the coefficient of $r^2$ in the expansion of $a_4(r)$.

\begin{prop}\label{a42} For locally conformally flat metrics,
the coefficient of $r^2$ in the expansion of $a_4(r)$ has the form
\begin{equation*}
360 a_{(4,2)} = 48 (n\!-\!4)(n\!-\!6) v_6 + 48 (n\!-\!4)(n\!-\!12)
v_2 v_4 - 12 (n\!-\!4)(n\!-\!16) v_2^3,
\end{equation*}
up to terms which involve derivatives of curvature.
\end{prop}

\begin{proof} The coefficient $a_{(4,2)}$ is a universal linear
combination (with coefficients depending on $n$) of the three terms $v_6$,
$v_2 v_4$ and $v_2^3$ and of terms which involve derivatives of curvature:
$|dv_2|^2$, $\Delta(v_4)$, $\Delta(v_2^2)$, $\delta(\Rho dv_2)$ and
$\Delta^2(v_2)$. Now, for the spaces $\s^p \times \h^q$, we have
\begin{multline*}
a_4(r) = (1-\lambda r^2)^{p-4} (1+\lambda r^2)^q a_4(P_2(\s^p)) \\ +
(1-\lambda r^2)^{p-2} (1+\lambda r^2)^{q-2} a_2(P_2(\s^p))
a_2(P_2(\h^q)) \\ + (1-\lambda r^2)^p (1+\lambda r^2)^{q-4}
a_4(P_2(\s^p))
\end{multline*}
with $\lambda = 1/4$. Formulas for the heat kernel coefficient $a_2$ and
$a_4$ of the conformal Laplacian on round spheres and hyperbolic spaces
were given in Section \ref{heat-spheres}. We use the resulting formula for
$360 a_{(4,2)}$ to represent this quantity as a linear combination of
$v_6$, $v_2 v_4$ and $v_2^3$ with quadratic coefficients in $n$. A
calculation shows that the general solution has the form
\begin{multline*}
(c_2 n^2 + c_1 n + 1152) v_6 \\ +
\left((96-c_2)n^2+(-1280-c_1+2/3c_2)n+(2624+2/3c_1)\right) v_2 v_4 \\
+ \left((-28+1/3c_2)n^2+(416+1/3c_1-1/3c_2)n+(-928- 1/3c_1) \right)
v_2^3
\end{multline*}
with arbitrary coefficients $c_1, c_2$. In particular, the solution
is unique once the coefficient of $v_6$ is known. Therefore, it
suffices to combine the result with those arguments in Section
\ref{a4-holo} which yield the coefficient of $v_6$.
\end{proof}

In the following section, we shall use an alternative method to
derive a fully explicit version of Proposition \ref{a42} for general
metrics.

\subsection{On the coefficient $a_4(r)$}\label{a4-holo}

We apply the method of Section \ref{structure-gen} to determine the first
two terms in the asymptotic expansion of $a_4(r)$. We continue to use the
notation of Section \ref{structure-gen}. The following result is a
consequence of Gilkey's formula for the heat kernel coefficient $a_4$ of
second-order operators of the form \eqref{L-gen}.

\begin{prop} The fourth-order heat kernel coefficient of the operator
$$
L = -(\Delta_g + g(d \eta, d) + b), \; \eta, b \in C^\infty(M)
$$
is given by
\begin{equation*}
360 a_4(L) = 2 |R|_g^2 - 2 |\Ric|_g^2 + 5 \scal^2 + 12 \Delta_g (\scal + 5
E) + 60 \scal E + 180 E^2,
\end{equation*}
where
$$
E = b - \frac{1}{4} |d\eta|_g^2 - \frac{1}{2} \Delta_g (\eta).
$$
\end{prop}

\begin{proof} By Theorem 4.18.16/c) in \cite{G-book}, the heat kernel
coefficient $a_4$ of an operator $L$ as in \eqref{L-gen} is given by
\begin{multline*}
360 a_4 = 2 |R|_g^2 - 2 |\Ric|_g^2 + 5 \scal^2 \\ + 12 \Delta_g (\scal) +
60 \tr_g(\nabla^2)(E) + 60 \scal E + 180 E^2 + 30 |\Omega|_g^2.
\end{multline*}
Here $\nabla$ and the endomorphism $E$ are given by \eqref{one-form-gen}
and \eqref{E-gen}. $\Omega$ is the curvature of $\nabla$. Now $\Omega$
vanishes in the present case, and for the endomorphism $E(r)$ of $\L$ we
have $\nabla (E(r)) = d E(r) \otimes \id$. Hence
$\tr_{g(r)}(\nabla^2)(E(r)) = \Delta_{g(r)}(E(r))$. This proves the claim.
\end{proof}

As a consequence, we obtain

\begin{corr}\label{a4r-gen} We have
\begin{multline*}
360 a_4(r) = \big(2 |R(r)|_{g(r)}^2 - 2 |\Ric(r)|_{g(r)}^2 + 5 \scal(r)^2
\\ + 12 \Delta_{g(r)}(\scal(r) + 5 (E(r)) + 60 \scal(r)
E(r) + 180 E(r)^2 \big) v(r)
\end{multline*}
with
$$
R(r) = R(g(r)), \quad \Ric(r) = \Ric(g(r)), \quad \scal(r) = \scal(g(r))
$$
and $E(r)$ as in Lemma \ref{E-final}.
\end{corr}

The latter result can be used to derive formulas for the coefficients of
the expansion of $a_4(r)$ in terms of $g$. In the following, we shall make
explicit the coefficient of $r^2$.

We start with the discussion in the special case of locally
conformally flat metrics. As before, we define $h(\rho)$ so that
$h(r^2) = g(r)$.

\begin{lemm}\label{R-expansion} Assume that $g$ is locally conformally
flat, i.e., $W=0$. Then
$$
R(r) = h \owedge \dot{h} - \frac{1}{2}(\dot{h} \owedge \dot{h}) r^2
$$
and
$$
\Ric(r) = -(n\!-\!2) \dot{h} - h \tr (h^{-1}\dot{h}) + (\dot{h} \tr
(h^{-1}\dot{h}) - \dot{h} h^{-1} \dot{h}) r^2.
$$
In these formulas, $\owedge$ is the Kulkarni-Nomizu product
\eqref{KN-product}, the dot denotes differentiation with respect to
$\rho$ and the dependence of $h$ on $\rho = r^2$ has been
suppressed.
\end{lemm}

\begin{proof} On the one hand, we have
$$
R(g_+) = W(g_+) - \Rho(g_+) \owedge g_+ = W(g_+) + \frac{1}{2} (g_+
\owedge g_+).
$$
But it is well-known that $W(g)=0$ implies $W(g_+)=0$ (see Chapter 7
of \cite{FG-final} or Section 6.14 of \cite{juhl-book}). Therefore,
we obtain
$$
R(g_+) = \frac{1}{2} (g_+ \owedge g_+).
$$
On the other hand, we recall that the curvature tensor $R$ of a
metric $g$ is given by
$$
R_{ijkl} = \frac{1}{2} \left[g_{ik,jl} + g_{jl,ik} - g_{jk,il} -
g_{il,jk} \right] + g_{pq} (\Gamma^p_{ik} \Gamma^q_{jl} -
\Gamma^p_{il} \Gamma^q_{jk}),
$$
where
$$
\Gamma_{ij}^k = \frac{1}{2} g^{kl} (g_{il,j} + g_{jl,i} - g_{ij,l}).
$$
It follows that the tangential components of $R(g_+)$ are given by
$$
R(g_+)_{ijkl} = \frac{1}{r^2} R(r)_{ijkl} + \frac{1}{r^2}
(\Gamma^0_{ik} \Gamma^0_{jl} - \Gamma^0_{il} \Gamma^0_{jk});
$$
here the upper index $0$ refers to $\partial/\partial r$. But using
$$
\Gamma_{ij}^0 = \frac{1}{r} g(r)_{ij} - \frac{1}{2} \dot{g}(r)_{ij},
$$
we obtain
\begin{multline*}
R(g_+)_{ijkl} = \frac{1}{r^2} R(r)_{ijkl} + \frac{1}{r^2}
\Big(\left(\frac{1}{r} g(r)_{ik} - \frac{1}{2} \dot{g}(r)_{ik}
\right) \left(\frac{1}{r} g(r)_{jl} - \frac{1}{2}
\dot{g}(r)_{jl}\right) \\ - \left(\frac{1}{r} g(r)_{il} -
\frac{1}{2} \dot{g}(r)_{il}\right) \left(\frac{1}{r} g(r)_{jk} -
\frac{1}{2} \dot{g}(r)_{jk} \right)\Big).
\end{multline*}
Now the first assertion follows by comparing both formulas for the
components $R(g_+)_{ijkl}$. The result for the Ricci tensor follows
by contraction. We omit the details.
\end{proof}

\begin{rem} Under the assumption $W=0$, we have $2 \ddot{h} = \dot{h}
h^{-1} \dot{h}$. This is a consequence of the explicit formula
$h(\rho) = g - \rho \Rho + \rho^2/4 \Rho^2$. This shows that the
formula for $\Ric(r)$ in Lemma \ref{R-expansion} is compatible with
\eqref{Ricci}.
\end{rem}

Lemma \ref{R-expansion} has the following consequences.

\begin{lemm}\label{curv-term} For locally conformally flat metrics,
we have the expansions
$$
|R(r)|^2_{g(r)} = |R|_g^2 + (4(n\!-\!4) \tr (\Rho^3) + (12 \J
|\Rho|^2) r^2 + \cdots
$$
and
$$
|\Ric(r)|^2_{g(r)} = |\Ric|^2_g + ((n\!-\!2)(n\!-\!4) \tr(\Rho^3) +
5(n\!-\!2) \J|\Rho|^2 + 2 \J^3) r^2 + \cdots.
$$
Moreover,
$$
|R|_g^2 = 4(n\!-\!2)|\Rho|^2 + 4 \J^2 \quad \mbox{and} \quad
|\Ric|_g^2 = (n\!-\!2)^2 |\Rho|^2 + (3n\!-\!4) \J^2.
$$
\end{lemm}

\begin{proof} In view of
$$
|R(r)|^2_{g(r)} = R(r)_{ijkl} R(r)_{abcd} g(r)^{ai} g(r)^{bj}
g(r)^{ck} g(r)^{dl},
$$
the coefficient of $r^2$ in the expansion of $|R(r)|^2_{g(r)}$ is
given by the sum
$$
2 g(\dot{R}(0),R) + 4 R_{ijkl} R_a{}^{jkl} \Rho^{ai},
$$
where the dot denotes differentiation with respect to $\rho$. First,
using $R = - \Rho \owedge g$, we find
$$
R_{ijkl} R_a{}^{jkl} \Rho^{ai} = 2(n\!-\!4) \tr(\Rho^3) + 6 \J
|\Rho|^2.
$$
Next, Lemma \ref{R-expansion} and $h(\rho) = g - \rho \Rho +
\rho^2/4 \Rho^2$ imply that
\begin{equation}\label{R-exp}
\dot{R}(0) = \frac{1}{2} (g \owedge \Rho^2 + \Rho \owedge \Rho).
\end{equation}
Therefore, the results
\begin{align*}
g(g \owedge \Rho^2,R) & = -4(n\!-\!2) \tr(\Rho^3) - 4 \J |\Rho|^2, \\
g(\Rho \owedge \Rho,R) & = 8 \tr(\Rho^3) - 8\J |\Rho|^2
\end{align*}
yield
$$
g(\dot{R}(0),R) = -2(n\!-\!4) \tr(\Rho^3) - 6 \J |\Rho|^2.
$$
Now the first claim follows by combining these formulas. Similarly,
we observe that
$$
|\Ric(r)|^2_{g(r)} = |\Ric|_g^2 + (2 g(\dot{\Ric}(0),\Ric) + 2
\Ric_{ij} \Ric_a^j \Rho^{ia}) r^2 + \cdots.
$$
Lemma \ref{R-expansion} implies that
\begin{equation}\label{Ric-exp}
\Ric(r) = \Ric + \frac{1}{2}(|\Rho|^2 g - n \Rho^2) r^2 + \cdots,
\end{equation}
i.e., $\dot{\Ric}(0) = \frac{1}{2} (|\Rho|^2 g - n \Rho^2)$. Now the
second claim follows by a direct calculation using $\Rho =
(n\!-\!2)\Ric + \J g$.
\end{proof}

The extension of Lemma \ref{curv-term} to general metrics requires
formulas for the first two terms of the expansions of $R(r)$ and
$\Ric(r)$ in terms of $g$. The following result provides that
information.

\begin{lemm}\label{R-general} The expansions of $R(r)$ and $\Ric(r)$
start with
$$
(W - g \owedge \Rho) + T r^2 + \cdots
$$
with
$$
T_{ijkl} = \frac{1}{2} (\nabla_k (\Co)_{lij} - \nabla_l (\Co)_{kij})
- \frac{1}{2} (R_i{}^a{}_{kl} \Rho_{aj} + R^a{}_{jkl} \Rho_{ai})
$$
and
$$
\Ric + \frac{1}{2} (|\Rho|^2 g - n \Rho^2 + \B) r^2 + \cdots.
$$
Here $\B$ and $\Co$ are the respective Bach and Cotton tensors
\eqref{cotton}. These expansions generalize \eqref{R-exp} and
\eqref{Ric-exp}.
\end{lemm}

\begin{proof} We recall the variational formula
\begin{multline*}
\dot{R}_{ijkl}[h] = (d/dt)|_0 (R(g+th)_{ijkl}) \\
= \frac{1}{2} ( \nabla_i \nabla_k h_{jl} + \nabla_j \nabla_l h_{ik}
- \nabla_i \nabla_l h_{jk} - \nabla_j \nabla_k h_{il} +
R_{ijk}{}^ah_{al} - R_{ijl}{}^a h_{ka})
\end{multline*}
(\cite{besse} Theorem 1.174). For $g(\rho) = g - \rho \Rho$, it
yields
$$
\dot{R}_{ijkl} = \frac{1}{2} (-\nabla_i \nabla_k \Rho_{jl} -
\nabla_j \nabla_l \Rho_{ik} + \nabla_i \nabla_l \Rho_{jk} + \nabla_j
\nabla_k \Rho_{il} - R_{ijk}{}^a \Rho_{al} + R_{ijl}{}^a \Rho_{ka}).
$$
By interchanging covariant derivatives, we find
\begin{multline*}
\dot{R}_{ijkl} = \frac{1}{2} (\nabla_k (\nabla_j \Rho_{il} -
\nabla_i \Rho_{jl}) + \nabla_l (\nabla_i \Rho_{jk} - \nabla_j
\Rho_{ik}) - R_{ijk}{}^a \Rho_{al} + R_{ijl}{}^a \Rho_{ak} \\ -
[\nabla_i,\nabla_k] \Rho_{jl} - [\nabla_j,\nabla_l] \Rho_{ik} +
[\nabla_j,\nabla_k] \Rho_{il} + [\nabla_i,\nabla_l] \Rho_{jk}).
\end{multline*}
This identity can be written as the sum of
\begin{equation*}
\dot{R}_{ijkl} = \frac{1}{2} (\nabla_k (\Co)_{lij} - \nabla_l
(\Co)_{kij}) + \frac{1}{2} (-R_{ijk}{}^a \Rho_{al} + R_{ijl}{}^a
\Rho_{ak})
\end{equation*}
and
\begin{align*}
& \frac{1}{2} (R_{ikj}{}^a \Rho_{al} + R_{ikl}{}^a \Rho_{aj}
+ R_{jli}{}^a \Rho_{ak} + R_{jlk}{}^a \Rho_{ai}) \\
& - \frac{1}{2} (R_{jki}{}^a \Rho_{al} + R_{jkl}{}^a \Rho_{ai} +
R_{ilj}{}^a \Rho_{ak} + R_{ilk}{}^a \Rho_{aj}).
\end{align*}
By the first Bianchi identity, the latter sum equals
$$
-\frac{1}{2} (R_{jik}{}^a \Rho_{al} + R_{kli} {} ^a \Rho_{aj} +
R_{ijl}{}^a \Rho_{ak} + R_{lkj}{}^{a} \Rho_{ai}).
$$
Summarizing yields
$$
\dot{R}_{ijkl} =  \frac{1}{2} (\nabla_k (\Co)_{lij} - \nabla_l
(\Co)_{kij}) - \frac{1}{2} (R_i{}^a{}_{kl} \Rho_{aj} + R^a{}_{jkl}
\Rho_{ai}).
$$
This proves the first assertion. Let $g_{(4)}$ be the coefficient of
$r^4$ in the expansion of $g(r)$. Then
$$
4(n-4) g_{(4)} = |\Rho|^2 g - 4 \Rho^2 - 2 \dot{\Ric}(0)
$$
by (6.9.14) of \cite{juhl-book} and
$$
4(n-4) g_{(4)} = (n-4) \Rho^2 - \B
$$
by \eqref{g-exp-4}. This proves the second assertion.
\end{proof}

The following result extends Lemma \ref{curv-term} to general
metrics.

\begin{lemm}\label{curv-term-gen} For general metrics, the coefficients
of $r^2$ in the expansions of the quantities $|R(r)|_{g(r)}^2$ and
$|\Ric(r)|_{g(r)}^2$ are
$$
4(n\!-\!4) \tr (\Rho^3) + 12 \J |\Rho|^2 + 2 \nabla_k(\Co)_{lij}
R^{ijkl} + 2 W_{ijkl} W^{ajkl} \Rho_a^i - 8 W_{ijkl}
\Rho^{jl}\Rho^{ik}
$$
and
$$
(n\!-\!2)(n\!-\!4) \tr(\Rho^3) + 5(n\!-\!2) \J|\Rho|^2 + 2 \J^3 +
(n\!-\!2)(\B,\Rho),
$$
respectively.
\end{lemm}

\begin{proof} It suffices to combine Lemma \ref{R-general} with the
arguments in the proof of Lemma \ref{curv-term}. Note also that
$$
(\Rho \owedge g)_{ijkl} (\Rho \owedge g)^{ajkl} \Rho_a^i = 2
(n\!-\!4) \tr (\Rho^3) + 6 \J |\Rho|^2.
$$
The proof is complete.
\end{proof}

Lemma \ref{R-expansion} (or \eqref{scal}) implies that the expansion
of $\scal(r)$ starts with
\begin{equation}\label{scal-expansion}
\scal(r) = \scal + ((n\!-\!2) |\Rho|^2 + \J^2) r^2 + \cdots.
\end{equation}

Next, we expand the potential $E(r)$. Lemma \ref{E-final} yields

\begin{lemm}\label{E-expand}
$$
E(r) = (n\!-\!2) v_2 + (2(n\!-\!4) v_4 - (n\!-\!3) v_2^2) r^2 +
\cdots \st E_0 + E_2 r^2 + \cdots.
$$
\end{lemm}

Now, by Corollary \ref{a4r-gen}, $360 a_{(4,2)}$ equals the
coefficient of $r^2$ in the expansion of the sum
\begin{multline}\label{exp-sum}
\big(2 |R(r)|_{g(r)}^2 - 2 |\Ric(r)|_{g(r)}^2 + 5 \scal(r)^2 \\
+ 12 \Delta_{g(r)}(\scal(r) + 5 E(r)) + 60 \scal(r) E(r) + 180 E(r)^2
\big) v(r).
\end{multline}
We use the above results to make the latter sum explicit. First, in the
conformally flat case, we find that
\begin{multline*}
(2|R(r)|_{g(r)}^2 - 2 |\Ric(r)|_{g(r)}^2 + 5 \scal(r)^2) v(r) \\
= (-2(n\!-\!6)(n\!-\!2)|\Rho|^2 + 2(10n^2\!-\!23n\!+\!18) \J^2) \\ +
(-2(n\!-\!6)(n\!-\!4) \tr(\Rho^3) + 3(7n^2\!-\!26n\!+\!32) \J
|\Rho|^2 - (10n^2\!-\!43n\!+\!42)\J^3) r^2 + \cdots.
\end{multline*}
Similarly, we obtain
\begin{multline*}
(60 \scal(r) E(r) + 180 E(r)^2) v(r) \\
= -15 (n^2\!-\!4)\J^2 + \frac{15}{2} ((n^2\!-\!2n\!-\!4) \J^3 -
2(n^2\!-\!8) \J|\Rho|^2) r^2 + \cdots.
\end{multline*}
These results imply that the coefficient of $r^2$ in \eqref{exp-sum}
is the sum of
\begin{multline}\label{v-contribution}
-2(n\!-\!4)(n\!-\!6) \tr(\Rho^3) + 6 (n\!-\!4)(n\!-\!9) \J
|\Rho|^2 - \frac{1}{2} (n\!-\!4)(5n\!-\!36) \J^3 \\
= 48 (n\!-\!4)(n\!-\!6) v_6 + 48 (n\!-\!4)(n\!-\!12) v_2 v_4 - 12
(n\!-\!4)(n\!-\!16) v_2^3
\end{multline}
and the coefficient of $r^2$ in the expansion of
\begin{equation}\label{pot-terms}
12 \Delta_{g(r)}(\scal(r) + 5 E(r)) v(r).
\end{equation}
We continue with an analysis of the contribution \eqref{pot-terms}. Lemma
\ref{conjugate} and \eqref{scal-expansion} yield the expansion
\begin{align*}
\Delta_{g(r)}(\scal(r)) v(r) & = - \delta (g(r)^{-1} v(r)
d(\scal(r))) \nonumber \\ & = \Delta(\scal) + \left(
\Delta((n\!-\!2) |\Rho|^2 + \J^2) - \delta ((\Rho + v_2) d(\scal))
\right) r^2 + \cdots,
\end{align*}
where $\delta$ and $\Delta$ are defined by $g$. Note that
$$
\scal = -4(n\!-\!1) v_2 \quad \mbox{and} \quad (n\!-\!2) |\Rho|^2 +
\J^2 = -8(n\!-\!2) v_4 + 4(n\!-\!1) v_2^2.
$$
Similarly, Lemma \ref{conjugate} implies that the expansion of the term
$\Delta_{g(r)}(E(r)) v(r)$ starts with
$$
\Delta (E_0) + \left(\Delta (E_2) - \delta((\Rho + v_2) dE_0)
\right)r^2 + \cdots.
$$
Hence, by
$$
E_0 = (n\!-\!2)v_2 \quad \mbox{and} \quad E_2 = 2(n\!-\!4) v_4 -
(n\!-\!3)v_2^2
$$
(see Lemma \ref{E-expand}), the coefficient of $r^2$ can be written
in the form
$$
2(n\!-\!4) \Delta (v_4) - (n\!-\!3) \Delta (v_2^2) - (n\!-\!2)
\delta((\Rho + v_2) dv_2).
$$
These results show that the coefficient of $r^2$ in the expansion of
\eqref{pot-terms} is given by the sum
\begin{equation*}
24(n\!-\!12) \Delta (v_4) - 12(n\!-\!11) \Delta (v_2^2) -
12(n\!-\!6) \delta((\Rho + v_2)dv_2).
\end{equation*}
Simplifying and summarizing, we have proved the following result.

\begin{prop}\label{a42-final} For locally conformally flat $(M^n,g)$
of dimension $n \ge 3$, we have
\begin{multline}\label{a42-closed}
360 a_{(4,2)} = 48 (n\!-\!4)(n\!-\!6) v_6 + 48 (n\!-\!4)(n\!-\!12)
v_2 v_4 - 12 (n\!-\!4)(n\!-\!16) v_2^3 \\ + 24 (n\!-\!12)
\Delta(v_4) - 6(n\!-\!16) \Delta(v_2^2) - 12 (n\!-\!6) \delta (\Rho
d v_2).
\end{multline}
Here we use the convention $\Delta = -\delta d$.
\end{prop}

Note that Proposition \ref{a42-final} improves Proposition
\ref{a42}. For general metrics, the corresponding formula for
$a_{(4,2)}$ contains a few more terms. In fact, we find

\begin{prop}\label{a42-gen} For general manifolds $(M^n,g)$ of dimension
$n \ge 3$, the formula for $a_{(4,2)}$ contains the additional terms
\begin{equation*}
4 W_{ijkl} W^{ajkl} \Rho_a^i - 16 W_{ijkl} \Rho^{ik} \Rho^{jl} -
\J |W|^2 + 4 \nabla_k (\Co)_{lij} R^{ijkl} - 8 (\B,\Rho).
\end{equation*}
\end{prop}

\begin{proof} The additional terms are contained in the
contributions
$$
(2|R(r)|_{g(r)}^2 - 2|\Ric(r)|_{g(r)}^2) v(r).
$$
These are described in Lemma \ref{curv-term-gen}. We omit the
details of the calculation and note only that the resulting
contribution $-2(n\!-\!2)(\B,\Rho)$ splits as the sum of
$$
-2(n\!-\!6) (\B,\Rho) \quad \quad \mbox{(from $48(n\!-\!4)(n\!-\!6)
v_6$)}
$$
and $-8(\B,\Rho)$.
\end{proof}

\begin{rem}\label{a42-Einstein} Let $g$ be Einstein. Then $\Rho = 2 c g$
and $v(r) = (1-cr^2)^n$ for some constant $c$. In this case,
Proposition \ref{a42-final} yields
$$
360 a_{(4,2)} = - 4 c^3 n(n\!-\!4)(n\!-\!6) (5n^2\!-\!18n\!+\!4) -
2c(n\!-\!4)|W|^2 = -c(n\!-\!4) 360 a_4.
$$
This relation fits with the consequence $a_{(4,2)} = - c (n\!-\!4)
a_4$ of \eqref{hc-einstein}.
\end{rem}

\begin{rem}\label{byproduct} As a byproduct of the calculations, we find
that the coefficient $ 360 a_4 = 360 a_{(4,0)} $ equals the sum
\begin{multline*}
(-2(n\!-\!6)(n\!-\!2)|\Rho|^2 + 2(10n^2\!-\!23n\!+\!18) \J^2) - 15
(n^2\!-\!4) \J^2 \\ + 12 (5 \Delta (E_0) + \Delta (\scal)) + 2
|W|^2.
\end{multline*}
This formula implies the identity \eqref{CL-a4} for $a_4$ of the
conformal Laplacian.
\end{rem}

Next we use Proposition \ref{a42-final} to derive a closed formula
for the coefficient $c_{(4,2)}$ for locally conformally flat
metrics. We start by determining the conformal variation of the
total integral of $a_{(4,2)}$. Let
$$
360 a_{(4,2)}' \st 48(n\!-\!4)(n\!-\!6) v_6 + 48(n\!-\!4)(n\!-\!12)
v_2 v_4 - 12 (n\!-\!4)(n\!-\!16) v_2^3.
$$
Then
$$
360 \left(\int_M a_{(4,2)} dvol \right)^\bullet[\varphi] = 360
\left(\int_M a_{(4,2)}' dvol \right)^\bullet[\varphi]
$$
by Proposition \ref{a42-final}. Now partial integration and the
formulas in the proof of Proposition \ref{a6-div} show that the
latter variation is given by the sum of the scaling term
$$
360 (n\!-\!6) \int_M \varphi a_{(4,2)}' dvol
$$
and
\begin{multline*}
\int_M \varphi \Big( 24(n\!-\!4)(n\!-\!12) \Delta(v_4) - 24
(n\!-\!4)(n\!-\!12) \delta(v_2+1/2 \Rho)dv_2 \\ - 18
(n\!-\!4)(n\!-\!16) \Delta(v_2^2) \Big) dvol,
\end{multline*}
where $\Delta = -\delta d$. On the other hand, \eqref{CV-ac} implies
$$
360 \left(\int_M a_{(4,2)} dvol \right)^\bullet[\varphi] = 360
(n\!-\!6) \int_M \varphi a_{(4,2)} dvol - 180 \int_M \varphi
c_{(4,2)} dvol.
$$
These identities hold true for arbitrary $\varphi \in C^\infty(M)$.
Hence a comparison using Proposition \ref{a42-final} shows that
\begin{equation}\label{c42}
-180 c_{(4,2)} = 48 (n\!-\!12) \Delta (v_4) + 48 (n\!-\!3) \delta
(\Rho dv_2) + 36 n \Delta (v_2^2).
\end{equation}
But
$$
\tilde{c}_{(4,2)} = 240 v_2^3 - 576 v_2 v_4 - 48 \Delta v_4 - 36
\Delta v_2^2 - 48 \delta(\Rho d v_2)
$$
is a conformal primitive of $-180 c_{(4,2)}$ in the sense that
$$
\left(\int_M \tilde{c}_{(4,2)} dv \right)^\bullet[\varphi] = (n-6)
\int_M \varphi \tilde{c}_{(4,2)} dv - 180 \int_M \varphi c_{(4,2)}
dv.
$$
Hence the modified heat kernel coefficient
$$
a_{(4,2)}' \st 360 a_{(4,2)} - \tilde{c}_{(4,2)}
$$
satisfies
$$
\left(\int_M a_{(4,2)}' dv \right)^\bullet[\varphi] = (n-6) \int_M
\varphi a_{(4,2)}' dv.
$$
A calculation shows that
\begin{multline*}
a_{(4,2)}' = 48 (n\!-\!6)(n\!-\!4) v_6 + 48 (n\!-\!6)(n\!-\!10) v_2
v_4 - 12 (n\!-\!6) (n\!-\!14) v_2^3 \\ + 24 (n\!-\!10) \Delta v_4 -
6 (n\!-\!22) \Delta v_2^2 - 12 (n\!-\!10) \delta(\Rho d v_2).
\end{multline*}

Finally, we discuss extremal properties of the functional
\begin{equation}
\A_{(4,2)} \st \int_{M^6} a_{(4,2)} \dvol
\end{equation}
for locally conformally flat metrics in the critical dimension
$n=6$.

\begin{prop}\label{a42-extreme} Let $M^6$ be closed with a locally
conformally flat Einstein metric $g$.\footnote{Thus $g$ is a metric
of constant sectional curvature, i.e., $(M^6,g)$ is a closed space
form.} Then the restriction of the scale-invariant functional
$\A_{(4,2)}$ to the conformal class $c = [g]$ has a local maximum at
$g$. The local maximum is strict (modulo rescalings) unless
$(M^6,g)$ is isometric to a (rescaled) round sphere.
\end{prop}

\begin{proof} We argue as in the proof of Theorem
\ref{extremal-RVC} using a formula for the second conformal
variation. Let $g$ be locally conformally flat Einstein and
$c=\scal(g)/(4n(n-1)) = \scal(g)/120$. Let $\gamma(t) =
e^{2\varphi(t)}$ be a curve in $c$ so that $\varphi(0) = 0$ and
$\varphi'(0) = \varphi$. Then the second conformal variation
$$
(\A_{(4,2)}(g))^{\bullet\bullet}[\gamma] = (\partial^2/\partial
t^2)|_0 (\A_{(4,2)}(\gamma(t)))
$$
of $\A_{(4,2)}$ at $g$ does not depend on $\psi = \varphi''(0)$ and
is given by the formula
\begin{equation}\label{scv-a42}
(\A_{(4,2)})^{\bullet \bullet}[\varphi] = -2 \int_{M^6} \varphi
\Delta (\Delta+24c) (\varphi) \dvol = -2 \int_{M^6} \varphi \Delta
(\Delta+4nc) (\varphi) \dvol
\end{equation}
at $g$. Now the spectrum of $\Delta = -\delta d$ consists only of
non-positive eigenvalues. If $c \le 0$, then it follows that the quadratic
form on the left-hand side of \eqref{scv-a42} is negative semi-definite
with kernel $\c$. On the other hand, if $c>0$, then the first non-trivial
eigenvalue $\lambda_1$ of $-\Delta$ satisfies $\lambda_1 \ge
\scal(g)/(n\!-\!1) = 24c$ \cite{obata}. In this case, the quadratic form
on the right-hand side of \eqref{scv-a42} is negative semi-definite with
kernel spanned by $\c$ and the eigenspace of the eigenvalue
$\scal(g)/(n-1)$ (being non-trivial only for rescaled spheres). In the
remaining part of the proof we establish the formula \eqref{scv-a42}.
First, we observe that, in general dimensions $n$ and for locally
conformally flat $g$, Proposition \ref{a42-final} shows that
$$
\int_{M^n} a_{(4,2)} \dvol = \int_{M^n} \bar{a}_{(4,2)} \dvol,
$$
where $360 \bar{a}_{(4,2)}$ is defined as the sum
$$
48(n\!-\!4)(n\!-\!6) v_6 + 48(n\!-\!4)(n\!-\!12) v_2 v_4 - 12
(n\!-\!4)(n\!-\!16) v_2^3.
$$
By partial integration and the formulas in the proof of Proposition
\ref{a6-div}, it follows that the conformal variation of $360 \A_{(4,2)}$
at a locally conformally flat metric $g$ is given by the sum of the
scaling term
$$
360 (n\!-\!6) \int_{M^n} \varphi \bar{a}_{(4,2)} \dvol
$$
and
\begin{multline*}
\int_{M^n} \varphi \Big( 24(n\!-\!4)(n\!-\!12) \Delta(v_4) - 24
(n\!-\!4)(n\!-\!12) \delta(v_2\!+\!1/2 \Rho)dv_2 \\ - 18
(n\!-\!4)(n\!-\!16) \Delta(v_2^2) \Big) \dvol
\end{multline*}
at $g$. Note that the latter integral vanishes at Einstein metrics. This
follows from the fact that, for Einstein metrics, $v_2$ and $v_4$ are
constant, and that $\Rho$ is a constant multiple of the metric. This
observation confirms that the functional $\A_{(4,2)}$ is critical at
Einstein metrics in dimension $n=6$. Now, for $n=6$ and for any locally
conformally flat metric $g$, the above results show that
\begin{multline}\label{fcv-a42}
360 \left(\int_{M^6} a_{(4,2)} \dvol \right)^\bullet[\varphi] \\
= \int_{M^6} \varphi (-288 \Delta (v_4) + 288 \delta (v_2\!+\!1/2 \Rho)
dv_2 + 360 \Delta(v_2^2) ) \dvol
\end{multline}
at $g$. We use this formula to prove \eqref{scv-a42} at the Einstein
metric $g$. Note that
$$
v_2=-6c, \quad v_4=15c^2, \quad v_2 g + 1/2 \Rho = -5gc \quad
\mbox{at $g$}.
$$
Now, by arguments as in the proof of Theorem \ref{extremal-RVC}, the
second conformal variation of $360 \A_{(4,2)}$ at the Einstein
metric $g$ equals
$$
\int_{M^6} \varphi \Delta (-288 v_4^\bullet[\varphi] + 1440 c
v_2^\bullet[\varphi] + 360 (v_2^2)^\bullet[\varphi]) \dvol;
$$
the independence on $\psi = \varphi''(0)$ follows from the vanishing of
the integrand in \eqref{fcv-a42} if $g$ is Einstein. A calculation using
the formulas for the conformal variation of $v_2$ and $v_4$ (displayed in
the proof of Proposition \ref{a6-div}) shows that the latter sum
simplifies to
$$
-720 \int_{M^6} \varphi \Delta (\Delta + 24c) (\varphi) \dvol.
$$
This completes the proof.
\end{proof}



\begin{thebibliography}{AJPS97}

\bibitem[AJPS97]{AJPS}
A.~Albeverio, J.~Jost, S.~Paycha and S.~Scarlatti, {\em A
Mathematical Introduction to String Theory}, London Mathematical
Society Lecture Note Series {\bf 225}, 1997.
\bibitem[A09]{albin}
P.~Albin, Renormalizing curvature integrals on Poincar\'e-Einstein
manifolds, {\em Advances in Math.} {\bf 221}, 1 (2009), 140--169.
\bibitem[A12]{alex}
S.~Alexakis, {\em The Decomposition of Global Conformal Invariants},
Annals of Mathematics Studies {\bf 182}, Princeton University Press,
2012, {x+449}.
\bibitem[A00]{A-book}
I.~Avramidi, {\em Heat Kernels and Quantum Gravity}, Lecture Notes
in Physics Monographs {\bf M64}, 2000.
\bibitem[A04]{A-GG}
I.~Avramidi, Gauged gravity via spectral asymptotics of non-Laplace
type operators, {\em J. High Energy Phys.} {\bf 7}, 030, 2004
(electronic). \url{arXiv:hep-th/0406026}
\bibitem[AB01]{AB}
I.~Avramidi and T.~Branson, Heat kernel asymptotics of operators
with non-Laplace principal part, {\em Rev. Math. Phys.} {\bf 13}, 7
(2001), 847--890. \url{arXiv:math-ph/9905001}
\bibitem[BJ10]{BJ}
H.~Baum and A.~Juhl, {\em Conformal Differential Geometry:
$Q$-curvature and Conformal Holonomy}. Oberwolfach Seminars {\bf
40}, 2010.
\bibitem[BFT00]{BFT}
F.~Bastianelli, S.~Frolov and A.~Tseytlin, Conformal anomaly of
{$(2,0)$} tensor multiplet in six dimensions and
AdS/CFT-correspondence, {\em J. High Energy Phys.} {\bf 2}, 13,
2000. \url{arXiv:hep-th/0001041}
\bibitem[BGM71]{BGM}
M.~Berger, P.~Gauduchon and E.~Mazet, {\em Le Spectre d'une
Variet\'e Riemannienne}. Lecture Notes in Mathematics {\bf 194}
(1971).
\bibitem[B87]{besse}
A.~Besse, {\em Einstein manifolds}, Ergebnisse der Mathematik und
ihrer Grenzgebiete {\bf 110}, Springer, 1987.
\bibitem[BGV92]{BGV}
N.~Berline, E.~Getzler and M.~Vergne, {\em Heat Kernels and Dirac
Operators}, Grundlehren {\bf 298}, Springer, 1992.
\bibitem[BD82]{BD}
N.~D.~Birrell and P.~C.~W.~Davies, {\em Quantum fields in curved
space}, Cambridge Monographs on Mathematical Physics {\bf 7},
Cambridge University Press, {1982}.
\bibitem[B95]{sharp}
T.~P.~Branson, Sharp inequalities, the functional determinant, and
the complementary series, {\em Trans. AMS} {\bf 347}, (1995),
3671--3742.
\bibitem[B96]{spec-ineq}
T.~Branson, Spectral theory of invariant operators, sharp
inequalities and representation theory. {\em Rend. Circ. Mat.
Palermo}, Serie II, Suppl. {\bf 46}, (1997), 29--54.
\bibitem[B04]{B-spec}
T.~Branson, $Q$-curvature and spectral invariants, {\em Rend. Circ.
Mat. Palermo}, Serie II, Suppl. {\bf 75}, (2005), 11--55.
\bibitem[B07]{sigma}
T.~Branson, $Q$-curvature, spectral invariants, and representation
theory. {\em SIGMA Symmetry Integrability  Geom. Methods Appl.} {\bf
3} (2007), paper 090, 31p.
\bibitem[BO86]{BO-index}
T.~Branson and B.~{\O}rsted, Conformal indices of Riemannian manifolds,
{\em Compositio Math.} {\bf 60}, 3, (1986), 261--293.
\bibitem[BO88]{BO-def}
T.~Branson and B.~{\O}rsted, Conformal deformation and the heat
operator, {\em Indiana University Math. Journal} {\bf 37}, 1,
(2008), 83--110.
\bibitem[BGO90]{BGO}
T.~Branson, P.~Gilkey and B.~{\O}rsted, Leading terms in the heat
invariants, {\em Proceedings AMS} {\bf 209}, 2, (1990), 437--450.
\bibitem[BG08]{origin}
T.~P.~Branson and A.~R.~Gover. Origins, applications and
generalisations of the {$Q$}-curvature. {\em Acta Appl. Math.} {\bf
102} no. 2-3 (2008), 131--146.
\bibitem[CW76]{CW}
R.~Cahn and J.~Wolf, Zeta functions and their asymptotic axpansions
for compact symmetric spaces of rank one, {\em Comment. Math.
Helvetici} {\bf 5}, (1976), 1--21.
\bibitem[C90]{camp}
R.~Camporesi, Harmonic analysis and propagators on homogeneous
spaces. {\em Physics Reports} {\bf 196}, No. 1-2, (1990), 1--134.
\bibitem[CEOY08]{what}
S.~-Y.~A.~Chang, M.~Eastwood, B.~{\O}rsted and P.~C.~Yang, What is
$Q$-curvature? {\em Acta Appl. Math.} {\bf 102}, (2008), 119–-125.
\bibitem[CF08]{CF}
S.-Y.~A.~Chang and H.~Fang, A class of variational functionals in
conformal geometry, {\em Intern. Math. Research Notices}, (2008),
008. \url{arXiv:0803.0333}
\bibitem[CFG12]{CFG}
S.-Y.~A.~Chang, H.~Fang and C.~R.~Graham, A note on renormalized volume
functionals. {\em Differential Geom. Appl.} (2014), {\bf 33}, Supplement,
246--258. \url{arXiv:1211.6422}
\bibitem[CY95]{CY}
S.-Y.~A.~Chang and P.~Yang, Extremal metrics of zeta function determinants
on 4-manifolds, {\em Annals of Mathematics} {\bf 142}, (1995), 171--212.
\bibitem[Ch84]{Ch}
I.~Chavel, {\em Eigenvalues in Riemannian Geometry}, Academic Press,
1984.
\bibitem[DW64]{dewitt}
B.~S.~DeWitt, Dynamical theory of groups and fields, in {\em
Relativit\'e, Groupes et Topologie (Lectures, Les Houches, 1963
Summer School of Theoret. Phys., Univ. Grenoble)}, {Gordon and
Breach}, {1964}, {585--820}.
\bibitem[DG75]{DG}
J.~J.~Duistermaat and V.~W.~Guillemin, The spectrum of positive
elliptic operators and periodic bicharacteristics, {\em Invent.
Math.} {\bf 29}, (1975), 39--79.
\bibitem[ES85]{ES}
M.~Eastwood and M.~Singer. A conformally invariant Maxwell gauge.
{\em Physics Letters} {\bf 107A}, 2, (1985), 73--74.
\bibitem[E97]{Erd}
J.~Erdmenger, Conformally covariant differential operators:
properties and applications. {\em Class. Quantum Grav.} {\bf 14}
(1997), 2061. \url{arXiv:hep-th/9704108}
\bibitem[FG85]{FG-Cartan}
C.~Fefferman and C.~R. Graham. Conformal Invariants. In 'Elie Cartan
et les Mathematiques d'Adjourd'hui.' {\em Ast\'erisque} Numero Hors
Serie (1985), 95--116.
\bibitem[FG02]{FG-P}
C.~Fefferman and C.~R.~Graham, $Q$-curvature and Poincar\'e metrics.
Math. Res. Lett. {\bf 9}, 2-3, (2002), 139--151.
\bibitem[FG12]{FG-final}
C.~Fefferman and C.~R.~Graham, {\em The Ambient Metric}. Annals of
Math. Studies {\bf 178}, Princeton University Press, 2012.
\url{arXiv:0710.0919}
\bibitem[FG12]{FG-J}
C.~Feffermann and C.~R.~Graham, Juhl's formulae for GJMS-operators
and $Q$-curvatures. {\em Journal Amer. Math. Soc.} {\bf 26}, 4,
(2013), 1191--1207. \url{arXiv:1203.0360}
\bibitem[Fo08]{Fo}
G.~B.~Folland, {\em Quantum Field Theory. A Tourist Guide for
Mathematicians}, Surveys and Monographs {\bf 149}, AMS, 2008.
\bibitem[FT82]{FT}
E.~S.~Fradkin and A.~A.~Tseytlin, Asymptotic freedom in extended
conformal supergravity. {\em Physics Letters} {\bf 110B}, 2, (1982),
117--121.
\bibitem[G75]{G-spec}
P.~Gilkey, The spectral geometry of a Riemannian manifold, {\em J.
Diff. Geom.} {\bf 10}, 4, (1975), 601--618.
\bibitem[G84]{G-book}
P.~Gilkey, {\em Invariance Theory, the Heat Equation, and the
Atiyah-Singer Index Theorem}, Publish of Perish Inc., 1984.
\bibitem[G01]{G-GP}
P.~Gilkey, {\em Geometric Properties of Natural Operators Defined by
the Riemann Curvature Tensor}. World Scientific, 2001.
\bibitem[GPS05]{GPS}
C.~Gordon, P.~Perry and D.~Sch\"{u}th, Isospectral and isoscattering
manifolds: a survey of techniques and examples, {\em Contemp.
Math.}, {\bf 387}, Amer. Math. Soc. (2005), 157--179.
\bibitem[G06]{gover-product}
A.~R.~Gover, Laplacian operators and $Q$-curvature on conformally
Einstein manifolds, {\em Math. Ann.} {\bf 336}, 2, (2006), 2,
311--334. \url{arXiv:math/0506037}
\bibitem[GL09]{GL}
A.~R.~Gover and P.~Leitner, A sub-product construction of
Poincar\'e-Einstein metrics, {\em Intern. J. Math.} {\bf 20},
(2009), 1263. \url{arXiv:math/0608044}
\bibitem[G00]{G-vol}
C.~R.~Graham, Volume and area renormalizations for conformally
compact Einstein metrics, {\em Rend. Circ. Mat. Palermo (2) Suppl.}
{\bf 63}, (2000), 31--42. \url{arXiv:math/9909042}
\bibitem[G09]{G-ext}
C.~R.~Graham, Extended obstruction tensors and renormalized volume
coefficients, {\em Advances in Math.} {\bf 220}, 6, (2009),
1956--1985. \url{arXiv:0810.4203}
\bibitem[GJMS92]{GJMS}
C.~R.~Graham, R.~Jenne, L.~J.~Mason and G.~A.~J.~Sparling,
Conformally invariant powers of the Laplacian. {I}. Existence. {\em
J. London Math. Soc.} {\bf 46}, 2, (1992), 557--565.
\bibitem[GJ07]{GJ-holo}
C.~R.~Graham and A.~Juhl, Holographic formula for $Q$-curvature,
{\em Advances in Math.} {\bf 216}, 2, (2007), 841--853.
\url{arXiv:0704.1673}
\bibitem[GZ03]{GZ}
C.~R.~Graham and M.~Zworski, Scattering matrix in conformal
geometry, {\em Invent. Math.} {\bf 152}, 1, (2003), 89--118.
\url{arXiv:math/0109089}
\bibitem[GSZ11]{GSZ}
K.~Groh, F.~Saueressig and and O.~Zanusso, Off-diagonal heat-kernel
expansion and its application to fields with differential
constraints. \url{arXiv:1112.4856}
\bibitem[GMS12]{GMS}
C.~Guillarmou, S.~Moroianu and J.-M.~Schlencker, The renormalized
volume and uniformization of conformal strctures.
\url{arXiv:1211.6705}
\bibitem[GL11]{GuLi}
B.~Guo and H.~Li, The second variational formula for the functional
$\int_M v^{(6)}(g) dV_g$, {\em Proc. Amer. Math. Soc.} {\bf 139},
(2011) 2911--2925.
\bibitem[H77]{H}
S.~Hawking, Zeta function regularization of path integrals in curved
spacetime, {\em Comm. Math. Phys.} {\bf 55}, (1977), 133--148.
\bibitem[HS98]{HS}
M.~Henningson and K.~Skenderis, The holographic Weyl anomaly. {\em
JHEP} {\bf 7} paper 23 (electronic), 1998.
\url{arXiv:hep-th/9806087}
\bibitem[ISTY00]{ISTY}
C.~Imbimbo, A.~Schwimmer, S.~Theisen and S.~Yankielowicz,
Diffeomorphisms and holographic anomalies, {\em Classical Quantum
Gravity} {\bf 17}, 5, (2000), 1129--1138. \url{arXiv:hep-th/9910267}
\bibitem[J09a]{juhl-book}
A.~Juhl, {\em Families of Conformally Covariant Differential
Operators, $Q$-Curvature and Holography}, volume {\bf 275} of
Progress in Mathematics. Birkh\"auser Verlag, 2009.
\bibitem[J09b]{juhl-power}
A.~Juhl, On conformally covariant powers of the Laplacian.
\url{arXiv:0905.3992}
\bibitem [J11]{holo-II}
A.~Juhl, Holographic formula for $Q$-curvature. II, {\em Advances in
Math.} {\bf 226}, 4, (2011), 3409--3425. \url{arXiv:1003.3989}
\bibitem[J13]{juhl-ex}
A.~Juhl, Explicit formulas for GJMS-operators and $Q$-curvatures.
{\em Geom. Funct. Anal.} {\bf 23}, (2013), 1278--1370.
\url{arXiv:1108.0273}
\bibitem [J14]{Q-recursive}
A.~Juhl, On the recursive structure of Branson's $Q$-curvatures,
{\em Math. Res. Lett.} {\bf 21}, 3, (2014), 1--13.
\url{arXiv:1004.1784}.
\bibitem [JK14]{JK}
A.~Juhl and C.~Krattentahler, Summation formulas for GJMS-operators
and $Q$-curvatures on the M\"{o}bius sphere, {\em J. Approx. Theory},
2014. \url{arXiv:0910.4840}
\bibitem[MP49]{MP}
S.~Minakshisundaram and A.~Pleijel, Some properties of the
eigenfunctions of the Laplace-operator on Riemannian manifolds, {\em
Canadian Journal of Mathematics} {\bf 1}, (1949), 242--256.
\bibitem[McS67]{MC-S}
H.P.~McKean and I.~M.~Singer, Curvature and the eigenvalues of the
Laplacian, {\em J. Diff. Geom.} {\bf 1}, (1967), 43--69.
\bibitem[MO09]{MO}
N.-M.~Moller and B.~{\O}rsted, Rigidity of conformal functionals on
spheres. {\em Int. Math. Res. Not.} {\bf 24} (2013).
\url{arXiv:0902.4067}
\bibitem[O62]{obata}
M.~Obata, Certain conditions for a Riemannian manifold to be
isometric to a sphere, {\em J. Math. Soc. Japan} {\bf 14}, 3,
(1962), 333--340.
\bibitem[OPS88]{OPS}
B.~Osgood, R.~Phillips and P.~Sarnak, Extremals of determinants of
Laplacians, {\em J. Funct. Anal.} {\bf 80}, 1, (1988), 148--211.
\bibitem[P08]{pan}
S.~Paneitz, A quartic conformally covariant differential operator
for arbitrary pseudo-Riemannian manifolds (summary), {\em SIGMA}
{\bf 4} (2008), paper 036, 3p. \url{arXiv:0803.4331}
\bibitem[PR87]{PR}
T.~Parker and S.~Rosenberg, Invariants of conformal {L}aplacians,
{\em J. Diff. Geom.} {\bf 25}, 2, (1987), 199--222.
\bibitem[P81]{Po}
A.~M. Polyakov. Quantum geometry of bosonic strings. {\em Phys.
Lett. B} {\bf 103}, 3, (1981), 207--210.
\bibitem[P13]{log}
R.~Ponge, The logarithmic singularities of the Green functions of
the conformal powers of the Laplacian. \url{arXiv:1306.3104}
\bibitem[RS71]{RS}
D.~B.~Ray and I.~M.~Singer, R-torsion and the Laplacian on
Riemannian manifolds, {\em Advances in Math.} {\bf 7}, (1971)
145--210.
\bibitem[R84]{Rieg}
R.~Riegert, A nonlocal action for the trace anomaly. {\em Phys.
Lett. B} {\bf 134}, no. 1--2, (1984), 56--60.
\bibitem[R97]{rosen-zeta}
S.~Rosenberg, The variation of the de Rham zeta function, {\em
Trans. AMS} {\bf 299}, 2, (1987), 535--557.
\bibitem[S01]{shubin}
M.~A.~Shubin, {\em Pseudodifferential operators and spectral
theory}, second edition, Springer 2001.
\bibitem[V03]{V}
D.~V.~Vassilevich, Heat kernel expansion: user's manual, {\em Phys.
Rep.} {\bf 388}, (2003), 279--360. \url{arXiv:hep-th/0306138}
\bibitem[V00]{viac}
J.~Viaclovsky, Conformal geometry, contact geometry, and the
calculus of variations, {\em Duke Math. J.} {\bf 101}, 2, {2000},
283--316.
\bibitem[W04]{Wein}
G.~Weingart, Combinatorics of heat kernel coefficients.
Habilitation, University Bonn, 2004.
\bibitem[W98]{witt}
E.~Witten. Anti de Sitter space and holography. {\em Adv. Theor.
Math. Phys.} {\bf 2}, 2, (1998), 253--291.\url{arXiv:hep-th/9802150}
\bibitem[Zee10]{zee}
A.~Zee, {\em Quantum Field Theory in a Nutshell}, Princeton
University Press, 2010.
\bibitem[Zeld07]{Z-inverse}
S.~Zelditch, The inverse spectral problem, Surveys in differential
geometry. 401--467. With an appendix by Johannes Sj\"ostrand and
Maciej Zworski. Int. Press, Somerville, MA, 2004.
\end{thebibliography}
\end{document}